\documentclass[a4paper,english,10pt]{article}
\usepackage[left=.9in, right=.9in, bottom=1.2in, top=1in]{geometry}
\makeatletter
\AtBeginDocument{\setlength\abovedisplayskip{5pt plus 2pt minus 1pt}
  \setlength\belowdisplayskip{5pt plus 2pt minus 1pt}}

\usepackage{eKN}

\usepackage{dsfont}
\usepackage[normalem]{ulem} 
\newcommand{\myuline}[1]{\uline{#1}}

\newcommand{\RHS}{right-hand side}
\newcommand{\Real}{\mathbb{R}}
\newcommand{\Complex}{\mathbb{C}}
\newcommand{\Natural}{\mathbb{N}}

\newcommand{\Laplace}{\bigtriangleup}
\newcommand{\ImagUnit}{i}

\newcommand{\supp}{\operatorname{supp}\,}

\newcommand{\eKN}{extremal Kerr--Newman}
\newcommand{\RNbh}{Reissner-Nordström}
\newcommand{\eRN}{extremal \RNbh}
\newcommand{\Manifold}{\mathcal{M}}

\newcommand{\ChristoffelTypeTwo}[3][]{
  \ifthenelse{\equal{#1}{}}
  {\tensor{\Gamma}{^{#2}_{#3}}} 
  {\tensor{\Gamma(#1)}{^{#2}_{#3}}}
}
\newcommand{\ChristoffelTypeOne}[3][]{
  \ifthenelse{\equal{#1}{}}
  {\tensor{\Gamma}{_{{#2#3}}}}
  {\tensor{\Gamma(#1)}{_{{#2#3}}}}
}

\newcommand{\Sphere}{\mathbb{S}}
\newcommand{\DeformationTensor}[2][]{
  \ifthenelse{\equal{#1}{}}
  {\tensor[^{({#2})}]{\pi}{}}
  {\tensor[^{({#2})}]{\pi}{#1}}
}
\newcommand{\DeformationTensorLinearized}[2][]{
  \ifthenelse{\equal{#1}{}}
  {\tensor[^{({#2})}]{\widecheck{\pi}}{}}
  {\tensor[^{({#2})}]{\widecheck{\pi}}{#1}}
}
\newcommand{\DeformationTensorErr}[2][]{
  \ifthenelse{\equal{#1}{}}
  {\tensor[^{({#2})}]{\tilde{\pi}}{}}
  {\tensor[^{({#2})}]{\tilde{\pi}}{#1}}
}

\newcommand{\CovariantDeriv}{\mathbf{D}}

\newcommand{\Metric}{\mathbf{g}}

\newcommand{\TransformScalarWaveOp}[1][]{
  \ifthenelse{\equal{#1}{}}
  {\widehat{\Box}^{(0)}}
  {\widehat{\Box}^{(0)}_{#1}}
}
\newcommand{\ScalarWaveOp}[1][]{
  \ifthenelse{\equal{#1}{}}
  {\Box^{(0)}}
  {\Box^{(0)}_{#1}}
}
\newcommand{\VectorWaveOp}[1][]{
  \ifthenelse{\equal{#1}{}}
  {\Box^{(1)}}
  {\Box^{(1)}_{#1}}
}
\newcommand{\TensorWaveOp}[1][]{
  \ifthenelse{\equal{#1}{}}
  {\Box^{(2)}}
  {\Box^{(2)}_{#1}}
}
\newcommand{\ScalarWaveConjOp}[1][]{
  \ifthenelse{\equal{#1}{}}
  {\overline{\Box}^{(0)}}
  {\overline{\Box}^{(0)}_{#1}}
}
\newcommand{\ScalarWaveLaplaceOp}[1][]{
  \ifthenelse{\equal{#1}{}}
  {\widehat{\Box}^{(0)}}
  {\widehat{\Box}^{(0)}_{#1}}
}
\newcommand{\ReducedWaveOp}[1][]{
  \ifthenelse{\equal{#1}{}}
  {\widetilde{\Box}^{(0)}}
  {\widetilde{\Box}^{(0)}_{#1}}
}

\newcommand{\JCurrent}[1]{\mathcal{P}^{#1}}
\newcommand{\KCurrent}[1]{\mathcal{Q}^{#1}}

\newcommand{\AxiSquareSumCoeff}{\mathbb{A}}

\newcommand{\rTrapping}{r_{\trap}}

\newcommand{\trap}{\operatorname{trap}}
\newcommand{\nontrap}{\cancel{\trap}}

\newcommand{\WeylQ}[1]{\operatorname{\Op}_W\left(#1\right)}

\newcommand{\TanSymClass}[1]{S_{\operatorname{tan}}^{#1}}
\newcommand{\TanOpClass}[1]{\Psi_{\operatorname{tan}}^{#1}}

\newcommand{\MixedSymClass}[2]{S^{#1}_{#2}}
\newcommand{\MixedOpClass}[2]{\Psi^{#1}_{#2}}
\newcommand{\Op}{\operatorname{Op}}
\newcommand{\Main}{\operatorname{princ}}
\newcommand{\Aux}{\bowtie}

\newcommand{\NablaAngular}{\slashed{\nabla}}

\newcommand{\LaplaceAngular}{\slashed{\Laplace}}

\newcommand{\EMTensor}{\mathbb{T}}

\newcommand{\MorNorm}{\operatorname{Mor}}
\newcommand{\MorrNorm}{\operatorname{Morr}}

\newcommand{\MorDualNorm}[1]{\MorNorm^{*}}
\newcommand{\MorrDualNorm}[1]{\MorrNorm^{*}}

\newcommand{\KN}{Kerr-Newman}

\newcommand{\Horizon}{\mathcal{H}}
\newcommand{\EventHorizon}{\mathcal{H}}

\newcommand{\EventHorizonFuture}{\mathcal{H}^+}

\newcommand{\DomainOfIntegration}{\mathcal{D}}

\newcommand{\FreqAngular}{\eta}
\newcommand{\FreqPhi}{\FreqAngular_{\varphi}}
\newcommand{\FreqTheta}{\FreqAngular_{\theta}}

\newcommand{\PrinSymb}{p}
\newcommand{\RescaledPrinSymb}{{\abs*{q}^2 \PrinSymb}}

\newcommand{\phiOEF}{\phi^*}
\newcommand{\phiIEF}{\phi_*}

\newcommand{\HawkingHorizon}{V_{\Horizon}}
\newcommand{\InfinityHawking}{V_{\NullInfinity}}

\newcommand{\MorawetzVF}{X}
\newcommand{\MorawetzSym}{\mathfrak{x}}
\newcommand{\MorawetzLagrangeCorr}{w}
\newcommand{\MorawetzLagrangeCorrSym}{\mathfrak{w}}
\newcommand{\MorawetzOneForm}{J}
\newcommand{\SquareDecomp}{\mathfrak{a}}
\newcommand{\SquareDecompOp}{\mathfrak{A}}

\newcommand{\SobolevDeg}[1]{H_{\operatorname{deg}}^{#1}}
\newcommand{\bSobolev}[1]{H_{\operatorname{b}}^{#1}}

\newcommand{\bSobolevH}[1]{H_{\operatorname{b},\Horizon}^{#1}}
\newcommand{\bSobolevI}[1]{H_{\operatorname{b},\NullInfinity}^{#1}}

\newcommand{\SobolevfinHI}[1]{H_{\operatorname{fin},\Horizon,\NullInfinity}^{#1}}

\newcommand{\bSobolevHI}[1]{H_{\operatorname{b},\Horizon,\NullInfinity}^{#1}}
\newcommand{\bSobolevTrap}[1]{H_{\operatorname{b,trap}}^{#1}}
\newcommand{\bSobolevHITrap}[1]{H_{\operatorname{b,trap},\Horizon,\NullInfinity}^{#1}}

\newcommand{\compactSobolev}[1]{H_{c}^{#1}}

\newcommand{\NullInfinity}{\mathscr{I}}

\newcommand{\bDiffH}{\operatorname{Diff}_{\operatorname{b},\EventHorizon}}
\newcommand{\bDiffI}{\operatorname{Diff}_{\operatorname{b},\NullInfinity}}
\newcommand{\bDiffHI}{\operatorname{Diff}_{\operatorname{b},\EventHorizon,\NullInfinity}}

\newcommand{\conormalSpaceH}[1]{\mathcal{O}_{\EventHorizon}^{#1}}
\newcommand{\conormalSpaceI}[1]{\mathcal{O}_{\NullInfinity}^{#1}}
\newcommand{\conormalSpaceHI}[1]{\mathcal{O}_{\EventHorizon,\NullInfinity}^{#1}}

\newcommand{\daxi}{{\varepsilon_0}}

\newcommand{\EF}{Eddington Finkelstein}
\newcommand{\iEF}{ingoing \EF}
\newcommand{\oEF}{outgoing \EF}
\newcommand{\BL}{Boyer-Lindquist}

\newcommand{\rAux}{\mathfrak{r}}
\newcommand{\KCurrentPert}[1]{\widetilde{K}^{#1}}

\newcommand{\KCurrentSym}[1]{\mathbf{k}^{#1}}
\newcommand{\KCurrentPertSym}[1]{\tilde{\mathbf{k}}^{#1}}
\newcommand{\JCurrentPert}[1]{\widetilde{J}^{#1}}
\newcommand{\renormpphi}{{\cancel{\Phi}}}
\newcommand{\diff}{{\mathrm{diff}}}

\newcommand{\Mor}{{\mathrm{Mor}}}
\newcommand{\RotationVF}{\Omega}

\newtheorem*{theorem*}{Theorem}
\newtheorem*{prop*}{Proposition}

\newtheoremstyle{named}{}{}{\itshape}{}{\bfseries}{.}{.5em}{\thmnote{#3}#1}
\theoremstyle{named}

\newcommand{\MM}{\mathcal{M}}

\def\TT{{\mathcal T}}
\def\PP{\mathcal{P}}

\def\D{{\bf D}}

\def\g{{\bf g}}

\def\RR{\mathcal{R}}

\def\a{{\alpha}}

\def\b{{\beta}}

\def\ga{\gamma}

\def\de{\delta}
\def\De{\Delta}

\def\la{\lambda}

\def\si{\sigma}
\def\th{\theta}

\def\nab{\nabla}

\def\ntrap{trap\mkern-18 mu\big/\,}

\def\pr{{\partial}}
\def\les{\lesssim}
\def\ges{\gtrsim}
\def\c{\cdot}

\def\div{{\mbox div\,}}

\def\nn{\nonumber}

\newcommand{\nabb}{\nab\mkern-13mu /\,}

\def\QQ{\mathcal{Q}}

\def\AA{\mathcal{A}}
\def\VV{\mathcal{V}}

\def\FF{\mathcal{F}}
\def\UU{\mathcal{U}}

\def\piX{\, ^{(X)}\pi}
\def\That{\widehat{T}}

\def\HH{\mathcal{H}}
\newcommand{\II}{\mathscr{I}}
\def\SSS{\mathbb{S}}

\def\gadot{\dot{\gamma}}
\def\e{{\bf e}}
\def\k{{\bf k}}

\def\g{{\bf g}}

\def\t{{\bf t}}

\def\lz{{\bf \ell_z}}

\def\rhoH{\rho_{\HH}}
\def\rhoI{\rho_{\II}}

\newcommand{\lapp}{\mbox{$\triangle \mkern-13mu /$\,}}

\newcommand{\bea}{\begin{eqnarray}}
  \newcommand{\eea}{\end{eqnarray}}
\def\beaa{\begin{eqnarray*}}
  \def\eeaa{\end{eqnarray*}}

\pgfplotsset{compat=1.18}

\usepackage[sortcites,
backend=biber,
date=year,
giveninits=true,
style=alphabetic,
doi=false,
isbn=false,
url=false,
sorting=nyt,
eprint=true]{biblatex}

\DeclareNameAlias{author}{family-given}
\AtEveryBibitem{
  \clearfield{month}
  \clearfield{url}
  \clearfield{urlyear}
  \clearfield{urlmonth} 
  \ifentrytype{online}{
      \clearfield{year}}
    {
      \clearfield{eprint}}
  }
\renewbibmacro*{volume+number+eid}{%
    \printfield{volume}%
    \setunit*{\addnbspace}
    \printfield{number}%
    \setunit{\addcomma\space}%
    \printfield{eid}}
\DeclareFieldFormat[article]{volume}{\textbf{#1}}
\DeclareFieldFormat[article]{number}{\mkbibparens{#1}}
\DeclareFieldFormat*{title}{\textit{#1}}
\DeclareFieldFormat{journaltitle}{#1\isdot}
\renewbibmacro{in:}{}
\begin{document}

\title{\LARGE \textbf{Wave decay and horizon instability on strongly charged extremal Kerr--Newman black holes}}

\author[1]{Allen Juntao Fang\footnote{allen.juntao.fang@uni-muenster.de}}
\author[2]{Elena Giorgi\footnote{elena.giorgi@columbia.edu}}
\author[3]{Jingbo Wan\footnote{jingbo.wan@sorbonne-universite.fr}}

\affil[1]{\small Mathematics M\"unster, Universit\"at M\"unster }
\affil[2]{\small Department of Mathematics, Columbia University}
\affil[3]{\small Laboratoire Jacques-Louis Lions de Sorbonne Universit\'e}

\maketitle

\begin{abstract} 
  We prove the first boundedness and pointwise decay result for the scalar wave equation on \emph{rotating extremal} black holes without any symmetry assumptions. The result applies to slowly rotating  (equivalently, strongly charged) extremal Kerr--Newman spacetimes. We establish uniform energy boundedness, integrated local energy decay, and a hierarchy of boundary-weighted estimates at the extremal horizon and at null infinity, from which inverse-polynomial pointwise decay follows in the entire exterior region. As a consequence, we also prove the expected Aretakis instability: for generic initial data, suitable transversal derivatives fail to decay along the event horizon, and higher transversal derivatives blow up asymptotically. The proof uses the $b$-structure of the wave operator near the two boundary hypersurfaces, together with a treatment of normally hyperbolic trapping on extremal Kerr--Newman.

\end{abstract}

\vspace{0.3cm}

\tableofcontents

\section{Introduction}

The study of the scalar wave equation on black hole spacetimes
occupies a central place in the mathematical theory of General
Relativity. On a fixed Lorentzian background $(\Manifold,\Metric)$,
the wave equation
\begin{equation}
  \label{eq:wave-equation-intro}
  \Box_{\Metric} \psi = F
\end{equation}
provides a basic linear model for the dynamics of fields propagating
on a black hole exterior. In the context of the black hole stability
problem, uniform boundedness and decay for solutions of
\eqref{eq:wave-equation-intro} are widely viewed as essential first
steps toward understanding the nonlinear dynamics of perturbations of
the underlying spacetime.

Over the past two decades, a robust vector-field framework has been
developed to study the dispersive properties of waves on black hole
backgrounds. The asymptotically flat, stationary black hole solutions
of the Einstein--Maxwell equations are described by the Kerr--Newman
family, a three-parameter family parametrized by the mass $M$, charge
$e$, and angular momentum parameter $a$, satisfying $a^2+e^2\leq M^2$.
Important special cases include the spherically symmetric vacuum
Schwarzschild solution $(a=e=0)$, the charged Reissner--Nordstr\"om
family $(a=0)$, and the axially symmetric vacuum Kerr family $(e=0)$.

For subextremal black holes, corresponding in the Kerr--Newman family to the strict inequality $a^2+e^2<M^2$, uniform boundedness and inverse-polynomial decay for solutions of the scalar wave equation are by now well-understood in several fundamental cases. The proofs combine energy estimates, integrated local energy decay estimates, redshift estimates near the event horizon, and $r^p$-weighted hierarchies near null infinity. Two geometric mechanisms are especially important: trapping, which forces a loss in local energy decay estimates \cite{alinhacEnergyMultipliersPerturbations2009}, and the positivity of the surface gravity, which gives rise to the redshift effect and provides robust control near the horizon \cite{dafermosRedshiftEffectRadiation2009}.


The extremal case displays fundamentally different behavior. Extremal
black holes arise when the parameters saturate the extremality
condition $a^2+e^2=M^2$.  In this case, the surface gravity of the
event horizon, now located at $\{r=M\}$, vanishes. Consequently, the redshift
effect, which plays a central role in subextremal stability proofs,
degenerates completely. This loss is not merely technical; it is tied to genuinely new dynamical phenomena at extremal
horizons.

The first manifestation of this new behavior was Aretakis' discovery
of a horizon instability for solutions of the wave equation on extremal
black hole backgrounds \cite{aretakisHorizonInstabilityExtremal2015}. Although
the solution itself may remain bounded, certain transversal derivatives
fail to decay along the event horizon and, in fact, higher-order
transversal derivatives in general grow polynomially along the
horizon. The mechanism is driven by conserved quantities along the
horizon which generate a hierarchy of instabilities. This phenomenon
shows that extremal horizons are qualitatively different from their
subextremal counterparts.

Nevertheless, exterior decay remains possible in several extremal settings,
provided the redshift estimate is replaced by $(r-M)^{-p}$-weighted estimates adapted to the
degenerate horizon. This has been carried out for waves on extremal
Reissner--Nordstr\"om and, under axisymmetry, on extremal Kerr.

In the full extremal Kerr problem without symmetry assumptions,
however, much less is known. The absence of a redshift estimate, the
presence of trapping, the coupling to superradiant effects, and the
lack of spherical symmetry make the analysis substantially more
challenging. These difficulties persist in the extremal Kerr--Newman
family, where both rotation and charge are present. The additional
charge parameter in the Kerr--Newman family makes it possible to study
extremal black holes which nevertheless are slowly rotating, in the precise sense
that $|a|/M$ is small while $a^2+e^2=M^2$.

The main purpose of this paper is to prove uniform boundedness, integrated local energy decay, and quantitative inverse-polynomial decay for scalar waves on slowly rotating extremal Kerr--Newman spacetimes, without imposing any symmetry assumptions.

\begin{theorem}[Main theorem, informal version of \zcref{thm:main-thm}]\label{main-thm-intro}
  Let $(\mathcal M,\Metric_{M,a,e})$ be an extremal Kerr--Newman spacetime, with $a^2+e^2=M^2$.
  For $\frac{|a|}{M}\ll 1$, every sufficiently regular solution\footnote{In \zcref{thm:main-thm}, the estimates allow an inhomogeneous term $F$, although the main decay result in \zcref{cor:pointwise-decay}  concerns the homogeneous equation.} of
  \[
    \Box_{\Metric_{M,a,e}}\psi=0
  \]
  arising from suitable initial data satisfies uniform boundedness of
  the energy, integrated local energy decay estimates, $(r-M)^{-p}$
  and $r^p$-weighted estimates at the event horizon and null
  infinity. 
  
  As a consequence, solutions satisfy the following
  pointwise decay:
  \begin{align*}
    \abs*{\psi(\tau, r, \theta, \phi)}\lesssim \frac{E_{\operatorname{init}}}{ r \, \tau^{\frac{1-\delta}{2}}},
  \end{align*}
  where $\tau$ is a suitable time function, $E_{\operatorname{init}}$ denotes a suitable higher-order initial energy  norm and $\delta \ll 1$ is related to the size of $\frac{|a|}{M}$.

  Moreover, for generic\footnote{More precisely, for data with non-vanishing horizon charge $H_{0}^{\text{eKN}}[\psi]$, defined in \zcref{cons-eKN}.} initial data the solutions exhibit the \emph{Aretakis instability}: suitable transversal derivatives fail to decay along the event horizon, and higher-order transversal derivatives grow asymptotically.
\end{theorem}

To the best of our knowledge, this is the first boundedness and decay theorem for scalar waves on a rotating extremal black hole background in the absence of symmetry.

\begin{remark}[Weak time decay]
  The decay rate obtained here, roughly
  $\tau^{-1/2+O(\frac{|a|}{M})}$, is weaker than the one obtained in
  the subextremal case with the same techniques. This loss is
  ultimately tied to the restricted range of weights available in the
  extremal horizon hierarchy in the presence of rotation. Although we do
not expect the decay rate obtained in the present work to be sharp, we
leave for future work the broader question of how
asymptotics vary along the extremal family, from
Reissner--Nordstr\"om toward Kerr.
\end{remark}

The proof of \zcref{main-thm-intro} follows the general vector-field
strategy developed for subextremal black holes: one combines energy estimates,
integrated local energy decay estimates, and weighted estimates near the
asymptotic ends. In the rotating extremal setting, however, this strategy has to
overcome three distinct difficulties:
\begin{enumerate}
\item the complete degeneration of the redshift effect at the event horizon;
\item superradiance, caused by the absence of a globally timelike stationary
  Killing field;
\item trapping, without the simplifying reduction available under
  axisymmetry.
\end{enumerate}

The slow rotation assumption allows the last two difficulties to be
separated.  For $|a|/M$ sufficiently small, superradiance can be
controlled by using a globally timelike vector field which agrees with
the stationary Killing field outside a compact region and whose
non-Killing contribution is supported away from both the ergoregion
and the trapped set. For slowly-rotating \eKN{} spacetimes, the
exact location of the trapped null geodesics becomes
frequency-dependent, but all trapped null geodesics are located within
an $O(a)$ neighborhood of $\{r=2M\}$, which is the photon sphere of \eRN. 
Trapping can then be treated by combining a
physical-space Morawetz current away from the trapped set with a
microlocal pseudodifferential correction near trapping, in the spirit
of Tataru--Tohaneanu's work \cite{tataruLocalEnergyEstimate2011}. This yields
a coercive integrated local energy decay estimate.  In these respects,
the arguments in the present work for slowly-rotating \eKN{} are closely related to those used for slowly-rotating Kerr.

The main new ingredient of the paper is the treatment of the first difficulty,
namely the absence of redshift at the extremal horizon. We develop a
$(r-M)^{-p}$-weighted hierarchy, first introduced in extremal Reissner-Nordstr\"om in \cite{angelopoulosLatetimeAsymptoticsWave2020}, and adapt it to the extremal Kerr--Newman geometry.
This hierarchy replaces the redshift estimate and provides the horizon control
needed to close the Morawetz argument without imposing axisymmetry.

In extremal Reissner--Nordstr\"om, the horizon hierarchy is closely
related, through the Couch-Torrence conformal inversion, to the usual
$r^p$-weighted hierarchy at null infinity \cite{couchConformalInvarianceSpatial1984}. In the rotating extremal
case this symmetry is lost.  In particular, while the asymptotically flat end of extremal Kerr--Newman retains the same leading structure relevant for the usual $r^p$-hierarchy near null infinity, the near-horizon geometry is no longer related to the asymptotic end. This is a manifestation of the fact that, while Kerr--Newman is asymptotically Minkowski near infinity, it is \emph{not asymptotically Reissner-Nordstr\"om} near the event horizon. The horizon hierarchy must therefore be constructed directly, and the rotational error terms restrict the range of admissible weights.
This restriction is ultimately responsible for the weaker time-decay rate in the main theorem, as the loss in the admissible horizon weights propagates through the energy-decay argument.

The exterior boundedness and decay estimates obtained in this paper
hold up to the event horizon and therefore allow one to exploit the
Aretakis conservation law directly. For generic initial data, this
conservation law implies non-decay of suitable transversal derivatives
along the horizon and, through the wave equation, a hierarchy of
polynomial growth for higher transversal derivatives. In this sense,
our decay theory recovers the expected Aretakis instability while
maintaining uniform control of the solution itself in the exterior.

\subsection{Previous results}

The analysis of waves on extremal black holes began with the work of
Aretakis on extremal Reissner--Nordstr\"om. In
\cite{aretakisStabilityInstabilityExtreme2011,
  aretakisStabilityInstabilityExtreme2011a}, Aretakis established the
first rigorous boundedness, decay, and instability results for the
linear scalar wave equation on an extremal black hole exterior. These
works showed that, although solutions remain pointwise bounded and
decay in the exterior, certain
transversal derivatives along the horizon fail to decay and higher
transversal derivatives blow up polynomially in advanced time. A key
feature of the celebrated Aretakis instability is the presence of conserved quantities
along the degenerate event horizon.

Subsequent work led to a more refined understanding of the exterior
dynamics on extremal
Reissner--Nordstr\"om. Angelopoulos--Aretakis--Gajic proved degenerate
Morawetz estimates up to and including the event horizon
\cite{angelopoulosTrappingEffectDegenerate2017}, and later obtained
sharp late-time asymptotics and Price's law tails for solutions to the
linear wave equation
\cite{angelopoulosLatetimeAsymptoticsWave2020}. Their analysis
introduced an $(r-M)^{-p}$-weighted hierarchy adapted to the
extremal horizon. Recently, Gajic extended these
results to charged scalar fields on extremal Reissner--Nordstr\"om
backgrounds \cite{gajicChargedScalarFields2026}, where the electromagnetic coupling introduces
features reminiscent of the dynamics outside spherical symmetry. 
For the
Einstein--Maxwell equations on extremal Reissner--Nordstr\"om,
Apetroaie proved stability and instability results for the coupled
gravitational and electromagnetic perturbations
\cite{apetroaieInstabilityGravitationalElectromagnetic2023}. In
the context of extremal Reissner-Nordstr\"om, we also recall results
on nonlinear wave equations
\cite{angelopoulosGlobalSphericallySymmetric2016,
  angelopoulosNonlinearScalarPerturbations2020,
  angelopoulosSemilinearWaveEquations2025}, scattering theory for the
linear wave \cite{angelopoulosNondegenerateScatteringTheory2020,
  angelopoulosMatchingConditionsScattering2026}, and on the black hole
interior \cite{gajicLinearWavesInterior2017,
  gajicInteriorDynamicalExtremal2019}.

The rotating extremal case is substantially more delicate: the analysis is complicated by the absence of spherical symmetry, the presence of superradiance and trapping, and the degeneracy of the redshift effect at the event horizon.
In the axisymmetric
setting, Aretakis proved boundedness and decay for solutions to the
scalar wave equation on extremal Kerr
\cite{aretakisDecayAxisymmetricSolutions2012}. A different physical-space proof of integrated local energy estimates
for axisymmetric waves on extremal Kerr was later obtained by the second and third author in
\cite{giorgiPhysicalspaceEstimatesAxisymmetric2024}.

Without symmetry assumptions, much less is known for extremal Kerr.
Teixeira da Costa proved mode stability for the Teukolsky equation on
extremal Kerr, ruling out exponentially growing mode solutions
\cite{teixeiradacostaModeStabilityTeukolsky2020}. More recently,
Gajic discovered stronger horizon instabilities associated to
non-axisymmetric azimuthal modes, with higher transversal derivatives displaying
polynomial growth, accompanied by oscillatory behavior along the event
horizon \cite{gajicAzimuthalInstabilitiesExtremal2023}.
These instabilities are not a direct consequence of the axisymmetric
Aretakis conservation laws and show that the non-axisymmetric dynamics
near the extremal Kerr horizon contain additional mechanisms. For conditional results on the
Maxwell equations on extremal Kerr see \cite{benomioMaxwellEquationsFull2025}, and for blow-up results for nonlinear wave equations on extremal Kerr see \cite{aretakisNonlinearInstabilityScalar2013}.

For extremal Kerr--Newman, the general horizon instability mechanism
was identified by Aretakis in \cite{aretakisHorizonInstabilityExtremal2015},
where conservation laws were shown
to hold for a broad class of extremal horizons, including
Kerr--Newman. 
However, in contrast with the extremal Reissner--Nordstr\"om case and
with the axisymmetric theory on extremal Kerr, prior to the present work, no general boundedness and decay theory
was available for the scalar wave equation on rotating extremal
Kerr--Newman backgrounds without symmetry assumptions.

This should be contrasted with the subextremal theory. For subextremal
Kerr, boundedness and decay for the scalar wave equation were developed through the works of, among others,
Dafermos--Rodnianski
\cite{dafermosBlackHoleStability2010}, Tataru--Tohaneanu
\cite{tataruLocalEnergyEstimate2011}, Andersson--Blue
\cite{anderssonHiddenSymmetriesDecay2015}, and Dafermos-Rodnianski-Shlapentokh-Rothman \cite{dafermosDecaySolutionsWave2016}. These results rely in an
essential way on the presence of a non-degenerate redshift effect at the
event horizon \cite{dafermosRedshiftEffectRadiation2009} and a robust physical-space method to deduce decay \cite{dafermosNewPhysicalSpaceApproach2010} introduced by Dafermos-Rodnianski (see also \cite{moschidis$r^p$WeightedEnergyMethod2016, angelopoulosVectorFieldApproach2018}), together with suitable treatments of trapping and
superradiance.

The preceding discussion concerns linear and nonlinear field equations on fixed extremal black hole backgrounds. Extremal horizons have also recently appeared in a more dynamical context, in connection with gravitational collapse and threshold phenomena. In spherical symmetry, Kehle--Unger \cite{kehleGravitationalCollapseExtremal2025} constructed
solutions of the Einstein--Maxwell--charged scalar field system forming
extremal Reissner--Nordstr\"om black holes in finite advanced time, and
related matter models have also been studied
\cite{kehleExtremalBlackHole2024,weissenbacherDecayNondecayMassless2024}. Further results on collapse, threshold dynamics, and moduli spaces
were obtained in
\cite{eastGravitationalCollapseVicinity2026,
  gadiouxFormationExtremalReissnernordstrom2026,
  angelopoulosNonlinearStabilityExtremal2026,
  angelopoulosModuliSpaceDynamical2026}.

\subsection{Sketch of the proof}

We give here a sketch of the proof of \zcref{main-thm-intro}. We start by explaining the role of the boundary-weighted hierarchies to deal with the degeneration of the redshift effect and how those hierarchies interact with the derivation of the energy-Morawetz estimates for the wave equation. 

\subsubsection{The role of the boundary-weighted hierarchies}

Following the work of \cite{dafermosNewPhysicalSpaceApproach2010}, it
is known that proving a hierarchy of weighted $L^2$-estimates for
solutions to the wave equation yields their (weighted) time decay. In
the case of \eKN, there are two relevant hierarchies: one capturing
weights towards null infinity, and the other capturing weights towards
the extremal horizon.

To capture the weighted hierarchy near null infinity $\NullInfinity$, we consider the natural boundary-defining function for $\NullInfinity=\{\rhoI=0\}$ given by 
\[\rhoI\vcentcolon= \frac{1}{r}.\]
Observe that
in the outgoing Eddington Finkelstein coordinates, the inverse Minkowski metric becomes
\begin{equation}\label{eq:expression-inverse-metric-minkowski-intro}
  r^2\Metric^{-1}_{\mathbb{M}^{1+3}}
  = 2\partial_u\partial_{\rhoI} + (\rhoI\partial_{\rhoI})^2 + \partial_\th^2 +\frac{1}{\sin^2\th} \partial_{\phi}^2, 
\end{equation}
where $u=t-r$ is the outgoing retarded time. In particular, the stationary part of the rescaled inverse metric, and
therefore the stationary part of the rescaled wave operator itself, is
an elliptic $b$-operator on Minkowski\footnote{In other words,
  the (non-rescaled) wave operator on Minkowski is a scattering operator.},
i.e. it is composed of vectorfields tangent to the boundary
$\{\rhoI=0\}$, a basis of which is given by
\begin{equation*}
  \curlyBrace*{\rhoI\partial_{\rhoI}=-r\partial_r, \partial_u, \NablaAngular},
\end{equation*}
where $\NablaAngular$ denotes the angular derivatives over $\Sphere^2$.  As a consequence, the stationary part of the rescaled wave
operator on Minkowski, which is the Euclidean Laplacian, enjoys
good mapping properties between weighted $b$-Sobolev spaces, as well
as weighted $L^{\infty}$ spaces (referred to as conormal
spaces). This is also connected with the classical invertibility of
the Euclidean Laplacian between weighted $L^{\infty}$ spaces.

The $b$-structure embedded in the wave operator plays an important
underlying role in the derivation of the $r^p$-weighted estimates, first
uncovered in the context of the wave equation by Dafermos-Rodnianski
\cite{dafermosNewPhysicalSpaceApproach2010} using $r^p\partial_r$
as a vectorfield multiplier\footnote{We see that
  $r^p\partial_r = r^{p-1}r\partial_r =
  -\rhoI^{1-p}\rho\partial_{\rhoI} = -\rhoI^{2-p}\partial_{\rhoI}$ is
  a weighted $b$-vectorfield.}. These estimates roughly state\footnote{In our notation, $\b=3-p$, resulting in the familiar range of $p\in(0,2)$.} that
\begin{equation*}
  \norm*{\psi}_{\bSobolevI{s, -\frac{\beta}{2}}(\Manifold_{\rhoI\le \rho_0}(\tau_1,\tau_2))}
  \lesssim \norm*{\Box_{\mathbb{M}^{1+3}}\psi}_{\bSobolevI{s, -\frac{\beta-4}{2}}(\Manifold_{\rhoI\le \rho_0}(\tau_1,\tau_2))},\qquad \beta\in (1,3),
\end{equation*}
where
$\norm*{\cdot}_{\bSobolevI{s, \gamma}(\Manifold(\tau_1,\tau_2))}$
denotes the $\rhoI^{\gamma}$ weighted $s$-regularity $b$-Sobolev norm
at $\NullInfinity$ (see already \zcref{sec:function-spaces} for the precise
definition), and the fact that the estimates hold on
$\Manifold_{\rhoI\le \rho_0}(\tau_1,\tau_2)$ reflects the fact that
the estimates are capturing behavior related to the asymptotic
geometric structure at $\rhoI = 0$. 

This $b$-structure persists in waves on
\emph{asymptotically flat} spacetimes, rather than just Minkowski
itself, leading to the $r^p$-weighted estimates holding for general,
asymptotically flat spacetimes
\cite{moschidis$r^p$WeightedEnergyMethod2016}. Kerr--Newman is also
asymptotically flat, and consequently has a similar hierarchy of
weighted estimates at null infinity.

The other hierarchy of weighted estimates in \eKN{} is located near
the event horizon. Recall that in the subextremal \KN{} family, the
positive surface gravity of the event horizon induces a redshift
effect at the event horizon which is a form of local exponential
damping for solutions to the wave equation
\cite{dafermosRedshiftEffectRadiation2009}. In the extremal limit, the
surface gravity of the event horizon vanishes, and the redshift effect
also vanishes along with it. Nonetheless, one is able to recover a
family of boundary-weighted estimates involving $b$-Sobolev spaces at
the event horizon similar to the $r^p$-estimates near null
infinity\footnote{This is reminiscent of the appearance of the
  $r^p$-weighted estimates in place of the cosmological redshift
  effect at the cosmological horizon in the vanishing-$\Lambda$ limit
  for Kerr--de Sitter black holes
  \cite{fangTeukolskySlowlyrotatingKerrde2026}.}. Moreover, the
horizon-weighted hierarchy of estimates on \eKN{} can be viewed as
degenerate form of the redshift estimates in that they capture what
remains of radial point structure of the standard redshift estimate
and its associated local exponential decay on subextremal Kerr--Newman
\cite{dafermosRedshiftEffectRadiation2009,vasyMicrolocalAnalysisAsymptotically2013}.

To capture the weighted hierarchy near the horizon $\EventHorizon$, we consider the natural boundary-defining function for $\EventHorizon=\{\rhoH=0\}$ given by 
\[\rhoH\vcentcolon= r-M.\]
Consider first the case of non-rotating extremal black holes.
For $a=0$, the inverse \eRN{} metric can be written in terms of the ingoing Eddington-Finkelstein coordinates as
\begin{equation}\label{eq:form-inverse-metric-eRN-intro}
  r^2\Metric^{-1}_{\operatorname{eRN}}
  = 2r^2\partial_v\partial_{\rhoH} + (\rhoH\partial_{\rhoH})^2 + \partial_\th^2 +\frac{1}{\sin^2\th} \partial_{\phi}^2, 
\end{equation}
where $v = t+r_{*}$ is the advanced tortoise time coordinate.
Comparing the expression for $r^2\Metric^{-1}_{\operatorname{eRN}}$ with the one in \zcref{eq:expression-inverse-metric-minkowski-intro}, one can notice that the \eRN{} inverse metric displays at the horizon the
same $b$-structure as the
asymptotically flat end at null infinity\footnote{This is
  also suggested by the Couch-Torrence transform \cite{couchConformalInvarianceSpatial1984}.}. This structure
appeared in \cite{angelopoulosLatetimeAsymptoticsWave2020}, which used
$(r-M)^{-p}\partial_r = \rhoH^{-p}\partial_{\rhoH}$ as a vectorfield
multiplier to show that a hierarchy of $(r-M)^{-p}$-weighted estimates
also holds for the wave equation on \eRN, that is, roughly\footnote{In our notation, $\a=1-p$, resulting the familiar range of $p\in(0,2)$.},
\begin{equation*}
  \norm*{\psi}_{\bSobolevH{s, -\frac{\alpha}{2}}(\Manifold_{\rhoH\le \rho_0}(\tau_1,\tau_2))}
  \lesssim \norm*{\Box_{\operatorname{eRN}}\psi}_{\bSobolev{s, -\frac{\alpha}{2}}(\Manifold_{\rhoH\le \rho_0}(\tau_1,\tau_2))},\qquad \alpha\in (-1,1),
\end{equation*}
where $\bSobolevH{s, \alpha}$ now represent the $b$-Sobolev spaces with respect to the event horizon rather
than null infinity.

 Unlike in the
asymptotically flat end, the passage from the spherically symmetric
\eRN{} spacetime (for $a=0$) to the axisymmetric \eKN{} spacetime (even for small $\frac{|a|}{M}$) has a large
effect on the $b$-structure of the wave equation at the horizon. To see
this, observe that the inverse \eKN{} metric can be written in \iEF{} coordinates as
\begin{equation*}
 (r^2+a^2\cos^2\th)\Metric^{-1}_{\mathrm{eKN}}
  ={} 2\left( (r^2+a^2)\partial_v + a\partial_{\phi} \right)\partial_{\rhoH}
  + (\rhoH\partial_{\rhoH})^2
  + \partial_{\theta}^2 + \big(a\sin\theta\partial_v + \frac{1}{\sin\theta}\partial_{\phi}\big)^2.
\end{equation*}
Comparing the expression for $\Metric^{-1}_{\operatorname{eKN}}$ with the one in \zcref{eq:form-inverse-metric-eRN-intro}, one sees that
in this case the $O(a)$ contributions to the inverse metric are no
longer lower-order, and in fact, at the horizon,
$a\partial_{\phi}\partial_{\rhoH}$ is the dominant term in the
stationary part of the inverse metric. This reflects the fact that
while \eKN{} is asymptotically flat, i.e. asymptotically Minkowski at
null infinity, \eKN{} is \emph{not} asymptotically \eRN{} at the extremal
event horizon. 

This is the main difficulty and the main novel structure in
\eKN. While the dominance of the $a\partial_{\phi}\partial_{\rhoH}$
term in the inverse metric in \eKN{} is problematic in general, for $\abs*{\frac{a}{M}}\ll1$, this can be overcome by giving
up a small amount of the hierarchy\footnote{In general, for $\frac{|a|}{M}<\frac 1 2$ the hierarchy holds for $\alpha \in (-1+4\frac{a^2}{M^2}, 1-4\frac{a^2}{M^2})$, see already \zcref{prop:horizon-weighted-estimate}.}; that is, by proving the hierarchy
of weighted estimates of the form
\begin{equation}
  \label{eq:intro:eKN-boundary-hierarchy}
  \norm*{\psi}_{\bSobolev{s, -\frac{\alpha}{2}}(\Manifold(\tau_1,\tau_2))}
  \lesssim \norm*{\Box_{\operatorname{eKN}}\psi}_{\bSobolev{s, -\frac{\alpha}{2}}(\Manifold(\tau_1,\tau_2))},\qquad \alpha\in (-1+\delta(a), 1-\delta(a)),
\end{equation}
where the range of weights is restricted by $O(a)$ compared to the
range of weights on \eRN.  Even though we do not expect the range
of weights that we prove to be sharp, it does appear that outside of symmetry assumptions, the
range of viable weights for the horizon boundary-weighted hierarchy is
restricted for \eKN{} compared to \eRN{} due to the dominance of the
rotation terms. This phenomenon of restriction in the available range of the horizon hierarchy was previously observed for azimuthal modes in extremal Kerr in \cite{gajicAzimuthalInstabilitiesExtremal2023}.

\subsubsection{Trapping and its interaction with the boundary hierarchy}

Another key geometric obstacle to decay for
waves on \eKN{} spacetimes is the family of trapped
null geodesics. Nonetheless, since the trapped set is normally
hyperbolic \cite{dyatlovSpectralGapsNormally2016}, a local pseudo-differential modification of the usual
vectorfield method can be used to prove an integrated local energy
decay, or Morawetz, estimate near trapping
\cite{tataruLocalEnergyEstimate2011}. A global estimate however, is
still required to connect the local estimates near trapping to the
boundary-weighted hierarchy of estimates discussed in the previous
section.

To this end, we prove an integrated local energy decay inequality
controlling a nondegenerate set of derivatives away from the trapped
set together with a degenerate control at the extremal horizon. We first construct, following the approach in
\cite{giorgiPhysicalspaceEstimatesAxisymmetric2024}, a Morawetz
multiplier that would be suitable for proving an integrated local
energy decay estimate for axisymmetric solutions.  This is achieved by
using as a multiplier a radial vector field
$X_{ax}=\mathcal{F}(r)\partial^{\mathrm{BL}}_r$ in Boyer-Lindquist
coordinates together with a Lagrangian correction $w_{ax}$ and a
one-form $J_{ax}$ to obtain a positive bulk term controlling, outside trapping,
\[
 \Big( r^{-3}|\That\psi|^2
  +\frac{\rho_{\HH}^4}{r^7}|\partial^{\mathrm{BL}}_r\psi|^2
  +\frac{\rho_{\HH}}{r^2}|\nabla\psi|^2
  +\frac{\rho_{\HH}^2}{r^6}|\psi|^2\Big) - O(|a|r^{-3})\That\psi\,\partial_\phi\psi,
\]
where $\That=\partial_t +\frac{a}{r^2+a^2}\partial_\phi$,
modulo the mixed term in $\That\psi\,\partial_\phi\psi$ and a
localized negative zeroth-order contribution supported very close to
the horizon.  

Away from both the horizon and from trapping, the mixed term is
perturbative for $|a|/M\ll1$. However, near the horizon, the mixed
term does not vanish at the event horizon\footnote{Equivalently, it lives in the weighted $b$-Sobolev space
$\bSobolev{1, 0}(\Manifold(\tau_1,\tau_2))$.} and therefore cannot be absorbed by all members of the horizon-weighted hierarchy. In particular, in the notation of
\zcref{eq:intro:eKN-boundary-hierarchy}, the mixed term dominates over
the boundary-weighted hierarchy for $\alpha>0$. Similarly, the localized zeroth-order term cannot be controlled by the boundary-weighted hierarchy for
$\alpha>0$. To resolve both of these issues at once, we use a
distinguished member of the horizon-weighted hierarchy, corresponding to $\alpha=-\de_1$ for $\de_1\ll 1$, as a
\emph{degenerate redshift estimate}, using the fact that
\begin{equation*}
  \int_{\Manifold_{\rhoH\le \rho_0}(\tau_1,\tau_2)}|\That\psi|\abs*{\partial_{\phi}\psi} + \abs*{\psi}^2
  \lesssim \norm*{\psi}_{\bSobolev{1, -\frac{\alpha}{2}}(\Manifold(\tau_1,\tau_2))}^2, \qquad \alpha \in (-1+\delta(a), -\de_1).
\end{equation*}
As a result, in our final estimate, we lose access to roughly
half the range of weights for the horizon boundary-weighted
hierarchy.
While our proof encounters these error terms, we do not expect that
these error terms are strictly necessary. Indeed, a more refined proof
for integrated local energy decay, which yields error terms that only
degenerate more strongly at the horizon, could avoid the need for a
separate degenerate redshift estimate.

\subsubsection{Horizon instability vs time decay}

Since the admissible range of horizon weights is effectively reduced,
the standard energy-flux and Sobolev argument yields only a weak
inverse-polynomial decay rate, of order
$\tau^{-(1-\delta(a)-\delta_1)/2}$, for the solution throughout the
exterior, up to and including the event horizon. Although we do not expect this decay rate to be sharp, even such weak decay along the horizon is sufficient, in conjunction with the Aretakis conservation law, to recover the horizon instability.
 Indeed, the conservation of a
quantity roughly of the form $\partial_r(r\psi)$, together with decay
of $\psi$ along the event horizon, implies that the ingoing
transversal derivative $\partial_r\psi$ generically fails to decay
there. Higher transversal derivatives then grow polynomially in time
by commuting with the wave equation.

\subsection{Outline of the paper}

We provide here an outline of the remainder of the paper.

\begin{itemize}
  \item In \zcref{sec:preliminaries} we introduce the necessary preliminaries on extremal Kerr--Newman spacetimes and the scalar wave equation.
  \item In \zcref{sec:main-theorem} we define the relevant energy norms and state the main theorem and pointwise decay and prove the horizon instability.
  \item In \zcref{sec:hierarchy} we prove the weighted hierarchies near the event horizon and null infinity.
  \item In \zcref{sec:Morawetz} we establish the energy and Morawetz estimates.
  \item In \zcref{sec:proof-main-theorem} we complete the proof of the main theorem.
\end{itemize}

\subsubsection*{Acknowledgments}

We are grateful to Dejan Gajic and Georgios Moschidis for helpful discussions on weighted hierarchies, and to Marios Apetroaie for useful discussions on horizon instability.

A.J.F. acknowledges  support from the Deutsche
Forschungsgemeinschaft (DFG, German Research Foundation) through
Germany’s Excellence Strategy EXC 2044/2 390685587, Mathematics
M\"{u}nster: Dynamics–Geometry–Structure, from the Alexander von
Humboldt Foundation in the framework of the Alexander von Humboldt
Professorship endowed by the Federal Ministry of Education and
Research, and from NSF award DMS-2303241. E.G. acknowledges the support of NSF Grants DMS-2306143, DMS-2336118 and of a grant of the Sloan Foundation. J.W. is supported by ERC-2023 AdG 101141855 BlaHSt.

\section{Preliminaries}\label{sec:preliminaries}

In this section we present the necessary preliminaries. \zcref{sec:eKN} reviews the main properties of the extremal Kerr–Newman spacetime, and \zcref{sec:wave-eq} summarizes basic results on the wave equation. \zcref{sec-AB} contains the definitions and computations required for the classical and pseudodifferential vector field methods, respectively.

\subsection{Extremal Kerr--Newman spacetime}
\label{sec:eKN}

The Kerr--Newman family consists of stationary, rotating, charged
black holes of mass $M$, angular momentum $Ma$ and charge $e$, solutions
to the Einstein-Maxwell equations, satisfying $a^2+e^2 \leq M^2 $. If
$a^2+e^2=M^2$, the metric describes the extremal Kerr--Newman black
hole.

The main ingredients recalled in this section are the coordinate systems regular at the event horizon and null infinity, the associated boundary-defining functions and conormal spaces, and the geometric quantities governing timelike multipliers and trapping.

\subsubsection{The metric in Boyer-Lindquist coordinates}\label{sec:metric}

For $a^2+e^2 \leq M^2 $, the Kerr--Newman metric in Boyer-Lindquist
coordinates
$(t, r, \th, \phi) \in \mathbb{R} \times (r_{+}, \infty) \times
\Sphere^2$ takes the form
\begin{equation}\label{metric-KN}
  \g_{\mathrm{BL}}
  =-\frac{\Delta}{|q|^2}\left( dt- a \sin^2\th d\phi\right)^2+\frac{|q|^2}{\Delta}dr^2+|q|^2d\th^2+\frac{\sin^2\th}{|q|^2}\left(a dt-(r^2+a^2) d\phi \right)^2,
\end{equation}
where
\begin{align*}
  \Delta \vcentcolon={}& r^2-2Mr+a^2+e^2=(r-r_{+}) (r-r_{-}), \qquad
                         |q|^2\vcentcolon=r^2+a^2\cos^2\th,
\end{align*}
and $r_{\pm}=M\pm \sqrt{M^2-a^2-e^2}$. In the case of extremal
Kerr--Newman spacetimes, we have
\begin{align*}
  \Delta=(r-M)^2,
\end{align*}
and the roots of $\Delta=0$ degenerate to $r_{+}=r_{-}=M$. The event
horizon is defined by $\HH=\{ r=M\}$ and the exterior region is given
by $\mathcal{D}\vcentcolon= \{ r> M\}$.  We denote $\MM=\{ r \geq M\}$
the exterior region together with the future event horizon.

The inverse metric in \BL{} coordinates takes the form
\begin{equation}
  \label{eq:inverse-metric-BL}
  |q|^2\g_{\mathrm{BL}}^{\a\b}
  = \Delta \partial_r^\a \partial_r^\b+\frac{1}{\Delta} \RR^{\a\b},
\end{equation}
where
\begin{align}
  \RR^{\a\b}
  ={}&  -(r^2+a^2)^2 \partial_t^\a \partial_t^\b-2a(r^2+a^2)\partial_t^{(\a} \partial_\phi^{\b)}-a^2  \partial_\phi^\a \partial_\phi^\b+ \Delta O^{\a\b}, \label{eq:RR-def}\\
  O^{\a\b}
  ={}& \partial_\th^\a  \partial_\th^\b  +\frac{1}{\sin^2\th} \partial_{\phi}^\a \partial_{\phi}^\b+2a\partial_t^{(\a} \partial_\phi^{\b)}+a^2 \sin^2\th \partial_t^\a \partial_t^\b, 
\end{align}
where $O^{\alpha \beta}$ is the positive-definite angular part of the
conformally rescaled inverse metric.  We denote
\begin{equation*}
  O^{\a\b}(\pr_\a \psi )(\pr_\b \psi )
  ={}|\partial_\theta \psi|^2+\big|\frac{1}{\sin\theta} \partial_\phi \psi + a\sin\theta \partial_t \psi\big|^2=\vcentcolon|\nabb \psi|^2,
\end{equation*}
and also $|q|^2|\nab \psi|^2 \vcentcolon=|\nabb \psi|^2$.
In what
follows, we will also use the shorthand notation
\begin{equation*}
 \renormpphi\vcentcolon=\frac{1}{\sin\th}\pr_\phi + a\sin\th \pr_t.
\end{equation*}

Boyer-Lindquist coordinates are singular at the horizon. We therefore
introduce the ingoing Eddington-Finkelstein coordinates, used near the
event horizon $\EventHorizon$, and the outgoing Eddington-Finkelstein
coordinates, used near null infinity $\NullInfinity$.  When there is
risk of ambiguity we will denote the coordinate vectorfields in \BL{}
coordinates as $\pr_t^{\mathrm{BL}}$, $\pr_r^{\mathrm{BL}}$,
$\pr_\th^{\mathrm{BL}}$, $\pr_\phi^{\mathrm{BL}}$.

\subsubsection{The metric in ingoing and outgoing Eddington-Finkelstein charts}

To remove the coordinate singularity in \zcref{metric-KN} at $\De=0$
describing the black hole event horizon, one can define the functions
\begin{align*}
  r^*=\int \frac{r^2+a^2}{\De}, \qquad \phiIEF=\phi+ \int \frac{a}{\De}, \qquad v=t+r^{*}
\end{align*}
and obtain the Kerr--Newman metric in the ingoing Eddington-Finkelstein
(iEF) coordinates
$(v, r, \th, \phiIEF)\in \mathbb{R} \times [r_{+}, \infty) \times
\Sphere^2$:
\begin{equation*}
  \begin{split}
    \g_{\mathrm{iEF}}
    ={}&- \frac{\De-a^2\sin^2\th}{\abs*{q}^2}dv^2
         + 2 dv dr -\frac{2a\sin^2\th\left(  (r^2+a^2)-\De\right)}{\abs*{q}^2} dv d\phiIEF \\
       &-2a\sin^2\th dr d\phiIEF +\abs*{q}^2 d\th^2 +\frac{\sin^2\th}{\abs*{q}^2}\left((r^2+a^2)^2-\De a^2\sin^2\th \right)(d\phiIEF)^2,
  \end{split}
\end{equation*}
which is regular at the horizon.  The inverse metric in iEF
coordinates takes the form
\begin{equation}
  \label{eq:inverse-metric:ingoing-EF}
  \abs*{q}^2\Metric^{-1}_{\mathrm{iEF}}
  ={} 2\left( (r^2+a^2)\partial_v + a\partial_{\phiIEF} \right)\partial_r + \Delta\partial_r^2 + \partial_{\theta}^2 + \left(a\sin\theta\partial_v + \frac{1}{\sin\theta}\partial_{\phiIEF}\right)^2.
\end{equation}
The coordinate vectorfields in iEF coordinates $\pr_v^{\mathrm{iEF}}$, $\pr_r^{\mathrm{iEF}}$, $\pr_\th^{\mathrm{iEF}}$, $\pr_{\phiIEF}^{\mathrm{iEF}}$ are related to the coordinate vectorfields in \BL{} coordinates:
\begin{align}\label{eq:coord-change1}
  \pr_r^{\mathrm{BL}} =\pr_r^{\mathrm{iEF}} +\frac{r^2+a^2}{\De}\,\pr_v^{\mathrm{iEF}} +\frac{a}{\De}\,\pr_{\phiIEF}^{\mathrm{iEF}}, \qquad 
  \pr_r^{\mathrm{iEF}} =\pr_r^{\mathrm{BL}} -\frac{r^2+a^2}{\De}\,\pr_t^{\mathrm{BL}} -\frac{a}{\De}\,\pr_\phi^{\mathrm{BL}} ,
\end{align}
while $\pr_t^{\mathrm{BL}}=\pr_v^{\mathrm{iEF}}$, $\pr_\phi^{\mathrm{BL}}=\pr_{\phiIEF}^{\mathrm{iEF}}$, $\pr_\th^{\mathrm{BL}}=\pr_\th^{\mathrm{iEF}}$. Henceforth we denote the coordinate vectorfields in iEF as $\pr_v$, $\pr_r^{\mathrm{iEF}}$, $\pr_\th$, $\pr_{\phiIEF}$.

Defining the outgoing Eddington-Finkelstein (oEF) coordinates $(u, r, \th, \phiOEF)\in \mathbb{R} \times [r_{+}, \infty) \times \Sphere^2$ with
\begin{equation*}
  r^{*} = \int \frac{r^2+a^2}{\Delta}, \qquad \phiOEF = \phi - \int \frac{a}{\Delta}, \qquad u = t - r^{*},
\end{equation*}
the metric and inverse metric take the form
\begin{align}
  \Metric_{\mathrm{oEF}}
  ={}& - \frac{\Delta - a^2\sin^2\theta}{\abs*{q}^2} du^2
       - 2dudr
       - \frac{2a\sin^2\theta\left((r^2+a^2)-\Delta\right)}{\abs*{q}^2}dud\phiOEF \nonumber\\
     & +2a\sin^2\th dr d\phiOEF
       +\abs*{q}^2 d\th^2
       +\frac{\sin^2\th}{\abs*{q}^2}\left((r^2+a^2)^2-\De a^2\sin^2\th \right)(d\phiOEF)^2, \nn\\
  \abs*{q}^2\Metric^{-1}_{\mathrm{oEF}}
  ={}& -2\left( (r^2+a^2)\partial_u + a\partial_{\phiOEF} \right)\partial_r
       + \Delta\partial_r^2
       + \partial_{\theta}^2
       + \left(a\sin\theta\partial_u  + \frac{1}{\sin\theta}\partial_{\phiOEF}\right)^2.\label{eq:inverse-metric-oEF}
\end{align}
This chart is adapted at future null infinity and is convenient for describing asymptotics for large $r$.
The coordinate vectorfields in oEF coordinates $\pr_u^{\mathrm{oEF}}$, $\pr_r^{\mathrm{oEF}}$, $\pr_\th^{\mathrm{oEF}}$, $\pr_{\phiOEF}^{\mathrm{oEF}}$ are related to the coordinate vectorfields in \BL{} coordinates:
\begin{align}\label{eq:relation-BL-oEF}
  \pr_r^{\mathrm{BL}}
  =\pr_r^{\mathrm{oEF}}
  -\frac{r^2+a^2}{\De}\,\pr_u^{\mathrm{oEF}}
  -\frac{a}{\De}\,\pr_{\phiOEF}^{\mathrm{oEF}}, \qquad 
  \pr_r^{\mathrm{oEF}}
  =\pr_r^{\mathrm{BL}}
  +\frac{r^2+a^2}{\De}\,\pr_t^{\mathrm{BL}}
  +\frac{a}{\De}\,\pr_\phi^{\mathrm{BL}} 
\end{align}
while $\pr_t^{\mathrm{BL}}=\pr_u^{\mathrm{oEF}}$,
$\pr_\phi^{\mathrm{BL}}=\pr_{\phiOEF}^{\mathrm{oEF}}$,
$\pr_\th^{\mathrm{BL}}=\pr_\th^{\mathrm{oEF}}$.
Henceforth we denote the coordinate vectorfields in oEF as $ \pr_u$, $\pr_r^{\mathrm{oEF}}$, $\pr_\th$, $\pr_{\phiOEF}$.

\subsubsection{Boundary-defining functions and conormal spaces}

It will be convenient for us to define the following boundary-defining
function for $\EventHorizon$:
\begin{equation*}
  \rho_{\EventHorizon}\vcentcolon= r-M.
\end{equation*}
The function $\rho_{\EventHorizon}$ can be used as an alternative radial coordinate in the iEF coordinate system $(v, \rho_{\EventHorizon}, \theta, \phiIEF)$ since $\partial_{\rho_{\EventHorizon}} = \partial_r^{\mathrm{iEF}}$. This viewpoint lets us regard the horizon as a boundary hypersurface and measure regularity using vector fields tangent to that boundary.

Similarly, it will be convenient to define the following boundary-defining function for $\II$:
\begin{equation*}
  \rho_{\II} \vcentcolon= r^{-1}. 
\end{equation*}
The function $\rho_{\II}$ can be used as an alternative radial coordinate in the oEF coordinate system $(u, \rho_{\II}, \theta, \phiOEF)$.
Observe that in particular 
\begin{equation*}
  \partial_{r}^{\mathrm{oEF}} = -\rhoI^2\partial_{\rhoI},\qquad
  r\partial_{r}^{\mathrm{oEF}} = -\rhoI\partial_{\rhoI}.
\end{equation*}

For $\Delta=(r-M)^2$, using the boundary defining function $\rho_{\EventHorizon}$, we can write the inverse metric \zcref{eq:inverse-metric:ingoing-EF} as 
\begin{equation}\label{eq:inverse-iEF-rho}
  \abs*{q}^2\Metric_{\mathrm{iEF}}^{-1} =
  2\left((r^2+a^2)\partial_v + a\partial_{\phiIEF}\right)\partial_{\rho_{\EventHorizon}} + (\rho_{\EventHorizon}\partial_{\rho_{\EventHorizon}})^2 + \partial_{\theta}^2 + \left(a\sin\theta\partial_v + \frac{1}{\sin\theta}\partial_{\phiIEF}\right)^2,
\end{equation}
and using the boundary defining function $\rhoI$, we can write the inverse
metric \zcref{eq:inverse-metric-oEF} as 
\begin{equation}\label{eq:inverse-oEF-rho}
  \abs*{q}^2\Metric^{-1}_{\mathrm{oEF}}
  ={} 2\left( \frac{r^2+a^2}{r^2}\partial_u + \frac{a}{r^2}\partial_{\phiOEF} \right)\partial_{\rho_{\II}}
  +  \Upsilon\left( \rho_{\II}\partial_{\rho_{\II}} \right)^2
  + \partial_{\theta}^2
  + \left(a\sin\theta\partial_u  + \frac{1}{\sin\theta}\partial_{\phiOEF}\right)^2, 
\end{equation}
where $\Upsilon\vcentcolon= \frac{\Delta}{r^2}$.

We observe that
\begin{itemize}
\item $\rhoH\partial_{\rhoH}, \partial_v, \partial_{\theta}, \partial_{\phiIEF}$ span\footnote{Up to singularities at the north/south poles of the spheres, where a different frame on the spheres is needed.}, the space $\mathcal{V}_{b,\EventHorizon}(\Manifold)$ of all smooth vector fields on $\Manifold$ which are tangent to the boundary $\EventHorizon$ of $\Manifold$,
\item $\rhoI\partial_{\rhoI}, \partial_u, \partial_{\theta},
  \partial_{\phiOEF}$ span the space
  $\mathcal{V}_{b,\NullInfinity}(\Manifold)$ of all smooth vector fields
  on $\Manifold$ which are tangent to the boundary $\NullInfinity$ of
  $\Manifold$. 
\end{itemize}

For $m\in \mathbb{N}$, we will use $\bDiffH^m(\Manifold)$ (or
$\bDiffI^m(\Manifold)$) to denote the space consisting of all finite
sums of up to $m$-fold products of members of
$\mathcal{V}_{b,\EventHorizon}(\Manifold)$ (or
$\mathcal{V}_{b,\NullInfinity}$ respectively). We write
$\bDiffH(\Manifold)=\cup_{m\geq 0} \bDiffH^m(\Manifold)$ and similarly
for $\bDiffI(\Manifold)$. We also define
\[\bDiffHI^m(\Manifold)=\bDiffH^m(\Manifold\cap \{r \leq 4M\}) \cup \bDiffI^m(\Manifold\cap \{ r \geq 4M\}).\]

For $\alpha\in \Real$, we define the \emph{conormal (to
  $\EventHorizon$) space}
$\conormalSpaceH{\alpha}(\Manifold) =
\rhoH^{\alpha}\conormalSpaceH{0}(\Manifold)$ to consist of all
$u\in \rhoH^{\alpha}L^{\infty}(\Manifold)$ such that
$Au\in \rhoH^{\alpha}L^{\infty}(\Manifold)$ for all
$A\in \bDiffH(\Manifold)$.  Similarly, for $\alpha\in \Real$, we
define the \emph{conormal (to $\NullInfinity$) space}
$\conormalSpaceI{\alpha}(\Manifold) =
\rhoI^{\alpha}\conormalSpaceI{0}(\Manifold)$ to consist of all
$u\in \rhoI^{\alpha}L^{\infty}(\Manifold)$ such that
$Au\in \rhoI^{\alpha}L^{\infty}(\Manifold)$ for all
$A\in \bDiffI(\Manifold)$.  Thus
$u \in \conormalSpaceH{\alpha}(\Manifold)$ (or
$u \in \conormalSpaceI{\alpha}(\Manifold)$) means that $u$ has size
$\rhoH^{\alpha}$ (or $\rhoI^{\alpha}$ respectively) and remains of the
same size after repeated differentiation by vector fields tangent to
the horizon (or null infinity respectively).

To help us keep track of the behavior of
global functions at both boundaries, we will denote by
$\conormalSpaceHI{\alpha, \beta}(\Manifold) =
\conormalSpaceH{\alpha}(\Manifold)\cap \conormalSpaceI{\beta}(\Manifold)$.
These spaces are natural for tracking polyhomogeneous behavior at the horizon and null infinity.
We record the following elementary properties:
\begin{enumerate}
\item Action by differential operators: if $V\in \rhoH^a\bDiffH\bigcap \rhoI^b\bDiffI$ and $u\in \conormalSpaceHI{\alpha,\beta}$, then $Vu\in \conormalSpaceHI{\alpha+a, \beta+b}$,
\item Product of conormal spaces: if $u \in \conormalSpaceHI{\alpha_1,\beta_1}, v\in\conormalSpaceHI{\alpha_2, \beta_2}$, then $uv\in \conormalSpaceHI{\alpha_1+\alpha_2,\beta_1+\beta_2}$,
\item Addition of conormal spaces: if $u \in \conormalSpaceHI{\alpha_1,\beta_1}, v\in\conormalSpaceHI{\alpha_2, \beta_2}$, then $u+v\in \conormalSpaceHI{\min\{\alpha_1,\alpha_2\},\min\{\beta_1,\beta_2\}}$.
\end{enumerate}

\subsubsection{Killing vectorfields}
The coordinate vectorfields
\begin{align*}
  T\vcentcolon=\partial_t=\partial_v=\partial_u, \qquad \Phi\vcentcolon=\partial_\phi^{\mathrm{BL}}=\partial_{\phiIEF}=\partial_{\phiOEF}
\end{align*}
are manifestly Killing for the Kerr--Newman metric. 
The stationary Killing vectorfield $T$ is asymptotically timelike as $r \to \infty$, and spacelike close to the horizon, in the ergoregion $\{ \Delta - a^2\sin^2\th <0\}$.
The vectorfield 
\begin{equation*}
  \That\vcentcolon=T+\frac{a}{r^2+a^2} \Phi
\end{equation*}
satisfies $\g(\That, \That)=-\frac{\Delta|q|^2}{(r^2+a^2)^2}$, so it is timelike in the exterior region $\mathcal{D}$ and null on the horizon $\HH$. This is the natural timelike combination of the stationary and axial Killing fields in the exterior and will serve as the model timelike multiplier away from the horizon.

We use the following shorthanded notations for the rescaled vector fields
\begin{align}\label{eq:def-VHH-VII}
  \HawkingHorizon\vcentcolon=(r^2+a^2)T +a\Phi=(r^2+a^2)\That, \qquad \InfinityHawking\vcentcolon=\frac{(r^2+a^2)}{r^2}T +\frac{a}{r^2}\Phi =\frac{r^2+a^2}{r^2}\That ,
\end{align}
where the weights are chosen so that $\HawkingHorizon$ is smooth at the horizon and $\InfinityHawking$ is the version adapted to null infinity.
From \zcref{eq:coord-change1} and \zcref{eq:relation-BL-oEF}, we have
\begin{align}\label{eq:relation-BL-iEF}
  \pr_r^{\mathrm{BL}}
  =\pr_r^{\mathrm{iEF}}
  +\rhoH^{-2}\HawkingHorizon, \qquad \pr_r^{\mathrm{BL}}
  =\pr_r^{\mathrm{oEF}}
  -\Upsilon^{-1}\InfinityHawking.
\end{align}

We define the Hawking Killing vectorfield 
\begin{align*}
  \That_{\HH}\vcentcolon=T+\frac{a}{r_{+}^2+a^2} \Phi,
\end{align*}
which satisfies
\begin{align*}
  \g(\That_{\HH}, \That_{\HH}) 
  ={}&  -\frac{\Delta}{|q|^2}\frac{(r_{+}^2+a^2\cos^2\th)^2}{(r_{+}^2+a^2)^2}+\frac{a^2\sin^2\theta}{|q|^2(r_{+}^2+a^2)^2}\big(r^2- r_{+}^2 \big)^2 \leq  \frac{-\Delta r_{+}^4+a^2 \big(r^2- r_{+}^2 \big)^2}{|q|^2(r_{+}^2+a^2)^2}.
\end{align*}
In extremal Kerr--Newman,
$-\Delta r_{+}^4+a^2 \big(r^2- r_{+}^2 \big)^2=-(r-M)^2\big(
M^4-a^2(r+M)^2\big)$, so for $\frac{|a|}{M}\ll 1$, the Hawking
vectorfield $\That_{\HH}$ is timelike sufficiently close to the event
horizon. We record this because it provides a timelike multiplier up
to the horizon in the slowly rotating regime.

\subsubsection{Trapped null geodesics}\label{section:trapped-null-geodesics}

Since the Morawetz vector field will be designed to degenerate
precisely at trapping, we record the form of trapped null geodesics
here.

Let $\ga(\la)$ be a null geodesic in Kerr--Newman spacetime. Using the
expression for the inverse of the metric given by
\zcref{eq:inverse-metric-BL}, along $\ga(\la)$, since
$\g(\gadot, \gadot)=0$ we have, with $\gadot_r=\pr_r^\a \gadot_\a$,
$\gadot_t=\pr_t^\a \gadot_\a$, $\gadot_\phi=\pr_\phi^\a \gadot_\a$
\begin{align*}
  0= |q|^2 \g^{\a\b} \gadot_\a \gadot_\b=\big( \De \pr_r^\a \pr_r^\b +\frac{1}{\De}\RR^{\a\b}\big)\gadot_\a \gadot_\b=
  \De\gadot_r \gadot_r+ \frac{1}{\De} \RR^{\a\b}\gadot_\a \gadot_\b
\end{align*}
with
\begin{equation}\label{eq:RR-ab-geodesics}
  \RR^{\a\b}\gadot_\a \gadot_\b= -(r^2+a^2) ^2\gadot_t \gadot_t- 2a(r^2+a^2) \gadot_t \gadot_\phi- a^2  \gadot_\phi \gadot_\phi  +\De O^{\a\b} \gadot_\a \gadot_\b.
\end{equation}

Since $T$ and $\Phi$ are Killing vectorfields we deduce that
$\gadot_t=\g(\gadot, T)$ and $\gadot_\phi= \g(\gadot, \Phi)$ are
constants of motion, i.e. constants along $\ga$, respectively called
the energy and the azimuthal angular momentum. We write
$\e\vcentcolon=-\g(\gadot, T)$ and $\lz=-\g(\gadot, \Phi)$. With this
convention, $\e$ is positive for future-directed null geodesics in the
asymptotically flat region.  We also define\footnote{Observe that
  $\k^2$ is a positive constant of motion by definition of $K$. }
$\k^2\vcentcolon=K^{\a\b} \gadot_\a\gadot_\b$, where $K$ denotes the
Carter Killing tensor in Kerr--Newman. Since $K$ is a Killing tensor,
$\k^2$ is also a constant of motion.

With these constants from \zcref{eq:RR-ab-geodesics}  we have
\begin{align*}
  \RR(r;   \e,\lz, \k^2)\vcentcolon=\RR^{\a\b}\gadot_\a \gadot_\b ={}& -(r^2+a^2) ^2 \e^2- 2a(r^2+a^2) \e  \lz - a^2  \lz^2  +\De \k^2\\
  ={}&-\big(( r^2+a^2)\e + a \lz\big)^2+\De\k^2,
\end{align*}
which is only a function of $r$ along  any fixed  null geodesic. 
Going back to the equation for null geodesics  we  infer that
\begin{align*}
  \De\Big(\frac{dr}{d\la} \Big)^2  =-\RR(r;  \e,\lz, \k^2),
\end{align*}
which is the radial equation for a null geodesic with constants of motion
$\e$, $\lz$ $\k^2$.

There exist null geodesics along which $\RR(r; \e,\lz, \k^2)=0$
i.e. $r$ remains constant. These are called spherical null geodesics, or
trapped null geodesics.  The $r$ values for which such solutions are
possible must then verify the equations
\begin{align*}
  \RR(r;  \e,\lz, \k^2)=\pr_r\RR(r;  \e,\lz, \k^2)=0.
\end{align*}

\begin{lemma}[Trapped null geodesics equation]\label{lemma:trapped-geodesics}
  All spherical null geodesics in extremal Kerr--Newman spacetime are
  either at the event horizon or are given by the equation
  \begin{equation}\label{eq:trapped-region-KN}
    \widetilde{\TT}_{\e, \lz}\vcentcolon= ( r^2-2Mr-a^2)\e - a \lz=0,
  \end{equation}
  which has a unique root $\rTrapping(\e, \lz)>M$.  Moreover, the
  exterior spherical null geodesics are unstable, in the sense that
  $\pr^2_r\RR(r; \e,\lz, \k^2) \leq 0$ at such orbits.
\end{lemma}
\begin{proof} 
  We solve for 
  \begin{align*}
    \RR(r; \e,\lz, \k^2)=-\big(( r^2+a^2)\e + a \lz\big)^2+(r-M)^2\k^2=0,\\
    \pr_r\RR(r;  \e,\lz, \k^2)=-4r \e \big(( r^2+a^2)\e + a \lz\big)+2\big(r-M\big)\k^2=0.
  \end{align*}
  Writing from the second equation $(r-M)\k^2=2r \e  \big(( r^2+a^2)\e + a \lz\big)$, and substituting in the first equation, we obtain
  \begin{align}\label{eq:intermediate-step-trapping}
    0=\big(( r^2+a^2)\e + a \lz\big)  \big(( r^2-2Mr-a^2)\e - a \lz \big).
  \end{align}
  Therefore, either the null geodesics satisfy $( r^2-2Mr-a^2)\e - a \lz=0$ or $( r^2+a^2)\e + a \lz=0$. 
  The latter implies $(r-M)\k^2=0$, so either $\k^2=0$ or $r-M=0$. 

  We now show that the trapped null geodesics are unstable, i.e. that $\pr^2_r\RR(r; \e,\lz, \k^2) \leq 0$; equivalently, the radial potential has a local maximum at the trapped orbit. We compute
  \begin{align*}
    - \pr^2_r\RR(r;  \e,\lz, \k^2)=4 \e  \big(( r^2+a^2)\e + a \lz\big)+8r^2 \e^2-2\k^2.
  \end{align*}
  Substituting from $\pr_r\RR=0$ that
  $2\k^2=\frac{4r}{r-M} \e  \big(( r^2+a^2)\e + a \lz\big)$ for
  trapped geodesics, we obtain
  \begin{align*}
    - \pr^2_r\RR(r;  \e,\lz, \k^2)= -\frac{4M}{r-M} \e  \big(( r^2+a^2)\e + a \lz\big)+8r^2 \e^2.
  \end{align*}
  Substituting from \zcref{eq:intermediate-step-trapping} that $\big(( r^2+a^2)\e + a \lz\big)=2r(r-M) \e$ for trapped geodesics, we conclude
  \begin{align*}
    - \pr^2_r\RR(r;  \e,\lz, \k^2)= 8r(r-M) \e^2\geq 0,
  \end{align*}
  as stated.
\end{proof}

In the axisymmetric setting $\lz=0$, the trapping condition reduces to
\begin{align*}
  \widetilde{\TT}(r)\vcentcolon=r^2-2Mr-a^2
\end{align*}
so trapping occurs at the single radius
$r_{\trap}=M+\sqrt{M^2+a^2}$. Outside axial
symmetry, for $|a|\neq 0$ there are null geodesics with constant $r$
for an open range of $r$.  In the very slowly rotating extremal case
the trapping region is localized near $r^{eRN}_{trap}\vcentcolon=2M$, known as
the photon sphere of extremal Reissner-Nordstr\"om.

\begin{remark}\label{remark-derivative-TT}
  From the computations in \zcref{lemma:trapped-geodesics}, one can
  see that the polynomial $\widetilde{\TT}$ can be obtained as the
  derivative of the geodesic potential $\frac{\Delta}{(r^2+a^2)^2}$,
  as
  \begin{equation*}
    z=\frac{(r-M)^2}{(r^2+a^2)^2}, \qquad  \pr_rz= -2\frac{(r-M)\widetilde{\TT}}{(r^2+a^2)^{3}}=-\frac{2\TT}{(r^2+a^2)^{3}},
  \end{equation*}
  for $\TT\vcentcolon=(r-M)\widetilde{\TT}$. This identity explains
  why $\TT$ will arise naturally in the construction of Morawetz
  weights below. The fact that $\TT$ vanishes at the event horizon is
  a feature of the extremal case.
\end{remark}

\subsection{The wave equation on extremal black holes}\label{sec:wave-eq}

The wave operator for a scalar function $\psi$ on a Lorentzian
manifold $(\Manifold, \Metric)$ is given by 
\begin{align*}
  \Box_{\Metric}\psi=\frac{1}{\sqrt{-\det \g}}\partial_\a ((\sqrt{-\det\g}) \g^{\a\b} \partial_\b \psi).
\end{align*}
In what follows, we consider the inhomogeneous wave equation on a
slowly rotating extremal Kerr--Newman spacetime,
\begin{equation}\label{eq:equation-extremal}
  \square_{\g}\psi=F, 
\end{equation}
where $\g$ denotes the metric of the extremal Kerr--Newman spacetime
with $e^2=M^2-a^2$ and $|a| \ll M$, and $F$ is a sufficiently regular
function.

We consider the Cauchy problem for the wave equation with initial data
prescribed on some spacelike hypersurface $\Sigma(\tau_1)$ (described
below) connecting the event horizon and null infinity, given by
\begin{align*}
  \psi |_{\Sigma(\tau_1)}=\psi_0 \in H^s_{\operatorname{loc}}(\Sigma(\tau_1)), \qquad N_{\Sigma(\tau_1)} \psi|_{\Sigma(\tau_1)}=\psi_1 \in H^{s-1}_{\operatorname{loc}}(\Sigma(\tau_1)), \qquad s > \frac 3 2 
\end{align*}
where $N_{\Sigma(\tau_1)}$ denotes the normal vector to $\Sigma(\tau_1)$. 

Moreover, recalling the boundary-defining function $\rhoI$ at null
infinity, we denote by $\widehat{\Sigma}(\tau)$ the manifold with
boundary obtained from $\Sigma(\tau)$ by adjoining the boundary
hypersurface $\{ \rhoI =0\}=S^2$. We assume that the radiation field
$\widecheck{\psi}=\rhoI^{-1}\psi=r\psi$ extends to
$\widehat{\Sigma}(\tau_1)$ and satisfies
\begin{align*}
  \widecheck{\psi}|_{\widehat{\Sigma}(\tau_1)} \in H^s(\widehat{\Sigma}(\tau_1)), \qquad N_{\widehat{\Sigma}(\tau_1)} (r\psi)|_{\widehat{\Sigma}(\tau_1)}\in H^{s-1}(\widehat{\Sigma}(\tau_1)), \qquad s> \frac 3 2.
\end{align*}

Standard results imply well-posedness and persistence of regularity
for the above Cauchy problem for the compactified equation, see for
example Proposition 3.2 in
\cite{gajicLatetimeAsymptoticsGeometric2023}. In particular, it
follows that $r\psi \in H^s(\widehat{\Sigma}(\tau))$, $s>\frac 3 2$,
for the time of existence of the solution. Since
$\widehat{\Sigma}(\tau)$ is 3-dimensional and $s>\frac 3 2$, Sobolev
embedding gives
$\widecheck{\psi}=r\psi \in L^\infty
(\widehat{\Sigma}(\tau))$. Therefore, towards null infinity we have
\[\lim_{r\to \infty} r \psi^2=\frac{1}{r}(r\psi)^2=0.\]

\subsubsection{Renormalized time foliation}
In this section, we specify our domain of integration and the various
boundary hypersurfaces and hypersurface normal conventions that we
use. We fix $\rho_\star < r_{+}$.

\begin{figure}[ht]
  \centering
  \begin{tikzpicture}[scale=0.7,every node/.style={scale=0.7}]


  \def \s{3} 
  \def \exts{0.2} 
  \def \t{0.4}
  \def \Tlen{.5}
  \def \lenNull{0.1} 
  \def \lenDomain{0.2} 

  \coordinate (tInf) at (0,\s); 
  \coordinate (EventZero) at (-\s,0); 
  \coordinate (CosmoZero) at (\s,0); 
  \coordinate (tNegInf) at (0,-\s);
  \coordinate (SigmaZeroEvent) at (-\s + 0.15*\s, 0.15*\s);
  \coordinate (SigmaTEvent) at ($(SigmaZeroEvent) + (\lenDomain*\s, \lenDomain*\s)$);
  \coordinate (SigmaZeroCosmo) at (\s - 0.15*\s, 0.15*\s);
  \coordinate (SigmaTCosmo) at ($(SigmaZeroCosmo) + (-\lenDomain*\s, \lenDomain*\s)$);

  \draw[shorten >= -10,name path=EventFuture] (tInf) --
  node[pos=0.5,left]{$\EventHorizon$} (EventZero) ;  
  \draw[shorten >= -10,name path=CosmoFuture,dashed] (tInf) --
  node[inner sep=10pt,scale=1.0,pos=0.5,right]{$\NullInfinity$} (CosmoZero) ; 

  

  \draw[dashed] (tInf)    -- (-\s - \exts*\s,-\exts*\s);
  \draw[dashed] (tInf)    -- ( \s + \exts*\s,-\exts*\s);


  \path[name path=EventLowerBound] (-1.6*\s, 0.6*\s)-- (-\s, 0)--
  (tInf);
  \path[name path=CosmoLowerBound] (tInf) -- (\s, 0) -- (1.6*\s, 0.6*\s);

  \coordinate (SigmaZeroH) at ($(SigmaZeroEvent) + (4*\lenNull*\s, -2*\lenNull*\s)$);
  \coordinate (SigmaTH) at ($(SigmaTEvent) + (2*\lenNull*\s, -\lenNull*\s)$);
  \coordinate (SigmaZeroI) at ($(SigmaZeroCosmo) + (-4*\lenNull*\s, -2*\lenNull*\s)$);
  \coordinate (SigmaTI) at ($(SigmaTCosmo) + (-2*\lenNull*\s, -  \lenNull*\s)$);
  \def \bottombend{15}
  \def \topbend{10}
  \path[fill=blue, fill opacity=0.1, draw=blue, thick]
  (SigmaZeroEvent) to[bend right=\bottombend] (SigmaZeroH) to [bend left = 20] node[midway,below,opacity=1]{$\Sigma(\tau_1)$} (SigmaZeroI) to[out=20,in=-135] (SigmaZeroCosmo) --
  (SigmaTCosmo) to[out=-135,in=15] (SigmaTI) to [bend right=20] node[midway,above,opacity=1]{$\Sigma(\tau_2)$} (SigmaTH) to[bend left=\topbend] (SigmaTEvent) -- cycle;
  \draw[blue, dashed] (SigmaZeroH) to [bend left = 10] (SigmaTH);
  \draw[blue, dashed] (SigmaZeroI) to [bend right = 10] (SigmaTI);





  
  \node[scale=0.5,fill=white,draw,circle,label=above:$i^+$]at(tInf){};

  \node[label=left:${r=M}$]at(SigmaZeroEvent){};
  \node[label=right:${r=\infty}$]at(SigmaZeroCosmo){};

\end{tikzpicture}

  \caption{Schematic representation of the foliation $\Sigma(\tau)$. The region
    $\Manifold(\tau_1,\tau_2)$, highlighted in blue, is bounded by the hypersurfaces
    $\Sigma(\tau_1)$ and $\Sigma(\tau_2)$, the event horizon
    $\EventHorizon(\tau_1,\tau_2)$, and null infinity
    $\NullInfinity(\tau_1,\tau_2)$.}
\end{figure}
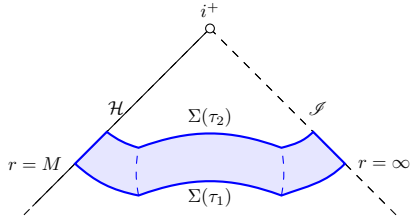

\begin{definition}[Renormalized time foliation]\label{def:tau}
  We define a foliation $\Sigma(\tau)=\{ \tau = \operatorname{constant}\}$
  with
  \begin{equation*}
    \tau=
    \begin{cases}
      v-h_{\EventHorizon}(r) \qquad \qquad \rhoH \leq \rho_\star, \\
      t \qquad \qquad \qquad \qquad \rhoH \geq \rho_\star, \quad \rhoI \geq r_{+}^{-2}\rho_\star, \\
      u+h_{\NullInfinity}(r) \qquad \qquad \rhoI \leq r_{+}^{-2}\rho_\star
    \end{cases}
  \end{equation*}
  where $h_{\HH}(r)$ and $h_{\NullInfinity}(r)$ are functions satisfying
  \begin{align}
    h_{\HH}'(r)={}&\frac{2a^2}{r^2+a^2 +\sqrt{(r^2+a^2)^2-8a^2\Delta}} , \qquad h_{\NullInfinity}'(r)={}\frac{a^2}{r^2\Upsilon},\label{eq:condition-h-II}
  \end{align}
  and such that the $\tau$ function is continuous at $\rhoH = \rho_\star$ and $\rhoI = r_{+}^{-2}\rho_\star$.
\end{definition}

The derivatives of $h_{\EventHorizon}(r)$ and $h_{\NullInfinity}(r)$
are chosen so that the level sets of $\tau$ are spacelike and
asymptotically null towards $\NullInfinity$.  Observe that the foliation $\Sigma(\tau)$ in the
case of extremal Reissner-Nordstr\"om ($a=0$) agrees with
$v=\text{const}$-type slices near the horizon, $t=\text{const}$ in the
intermediate region, and $u=\text{const}$-type slices near null
infinity.

We will use the following notation to denote the spacetime domains
\begin{equation*}
  \Manifold(\tau_1,\tau_2) \vcentcolon= \Manifold\cap \{\tau\in [\tau_1,\tau_2]\}. 
\end{equation*}
The boundaries of $\Manifold(\tau_1,\tau_2)$ are given by
\begin{equation*}
  \Sigma(\tau_1)\vcentcolon=\{\tau=\tau_1\}, \quad  \Sigma(\tau_2)\vcentcolon=\{\tau=\tau_2\}, \quad  \EventHorizon(\tau_1,\tau_2) \vcentcolon = \EventHorizon\cap \{\tau\in [\tau_1,\tau_2]\},\quad
  \NullInfinity(\tau_1,\tau_2) \vcentcolon =  \NullInfinity\cap \{\tau\in [\tau_1,\tau_2]\}.
\end{equation*}
We also denote for some $\rho_0 \leq \rho_\star$ and
$\rho_1 \leq r_{+}^{-2}\rho_\star$ sufficiently small to be chosen
later, the following near boundaries regions:
\begin{align*}
  \Sigma_{\EventHorizon}(\tau)&=\Sigma(\tau) \cap \{ \rhoH \le \rho_0\}, \qquad \Sigma_{\NullInfinity}(\tau)=\Sigma(\tau) \cap \{ \rhoI\le \rho_1\}, \\
  \MM_{\EventHorizon}(\tau_1, \tau_2) &= \MM(\tau_1, \tau_2) \cap \{ \rhoH \le \rho_0\}, \qquad \MM_{\NullInfinity}(\tau_1, \tau_2) = \MM(\tau_1, \tau_2) \cap\{ \rhoI\le \rho_1\}.
\end{align*}
For any hypersurface defined by a scalar function $f$, we adopt the
convention
\begin{equation*}
  N_{\{f=\mathrm{const}\}}
  \vcentcolon= -\g^{\alpha\beta}(\partial_\alpha f)\,\partial_\beta
\end{equation*}
with natural induced volume form defined by interior derivative. In
particular, we use the following conventions for normals to the
boundary hypersurfaces
\begin{equation}
  \label{eq:boundary-normal-convention}
  N_{\Sigma(\tau)} \vcentcolon= -\Metric^{\alpha\beta}\partial_{\alpha}\tau\partial_{\beta},\qquad
  N_{\EventHorizon} \vcentcolon= -\Metric^{\alpha\beta}\partial_{\alpha}r\partial_{\beta},\qquad N_{\NullInfinity} \vcentcolon= -\Metric^{\alpha\beta}\partial_{\alpha}r\partial_{\beta}.
\end{equation}

\begin{lemma}\label{lemma:normal-vector-tau}
  The level sets $\Sigma(\tau)$ of $\tau$ in \zcref{def:tau} form a
  spacelike foliation in the exterior region whose leaves become
  asymptotically null towards $\NullInfinity$.  Moreover, in the three
  regions where $\tau$ is defined by $v-h_{\EventHorizon}(r)$, $t$,
  and $u+h_{\NullInfinity}(r)$, the timelike normal vector
  $N_{\Sigma(\tau)}$ is given respectively by
  \begin{equation*}
    |q|^2    N_{\Sigma(\tau)}=
    \begin{cases}
      \left(-(r^2+a^2)+\De h_{\EventHorizon}'\right)\partial_{\rhoH}+
      h_{\EventHorizon}' \HawkingHorizon 
      -a\sin\th \renormpphi  & \rhoH \leq \rho_\star, \\
      \frac{(r^2+a^2)}{\De} \HawkingHorizon -a\sin\th\renormpphi  & \rhoH \geq \rho_\star, \quad \rhoI \geq r_{+}^{-2}\rho_\star, \\
      -\partial_{\rhoI}+ \frac{a^2}{\Upsilon} \InfinityHawking-a\sin\th\renormpphi  &  \rhoI \leq r_{+}^{-2}\rho_\star,
    \end{cases}
  \end{equation*}
  where $h_{\EventHorizon}'$ is given by
  \zcref{eq:condition-h-II}. Moreover, the normal vectors of
  $\EventHorizon$ and $\NullInfinity$ are given by
  \begin{align}\label{eq:NHH-NII}    
    |q|^2  N_{\EventHorizon} 
    ={} - \HawkingHorizon,\qquad
    N_{\NullInfinity}={} -\InfinityHawking .
  \end{align}
\end{lemma}
\begin{proof}
  Near the horizon, we compute $d\tau=dv-h_{\EventHorizon}'(r)\,dr$
  and therefore using \zcref{eq:inverse-metric:ingoing-EF}, we compute
  \begin{align*}
    \abs*{q}^2\Metric_{\mathrm{iEF}}^{-1}(d\tau,d\tau)
    ={}& \De \bigl(h_{\EventHorizon}'(r)\bigr)^2
         -2(r^2+a^2)h_{\EventHorizon}'(r)
         +a^2\sin^2\th\\
    \leq{}& \De \bigl(h_{\EventHorizon}'(r)\bigr)^2 -2(r^2+a^2)h_{\EventHorizon}'(r) +a^2.
  \end{align*}
  Analyzing the roots of 
  \[
    \De x^2-2(r^2+a^2)x+a^2=0,
  \]
  we see that the above is negative whenever 
  \[
    (r^2+a^2)-\sqrt{(r^2+a^2)^2-a^2\De}\leq \Delta h_{\EventHorizon}'(r)\leq (r^2+a^2)+\sqrt{(r^2+a^2)^2-a^2\De}
  \]
  which is clearly satisfied from the condition
  \zcref{eq:condition-h-II}. This implies that for
  $\rhoH \leq \rho_\star$, we have
  $\Metric_{\mathrm{iEF}}^{-1}(d\tau,d\tau)<0$.  Using
  \zcref{eq:inverse-metric:ingoing-EF} again, we compute
  \begin{align*}
    \abs*{q}^2\Metric^{-1}_{\mathrm{iEF}}(dv,\cdot)
    ={}&
         (r^2+a^2)\partial_r+a^2\sin^2\th\,\partial_v+a\partial_{\phiIEF}, \\
    \abs*{q}^2\Metric^{-1}_{\mathrm{iEF}}(dr,\cdot)
    ={}&
         (r^2+a^2)\partial_v+\De\,\partial_r+a\partial_{\phiIEF}.
  \end{align*}
  Hence
  \begin{align*}
    \abs*{q}^2\D\tau
    ={}&
         \abs*{q}^2\Metric^{-1}_{\mathrm{iEF}}(dv,\cdot)
         -h_{\HH}'(r)\abs*{q}^2\Metric^{-1}_{\mathrm{iEF}}(dr,\cdot)\\
    ={}&{}
         \left(
         -(r^2+a^2)h_{\HH}'(r)+a^2\sin^2\th
         \right)\partial_v- \left(\De h_{\EventHorizon}'(r)-(r^2+a^2)\right)\partial_{\rhoH}
         +a\left(
         1-h_{\HH}'(r)
         \right)\partial_{\phiIEF}.
  \end{align*}
  Using that $  \pr_v =\frac{1}{|q|^2}\HawkingHorizon -\frac{a\sin\th}{|q|^2} \renormpphi$, $\partial_{\phiIEF}=\frac{(r^2+a^2)\sin\th}{|q|^2}\renormpphi-\frac{1}{|q|^2}a\sin^2\th \HawkingHorizon$, we obtain the stated formula.

  In the intermediate region, we have $d\tau=dt$ and therefore using
  \zcref{eq:inverse-metric-BL} we compute
  \begin{align*}
    |q|^2 \g_{\mathrm{BL}}^{-1}(d\tau, d\tau)&={}-\frac{(r^2+a^2)^2}{\Delta}+a^2\sin^2\th= \frac{-(r^2+a^2)^2+\Delta a^2\sin^2\th}{\Delta}\\
                                             & \leq \frac{-(r^2+a^2)^2+\Delta a^2}{\Delta} = -\frac{r^4+a^2r^2+a^2M (2r-M)+a^4}{\Delta}  < 0
  \end{align*}
  and
  \[
    \abs*{q}^2\D\tau
    ={}\abs*{q}^2\Metric^{-1}_{\mathrm{BL}}(dt,\cdot)
    =
    \left(-\frac{(r^2+a^2)^2}{\De}+a^2\sin^2\th\right)\partial_t
    +a\left(1-\frac{r^2+a^2}{\De}\right)\partial_\phi.
  \]
  Using that
  $ \pr_t =\frac{1}{|q|^2}\HawkingHorizon -\frac{a\sin\th}{|q|^2}
  \renormpphi$,
  $\partial_{\phi}=\frac{(r^2+a^2)\sin\th}{|q|^2}\renormpphi-\frac{1}{|q|^2}a\sin^2\th
  \HawkingHorizon$, we obtain the stated formula.

  Near null infinity, we compute $d\tau=du + h_{\NullInfinity}'(r) dr$
  and therefore using \zcref{eq:inverse-metric-oEF}
  \begin{align*}
    \abs*{q}^2\Metric^{-1}_{\mathrm{oEF}}(d\tau,d\tau)
    ={}&
         \De \bigl(h_{\NullInfinity}'(r)\bigr)^2
         -2(r^2+a^2)h_{\NullInfinity}'(r)
         +a^2\sin^2\th\\
       &\leq \De \bigl(h_{\NullInfinity}'(r)\bigr)^2
         -2(r^2+a^2)h_{\NullInfinity}'(r)
         +a^2,
  \end{align*}
  which, as above, is negative as long as 
  \begin{align*}
    \frac{r^2+a^2}{r^2}-\frac{\sqrt{(r^2+a^2)^2-a^2\De}}{r^2}\leq  \Upsilon h_{\NullInfinity}'(r)\leq \frac{r^2+a^2}{r^2}+\frac{\sqrt{(r^2+a^2)^2-a^2\De}}{r^2}.
  \end{align*}
  The above is satisfied by the condition \zcref{eq:condition-h-II} since 
  \[
    \frac{a^2}{r^2}\geq \frac{r^2+a^2}{r^2}-\frac{\sqrt{(r^2+a^2)^2-a^2\De}}{r^2}= \frac{a^2\Upsilon}{r^2+a^2+\sqrt{(r^2+a^2)^2-a^2\De}}.
  \]
  Moreover, since 
  \begin{align*}
    \lim_{r\to \infty} \Big[\frac{a^2}{r^2}- \frac{a^2\Upsilon}{r^2+a^2+\sqrt{(r^2+a^2)^2-a^2\De}}\Big]=0
  \end{align*}
  we have that $\Metric^{-1}_{\mathrm{oEF}}(d\tau,d\tau) \to0^-$ as $r\to\infty$.
  i.e. the foliation is asymptotically null as $\rhoI \to 0$.
  Using \zcref{eq:inverse-metric-oEF} again, we compute 
  \begin{align*}
    \abs*{q}^2\Metric^{-1}_{\mathrm{oEF}}(du,\cdot)
    ={}&
         -(r^2+a^2)\partial_r+a^2\sin^2\th\,\partial_u+a\partial_{\phiOEF},\\
    \abs*{q}^2\Metric^{-1}_{\mathrm{oEF}}(dr,\cdot)
    ={}&
         -(r^2+a^2)\partial_u+\De\,\partial_r-a\partial_{\phiOEF}.
  \end{align*}
  Hence
  \begin{align*}
    \abs*{q}^2\D\tau
    ={}&
         \abs*{q}^2\Metric^{-1}_{\mathrm{oEF}}(du,\cdot)
         +h_{\NullInfinity}'(r)\abs*{q}^2\Metric^{-1}_{\mathrm{oEF}}(dr,\cdot)\\
    ={}&{}
         \left(
         -(r^2+a^2)h_{\NullInfinity}'(r)+a^2\sin^2\th
         \right)\partial_u+ \left(\De h_{\NullInfinity}'(r)-(r^2+a^2)\right)\partial_{r}
         +a\left(
         1-h_{\NullInfinity}'(r)
         \right)\partial_{\phiOEF}.
  \end{align*}
  Using that
  $ \pr_u =\frac{r^2}{|q|^2}\InfinityHawking -\frac{a\sin\th}{|q|^2}
  \renormpphi$,
  $\partial_{\phiOEF}=\frac{(r^2+a^2)\sin\th}{|q|^2}\renormpphi-\frac{r^2}{|q|^2}a\sin^2\th
  \InfinityHawking$ and
  $ \partial_{r}^{\mathrm{oEF}} = -\rhoI^2\partial_{\rhoI}$, we obtain
  \begin{align*}
    \abs*{q}^2N_{\Sigma_{\rhoI\le r_{+}^{-2}\rho_\star}(\tau)}
    ={}&{} \big(\Upsilon h_{\NullInfinity}'(r)-\frac{r^2+a^2}{r^2}\big)\partial_{\rhoI}+r^2 h_{\NullInfinity}'(r) \InfinityHawking-a\sin\th \renormpphi.
  \end{align*}
  The choice \zcref{eq:condition-h-II} makes the coefficient of
  $\partial_{\rhoI}$ equal to $-1$, and we obtain the stated
  formula. Finally, using \zcref{eq:inverse-metric:ingoing-EF} and
  \zcref{eq:inverse-metric-oEF} respectively we can similarly compute
  $N_\EventHorizon$ and $N_\NullInfinity$.
\end{proof}

We use the following convention for integrals over the boundary
hypersurfaces
\begin{equation*}
  \begin{split}
    \int_{\Sigma(\tau)}f ={}& \int_{r=M}^{r=\infty}\int_{\Sphere^2}f\, |q|^2drd\mathring{\gamma}
                              ,\\
    \int_{\EventHorizon(\tau_1,\tau_2)}f
    ={}& \int_{\tau_1}^{\tau_2}\int_{\Sphere^2}f\, |q|^2d\tau d\mathring{\gamma},\\
    \int_{\NullInfinity(\tau_1,\tau_2)}f
    ={}& \int_{\tau_1}^{\tau_2}\int_{\Sphere^2}f\, |q|^2d\tau d\mathring{\gamma},
  \end{split}  
\end{equation*}
where $d\mathring{\gamma}$ denotes the volume form of the round unit sphere $\mathbb{S}^2$.

\subsubsection{Trapping region}

Let $A_1, A_2, A_3$ be parameters with $\frac{3}{2}A_1 < A_2< M < A_3$
such that for $|a| \ll M$ sufficiently small all trapping radii lie in
$[A_2+M, A_3+M]$, or $\rhoH \in [A_2, A_3]$. We denote
\begin{align*}
  \Manifold_{\operatorname{trap}}(\tau_1, \tau_2)=\MM(\tau_1, \tau_2) \cap \{ \rhoH \in [A_2, A_3]\}
\end{align*}
and $\MM_{\ntrap}(\tau_1, \tau_2)$ the complement of
$\Manifold_{\operatorname{trap}}(\tau_1, \tau_2)$ in
$\MM(\tau_1, \tau_2)$.

Similarly, for $|a| \ll M$ sufficiently small the ergoregion is
contained in $\{\rhoH < A_1\}$ and the Hawking Killing vectorfield
$\That_{\HH}$ is timelike in $\{\rhoH< A_1\}$. In particular, we
choose $A_1, A_2, A_3$ so that the ergoregion and trapping region are
separated in physical space.

\subsubsection{Expressions and commutators for the wave operator}

Expressing the divergence form of $\Box_\Metric$ in \iEF{} and \oEF{}
coordinates and using the the expression for the inverse metrics
\zcref{eq:inverse-iEF-rho} and \zcref{eq:inverse-oEF-rho}, we deduce
that the wave operator is given by
\begin{align}
  |q|^2   \Box_{\Metric}={}&\left( \rhoH\partial_{\rhoH} \right)^2 + \rhoH\partial_{\rhoH}+2\HawkingHorizon\partial_{\rhoH}+2r\pr_v +\lapp_{\SSS^2}+a^{2}\sin^{2}\theta \pr_v^2+2a\pr_v\pr_{\phiIEF},\label{wave-iEF}\\
  |q|^2\Box_{\Metric}
  ={}&     \Upsilon(\rhoI\partial_{\rhoI})^2+\big( \Upsilon- \frac{\Delta'}{r} \big)(\rhoI\partial_{\rhoI})+2\InfinityHawking\partial_{\rhoI}-2\rhoI^{-1}\partial_u+\lapp_{\SSS^2} +2a \partial_u\partial_{\phiOEF}+a^2\sin^2\th \partial_u^2,\label{eq:wave-oEF}
\end{align}
where
$\lapp_{\SSS^2}=\frac{1}{\sin\th}\partial_\th(\sin\th
\partial_\th)+\frac{1}{\sin^2\th}\partial_\phi^2$ denotes the standard
Laplacian on the unit sphere.

We denote the rotational vectorfields as
\begin{align*}
  \RotationVF^{\mathrm{iEF}}_1 = \sin\phiIEF\partial_{\theta} + \cot\theta\cos\phiIEF \partial_{\phiIEF},\qquad
  \RotationVF^{\mathrm{iEF}}_2 = -\cos\phiIEF\partial_{\theta} + \cot\theta\sin\phiIEF \partial_{\phiIEF},\qquad
  \RotationVF^{\mathrm{iEF}}_3 = -\partial_{\phiIEF},\\
  \RotationVF^{\mathrm{oEF}}_1 = \sin\phiOEF\partial_{\theta} + \cot\theta\cos\phiOEF \partial_{\phiOEF},\qquad
  \RotationVF^{\mathrm{oEF}}_2 = -\cos\phiOEF\partial_{\theta} + \cot\theta\sin\phiOEF \partial_{\phiOEF},\qquad
  \RotationVF^{\mathrm{oEF}}_3 = -\partial_{\phiOEF},
\end{align*}
from which we can easily check that, in both cases,
\begin{equation}
  \label{eq:commutators:ingoing:rotation:rotation-commutation-prop}
  \left[ \RotationVF_1, \RotationVF_3 \right] = - \RotationVF_2,\qquad
  \left[ \RotationVF_2, \RotationVF_3 \right] =  \RotationVF_1.
\end{equation}

We list some commutators with the wave operator.

\begin{lemma}\label{lem:weighted-hierarchy-commutators}
  We have
  \begin{align*}
    \left[\abs*{q}^2\Box_{\Metric}, \rhoH\partial_{\rhoH}  \right]
    ={}& \abs*{q}^2\Box_{\Metric}
         - \left( \rhoH\partial_{\rhoH} \right)^2
         - \rhoH\partial_{\rhoH}
         - 2r\pr_v \\
       &-\lapp_{\SSS^2}-a^{2}\sin^{2}\theta \pr_v^2-2a\pr_v\pr_{\phiIEF}
         - 4r\rhoH\partial_{\rhoH}\pr_v
         - 2\rhoH\partial_v, \\
    \left[\abs*{q}^2\Box_{\Metric}, \rhoI\partial_{\rhoI}  \right]
    ={}&|q|^2\Box_{\Metric}-(\rhoI\partial_{\rhoI})^2
         +(\rhoI\partial_{\rhoI})\\
       &-\lapp_{\SSS^2}-a^2\sin^2\theta \pr_u^2
         -2a\pr_u\pr_{\phiOEF}+\conormalSpaceI{1}(\Manifold)\bDiffI^1+\conormalSpaceI{1}(\Manifold)(\rhoI \partial_{\rhoI})^{\leq 2},
  \end{align*}
  and 
  \begin{align*}
    \left[\abs*{q}^2\Box_{\Metric}, \RotationVF^{\mathrm{iEF}}_1  \right]&=-2a\left( \partial_{\rhoH}+ \partial_v \right)\RotationVF^{\mathrm{iEF}}_2 -2a^2\sin\th\cos\th\sin\phiIEF\partial_v^2, \\
    \left[\abs*{q}^2\Box_{\Metric}, \RotationVF^{\mathrm{iEF}}_2  \right]&= 2a\left( \partial_{\rhoH}+ \partial_v \right)\RotationVF^{\mathrm{iEF}}_1 +2a^2\sin\th\cos\th\cos\phiIEF\partial_v^2,\\
    \left[\abs*{q}^2\Box_{\Metric}, \RotationVF^{\mathrm{oEF}}_1  \right]&=-2a\left( \rhoI^2\partial_{\rhoI}+ \partial_u \right)\RotationVF^{\mathrm{oEF}}_2 -2a^2\sin\th\cos\th\sin\phiOEF\partial_u^2, \\
    \left[\abs*{q}^2\Box_{\Metric}, \RotationVF^{\mathrm{oEF}}_2  \right]&= 2a\left( \rhoI^2\partial_{\rhoI}+ \partial_u \right)\RotationVF^{\mathrm{oEF}}_1 +2a^2\sin\th\cos\th\cos\phiOEF\partial_u^2.
  \end{align*}
\end{lemma}
\begin{proof}
  Using the expression for the wave operator in \zcref{wave-iEF}, we compute 
  \begin{align*}
    \left[\abs*{q}^2\Box_{\Metric}, \rhoH\partial_{\rhoH}  \right]
    ={}& \left[ 2\HawkingHorizon\partial_{\rhoH} +2r\pr_v ,  \rhoH\partial_{\rhoH}\right]\\
       &  + \left[ \left( \rhoH\partial_{\rhoH} \right)^2 + \rhoH\partial_{\rhoH}
         +\lapp_{\SSS^2}+a^{2}\sin^{2}\theta \pr_v^2+2a\pr_v\pr_{\phiIEF}, \rhoH\partial_{\rhoH} \right]\\
    ={}& 2\HawkingHorizon\partial_{\rhoH}
         - 4r\rhoH\partial_{\rhoH}\pr_v
         - 2\rhoH\partial_v   ,
  \end{align*}
  where we used that $[\HawkingHorizon, \rhoH\partial_{\rhoH}]=-2r\rhoH \partial_v$.
  By rewriting the first term on the \RHS{} of
  the above using the wave operator itself we obtain the stated.
  Similarly, using \zcref{eq:wave-oEF}
  \begin{align*}
    \left[\abs*{q}^2\Box_{\Metric}, \rhoI\partial_{\rhoI}  \right]
    ={}& \left[ 2\InfinityHawking\partial_{\rhoI} -2\rhoI^{-1}\pr_u
         ,  \rhoI\partial_{\rhoI}\right]  + \left[ \Upsilon\left( \rhoI\partial_{\rhoI} \right)^2
         + \Upsilon\rhoI\partial_{\rhoI}
         - \frac{\Delta'}{r}(\rhoI\partial_{\rhoI}), \rhoI\partial_{\rhoI} \right]\\
       & +\left[ \lapp_{\SSS^2}+a^2\sin^2\theta \pr_u^2
         +2a\pr_u\pr_{\phiOEF}, \rhoI\partial_{\rhoI} \right] \\
    ={}& 2\InfinityHawking\partial_{\rhoI}-2\rhoI^{-1}\partial_u+\conormalSpaceI{1}(\Manifold)(\rhoI \partial_{\rhoI})^{\leq 2}+\conormalSpaceI{1}(\Manifold)\bDiffI^1
  \end{align*}
  where we used that the third commutator vanishes and
  \begin{align*}
    [\InfinityHawking, \rhoI \partial_{\rhoI}]&=-2a\rhoI^2(\partial_{\phiOEF} + a \partial_u)=\conormalSpaceI{2}(\Manifold)\bDiffI^1\\
    \left[ \Upsilon\left( \rhoI\partial_{\rhoI} \right)^2
    + \Upsilon\rhoI\partial_{\rhoI}
    - \frac{\Delta'}{r}(\rhoI\partial_{\rhoI}), \rhoI\partial_{\rhoI} \right]
                                              &= \conormalSpaceI{1}(\Manifold)(\rhoI \partial_{\rhoI})^{\leq 2}.
  \end{align*}
  By rewriting the first term on the \RHS{} of
  the above using the wave operator itself we obtain the stated.

  Using again
  \zcref{wave-iEF}, we have  for
  $i\in \curlyBrace*{1, 2}$
  \begin{equation*}
    \begin{split}
      \left[\abs*{q}^2\Box_{\Metric}, \RotationVF^{\mathrm{iEF}}_i  \right]
      ={}& \left[ 2a\partial_{\phiIEF}\partial_{\rhoH} +  a^2\sin^2\theta\partial_v^2 + 2a\partial_v\partial_{\phiIEF}, \RotationVF^{\mathrm{iEF}}_i \right]\\
         & + \left[ \LaplaceAngular_{\Sphere^2} + (\rhoH\partial_{\rhoH})^2 + \rhoH\partial_{\rhoH} + 2r\partial_v + 2(r^2+a^2)\partial_v\partial_{\rhoH}, \RotationVF^{\mathrm{iEF}}_i \right].
    \end{split}
  \end{equation*}
  Noticing that the second line of the above 
  vanishes, we can compute the first line with the
  assistance of
  \zcref{eq:commutators:ingoing:rotation:rotation-commutation-prop} and obtain the stated.
  Similarly, using \eqref{eq:wave-oEF}, we have  for
  $i\in \curlyBrace*{1, 2}$ 
  \begin{equation*}
    \begin{split}
      \left[\abs*{q}^2\Box_{\Metric}, \RotationVF^{\mathrm{oEF}}_i  \right]
      ={}& \left[ 2\frac{a}{r^2}\partial_{\phiOEF}\partial_{\rhoI} +  a^2\sin^2\theta\partial_u^2 + 2a\partial_u\partial_{\phiOEF}, \RotationVF^{\mathrm{oEF}}_i \right]\\
         & + \left[ \LaplaceAngular_{\Sphere^2} + \Upsilon(\rhoI\partial_{\rhoI})^2
           + \Upsilon(\rhoI\partial_{\rhoI})
           - \frac{\Delta'}{r} (\rhoI\partial_{\rhoI}) + 2\rhoI\partial_u + 2\frac{r^2+a^2}{r^2}\partial_u\partial_{\rhoI}, \RotationVF^{\mathrm{oEF}}_i \right].
    \end{split}
  \end{equation*}
  and as above we obtain the stated commutator.
\end{proof}

\begin{definition}[Radiation field]
  We define the \emph{radiation field} associated to $\psi$ as 
  \begin{align}\label{eq:def-radiation-field}
    \widecheck{\psi}\vcentcolon=r\psi.
  \end{align}
\end{definition}

We collect here the relation between the wave equation satisfied by
$\widecheck{\psi}$.
\begin{lemma}[Radiation field equation]
  \label{lemma:wave-radiation-field}
  The radiation field $\widecheck{\psi}$ of
  $\psi$ as defined in \zcref{eq:def-radiation-field} satisfies
  \begin{equation*}
    \abs*{q}^2\Box_{\Metric}\widecheck{\psi}
    = r\abs*{q}^2\Box_{\Metric}\psi
    - 2\left( \Upsilon\rhoI\partial_{\rhoI} + \rhoI^{-1}\InfinityHawking \right)\widecheck{\psi}
    + \conormalSpaceI{1}(\Manifold)\widecheck{\psi}.
  \end{equation*}
\end{lemma}
\begin{proof}
  From \zcref{eq:wave-oEF}, 
  we compute that 
  \begin{align*}
    [|q|^2\Box_{\Metric},r]
    =
    \Upsilon[(\rhoI\partial_{\rhoI})^2,r]
    + \big( \Upsilon- \frac{\Delta'}{r} \big)[\rhoI\partial_{\rhoI},r]
    +2\InfinityHawking[\partial_{\rhoI},r].
  \end{align*}
  Using that
  \[[\partial_{\rhoI},r]=-\rhoI^{-2},\quad [\rhoI\partial_{\rhoI},r]=-\rhoI^{-1},\quad [(\rhoI\partial_{\rhoI})^2,r]=\rhoI^{-1}-2\pr_{\rhoI},\]
  we deduce
  \begin{align*}
    [|q|^2\Box_{\Metric},r]\psi
    ={}& \Upsilon(\rhoI^{-1}\psi-2\pr_{\rhoI}\psi)+\big( \Upsilon- \frac{\Delta'}{r} \big)
         (-\rhoI^{-1}\psi)
         -2\InfinityHawking \rhoI^{-2}\psi\\
    ={}& \Upsilon(\widecheck{\psi}-2\pr_{\rhoI}(\rhoI \widecheck{\psi}))-\big( \Upsilon- \frac{\Delta'}{r} \big)
         \widecheck{\psi}
         -2\rhoI^{-1}\InfinityHawking \widecheck{\psi}\\
    ={}& \Upsilon(\widecheck{\psi}-2\rhoI\pr_{\rhoI} \widecheck{\psi}-2 \widecheck{\psi})-\big( \Upsilon- \frac{\Delta'}{r} \big)
         \widecheck{\psi}
         -2\rhoI^{-1}\InfinityHawking \widecheck{\psi}\\
    ={}& -2\Upsilon\rhoI\pr_{\rhoI} \widecheck{\psi}-\big( 2\Upsilon- \frac{\Delta'}{r} \big)
         \widecheck{\psi}
         -2\rhoI^{-1}\InfinityHawking \widecheck{\psi}
  \end{align*}
  which implies the stated since
  $2\Upsilon- \frac{\De'}{r}\in \conormalSpaceI{1}(\Manifold)$.
\end{proof}

\subsection{Energy-momentum tensor and divergence theorem}\label{sec-AB}

In this section, we introduce the main vectorfield multiplier
framework we will use to prove the relevant estimates. 

\begin{definition}
  We define the \emph{energy-momentum tensor} of a real-valued scalar
  function $\psi$ as the following symmetric 2-tensor
  \begin{equation*}
    \EMTensor[\psi]_{\mu\nu}= \pr_\mu\psi \pr_\nu \psi -\frac{1}{2} \g_{\mu\nu} \pr_\lambda \psi \pr^\lambda \psi.
  \end{equation*}
\end{definition}

Let $X$ be a vectorfield, $w$ be a function and $J$ a one-form. The
current associated to the multiplier triplet $(X, w, J)$ is defined as
\begin{equation}\label{definition-of-P}
  \PP_\mu^{(X, w, J)}[\psi]\vcentcolon=\EMTensor[\psi]_{\mu\nu} X^\nu
  + \frac{1}{2}\partial_{\mu}\big( w\abs*{\psi}^2 \big)
  - \partial_{\mu}w \abs*{\psi}^2
  + J_{\mu} |\psi|^2.
\end{equation}
A standard computation yields for the divergence of $\PP$:
\begin{align}\label{le:divergPP-gen}
  \begin{split}
    \D\cdot \JCurrent{(X, w, J)}[\psi]={}&  \EMTensor[\psi]  \c\piX+  w (\pr_\lambda \psi \pr^\lambda \psi)-\frac{1}{2} \square_\g w |\psi|^2+ \div(J |\psi|^2)\\
                                         &+ \big(X(\psi)+ w\psi\big) \Box_\g \psi ,
  \end{split}
\end{align}
where $\D$ denotes the covariant derivative of $\Metric$,
$\piX_{\mu\nu}=\D_{(\mu} X_{\nu)}=\frac{1}{2} \big(\D_\mu X_\nu +
\D_\nu X_\mu \big)$ is the deformation tensor of the vectorfield $X$,
and $\div$ is acting on the vectorfield dual to $J$. Recall that if
$X$ is a Killing vectorfield, then $\piX=0$.  Equivalently, we can
write
\begin{align}
  \label{le:divergPP-gen-2a}
  \begin{split}
    \D\cdot\JCurrent{(X, w, J)}[\psi]={}& (K^{(X,w)})^{\mu\nu}\pr_\mu \psi \pr_\nu \psi -\frac{1}{2} \square_\g w |\psi|^2+ \div(J |\psi|^2)+ \big(X(\psi)+ w\psi\big) \Box_\g \psi,
  \end{split}
\end{align}
where we denote 
\begin{align}\label{eq:definition-K}
  (K^{(X,w)})^{\mu\nu}&\vcentcolon=\piX^{\mu\nu}+ \big(w-\frac{1}{2}\D_\lambda X^\lambda\big)\g^{\mu\nu} .
\end{align}

In what follows, we also denote
\begin{align*}
  \QQ^{(X, w, J)}[\psi]:=\D\cdot\JCurrent{(X, w, J)}[\psi]-\big(X(\psi)+ w\psi\big) \Box_\g \psi.
\end{align*}

The energies associated to  $(X, w, J)$ are defined as
\begin{align*}
  E^{(X, w, J)}[\psi](\tau)
  ={}&\int_{\Sigma(\tau)} \JCurrent{(X, w, J)}[\psi]\cdot N_{\Sigma(\tau)}, \\
  E_{\EventHorizon}^{(X, w, J)}[\psi](\tau_1, \tau_2)
  ={}&\int_{\EventHorizon(\tau_1, \tau_2)} \JCurrent{(X, w, J)}[\psi]\cdot (-N_{\EventHorizon}), \\
  E_{\NullInfinity}^{(X, w, J)}[\psi](\tau_1, \tau_2)
  ={}&\int_{\NullInfinity(\tau_1, \tau_2)} \JCurrent{(X, w, J)}[\psi]\cdot (-N_{\NullInfinity}),
\end{align*}
where recall that the normal vectors are defined in
\zcref{eq:boundary-normal-convention}. With our normal conventions,
the horizon and null-infinity fluxes are written against
$-N_{\EventHorizon}$ and $-N_{\NullInfinity}$ so that positive
currents yield nonnegative outward fluxes. In particular, applying the
divergence theorem to $\Manifold(\tau_1, \tau_2)$ we obtain
\begin{equation}\label{eq:general-divergence-theorem}
  E^{(X, w, J)}[\psi](\tau_2)
  + E_{\EventHorizon}^{(X, w, J)}[\psi](\tau_1, \tau_2)+E_{\NullInfinity}^{(X, w, J)}[\psi](\tau_1, \tau_2)
  + \int_{\Manifold(\tau_1,\tau_2)} \D\cdot\JCurrent{(X, w, J)}[\psi]
  = E^{(X, w, J)}[\psi](\tau_1). 
\end{equation}

\subsubsection{Current computations in ingoing and outgoing Eddington Finkelstein}

Here we collect general computations needed in the vectorfield method
in iEF coordinates $(v,r,\th,\phiIEF)$ and in oEF coordinates
$(u, r, \th, \phiOEF)$.

\begin{lemma}
  \label{lemma:bulk:X-of-r:basic-computation}
  Let $X$ be a vector field in either ingoing or outgoing EF
  coordinates whose components depend only on $r$ and such that
  $X^\th=0$.  Then $(K^{(X,0)})$ defined in \zcref{eq:definition-K} in
  both cases is given by
  \begin{equation*}
    2 |q|^2(K^{(X,0)})^{\mu\nu} =|q|^2\g^{\mu r} \partial_{r} X^\nu + |q|^2\g^{\nu r} \partial_{r} X^\mu - |q|^2\g^{\mu \nu} \partial_{r} X^{r} - X^{r} \partial_{r}(|q|^2 \g^{\mu\nu}).
  \end{equation*}
\end{lemma}
\begin{proof}
  From the expression in \zcref{eq:definition-K},  see also \cite{stoginNonlinearWaveDynamics2017}, we have
  \begin{align*}
    2(K^{(X,0)})^{\mu\nu}={}&2\piX^{\mu\nu}-\D_\lambda X^\lambda\g^{\mu\nu}\\
    ={}&\g^{\mu\lambda}\pr_\lambda X^\nu + \g^{\nu\lambda}\pr_\lambda X^\mu - X^\lambda\pr_\lambda \g^{\mu\nu}-\D_\lambda X^\lambda\g^{\mu\nu}\\
    ={}&\g^{\mu r}\partial_{r} X^\nu + \g^{\nu r}\partial_{r} X^\mu - X^r\partial_{r} \g^{\mu\nu}-\D_\lambda X^\lambda\g^{\mu\nu},
  \end{align*}
  where we used that $X^\mu$ only depend on $r$ and $X^\th=0$.
  Using that   $\sqrt{|\g|}=|q|^2\sin\th$, we have 
  \begin{align*}
    \D\cdot X ={}&\frac{1}{|q|^2\sin\th}\pr_\lambda (|q|^2 \sin\th X^\lambda)=\partial_{r} X^{r} + \frac{\partial_{r} |q|^2}{|q|^2}X^{r},
  \end{align*}
  Combining the above we have
  \begin{align*}
    2 (K^{(X,0)})^{\mu\nu}
    ={}&\g^{\mu r}\partial_{r} X^\nu + \g^{\nu r}\partial_{r} X^\mu - X^r\partial_{r} \g^{\mu\nu}+(-\partial_{r} X^{r} - \frac{\partial_{r} |q|^2}{|q|^2}X^{r})\g^{\mu\nu}\\
    ={}&\g^{\mu r}\partial_{r} X^\nu + \g^{\nu r}\partial_{r} X^\mu -\g^{\mu\nu}\partial_{r}X^{r} - \frac{1}{|q|^2}X^{r}\partial_{r}(|q|^2\g^{\mu\nu}),
  \end{align*}
  as stated.
\end{proof}

We now compute the divergence of the current for vectorfields of the form\footnote{From \zcref{eq:def-VHH-VII} we have $(r^2+a^2)\That=(r^2+a^2) \pr_v + a \pr_{\phiIEF}$ in \iEF{} and $(r^2+a^2)\That=(r^2+a^2) \pr_u + a \pr_{\phiOEF}$ in \oEF{} coordinates.}
\begin{align*}
  X_{(1)}^{\mathrm{iEF}}= f(r) \pr_r^{\mathrm{iEF}}, \qquad  X_{(1)}^{\mathrm{oEF}}= f(r) \pr_r^{\mathrm{oEF}}, \qquad X_{(2)}=h(r)(r^2+a^2)\That=\begin{cases}
    h(r)\HawkingHorizon, \text{ near } \HH^+\\
    h(r)r^2\InfinityHawking, \text{ near } \II^+.
  \end{cases}
\end{align*}

\begin{lemma}
  \label{prop:fprr-HH}
  For multiplier $X_{(1)}= f(r) \pr_r$ in iEF and oEF coordinates respectively we have
  \begin{align*}
    |q|^2 \QQ^{(X^{\mathrm{iEF}}_{(1)},0, 0)}[\psi]
    ={}&\frac{1}{2} (\De  f'- f \De')|\partial^{\mathrm{iEF}}_{r} \psi|^2-\frac{1}{2} f' |\NablaAngular\psi|^2 -2rf \partial^{\mathrm{iEF}}_{r} \psi \pr_v\psi
         , \\
    |q|^2\QQ^{(X^{\mathrm{oEF}}_{(1)}, 0, 0)}[\psi]
    ={}& \frac{1}{2} (\De  f'- f\De')|\partial^{\mathrm{oEF}}_{r} \psi|^2-\frac{1}{2} f' |\NablaAngular\psi|^2+2rf \partial^{\mathrm{oEF}}_{r} \psi \pr_u\psi.
  \end{align*}
  For multiplier $X_{(2)}= h(r)(r^2+a^2) \That$ in \iEF{} and \oEF{} coordinates respectively we have 
  \begin{align*}
    |q|^2 \QQ^{(X^{\mathrm{iEF}}_{(2)}, 0, 0)}[\psi]
    ={}& h' |\HawkingHorizon\psi|^2+    2h r (\HawkingHorizon\psi)  \pr_v \psi+ \Delta  h'  (\HawkingHorizon\psi) \partial^{\mathrm{iEF}}_{r} \psi +   2h r\De\pr_v \psi \partial^{\mathrm{iEF}}_{r} \psi, \\
    |q|^2 \QQ^{(X^{\mathrm{oEF}}_{(2)}, 0, 0)}[\psi]
    ={}&- h' r^4|\InfinityHawking\psi|^2 -  2h r^3 (\InfinityHawking\psi) \pr_u \psi+ \Delta  h'  r^2(\InfinityHawking\psi) \partial^{\mathrm{oEF}}_{r} \psi+   2h r\De\pr_u \psi \partial_r^{\mathrm{oEF}} \psi.  
  \end{align*}
\end{lemma}
\begin{proof}
  Recall from \zcref{eq:inverse-metric:ingoing-EF} that the components
  of inverse metric in \iEF{} coordinates $(v,r,\th, \phiIEF)$ are
  given by
  \begin{gather*}
    |q|^2\g^{vv}=a^2\sin^2\th,\quad |q|^2\g^{vr}=r^2+a^2,\quad |q|^2\g^{rr}=\De,\\
    |q|^2\g^{\phiIEF\phiIEF}=\frac{1}{\sin^2\th},\quad |q|^2\g^{r\phiIEF}=|q|^2\g^{v\phiIEF}=a, \quad |q|^2\g^{\th\th}=1,
  \end{gather*}
  and therefore
  \begin{equation*}
    \begin{gathered}
      \partial^{\mathrm{iEF}}_{r}(|q|^2\g^{vv}) ={} \partial^{\mathrm{iEF}}_{r}(|q|^2\g^{\phiIEF\phiIEF})=\partial^{\mathrm{iEF}}_{r}(|q|^2\g^{r\phiIEF})=\partial^{\mathrm{iEF}}_{r}(|q|^2\g^{v\phiIEF})=\partial^{\mathrm{iEF}}_{r}(|q|^2\g^{\th\th})=0,\\
      \partial^{\mathrm{iEF}}_{r}(|q|^2\g^{vr})={}2r,\quad \partial^{\mathrm{iEF}}_{r}(|q|^2\g^{rr})=\De'.
    \end{gathered}
  \end{equation*}
  Recall from \zcref{eq:inverse-metric-oEF} that the components of inverse metric in \oEF{} coordinates $(u,r,\th,\phiOEF)$ are given by
  \begin{gather*}
    |q|^2\g^{uu}=a^2\sin^2\th,\quad |q|^2\g^{ur}=-(r^2+a^2),\quad |q|^2\g^{rr}=\De,\\ |q|^2\g^{\phiOEF\phiOEF}=\frac{1}{\sin^2\th},\quad |q|^2\g^{r\phiOEF}=-|q|^2\g^{u\phiOEF}=-a, \quad |q|^2\g^{\th\th}=1,
  \end{gather*}
  and therefore
  \begin{equation*}
    \begin{gathered}
      \partial^{\mathrm{oEF}}_{r}(|q|^2\g^{uu}) ={} \partial^{\mathrm{oEF}}_{r}(|q|^2\g^{\phiOEF\phiOEF})=\partial^{\mathrm{oEF}}_{r}(|q|^2\g^{r\phiOEF})=\partial^{\mathrm{oEF}}_{r}(|q|^2\g^{u\phiOEF})=\partial^{\mathrm{oEF}}_{r}(|q|^2\g^{\th\th})=0,\\
      \partial^{\mathrm{oEF}}_{r}(|q|^2\g^{ur})={}-2r,\quad \partial^{\mathrm{oEF}}_{r}(|q|^2\g^{rr})=\De'.
    \end{gathered}
  \end{equation*}
  Using \zcref{lemma:bulk:X-of-r:basic-computation}, we then deduce 
  \begin{equation*}
    \begin{split}
      2 |q|^2(K^{(X^{\mathrm{iEF}}_{(1)},0)})^{vv}&=2 |q|^2(K^{(X^{\mathrm{oEF}}_{(1)},0)})^{uu}  ={}-  a^2\sin^2\th  f',\\
      2  |q|^2(K^{(X^{\mathrm{iEF}}_{(1)},0)})^{vr}&=-2  |q|^2(K^{(X^{\mathrm{oEF}}_{(1)},0)})^{ur} ={}- 2r f,\\
      2  |q|^2(K^{(X^{\mathrm{iEF}}_{(1)},0)})^{rr}&=2  |q|^2(K^{(X^{\mathrm{oEF}}_{(1)},0)})^{rr} ={}\De  f'- f \De',\\
      2  |q|^2(K^{(X^{\mathrm{iEF}}_{(1)},0)})^{\phiIEF\phiIEF}&=2  |q|^2(K^{(X^{\mathrm{oEF}}_{(1)},0)})^{\phiOEF\phiOEF} ={}-  \frac{1}{\sin^2\th}  f',\\
      2  |q|^2(K^{(X^{\mathrm{iEF}}_{(1)},0)})^{r\phiIEF}&=2  |q|^2(K^{(X^{\mathrm{oEF}}_{(1)},0)})^{r\phiOEF} ={}0,\\
      2  |q|^2(K^{(X^{\mathrm{iEF}}_{(1)},0)})^{v\phiIEF}&=2  |q|^2(K^{(X^{\mathrm{oEF}}_{(1)},0)})^{u\phiOEF} ={}-  a f',\\
      2  |q|^2(K^{(X^{\mathrm{iEF}}_{(1)},0)})^{\th\th}&=2  |q|^2(K^{(X^{\mathrm{oEF}}_{(1)},0)})^{\th\th}  ={}-f' ,
    \end{split}    
  \end{equation*}
  where recall that for a function of $r$ only $\pr_r^{\mathrm{iEF}}f=\pr_r^{\mathrm{oEF}}f=\pr_r^{\mathrm{BL}}f=f'$.
  From \zcref{le:divergPP-gen-2a}, we deduce 
  \begin{align*}
    2 |q|^2\D\cdot \PP ^{(X^{\mathrm{iEF}}_{(1)}, 0, 0)}[\psi]
    ={}& (\De f'- f \De')|\partial^{\mathrm{iEF}}_{r} \psi|^2-4rf \partial^{\mathrm{iEF}}_{r} \psi \pr_v\psi-f' |\pr_\th\psi|^2\\
       &-  a^2\sin^2\th f' |\pr_v\psi|^2 -  \frac{1}{\sin^2\th} f' |\pr_{\phiIEF}\psi|^2-2  a f' \pr_v \psi \pr_{\phiIEF}\psi  + 2f\partial^{\mathrm{iEF}}_{r}\psi  |q|^2\Box_\g \psi, \\
    2 |q|^2\D\cdot \PP ^{(X^{\mathrm{oEF}}_{(1)}, 0, 0)}[\psi]
    ={}& (\De  f'- f \De')|\partial^{\mathrm{oEF}}_{r} \psi|^2+4rf \partial^{\mathrm{oEF}}_{r} \psi \pr_u\psi- f' |\pr_\th\psi|^2\\
       &-  a^2\sin^2\th f' |\pr_u\psi|^2 -  \frac{1}{\sin^2\th} f' |\pr_{\phiOEF}\psi|^2-2 a f' \pr_u \psi \pr_{\phiOEF}\psi  + 2f\partial^{\mathrm{oEF}}_{r}\psi  |q|^2\Box_\g \psi.   
  \end{align*}
  Recalling from \zcref{eq:def-VHH-VII} that $\renormpphi=a\sin\th \pr_v +\frac{1}{\sin\th} \pr_{\phiIEF}= a\sin\th \pr_u +\frac{1}{\sin\th} \pr_{\phiOEF}$,
  we obtain 
  \begin{align*}
    |q|^2\D\cdot \PP ^{(X^{\mathrm{iEF}}_{(1)}, 0, 0)}[\psi]
    ={}& \frac{1}{2} (\De  f'- f \De')|\partial^{\mathrm{iEF}}_{r} \psi|^2-\frac{1}{2} f'|\renormpphi\psi|^2- \frac{1}{2} f' |\pr_\th\psi|^2\\
       & -2rf \partial^{\mathrm{iEF}}_{r} \psi \pr_v\psi
         + f\partial^{\mathrm{iEF}}_{r}\psi  |q|^2\Box_\g \psi, \\
    |q|^2\D\cdot \PP ^{(X^{\mathrm{oEF}}_{(1)}, 0, 0)}[\psi]
    ={}& \frac{1}{2} (\De  f'- f \De')|\partial^{\mathrm{oEF}}_{r} \psi|^2-\frac{1}{2} f'|\renormpphi\psi|^2- \frac{1}{2} f' |\pr_\th\psi|^2\\
       &+2rf \partial^{\mathrm{oEF}}_{r} \psi \pr_u\psi
         + f\partial^{\mathrm{oEF}}_{r}\psi  |q|^2\Box_\g \psi.
  \end{align*}
  Finally, recalling the notation $|\nabb\psi|^2$, we obtain
  the stated result for $X_{(1)}$.

  For $X^{\mathrm{iEF}}_{(2)}=h (r^2+a^2) \pr_v + ah \pr_{\phiIEF}$ and $X^{\mathrm{oEF}}_{(2)}=h (r^2+a^2) \pr_u + ah \pr_{\phiOEF}$, using \zcref{lemma:bulk:X-of-r:basic-computation}, we deduce 
  \begin{equation*}
    \begin{split}
      2|q|^2(K^{(X^{\mathrm{iEF}}_{(2)},0)})^{vv} ={}&-2|q|^2(K^{(X^{\mathrm{oEF}}_{(2)},0)})^{uu}= 2(r^2+a^2)^2 h' +  4h r(r^2+a^2),\\
      2|q|^2(K^{(X^{\mathrm{iEF}}_{(2)},0)})^{vr} ={}&2|q|^2(K^{(X^{\mathrm{oEF}}_{(2)},0)})^{ur}= \De(r^2+a^2)h' +  2h r\De,\\
      2|q|^2(K^{(X^{\mathrm{iEF}}_{(2)},0)})^{rr} ={}&2|q|^2(K^{(X^{\mathrm{oEF}}_{(2)},0)})^{rr}=0,\\
      2|q|^2(K^{(X^{\mathrm{iEF}}_{(2)},0)})^{\phiIEF\phiIEF} ={}&-2|q|^2(K^{(X^{\mathrm{oEF}}_{(2)},0)})^{\phiOEF\phiOEF} =  2a^2h',\\
      2|q|^2(K^{(X^{\mathrm{iEF}}_{(2)},0)})^{r\phiIEF} ={}&2|q|^2(K^{(X^{\mathrm{oEF}}_{(2)},0)})^{r\phiOEF} =a\De h' ,\\
      2|q|^2(K^{(X^{\mathrm{iEF}}_{(2)},0)})^{v\phiIEF} ={}&-2|q|^2(K^{(X^{\mathrm{oEF}}_{(2)},0)})^{u\phiOEF}= 2ah'(r^2+a^2) + 2arh,\\
      2|q|^2(K^{(X^{\mathrm{iEF}}_{(2)},0)})^{\th\th} ={}&2|q|^2(K^{(X^{\mathrm{oEF}}_{(2)},0)})^{\th\th}=0.
    \end{split}    
  \end{equation*}
  From \zcref{le:divergPP-gen-2a}, we deduce
  \begin{align*}
    |q|^2 \D\cdot \JCurrent{(X^{\mathrm{iEF}}_{(2)}, 0, 0)}[\psi]={}&\big( (r^2+a^2)^2  h' +  2h r(r^2+a^2)\big) |\pr_v \psi|^2+ \big( \De(r^2+a^2)  h'+  2h r\De\big)\pr_v \psi \partial^{\mathrm{iEF}}_{r} \psi \\
                                                                    &+ a^2h' |\pr_{\phiIEF} \psi|^2+  a \Delta  h' \pr^{\mathrm{iEF}}_r\psi \pr_{\phiIEF} \psi + \big(2a(r^2+a^2) h' + 2arh \big)\pr_v \psi \pr_{\phiIEF} \psi\\
                                                                    &+ h (\HawkingHorizon\psi) |q|^2 \Box_\g \psi, \\
    |q|^2 \D\cdot \JCurrent{(X^{\mathrm{oEF}}_{(2)}, 0, 0)}[\psi]={}&-\big( (r^2+a^2)^2 h' +  2h r(r^2+a^2)\big) |\pr_u \psi|^2+ \big( \De(r^2+a^2) h' +  2h r\De\big)\pr_u \psi \partial_r^{\mathrm{oEF}} \psi \\
                                                                    &- a^2h' |\pr_{\phiOEF} \psi|^2+  a \Delta h' \pr^{\mathrm{oEF}}_r \psi \pr_{\phiOEF} \psi - \big(2a(r^2+a^2) h' + 2arh \big)\pr_u \psi \pr_{\phiOEF} \psi\\
                                                                    &+ h (r^2\InfinityHawking\psi) |q|^2 \Box_\g \psi. 
  \end{align*}
  Recalling that $\HawkingHorizon=(r^2+a^2) \pr_v + a \pr_{\phiIEF}$ and $r^2\InfinityHawking=(r^2+a^2) \pr_u + a \pr_{\phiOEF}$ respectively, we obtain the stated result for $X_{(2)}$.  
\end{proof}

\subsubsection{Boundary term computations in ingoing and outgoing Eddington Finkelstein}

We now compute the boundary terms obtained from the vectorfields $X_{(1)}$ and $X_{(2)}$.

\begin{lemma}\label{lemma:boundary-terms-ieF}
  For multiplier $X_{(1)}= f(r) \pr^{\mathrm{iEF}}_r$ in \iEF{}
  coordinates we have
  \begin{align*}
    \JCurrent{(X_{(1)},0,0)}[\psi]\cdot N_{\Sigma_{\EventHorizon}(\tau)}
    ={}&|q|^{-2}f \Big[-\big(   r^2+a^2-\frac{1}{2}\Delta\,h_{\HH}' \big) |\pr_{\rhoH}\psi|^2 -\frac{1}{2} h_{\HH}'\abs*{\NablaAngular\psi}^2-a\sin\th 
         (\pr_{\rhoH}\psi )(\renormpphi\psi)\Big],
    \\
    \JCurrent{(X_{(1)},0,0)}[\psi]\cdot (-N_{\EventHorizon})
    ={}&-\frac{1}{2} |q|^{-2} f\big( |\rhoH \pr_{\rhoH} \psi|^2 +\abs*{\NablaAngular\psi}^2\big).
  \end{align*}
  For multiplier $X_{(2)}= h(r) \HawkingHorizon$ in \iEF{} coordinates we have 
  \begin{align*}
    \JCurrent{(X_{(2)},0,0)}[\psi]\cdot N_{\Sigma_{\EventHorizon}(\tau)}
    ={}& |q|^{-2} h \Big[ \frac{1}{2} \big( r^2+a^2-\frac 1 2 \Delta\,h_{\HH}' \big)  \abs*{\rhoH\partial_{\rhoH}\psi}^2 + \frac{1}{2} \bigl(r^2+a^2\bigr)   \abs*{\NablaAngular\psi}^2\nonumber\\
       &+  h_{\HH}'  \big(  \HawkingHorizon\psi   +  \frac{1}{2}\rhoH^2  \partial_{\rhoH}\psi  \big)^2-a \sin\th  (\HawkingHorizon\psi) (\renormpphi \psi)\Big],\\
    \JCurrent{(X_{(2)},0,0)}[\psi]\cdot (-N_{\EventHorizon})
    ={}&\abs*{q}^{-2}h \left(\frac{1}{2}\big(  \HawkingHorizon\psi + 2\rhoH^2 \partial_{\rhoH}\psi\big)^2+\frac{1}{2} |\HawkingHorizon\psi|^2 +\frac{1}{2} \rhoH^2 \abs*{\NablaAngular\psi}^2 
         \right).
  \end{align*} 
\end{lemma}
\begin{proof}
  We first record the basic stress-energy pairings in the
  $(\partial_{\rhoH}, \HawkingHorizon, \renormpphi)$ frame. From
  \zcref{definition-of-P}, we have that
  $\JCurrent{(X,0,0)}[\psi]\cdot N= \EMTensor(X, N)$, and using the
  expression of the metric in iEF coordinates we compute
  \begin{align}\label{eq:EMtensor-rhoH-H}
    \begin{split}
      \EMTensor(\partial_{\rhoH}, \partial_{\rhoH})[\psi] ={}& \abs*{\partial_{\rhoH}\psi}^2,\\
      \EMTensor(\partial_{\rhoH}, \HawkingHorizon)[\psi]
      ={}&   -  \frac{1}{2} \big(\abs*{\rhoH\partial_{\rhoH}\psi}^2 +\abs*{\NablaAngular\psi}^2\big) ,\\
      \EMTensor(\HawkingHorizon, \HawkingHorizon)[\psi]
      = {}& |\HawkingHorizon\psi|^2 +\frac{1}{2} \rhoH^2\big(2 \HawkingHorizon\psi \pr_{\rhoH} \psi  + |\rhoH \pr_{\rhoH} \psi|^2 + \abs*{\NablaAngular\psi}^2 \big),\\
      \EMTensor(\partial_{\rhoH}, \renormpphi)[\psi]
      = {}& \pr_{\rhoH} \psi \renormpphi\psi, \\
      \EMTensor(\HawkingHorizon, \renormpphi)[\psi]
      = {}& \HawkingHorizon\psi \renormpphi\psi .
    \end{split}
  \end{align}
  We deduce from \zcref{eq:NHH-NII},
  \begin{align*}
    \JCurrent{(X_{(1)},0,0)}[\psi]\cdot N_{\EventHorizon}
    ={}&f \EMTensor(\partial_{\rhoH}, N_{\EventHorizon})[\psi]\\
    =&-|q|^{-2}f \EMTensor(\partial_{\rhoH}, \HawkingHorizon )=\frac{1}{2} |q|^{-2} f\big( |\rhoH \pr_{\rhoH} \psi|^2 +\abs*{\NablaAngular\psi}^2\big),\\
    \JCurrent{(X_{(2)},0,0)}[\psi]\cdot N_{\EventHorizon}
    ={}&h\EMTensor(\HawkingHorizon, N_{\EventHorizon})[\psi]\\
    =&-|q|^{-2}h\EMTensor(\HawkingHorizon, \HawkingHorizon )[\psi]= -\abs*{q}^{-2}h \left(2\big(\frac{1}{2}  \HawkingHorizon\psi + \rhoH^2 \partial_{\rhoH}\psi\big)^2+\frac{1}{2} |\HawkingHorizon\psi|^2 +\frac{1}{2} \Delta \abs*{\NablaAngular\psi}^2 
       \right),
  \end{align*}
  as stated. 
  Similarly, from the expression of $N_{\Sigma(\tau)}$ in \zcref{lemma:normal-vector-tau}
  we deduce
  \begin{align*}
    \EMTensor(\pr_{\rhoH}, |q|^2N_{\Sigma_{\EventHorizon}(\tau)})[\psi]
    ={}& \frac{1}{2}\left( \Delta\,h_{\HH}'(r)-   2\bigl(r^2+a^2\bigr) \right) |\pr_{\rhoH}\psi|^2 -\frac{1}{2} h_{\HH}'(r)\abs*{\NablaAngular\psi}^2 -a\sin\th
         (\pr_{\rhoH}\psi )(\renormpphi\psi)\\
    \EMTensor(\HawkingHorizon, |q|^2N_{\Sigma_{\EventHorizon}(\tau)})[\psi] 
    ={}& \frac{1}{4} \left( -\Delta\,h_{\HH}'(r)+   2\bigl(r^2+a^2\bigr) \right)  \abs*{\rhoH\partial_{\rhoH}\psi}^2 + \frac{1}{2} \bigl(r^2+a^2\bigr)   \abs*{\NablaAngular\psi}^2\\
       &+ \frac{1}{2} h_{\HH}'(r)  \big(
         2  |\HawkingHorizon\psi|^2
         +2  \rhoH \HawkingHorizon\psi (\rhoH\pr_{\rhoH} \psi)
         +  \frac{1}{2}\rhoH^4  \abs*{\partial_{\rhoH}\psi}^2  \big)
         -a \sin\th (\HawkingHorizon\psi) (\renormpphi\psi).
  \end{align*}
  as stated.
\end{proof}

\begin{lemma}\label{lemma:boundary-terms-oEF}
  For $\rhoI \leq \rho_1$ with $\rho_1 \ll r_{+}^{-1}$
  sufficiently small, we have
  \begin{align*}
    \JCurrent{(-\partial_{\rhoI}
    + \frac{a^2}{\Upsilon}\InfinityHawking, 0, 0)}[\psi]\cdot |q|^2N_{\Sigma_{\NullInfinity}(\tau)}
    &\gtrsim \abs*{\partial_{\rhoI}\psi}^2
      + |a|\abs*{\InfinityHawking\psi}^2
      +|a|\abs*{\NablaAngular\psi}^2
      ,     \\
    \JCurrent{(-\partial_{\rhoI}
    + \frac{a^2}{\Upsilon}\InfinityHawking, 0, 0)}[\psi]\cdot (-N_{\NullInfinity})
    &\gtrsim  \abs*{\rhoI\partial_{\rhoI}\psi}^2+|a|\abs*{\InfinityHawking\psi}^2
      + \abs*{\NablaAngular\psi}^2. 
  \end{align*}
\end{lemma}
\begin{proof}
  Using the expression of the metric in oEF coordinates we compute 
  \begin{align}\label{eq:T-outgoing}
    \begin{split}
      \EMTensor(\partial_{\rhoI}, \partial_{\rhoI})[\psi]
      ={}& \abs*{\partial_{\rhoI}\psi}^2,\\
      \EMTensor(\partial_{\rhoI}, \InfinityHawking)[\psi]
      ={}& - \frac{1}{2}\left( \Upsilon\abs*{\rhoI\partial_{\rhoI}\psi}^2 + \abs*{\NablaAngular\psi}^2 \right),\\
      \EMTensor(\InfinityHawking, \InfinityHawking)[\psi]
      ={}& \abs*{\InfinityHawking\psi}^2
           + \frac{1}{2}\rhoI^2\Upsilon\left(
           2\InfinityHawking\psi\partial_{\rhoI}\psi
           + \Upsilon |\rhoI\partial_{\rhoI}\psi|^2
           + \abs*{\NablaAngular\psi}^2
           \right),\\
      \EMTensor(\partial_{\rhoI}, \renormpphi)[\psi]
      ={}& \partial_{\rhoI}\psi \,\renormpphi\psi,\\
      \EMTensor(\InfinityHawking, \renormpphi)[\psi]
      ={}& \InfinityHawking\psi\, \renormpphi\psi.   
    \end{split} 
  \end{align}
  From \zcref{lemma:normal-vector-tau}, since $\rho_1 \leq r_{+}^{-2}\rho_\star$, we can write 
  \begin{equation*}
    |q|^2 N_{\Sigma_{\NullInfinity}(\tau)}
    =
    -\partial_{\rhoI}
    + \lambda\InfinityHawking
    + Y \renormpphi,
  \end{equation*}
  with $\lambda=\frac{a^2}{\Upsilon}$, $Y=-a\sin\th$. 
  We then deduce 
  \begin{align*}
    \EMTensor(|q|^2N_{\Sigma_{\NullInfinity}}, |q|^2N_{\Sigma_{\NullInfinity}})[\psi]
    ={}& \abs*{q}^4\abs*{N_{\Sigma_{\NullInfinity}}\psi}^2
         - \frac{1}{2\abs*{q}^2}\Metric(\abs*{q}^2N_{\Sigma_{\NullInfinity}}, \abs*{q}^2N_{\Sigma_{\NullInfinity}})
         \left(
         2\InfinityHawking\psi\partial_{\rho_{\II}}\psi
         + \Upsilon|\rho_{\II}\partial_{\rho_{\II}}\psi|^2
         +\abs*{\NablaAngular\psi}^2
         \right),\\
    \EMTensor(Y\renormpphi, \abs*{q}^2N_{\Sigma_{\NullInfinity}})[\psi]
    ={}& Y\renormpphi\psi \abs*{q}^2N_{\Sigma_{\NullInfinity}}\psi
         - \frac{1}{2}\abs*{Y}^2\left(
         2\InfinityHawking\psi\partial_{\rho_{\II}}\psi
         + \Upsilon | \rho_{\II}\partial_{\rho_{\II}}\psi|^2
         +\abs*{\NablaAngular\psi}^2
         \right).
  \end{align*}
  and therefore, using that 
  \begin{align*}
    \Metric(\abs*{q}^2N_{\Sigma_{\NullInfinity}}, \abs*{q}^2N_{\Sigma_{\NullInfinity}})=|q|^2 \big(-2\lambda -\rhoI^2\Upsilon\lambda^2 +  Y^2 \big)
  \end{align*}
  we have
  \begin{align*}
    & \EMTensor(\abs*{q}^2N_{\Sigma_{\NullInfinity}}-Y\renormpphi, \abs*{q}^2N_{\Sigma_{\NullInfinity}})[\psi]\\
    ={}& \abs*{\abs*{q}^2N_{\Sigma_{\NullInfinity}}\psi}^2
         - Y\renormpphi\psi \abs*{q}^2N_{\Sigma_{\NullInfinity}}\psi
         +\left( \lambda + \frac{1}{2}  \rhoI^2\Upsilon\lambda^2  \right)\left(
         2\InfinityHawking\psi\partial_{\rho_{\II}}\psi
         + \Upsilon|\rho_{\II}\partial_{\rho_{\II}}\psi |^2
         + \abs*{\NablaAngular\psi}^2
         \right)\\
    ={}& \frac{1}{2}\abs*{\abs*{q}^2N_{\Sigma_{\NullInfinity}}\psi}^2
         + \frac{1}{2}\abs*{\abs*{q}^2N_{\Sigma_{\NullInfinity}}\psi - Y\renormpphi\psi}^2
         + \left( \lambda + \frac{1}{2}  \rhoI^2\Upsilon\lambda^2   - \frac{1}{2}Y^2  \right)\abs*{\renormpphi\psi}^2\\
    &+ \left( \lambda + \frac{1}{2}  \rhoI^2\Upsilon\lambda^2  \right)\left(
      2\InfinityHawking\psi\partial_{\rho_{\II}}\psi
      + \Upsilon| \rho_{\II}\partial_{\rho_{\II}}\psi |^2
      + \abs*{\partial_{\theta}\psi}^2
      \right)\\
    ={}& \frac{1}{2}\abs*{\abs*{q}^2N_{\Sigma_{\NullInfinity}}\psi}^2
         + \frac{1}{2} \abs*{\partial_{\rhoI}\psi}^2
         +  \frac{1}{2}\lambda^2 \abs*{\InfinityHawking\psi}^2
         + \left( \lambda + \frac{1}{2}  \rhoI^2\Upsilon\lambda^2  - \frac{1}{2}Y^2  \right)\abs*{\renormpphi\psi}^2\\
    &  + \left( \lambda +   \rhoI^2\Upsilon\lambda^2  \right)
      \InfinityHawking\psi\partial_{\rho_{\II}}\psi+\left( \lambda + \frac{1}{2}  \rhoI^2\Upsilon\lambda^2 \right)\left(
      \Upsilon |\rho_{\II}\partial_{\rho_{\II}}\psi |^2
      + \abs*{\partial_{\theta}\psi}^2
      \right)
      ,
  \end{align*}
  which can be written as 
  \begin{align*}
    \EMTensor(\abs*{q}^2N_{\Sigma_{\NullInfinity}}-Y\renormpphi, \abs*{q}^2N_{\Sigma_{\NullInfinity}})[\psi]  ={}&{}\frac{1}{2}\abs*{\abs*{q}^2N_{\Sigma_{\NullInfinity}}\psi}^2
                                                                                                                   + \frac{1}{2}  \abs*{(\partial_{\rhoI}
                                                                                                                   +  \lambda \InfinityHawking)\psi}^2
                                                                                                                   + \left( \lambda  - \frac{1}{2}Y^2  \right)\abs*{\renormpphi\psi}^2+ \lambda\abs*{\partial_{\theta}\psi}^2 \\
                                                                                                                 &  + \conormalSpaceI{2}(\Manifold)
                                                                                                                   \left(|\partial_{\rhoI}\psi|^2+\InfinityHawking\psi\partial_{\rho_{\II}}\psi+ \abs*{\renormpphi\psi}^2+\abs*{\partial_{\theta}\psi}^2\right).
  \end{align*}
  Observe that $\lambda - \frac{1}{2} Y^2 \geq a^2 \big( \frac{1}{\Upsilon}-\frac{1}{2} \big)\geq 0$ and, noticing that $\abs*{q}^2N_{\Sigma_{\NullInfinity}}$, $\partial_{\rhoI}
  +  \lambda \InfinityHawking$ and $\renormpphi\psi$ are three linearly independent combinations of $\partial_{\rhoI}$, $\InfinityHawking$ and $\renormpphi$, we see that for $\rhoI\le \rho_1$ for $\rho_1$
  sufficiently small one can absorb the second line above and rearrange it so that,
  \begin{equation*}
    \EMTensor(\abs*{q}^2N_{\Sigma_{\NullInfinity}}-Y\renormpphi, \abs*{q}^2N_{\Sigma_{\NullInfinity}})[\psi]
    \ges \abs*{\partial_{\rhoI}\psi}^2
    + |a|\abs*{\InfinityHawking\psi}^2
    + |a|\abs*{\NablaAngular\psi}^2.
  \end{equation*}

  Similarly, recalling from \zcref{lemma:normal-vector-tau}  that $ N_{\NullInfinity}= -\InfinityHawking $
  we compute that
  \begin{align*}
    \EMTensor(|q|^2N_{\Sigma_{\NullInfinity}} - Y\renormpphi, -N_{\NullInfinity})
    ={}& \EMTensor\left(-\partial_{\rhoI} + \lambda\InfinityHawking, \InfinityHawking\right)\\
    ={}& \lambda\abs*{\InfinityHawking\psi}^2
         + \frac{\lambda}{2}\rhoI^2\Upsilon \left(
         2\InfinityHawking\psi\partial_{\rhoI}\psi
         + \Upsilon(\rhoI\partial_{\rhoI}\psi)^2 
         + \abs*{\NablaAngular\psi}^2\right)\\
       & + \frac{1}{2}\left( \Upsilon\abs*{\rhoI\partial_{\rhoI}\psi}^2 + \abs*{\NablaAngular\psi}^2 \right)\\
    ={}& \lambda\abs*{\InfinityHawking\psi}^2
         + \frac{1}{2}\abs*{\rhoI\partial_{\rhoI}\psi}^2
         + \frac{1}{2} \abs*{\NablaAngular\psi}^2
         + \conormalSpaceI{2}(\Manifold)| \bDiffI\psi |^2.
  \end{align*}
  In particular, for
  $\rhoI<\rho_1$ sufficiently small we obtain the stated bound. 
\end{proof}

\subsubsection{Hardy inequalities}

We collect here the following bounds that will be used as Hardy inequalities.

\begin{lemma}[Hardy inequality in divergence form]\label{lemma:hardy-pointwise-rhoH}
  Let $\psi\in C^1(\Manifold)$. Then, for every $\beta\in\Real$,
  \begin{align*}
    \rhoH^\beta |\rhoH\pr_{\rhoH}\psi|^2
    &\geq
      \frac{(\beta+1)^2}{4}\rhoH^\beta|\psi|^2
      -
      |q|^2
      \D_\mu\left(
      \frac{\beta+1}{2}
      \rhoH^{\beta+1}|\psi|^2
      |q|^{-2}(\partial_{\rhoH})^\mu
      \right),
    \\
    \rhoI^\beta |\rhoI\pr_{\rhoI}\psi|^2
    &\geq
      \frac{(\beta-1)^2}{4}\rhoI^\beta|\psi|^2
      -
      |q|^2
      \D_\mu\left(
      \frac{\beta-1}{2}
      \rhoI^{\beta+1}|\psi|^2
      |q|^{-2}(\partial_{\rhoI})^\mu
      \right).
  \end{align*}
\end{lemma}

\begin{proof}
  For any $c\in\Real$, we have the pointwise identity
  \begin{align*}
    0
    &\leq
      \left(
      \rho^{\frac{\beta+2}{2}}\pr_\rho\psi
      +
      c\rho^{\frac\beta2}\psi
      \right)^2
    \\
    &=
      \rho^\beta|\rho\pr_\rho\psi|^2
      +
      2c\rho^{\beta+1}\psi\pr_\rho\psi
      +
      c^2\rho^\beta\psi^2
    \\
    &=
      \rho^\beta|\rho\pr_\rho\psi|^2
      +
      c\pr_\rho(\rho^{\beta+1}\psi^2)
      -
      c(\beta+1)\rho^\beta\psi^2
      +
      c^2\rho^\beta\psi^2 .
  \end{align*}
  Therefore
  \begin{equation}\label{eq:hardy-basic-c}
    \rho^\beta|\rho\pr_\rho\psi|^2
    \geq
    \big(c(\beta+1)-c^2\big)\rho^\beta\psi^2
    -
    c\pr_\rho(\rho^{\beta+1}\psi^2).
  \end{equation}
  Using that in iEF and oEF coordinates,  $\sqrt{|\g|}=\sin\th|q|^2$ we have
  \begin{align*}
    \D_\mu\Big(|q|^{-2}(\pr_{r})^\mu\Big)
    =\frac{1}{\sqrt{|\g|}}\pr_\mu\big(\sqrt{|\g|}|q|^{-2}(\pr_{r})^\mu\big)
    =\frac{1}{\sin\th|q|^2}\pr_{r}\big(\sin\th\big)=0,
  \end{align*}
  and therefore  
  \begin{align*}
    \D_\mu\Big(|q|^{-2}(\pr_{\rhoH})^\mu\Big)
    =0, \qquad \D_\mu\Big(|q|^{-2}\rhoI^2(\pr_{\rhoI})^\mu\Big)
    =0.
  \end{align*}

  We first apply this with $\rho=\rhoH$. Writing
  \[
    \pr_{\rhoH}(\rhoH^{\beta+1}\psi^2)
    =
    |q|^2
    \D_\mu\left(
      \rhoH^{\beta+1}\psi^2
      |q|^{-2}(\partial_{\rhoH})^\mu
    \right),
  \]
  and substituting into \zcref{eq:hardy-basic-c} gives
  \[
    \rhoH^\beta|\rhoH\pr_{\rhoH}\psi|^2
    \geq
    \big(c(\beta+1)-c^2\big)\rhoH^\beta\psi^2
    -
    |q|^2
    \D_\mu\left(
      c\rhoH^{\beta+1}\psi^2
      |q|^{-2}(\partial_{\rhoH})^\mu
    \right).
  \]
  Choosing $c=\frac{\beta+1}{2}$
  gives the first stated bound.
  We now consider $\rho=\rhoI$. Writing
  \[
    \pr_{\rhoI}(\rhoI^{\beta+1}\psi^2)
    =
    |q|^2
    \D_\mu\left(
      \rhoI^{\beta+1}\psi^2
      |q|^{-2}(\partial_{\rhoI})^\mu
    \right)
    +
    2\rhoI^\beta\psi^2 
  \]
  and substituting into \zcref{eq:hardy-basic-c} gives
  \begin{align*}
    \rhoI^\beta|\rhoI\pr_{\rhoI}\psi|^2
    &\geq
      \big(c(\beta+1)-c^2\big)\rhoI^\beta\psi^2
      -
      c|q|^2
      \D_\mu\left(
      \rhoI^{\beta+1}\psi^2
      |q|^{-2}(\partial_{\rhoI})^\mu
      \right)
      -
      2c\rhoI^\beta\psi^2
    \\
    &=
      \big(c(\beta-1)-c^2\big)\rhoI^\beta\psi^2
      -
      |q|^2
      \D_\mu\left(
      c\rhoI^{\beta+1}\psi^2
      |q|^{-2}(\partial_{\rhoI})^\mu
      \right).
  \end{align*}
  Choosing $c=\frac{\beta-1}{2}$
  gives the second stated bound.
\end{proof}

We also state the following integrated Hardy inequality, proved in a
similar way, see for example
\cite{gajicLatetimeAsymptoticsGeometric2023}.

\begin{lemma}[Lemma 2.2 in \cite{gajicLatetimeAsymptoticsGeometric2023}]\label{lemma:hardy-dejan}
  Let $f\in C^1([a,b])$ with $a, b \in \mathbb{R}_{\geq 0}$ such that
  $a<b$. Then for $\gamma \in \mathbb{R}\setminus \{-1\}$:
  \begin{align*}
    \int_a^b \rho^\gamma |f|^2(\rho) d\rho \leq \frac{2}{\gamma+1} \Big[b^{\gamma+1}|f|^2(b)-a^{\gamma+1}|f|^2(a)\Big]+\frac{4}{(\gamma+1)^2}\int_a^b \rho^{\gamma+2}\big| \frac{df}{d\rho}\big|^2(\rho) d\rho.
  \end{align*}
\end{lemma}

In what follows we denote $\chi$ any non-negative bump function
supported in $[-2,2]$ and identically equal to $1$ on $[-1,1]$.
We deduce the following cutoff version of \zcref{lemma:hardy-dejan}. 
\begin{corollary}[Cutoff Hardy inequality]
  \label{corollary-hardy-dejan-cutoff}
  Let $f\in C^1([0,c])$ with $c>0$ and let $\gamma>-1$. Then
  \begin{align*}
    \int_0^{c}\rho^{\gamma}
    \abs*{\rho\partial_{\rho}f}^2\,d\rho
    &\geq
      \frac{(\gamma+1)^2}{8}
      \int_0^{\frac{c}{2}}\rho^{\gamma}
      \abs*{f}^2\,d\rho
      +
      \lim_{\rho\to0}
      \frac{\gamma+1}{2}\rho^{\gamma+1}\abs*{f}^2      
      -O(c^{-2})
      \int_{\frac12 c}^{c}
      \rho^{\gamma+2}|f|^2\,d\rho .
  \end{align*}
\end{corollary}

\begin{proof}
  Apply \zcref{lemma:hardy-dejan} to $\widetilde{\chi}_c f$, where
  \[
    \widetilde{\chi}_c(\rho)=\chi\left(\frac{\rho}{c/2}\right).
  \]
  . Since
  $\widetilde{\chi}_c(c)=0$, we get
  \begin{align*}
    \int_0^{c}\rho^{\gamma}
    \abs*{\rho\partial_{\rho}(\widetilde{\chi}_c f)}^2\,d\rho
    &\geq
      \frac{(\gamma+1)^2}{4}
      \int_0^{c}\rho^{\gamma}
      \abs*{\widetilde{\chi}_c f}^2\,d\rho
      -
      \frac{\gamma+1}{2}
      \evalAt*{\rho^{\gamma+1}
      \abs*{\widetilde{\chi}_c f}^2}_{\rho=0}^{\rho=c} \\
    &\geq
      \frac{(\gamma+1)^2}{4}
      \int_0^{\frac c2}\rho^{\gamma}|f|^2\,d\rho
      +
      \lim_{\rho\to0}
      \frac{\gamma+1}{2}\rho^{\gamma+1}|f|^2 .
  \end{align*}
  Expanding the left-hand side gives
  \begin{align*}
    \rho\partial_\rho(\widetilde{\chi}_c f)
    =
    \widetilde{\chi}_c \rho\partial_\rho f
    +
    \rho\widetilde{\chi}_c' f .
  \end{align*}
  Hence, by Cauchy's inequality,
  \begin{align*}
    \int_0^{c}\rho^{\gamma}
    \abs*{\rho\partial_{\rho}(\widetilde{\chi}_c f)}^2\,d\rho
    &\leq
      2\int_0^{c}\rho^{\gamma}
      \abs*{\rho\partial_{\rho}f}^2\,d\rho
      +
      2\int_{\frac c2}^{c}
      \rho^{\gamma+2}
      |\widetilde{\chi}_c'|^2|f|^2\,d\rho \\
    &\leq
      2\int_0^{c}\rho^{\gamma}
      \abs*{\rho\partial_{\rho}f}^2\,d\rho
      +
      O(c^{-2})
      \int_{\frac c2}^{c}
      \rho^{\gamma+2}|f|^2\,d\rho .
  \end{align*}
  Combining the two estimates and changing the implicit constant gives the
  claim.
\end{proof}

We will use cutoff functions to localize multipliers close to the event horizon or null infinity. 
For $\chi_c\vcentcolon=\chi\big(\frac{\rho}{c} \big)$ with $c>0$ we deduce from \zcref{le:divergPP-gen}
\begin{align}\label{current-cutoff}
  \begin{split}
    \D\cdot \JCurrent{( \chi_c X,  0, \chi_c J)}[\psi]={}&  \EMTensor[\psi]  \c {}^{(\chi_{c} X)}\pi + \div (\chi_c J|\psi|^2)+ \chi_c X(\psi) \Box_\g \psi \\
    ={}& \chi_c \D\cdot \JCurrent{( X, 0, J)}[\psi]+  \chi_c'\EMTensor (X, \pr_{\rho}) [\psi] +\chi_{c}' J^{\rho}|\psi|^2\\
    ={}& \chi_c \D\cdot \JCurrent{( X, 0, J)}[\psi]- O(c^{-1})\mathds{1}_{\{c \leq \rho \leq 2c\}}\big(\EMTensor (X, \pr_{\rho}) [\psi]+J^{\rho}|\psi|^2\big).
  \end{split}
\end{align}

\subsubsection{Divergence theorem for mixed pseudo-differential operators}
\label{sec:mixed-PDO-class}\label{sec:pseudo-diff:wave-commutation}

To treat trapping more sharply, we will replace classical radial
multipliers by pseudodifferential ones localized in phase space. For
this reason in this section, we define a class of pseudo-differential
multipliers which are pseudo-differential in space but differential in
time. This class of operators will allow for some flexibility to
account for the frequency-dependent behavior of trapping for extremal
Kerr--Newman, while also allowing the multiplier to produce local
energy decay estimates. We refer the reader to
\cite{maEnergyMorawetzEstimatesWave2024,
  lindbladLocalEnergyEstimate2020, tataruLocalEnergyEstimate2011} for
in-depth introductions to such operators and their properties.

\begin{definition}[$x_0$-tangential symbols on $\Real^d$]
  For $m\in \Real$, let
  $\TanSymClass{m}\subset C^{\infty}(\Real^{d-1}\times \Real^{d-1})$, the
  set of $x_0$-tangential symbols of order $m$, consist of functions
  $a(x, \xi): \Real^{d}\times \Real^{d-1}\to
  \Complex$ such that for all multi-indices $\alpha,\beta$,
  \begin{equation*}
    \abs*{\partial_x^{\alpha}\partial_{\xi}^{\beta}a(x,\xi)}\le C_{\alpha,\beta}\bangle*{\xi}^{m-\abs*{\beta}},
  \end{equation*}
  for all $x\in \Real_{x_0}\times \Real^{d-1}$, and let
  $\TanOpClass{m}$ denote the associated operator class. 
\end{definition}

\begin{definition}[$x_0$-mixed symbols on $\Real^d$]
  \label{def:mixed-sym-class}
  For $m\in \Real$, $n\in \Natural$, we define the class
  $\MixedSymClass{m}{n}(\Real^d)\vcentcolon=\MixedSymClass{m}{n}[x_0](\Real^d)\subset
  C^{\infty}(\Real^d\times \Real^d)$ of \emph{mixed symbols of order
    $(m,n)$}, such that, decomposing $\xi = (\xi_0,\xi')$,
  \begin{equation}\label{eq:a-x-xi}
    a(x,\xi) = \sum_{j=0}^na_{m-j}(x,\xi')\xi_0^j, \qquad a_{m-j}\in \TanSymClass{m-j}(\Real^d),
  \end{equation}
  and let $\MixedOpClass{m}{n}$ denote the associated operator class. 
\end{definition}

We define the Weyl quantization of a mixed symbol.
\begin{definition}[Weyl quantization of mixed symbol]
  Let $(m,n)\in \Real\times\Natural$, and $a\in \MixedSymClass{m}{n}(\Real^d)$ as in \zcref{eq:a-x-xi}.
  Then $\WeylQ{a}$, the \emph{Weyl quantization of $a$}, is given by
  \begin{equation*}
    \WeylQ{a}=\sum_{j=0}^n\sum_{k=0}^j2^{-k}\binom{j}{k}\WeylQ{D_{x_0}^ka_{m-j}}D^{j-k}_{x_0},
  \end{equation*}
  where $\WeylQ{D_{x_0}^ka_{m-j}}$ is the Weyl quantization in
  $\Real^{d-1}$ of the $x_0$-tangential symbol $D_{x_0}^ka_{m-j}$. 
\end{definition}

In this section, we present the main equivalent of the divergence
theorem in \eqref{eq:general-divergence-theorem} that we will use when
considering pseudo-differential multipliers. We recall that throughout
this section we will use the terms skew-symmetric and symmetric to be
defined with respect to the standard Hermitian $L^2$ inner product structure.

\begin{lemma}
  \label{lemma:PsiDO-divergence-theorem}
  Let $\psi$ be a function such that
  $\supp\psi\subset \Manifold_{\operatorname{trap}}$. Let
  \begin{equation*}
    \widetilde{\MorawetzVF}\vcentcolon=
    \widetilde{\MorawetzVF}_1
    + \widetilde{\MorawetzVF}_0\partial_t      
  \end{equation*}
  be a first-order anti-symmetric properly supported pseudodifferential operator on
  $\Manifold$ where
  $\widetilde{\MorawetzVF}_i\in
  \TanOpClass{i}(T^{*}\Manifold)$ and let
  \begin{equation*}
    \widetilde{\MorawetzLagrangeCorr}\vcentcolon=
    \widetilde{\MorawetzLagrangeCorr}_0
    + \widetilde{\MorawetzLagrangeCorr}_{-1}\partial_t
  \end{equation*}
  be a zero-order symmetric properly supported pseudodifferential operator on
  $\Manifold$, where
  $\widetilde{\MorawetzLagrangeCorr}_i\in \TanSymClass{i}(T^{*}\Manifold)$. 
  Then, 
  \begin{equation} \label{eq:PsiDO-divergence-theorem}
    -\bangle*{\Box_{\Metric}\psi, \big(\widetilde{\MorawetzVF}+\widetilde{\MorawetzLagrangeCorr}\big)\psi}_{L^2(\Manifold(\tau_1,\tau_2))}
    = \KCurrentPert{\widetilde{\MorawetzVF}, \widetilde{\MorawetzLagrangeCorr}}[\psi](\tau_1,\tau_2)
    + \JCurrentPert{\widetilde{\MorawetzVF}, \widetilde{\MorawetzLagrangeCorr}}[\psi](\tau_2)
    - \JCurrentPert{\widetilde{\MorawetzVF}, \widetilde{\MorawetzLagrangeCorr}}[\psi](\tau_1),
  \end{equation}
  where
  \begin{align}
    2\KCurrentPert{\widetilde{\MorawetzVF}, \widetilde{\MorawetzLagrangeCorr}}[\psi](\tau_1,\tau_2)
    \vcentcolon={}& \bangle*{[\widetilde{\MorawetzVF}, \Box_{\Metric}]\psi,\psi}_{L^2(\Manifold(\tau_1,\tau_2))}
                    - \bangle*{\left(\widetilde{\MorawetzLagrangeCorr}\Box_{\Metric} +\Box_{\Metric}\widetilde{\MorawetzLagrangeCorr}\right)\psi, \psi}_{L^2(\Manifold(\tau_1,\tau_2))},
                    \label{eq:KCurrentPert-def}
  \end{align}
  and 
  \begin{align}
    \JCurrentPert{\widetilde{\MorawetzVF}, \widetilde{\MorawetzLagrangeCorr}}[\psi](\tau)
    \lesssim{}& \norm*{\psi}^2_{H^1_c(\Sigma(\tau))}
                \label{eq:JCurrentPert-def}   
  \end{align}
  where $\norm*{\psi}_{\compactSobolev{s}(\Omega)} \vcentcolon={} \norm*{\chi_{\operatorname{comp}}\psi}_{H^s(\Omega)}$
  for an arbitrary cutoff $\chi_{\operatorname{comp}}(r)$ supported away from $\EventHorizon$ and $\NullInfinity$.
\end{lemma}
\begin{proof}
  Recall that the wave operator $\Box_{\Metric}$ is a symmetric
  operator as an operator on $L^2(\Manifold)$. %
  Then, we have that 
  \begin{align*}
    2\Re\bangle*{\Box_{\Metric}\psi, \widetilde{\MorawetzVF}_1\psi}_{L^2(\Manifold(\tau_1,\tau_2))}
    ={}& \bangle*{\psi, \Box_{\Metric}^{*}\widetilde{\MorawetzVF}_1\psi}_{L^2(\Manifold(\tau_1,\tau_2))}
         + \bangle*{\widetilde{\MorawetzVF}_1^{*}\Box_{\Metric}\psi, \psi}_{L^2(\Manifold(\tau_1,\tau_2))}
         - \evalAt*{\JCurrentPert{\widetilde{\MorawetzVF}_1,0}[\psi](\tau)}_{\tau=\tau_1}^{\tau=\tau_2}
    \\
    ={}& \bangle*{\psi, \Box_{\Metric}\widetilde{\MorawetzVF}_1\psi}_{L^2(\Manifold(\tau_1,\tau_2))}
         - \bangle*{\widetilde{\MorawetzVF}_1\Box_{\Metric}\psi, \psi}_{L^2(\Manifold(\tau_1,\tau_2))}
         - \evalAt*{\JCurrentPert{\widetilde{\MorawetzVF}_1,0}[\psi](\tau)}_{\tau=\tau_1}^{\tau=\tau_2}
    \\
    ={}& \Re\bangle*{[\Box_{\Metric},\widetilde{\MorawetzVF}_1]\psi,\psi}_{L^2(\Manifold(\tau_1,\tau_2))}
         - \evalAt*{\JCurrentPert{\widetilde{\MorawetzVF}_1,0}[\psi](\tau)}_{\tau=\tau_1}^{\tau=\tau_2},
  \end{align*}
  where 
  \begin{equation*}
    \JCurrentPert{\widetilde{\MorawetzVF}_1,0}[\psi](\tau)
    \lesssim{} \norm*{\psi}^2_{H^1_c(\Sigma(\tau))}.
  \end{equation*}
  Similarly, we also have that
  \begin{align*}
    & 2\Re\bangle*{\Box_{\Metric}\psi, \widetilde{\MorawetzVF}_0\partial_t  \psi}_{L^2(\Manifold(\tau_1,\tau_2))}\\
    ={}&\bangle*{\Box_{\Metric}\psi, \widetilde{\MorawetzVF}_0\partial_t \psi}_{L^2(\Manifold(\tau_1,\tau_2))}
         + \bangle*{\widetilde{\MorawetzVF}_0\partial_t \psi, \Box_{\Metric}\psi}_{L^2(\Manifold(\tau_1,\tau_2))}
         - \evalAt*{\JCurrentPert{\widetilde{\MorawetzVF}_0\partial_t,0}[\psi](\tau)}_{\tau=\tau_1}^{\tau=\tau_2}
    \\
    ={}& - \bangle*{\widetilde{\MorawetzVF}_0\partial_t \Box_{\Metric}\psi, \psi}_{L^2(\Manifold(\tau_1,\tau_2))}
         + \bangle*{\Box_{\Metric}(\widetilde{\MorawetzVF}_0\partial_t \psi), \psi}_{L^2(\Manifold(\tau_1,\tau_2))}
         - \evalAt*{\JCurrentPert{\widetilde{\MorawetzVF}_0\partial_t,0}[\psi](\tau)}_{\tau=\tau_1}^{\tau=\tau_2}
    \\
    ={}& \Re\bangle*{[\Box_{\Metric},\widetilde{\MorawetzVF}_0\partial_t]\psi,\psi}_{L^2(\Manifold(\tau_1,\tau_2))}
         - \evalAt*{\JCurrentPert{\widetilde{\MorawetzVF}_0\partial_t,0}[\psi](\tau)}_{\tau=\tau_1}^{\tau=\tau_2},
  \end{align*}
  where
  \begin{align*}
    \JCurrentPert{\widetilde{\MorawetzVF}_0\partial_t,0}[\psi](\tau)
    \lesssim{} \norm*{\psi}^2_{H^1_c(\Sigma(\tau))}.
  \end{align*}
  Now let us consider the result of multiplying by the Lagrangian correction.
  We have that
  \begin{align*}
    & 2\Re\bangle*{\Box_{\Metric}\psi, \widetilde{\MorawetzLagrangeCorr}_0\psi}_{L^2(\Manifold(\tau_1,\tau_2))}\\
    ={}& \bangle*{\psi, \Box_{\Metric}^{*}\widetilde{\MorawetzLagrangeCorr}_0\psi}_{L^2(\Manifold(\tau_1,\tau_2))}
         + \bangle*{\widetilde{\MorawetzLagrangeCorr}_0^{*}\Box_{\Metric}\psi, \psi}_{L^2(\Manifold(\tau_1,\tau_2))}
         - \evalAt*{\JCurrentPert{0, \MorawetzLagrangeCorr_0}[\psi](\tau)}_{\tau=\tau_1}^{\tau=\tau_2}
    \\
    ={}& \bangle*{\psi, \Box_{\Metric}\widetilde{\MorawetzLagrangeCorr}_0\psi}_{L^2(\Manifold(\tau_1,\tau_2))}
         + \bangle*{\widetilde{\MorawetzLagrangeCorr}_0\Box_{\Metric}\psi, \psi}_{L^2(\Manifold(\tau_1,\tau_2))}
         - \evalAt*{\JCurrentPert{0, \MorawetzLagrangeCorr_0}[\psi](\tau)}_{\tau=\tau_1}^{\tau=\tau_2}
    \\
    ={}& \Re\bangle*{\left(\Box_{\Metric}\widetilde{\MorawetzLagrangeCorr}_0 + \widetilde{\MorawetzLagrangeCorr}_0\Box_{\Metric}\right)\psi,\psi}_{L^2(\Manifold(\tau_1,\tau_2))}
         - \evalAt*{\JCurrentPert{0, \MorawetzLagrangeCorr_0}[\psi](\tau)}_{\tau=\tau_1}^{\tau=\tau_2},
  \end{align*}
  where
  \begin{align*}
    \JCurrentPert{0, \MorawetzLagrangeCorr_0}[\psi](\tau)
    \lesssim{} \norm*{\psi}^2_{H^1_c(\Sigma(\tau))}
    .
  \end{align*}
  Similarly, we also have that
  \begin{align*}
    &2\Re\bangle*{\Box_{\Metric}\psi, \left(\widetilde{\MorawetzLagrangeCorr}_{-1}\partial_t + \partial_t\widetilde{\MorawetzLagrangeCorr}_{-1}\right)\psi}_{L^2(\Manifold(\tau_1,\tau_2))}\\
    ={}& \bangle*{\psi, \Box_{\Metric}^{*}\widetilde{\MorawetzLagrangeCorr}_{-1}\partial_t\psi}_{L^2(\Manifold(\tau_1,\tau_2))}
         + \bangle*{\left(\widetilde{\MorawetzLagrangeCorr}_{-1}\partial_t\right)^{*}\Box_{\Metric}\psi, \psi}_{L^2(\Manifold(\tau_1,\tau_2))}\\
    ={}& \bangle*{\psi, \Box_{\Metric}\left(\widetilde{\MorawetzLagrangeCorr}_{-1}\partial_t\psi\right)}_{L^2(\Manifold(\tau_1,\tau_2))}
         + \bangle*{\left(\widetilde{\MorawetzLagrangeCorr}_{-1}\partial_t\right)\Box_{\Metric}\psi, \psi}_{L^2(\Manifold(\tau_1,\tau_2))}
         - \evalAt*{\JCurrentPert{0, \widetilde{\MorawetzLagrangeCorr}_{-1}\partial_t}[\psi](\tau)}_{\tau=\tau_1}^{\tau=\tau_2}
    \\
    ={}& \Re\bangle*{\left(\Box_{\Metric}\widetilde{\MorawetzLagrangeCorr}_{-1}\partial_t + \widetilde{\MorawetzLagrangeCorr}_{-1}\partial_t \Box_{\Metric}\right)\psi,\psi}_{L^2(\Manifold(\tau_1,\tau_2))}
         - \evalAt*{\JCurrentPert{0, \widetilde{\MorawetzLagrangeCorr}_{-1}\partial_t}[\psi](\tau)}_{\tau=\tau_1}^{\tau=\tau_2},
  \end{align*}
  where
  \begin{align*}
    \JCurrentPert{0, \widetilde{\MorawetzLagrangeCorr}_{-1}\partial_t}[\psi](\tau)
    \lesssim{} \norm*{\psi}^2_{H^1_c(\Sigma(\tau))}.
  \end{align*}
  We then conclude by using the definitions of
  $\KCurrentPert{\widetilde{\MorawetzVF},\widetilde{\MorawetzLagrangeCorr}}[\psi]$,
  $\JCurrentPert{\widetilde{\MorawetzVF},\widetilde{\MorawetzLagrangeCorr}}[\psi]$
  in \zcref[noname]{eq:KCurrentPert-def} and \zcref[noname]{eq:JCurrentPert-def}. 
\end{proof}

\section{Main theorem}\label{sec:main-theorem}

In this section we state the main theorem. In
\zcref{sec:function-spaces} we define the energy norms and function
spaces, suitably weighted at the event horizon and null infinity,
appearing in the main theorem.  In \zcref{sec:statement-theorem} we
give the precise statement of the main theorem, and in \zcref{sec:decay-blow} we present decay and blow-up results as corollaries of the main theorem.

\subsection{Function spaces and energy norms}\label{sec:function-spaces}

We introduce several function spaces capturing regularity at the horizon via $b$-Sobolev spaces, decay at null infinity via weighted Sobolev norms, and degeneracy at trapping via microlocal norms.

We denote $A^{\Horizon}_s, A^{\NullInfinity}_s
$ a basis of $\bDiffH^s(\Manifold),\bDiffI^s(\Manifold)
$
respectively, and $A^{\Horizon,\NullInfinity}_s
$ a basis of $
\bDiffHI^s(\Manifold)
$.

For $\Omega\subset \Manifold$ and a real function $\psi: \Manifold\to \Real$, we define:
\begin{itemize}
\item the \emph{standard (compactified) Sobolev norm} for the Sobolev space $H^s(\Omega)$: 
  \begin{align*}
    \norm*{\psi}_{H^1(\Omega)}^2 \vcentcolon={}& \int_{\Omega\cap\{r\le 4M\}}\abs*{\partial_v\psi}^2+\abs*{\NablaAngular\psi}^2+\abs*{\partial_{\rhoH}\psi}^2+\abs*{\psi}^2
    \\
                                               &+ \int_{\Omega\cap \{r\ge 4M\} }\abs*{\partial_u\psi}^2+\abs*{\NablaAngular\psi}^2+\abs*{\partial_{\rhoI}\psi}^2+\abs*{\psi}^2, \\
    \norm*{\psi}_{H^s(\Omega)} \vcentcolon={}& \sum_{i=0}^{s-1}\norm*{A^{\Horizon,\NullInfinity}_i \psi}_{H^1(\Omega)},
  \end{align*} 
  and its weighted version, i.e. the \emph{weighted Sobolev norms} at $\EventHorizon$ and $\NullInfinity$ for the weighted Sobolev spaces $H_{\EventHorizon}^{s,\gamma}(\Omega) = \rhoH^{\gamma}H^{s}(\Omega)$ and $H_{\NullInfinity}^{s,\gamma}(\Omega) = \rhoI^{\gamma}H^{s}(\Omega)$:
  \begin{align*}
    \norm*{\psi}_{H_{\EventHorizon}^{1, \gamma}(\Omega)}^2
    \vcentcolon={}& \int_{\Omega\cap \{\rhoH\leq \rho_0\}}\rhoH^{-2\gamma} \big(\abs*{\partial_v\psi}^2+\abs*{\NablaAngular\psi}^2+\abs*{\partial_{\rhoH}\psi}^2+\abs*{\psi}^2 \big),\\
    \norm*{\psi}_{H_{\EventHorizon}^{s, \gamma}(\Omega)}
    \vcentcolon={}& \sum_{i=0}^{s-1}\norm*{A^{\Horizon}_i \psi}_{H_{\EventHorizon}^{1, \gamma}(\Omega)}
  \end{align*}
  and 
  \begin{align*}
    \norm*{\psi}_{H_{\NullInfinity}^{1, \gamma}(\Omega)}^2
    \vcentcolon={}& \int_{\Omega\cap \{\rhoI\leq \rho_1\}}\rhoI^{-2\gamma}\big(\abs*{\partial_u\psi}^2+\abs*{\NablaAngular\psi}^2+\abs*{\partial_{\rhoI}\psi}^2+\abs*{\psi}^2 \big), \\
    \norm*{\psi}_{H_{\NullInfinity}^{s, \gamma}(\Omega)}
    \vcentcolon={}& \sum_{i=0}^{s-1}\norm*{A^{\NullInfinity}_i \psi}_{H_{\NullInfinity}^{1, \gamma}(\Omega)}
  \end{align*}
  for some $\rho_0 \leq  r_{+}$, $\rho_1 \leq r_{+}^{-1}$.

  We also define the \emph{combined weighted Sobolev norm} 
  \begin{gather*}
    \norm*{\psi}_{H_{\EventHorizon, \NullInfinity}^{s, \gamma_1, \gamma_2}(\Omega)}\vcentcolon=\norm*{\psi}_{H_{\EventHorizon}^{s, \gamma_1}(\Omega)}+\norm*{\psi}_{H_{\NullInfinity}^{s, \gamma_2}(\Omega)}+\norm*{\psi}_{\compactSobolev{s}(\Omega)},
  \end{gather*}
  where $H^s_c(\Omega)$ is \emph{compact Sobolev space} with norm
  \begin{align*}
    \norm*{\psi}_{\compactSobolev{s}(\Omega)} \vcentcolon={}& \norm*{\chi_{\operatorname{comp}}\psi}_{H^s(\Omega)}
  \end{align*}
  for an arbitrary cutoff $\chi_{\operatorname{comp}}(r)$ supported away from $\EventHorizon$ and $\NullInfinity$.

  These norms control the first derivatives and the zero-th order term of $\psi$ with respect to the standard derivative $(\partial_v, \NablaAngular, \pr_{\rhoH})$ close to the event horizon and $(\partial_u, \NablaAngular, \pr_{\rhoI})$ close to null infinity.

\item the \emph{$b$-Sobolev norms} for the b-Sobolev spaces $\bSobolev{s}(\Omega)$:
  \begin{align*}
    \norm*{\psi}_{\bSobolev{0}(\Omega)}^2 \vcentcolon={}& \int_{\Omega} |\psi|^2, \\
    \norm*{\psi}_{\bSobolev{1}(\Omega)}^2 \vcentcolon={}& \int_{\Omega\cap\{r\le 4M\}}\abs*{\partial_v\psi}^2+\abs*{\NablaAngular\psi}^2+\abs*{\rhoH\partial_{\rhoH}\psi}^2+\abs*{\psi}^2
    \\
                                                        &+ \int_{\Omega\cap \{r\ge 4M\} }\abs*{\partial_u\psi}^2+\abs*{\NablaAngular\psi}^2+\abs*{\rhoI\partial_{\rhoI}\psi}^2+\abs*{\psi}^2, \\
    \norm*{\psi}_{\bSobolev{s}(\Omega)} \vcentcolon={}& \sum_{i=0}^{s-1}\norm*{A^{\Horizon,\NullInfinity}_i \psi}_{\bSobolev{1}(\Omega)}.
  \end{align*}
  and its weighted version, i.e. the \emph{weighted $b$-Sobolev norms} at $\EventHorizon$ and $\NullInfinity$ for the weighted $b$-Sobolev spaces $\bSobolevH{s,\gamma}(\Omega) = \rhoH^{\gamma}\bSobolev{s}(\Omega)$ and $\bSobolevI{s,\gamma}(\Omega) = \rhoI^{\gamma}\bSobolev{s}(\Omega)$:
  \begin{align*}
    \norm*{\psi}_{\bSobolevH{0, \gamma}(\Omega)}^2
    \vcentcolon={}& \int_{\Omega\cap \{\rhoH\leq \rho_0\}}\rhoH^{-2\gamma} \abs*{\psi}^2 ,\\
    \norm*{\psi}_{\bSobolevH{1, \gamma}(\Omega)}^2
    \vcentcolon={}& \int_{\Omega\cap \{\rhoH\leq \rho_0\}}\rhoH^{-2\gamma} \big(\abs*{\partial_v\psi}^2+\abs*{\NablaAngular\psi}^2+\abs*{\rhoH\partial_{\rhoH}\psi}^2+\abs*{\psi}^2 \big),\\
    \norm*{\psi}_{\bSobolevH{s, \gamma}(\Omega)}
    \vcentcolon={}& \sum_{i=0}^{s-1}\norm*{A^{\Horizon}_i \psi}_{\bSobolevH{1, \gamma}(\Omega)}
  \end{align*}
  and 
  \begin{align*}
    \norm*{\psi}_{\bSobolevI{0, \gamma}(\Omega)}^2
    \vcentcolon={}& \int_{\Omega\cap \{\rhoI\leq \rho_1\}}\rhoI^{-2\gamma}\abs*{\psi}^2 ,\\
    \norm*{\psi}_{\bSobolevI{1, \gamma}(\Omega)}^2
    \vcentcolon={}& \int_{\Omega\cap \{\rhoI\leq \rho_1\}}\rhoI^{-2\gamma}\big(\abs*{\partial_u\psi}^2+\abs*{\NablaAngular\psi}^2+\abs*{\rhoI\partial_{\rhoI}\psi}^2+\abs*{\psi}^2 \big), \\
    \norm*{\psi}_{\bSobolevI{s, \gamma}(\Omega)}
    \vcentcolon={}& \sum_{i=0}^{s-1}\norm*{A^{\NullInfinity}_i \psi}_{\bSobolevI{1, \gamma}(\Omega)},
  \end{align*}
  for some $\rho_0 \leq  r_{+}$, $\rho_1 \leq r_{+}^{-1}$. 

  We also define the \emph{combined weighted $b$-Sobolev norm} 
  \begin{align*}
    \norm*{\psi}_{\bSobolevHI{s, \gamma_1, \gamma_2 }(\Omega)}&\vcentcolon=\norm*{\psi}_{\bSobolevH{s, \gamma_1}(\Omega)}+\norm*{\psi}_{\bSobolevI{s, \gamma_2}(\Omega)}+\norm*{\psi}_{\compactSobolev{s}(\Omega)}.
  \end{align*}

  These norms control the first derivatives and the zero-th order term of $\psi$ with respect to the $b$-derivative $(\partial_v, \NablaAngular, \rhoH\pr_{\rhoH})$ close to the event horizon and $(\partial_u, \NablaAngular, \rhoI\pr_{\rhoI})$ close to null infinity.

  Observe that for $\rho_0 \ll r_{+}$, $\rho_1 \ll r_{+}^{-1}$
  \begin{align*}
    \norm*{\psi}_{\bSobolevHI{s, \gamma_1, \gamma_2}(\Omega)} \gtrsim  \norm*{\psi}_{\bSobolevHI{s, \lambda_1, \lambda_2}(\Omega)} \qquad \text{if $\gamma_1 \geq \lambda_1$, $\gamma_2 \geq \lambda_2$}.
  \end{align*}

\item the \emph{combined weighted mixed Sobolev and b-Sobolev norms} which are Sobolev norms at $\EventHorizon$ and b-Sobolev norms at $\NullInfinity$, given by
  \begin{align*}
    \norm*{\psi}_{H_{\EventHorizon, \operatorname{b}\NullInfinity}^{s, \gamma_1, \gamma_2 }(\Omega)}\vcentcolon=\norm*{\psi}_{H_{\EventHorizon}^{s, \gamma_1}(\Omega)}+\norm*{\psi}_{\bSobolevI{s, \gamma_2}(\Omega)}+\norm*{\psi}_{\compactSobolev{s}(\Omega)}.
  \end{align*}

\item the \emph{auxiliary degenerate bulk norm}, given by
  \begin{equation*}
    \begin{split}
      \norm*{\psi}^2_{\SobolevDeg{1}(\Manifold(\tau_1,\tau_2))}
      \vcentcolon={}& \int_{\Manifold(\tau_1,\tau_2)\cap \{ r \leq 4M\}}\abs*{\HawkingHorizon\psi}^2
                      + \rhoH^2\left( \abs*{\rhoH\partial_{\rhoH}\psi}^2 + \abs*{\NablaAngular\psi}^2 + \abs*{\psi}^2 \right)\\
                    & + \int_{\Manifold(\tau_1,\tau_2)\cap \{ r \geq 4M\}}
                      \rhoI^3\abs*{\partial_u\psi}^2
                      + \rhoI^5\left( \abs*{\rhoI\partial_{\rhoI}\psi}^2 + \abs*{\NablaAngular\psi}^2 + \abs*{\psi}^2 \right),\\
      \norm*{\psi}_{\SobolevDeg{s}(\Manifold(\tau_1,\tau_2))}={}& \sum_{i=0}^{s-1}\norm*{A^{\Horizon,\NullInfinity}_i\psi}_{\SobolevDeg{1}(\Manifold(\tau_1,\tau_2))}
                                                                  .
    \end{split}    
  \end{equation*}
\end{itemize}
Finally, we define the \emph{final weighted norm}, appearing in the statement of the main theorem, as given by 
\begin{align*}
  \norm*{\psi}_{\SobolevfinHI{s, -\frac{\alpha+1}{2}, -\frac{\beta+1}{2}}(\Sigma(\tau))}\vcentcolon =
  \norm*{\psi}_{H_{\EventHorizon}^{s, -\frac{\alpha+1}{2}}(\Sigma_{\EventHorizon}(\tau))}
  +\norm*{\widecheck{\psi}}_{H_{\NullInfinity}^{s, -\frac{\b+3}{2}}(\Sigma_{\NullInfinity}(\tau))}+\norm*{\psi}_{\compactSobolev{s}(\Sigma(\tau))}
\end{align*}   
where recall\footnote{The shift from $-\frac{\b+1}{2}$ to $-\frac{\b+3}{2}$ in the radiation field term reflects the relation $\psi=\rhoI \widecheck{\psi}$.} that $\widecheck{\psi}=r\psi$. 

\medskip

In order to introduce the Sobolev norms which are degenerate at trapping, we write the principal symbol of the scalar wave as
\begin{equation}
  \label{eq:trapping-norm:sigmai-def}
  p = p_{M,a,Q} \vcentcolon= \Metric^{-1}(dt, dt)(\sigma-\sigma_1)(\sigma-\sigma_2),
\end{equation}
where $\sigma_i = \sigma_i(r, \omega;\xi, \eta)$ are distinct smooth
1-homogeneous real symbols with respect to $(\xi, \eta)$, the spatial frequency variables.
Define then 
\begin{equation}
  \label{eq:trapping-norm:elli-def}
  \ell_i(r,\phi;\xi,\eta)
  = r - r_i(\sigma_i,\eta), \qquad \qquad r_i(\sigma_i,\eta) = \rTrapping(\sigma_i, \eta),
\end{equation}
where 
$\rTrapping(\sigma, \eta)>M$ solves
\begin{equation*}
  \widetilde{\TT}_{\sigma, \FreqPhi}(\rTrapping(\sigma, \eta)) = 0,
\end{equation*}
with $\widetilde{\TT}_{\sigma, \FreqPhi}$ given by \zcref{eq:trapped-region-KN}.
We then define the following norms. 

\begin{itemize}
\item the \emph{trapped Sobolev norm}: 
  \begin{align*}
    \norm*{\psi}_{\bSobolevTrap{1}(\Manifold(\tau_1, \tau_2))}^2 
    \vcentcolon={}& \int_{\Manifold_{\operatorname{trap}}(\tau_1,\tau_2)}\abs*{\partial_r^{\mathrm{BL}}\psi}^2 + \abs*{\psi}^2
    \\
                  &+ \int_{\Manifold_{\operatorname{trap}}(\tau_1,\tau_2)}\sum_{i\neq j}\abs*{\mathring{\chi}\left(1-\frac{r_i(\sigma_i,\eta)}{r}\right)\left(D_t-\sigma_j(D,x)\right)(\mathring{\chi}\psi)}^2.  
  \end{align*}
\item the \emph{weighted trapped $b$-Sobolev norms} at $\EventHorizon$ and $\NullInfinity$:
  \begin{equation*}
    \begin{split}
      \norm*{\psi}_{\bSobolevHITrap{1, \gamma_1, \gamma_2}(\Manifold(\tau_1, \tau_2))}^2 
      \vcentcolon={}&\norm*{\psi}_{\bSobolevHI{1, \gamma_1, \gamma_2 }(\Manifold_{\nontrap}(\tau_1, \tau_2))}^2
                      + \norm*{\psi}_{\bSobolevTrap{1}(\Manifold(\tau_1, \tau_2))}^2 ,
    \end{split}   
  \end{equation*}
  where $D_t = -i\partial_t$ and $\sigma_j(D,x)$ denotes the Weyl quantization of the symbol $\sigma_j$ defined in \zcref{eq:trapping-norm:sigmai-def}, and $\mathring{\chi}$ is a cutoff localizing to the
  trapping region
  .

\end{itemize}

\subsection{Statement of the main theorem}\label{sec:statement-theorem}

We are now ready to state our main theorem.
\begin{theorem}[Main Theorem]\label{thm:main-thm}
  Let $\psi$ solve the inhomogeneous wave equation 
  \[\Box_\g \psi = F\]
  in an extremal Kerr--Newman spacetime with $\mathbf{a}\vcentcolon=\frac{|a|}{M}\ll 1$. 
  Then for
  \begin{align*}
    \a \in \big( -1+4\mathbf{a}^2 ,  -\de_1\big), \qquad  \qquad \b \in \big( 1, 3\big)
  \end{align*}
  with $\de_1 \ll 1$, the following $\EventHorizon$- and $\NullInfinity$-weighted energy estimate holds true for $s \in \mathbb{N}$, $s \geq 1$ and any $\tau_1<\tau_2$:                   \begin{align}\label{eq:final-thm-with-hierarchies}
    \begin{split}
      & \norm*{\psi}_{\SobolevfinHI{s, -\frac{\alpha+1}{2}, -\frac{\beta+1}{2}}(\Sigma(\tau_2))}+ \norm*{\psi}_{\bSobolevH{s, - \frac{\alpha+1}{2}}(\EventHorizon(\tau_1,\tau_2))} +\norm*{\widecheck{\psi}}_{\bSobolevI{s, -\frac{\b+1}{2}}(\NullInfinity(\tau_1, \tau_2))}
        + \norm*{\psi}_{\bSobolevHI{s-1, -\frac{\a}{2}, -\frac \b 2}(\Manifold(\tau_1, \tau_2))}\\
      \les{}& \norm*{\psi}_{\SobolevfinHI{s, -\frac{\alpha+1}{2}, -\frac{\beta+1}{2}}(\Sigma(\tau_1))}+\norm*{F}_{\bSobolevHI{s-1,-\frac \a 2,-\frac{\b-4}{2}}(\Manifold(\tau_1,\tau_2))}.
    \end{split}
  \end{align}   
\end{theorem}
\begin{proof}
  See \zcref{proof:main-them-s1} for the proof for $s=1$ and \zcref{proof:main-them-s2} for $s \geq 2$.
\end{proof}

\begin{remark}\label{rem:main-theorem-s=1}
  In \zcref{eq:final-thm-with-hierarchies} we recognize a familiar loss of derivative in the bulk norm $\norm*{\psi}_{\bSobolevHI{s-1, -\frac{\a}{2}, -\frac \b 2}(\Manifold(\tau_1, \tau_2))}$  due to the trapping phenomenon of null geodesics. For $s=1$, we prove the following more precise statement:
  \begin{align}\label{eq:main-theorem-s=1}
    \begin{split}
      &\norm*{\psi}_{\SobolevfinHI{1, -\frac{\alpha+1}{2}, -\frac{\beta+1}{2}}(\Sigma(\tau_2))}+ \norm*{\psi}_{\bSobolevH{1, - \frac{\alpha+1}{2}}(\EventHorizon(\tau_1,\tau_2))} +\norm*{\widecheck{\psi}}_{\bSobolevI{1, -\frac{\b+1}{2}}(\NullInfinity(\tau_1, \tau_2))}+
        \norm*{\psi}_{\bSobolevHITrap{1, -\frac{\a}{2}, -\frac \b 2}(\Manifold(\tau_1, \tau_2))}\\
      &\les \norm*{\psi}_{\SobolevfinHI{1, -\frac{\alpha+1}{2}, -\frac{\beta+1}{2}}(\Sigma(\tau_1))}+\norm*{F}_{\bSobolevHI{0,-\frac \a 2,-\frac{\b-4}{2}}(\Manifold(\tau_1,\tau_2))},
    \end{split}
  \end{align}
  with the weighted trapped norm in the bulk. Since 
  \[\norm*{\psi}_{\bSobolevHITrap{1, -\frac{\a}{2}, -\frac \b 2}(\Manifold(\tau_1, \tau_2))} \gtrsim \norm*{\psi}_{\bSobolevHI{0, -\frac{\a}{2}, -\frac \b 2}(\Manifold(\tau_1, \tau_2))}\]
  the estimate \zcref{eq:main-theorem-s=1}, proved in \zcref{proof:main-them-s1}, implies \zcref{eq:final-thm-with-hierarchies} for $s=1$.
\end{remark}

The proof of \zcref{thm:main-thm} is obtained in the following steps.
\begin{enumerate}
\item First, we obtain general $\EventHorizon$- and $\NullInfinity$-weighted hierarchies for solutions to the wave equations on extremal Kerr--Newman spacetime, not necessarily slowly rotating ones. Those estimates are localized near the event horizon and null infinity respectively, so they are conditional on some on a compact region. These estimates are obtained in \zcref{sec:hierarchy}.
\item We obtain unconditional energy-Morawetz estimates by combining a choice of differential multipliers from axial symmetry outside trapping and a choice of pseudodifferential multipliers at trapping. To close the energy-Morawetz estimates we also need to use a degenerate redshift estimate, which corresponds to an element of the $\EventHorizon$-hierarchy. These estimates are obtained in \zcref{sec:Morawetz}.
\item Finally, we obtain the proof of \zcref{thm:main-thm} for $s=1$ in \zcref{proof:main-them-s1} by combining the weighted hierarchies and the Energy-Morawetz estimates, and we obtain the higher order derivatives version of \zcref{thm:main-thm} in \zcref{proof:main-them-s2}.
\end{enumerate}

In the proof of \zcref{thm:main-thm} we use the following constants:
\begin{itemize}
\item $\delta_{1}<\frac 1 2 $ is the small constant appearing in the weights for the degenerate redshift estimate in \zcref{prop:extremal-redshift}.
\item $\daxi>0$ is the small constant in the construction of the axially symmetric choice in \zcref{prop:ax-symm-choice},
\item $c_{red}$ is the small constant which multiplies the degenerate redshift multiplier obtained in \zcref{prop:extremal-redshift}, 
\item $c_{\That}$ is the small constant which multiplies the Lagrangian corrector $w_{\That}$ used in the proof of \zcref{prop:Mor-axially-sym:main-bulk},
\item $C_{\That}$ is the large constant which multiplies the energy multiplier in the proof of \zcref{prop:energy-Morawetz}.
\end{itemize}
We take
\begin{gather*}
  C_{\That} \gg 1, \quad 0<\de_1 \ll 1, \qquad c_{\That} \ll 1, \quad c_{red} \ll 1, \qquad 
  \daxi \ll1,\qquad
  \mathfrak{a}\ll 1.
\end{gather*}

\subsection{Applications: decay and blow-up}\label{sec:decay-blow}

Following standard techniques, the weighted energy boundedness statement of \zcref{thm:main-thm} implies pointwise decay estimates for the solution, up to and including the event horizon. In this section, we restrict to the case of $F=0$, although the decay estimates would also remain valid if $F$ satisfied a suitable time-decay assumption. 

\begin{corollary}[Pointwise decay estimate]\label{cor:pointwise-decay}
  For solutions $\psi$ to the wave equation 
  \[\Box_\g \psi = 0\]
  on extremal Kerr--Newman with $\mathbf{a}\vcentcolon=\frac{|a|}{M}\ll 1$,  the following pointwise decay estimate holds true 
  \begin{align*}
    | \psi(\tau, r , \theta, \phi)| \les
    \frac{\sqrt{E_{\operatorname{init}}^{4}}}
    {r \, \langle \tau\rangle^{\frac{1-\de-\de_1}{2}}}, \qquad r\in [M, \infty)
  \end{align*}
  where $\de_1 \ll 1$, $\de> 4\mathbf{a}^2$, with $\de+\de_1<1$ and
  \begin{align*}
    E_{\operatorname{init}}^s=\norm*{\psi}_{\SobolevfinHI{s, -\frac{\delta}{2}, -\frac{2+\delta}{2}}(\Sigma(\tau_1))}^2.
  \end{align*}
\end{corollary}
\begin{proof}
  See \zcref{proof:corollary-decay}.
\end{proof}

Finally, combining the decay estimate of \zcref{cor:pointwise-decay} and the Aretakis conservation law along the event horizon, we deduce the instability result along the event horizon.
We first recall here the Aretakis conservation law, proved as a particular case of the one obtained in
\cite{aretakisHorizonInstabilityExtremal2015}. 
\begin{proposition}[Proposition 3.1 in \cite{aretakisHorizonInstabilityExtremal2015}] 
  For solutions $\psi$ to the wave equation 
  \[\Box_\g \psi = 0\]
  on extremal Kerr--Newman, the quantity
  \begin{equation}
    H_{0}^{\text{eKN}}[\psi](\tau)\vcentcolon=\int_{\Sigma(\tau)\cap \EventHorizon}\Big(a^2\sin^{2}\theta \pr_v\psi + 2(M^2+a^2)  \pr_{\rhoH}\psi + 2M \psi \Big) d\mathring{\gamma},
    \label{cons-eKN}
  \end{equation}
  is conserved along $\EventHorizon$, i.e. it is independent of $\tau$.
\end{proposition}
\begin{proof}
  Using \zcref{wave-iEF}, we rewrite the wave operator as 
  \begin{equation}\label{eq:box-for-conservation-law}
    \begin{split}
      |q|^2\Box_\g \psi={}& \That\Big(a^2\sin^2\th \pr_v\psi+ 2(r^2+a^2)\pr_{\rhoH}\psi + 2 r \psi  \Big)\\
                          &+\pr_{\phiIEF} \Big(\frac{|q|^2+r^2+a^2}{r^2+a^2} a \pr_v\psi - 2  \frac{ar}{r^2+a^2} \psi \Big)+\lapp_{\SSS^2}\psi-(\rhoH\pr_{\rhoH})^2\psi-\rhoH\pr_{\rhoH}\psi,
    \end{split}
  \end{equation}
  where recall that $\That=\partial_v+\frac{a}{r^2+a^2}\partial_{\phiIEF}$.
  Restricting it to $\EventHorizon$
  and integrating along $\Sigma(\tau)\cap \HH$ and then in $\tau$, where we observe that $\partial_\tau|_{\EventHorizon}=\partial_v|_{\EventHorizon}$, we obtain 
  \begin{equation*}
    H_{0}^{\text{eKN}}[\psi](\tau)-H_{0}^{\text{eKN}}[\psi](0)=-\int^{\tau}_0 \int_{\Sigma(\tau')\cap \HH}\left((\rhoH\pr_{\rhoH})^2\psi+\rhoH\pr_{\rhoH}\psi\right)d\si d\tau'=0,
  \end{equation*}
  therefore 
  $H_{0}^{\text{eKN}}[\psi](\tau)$ is conserved along $\HH$. 
\end{proof}

Notice that
$H_{0}^{\text{eKN}}[\psi]$ is a constant which depends only on the initial data and is generically non-zero.
As a consequence, we obtain the following manifestation of the Aretakis instability.

\begin{corollary}[Aretakis instability]\label{cor:aretakis}
  For solutions $\psi$ to the wave equation 
  \[\Box_\g \psi = 0\]
  on extremal Kerr--Newman with $\mathbf{a}\vcentcolon=\frac{|a|}{M}\ll 1$, with $E_{\operatorname{init}}^s<\infty$ for sufficiently high $s$,
  \begin{enumerate}
  \item the first transversal derivative of $\psi$ does not decay along $\EventHorizon$, i.e. there exists a constant $c$ depending on $M$ and $a$ such that
    \begin{align*}
      \sup_{\Sigma(\tau) \cap \EventHorizon} |\partial_{\rhoH}\psi| \geq c \big|H_{0}^{\text{eKN}}[\psi]\big|,
    \end{align*}
  \item there is a second order derivative of $\psi$ which blows up
    asymptotically along $\EventHorizon$.
  \end{enumerate}
\end{corollary}
\begin{proof}
  We differentiate \zcref{eq:box-for-conservation-law} with respect to $\partial_{\rhoH}$ and evaluate at the event horizon:
  \begin{align*}
    \pr_{\rhoH}(|q|^2\Box_\g \psi)\Bigg|_{\HH^+}
    ={}& \pr_v \Big(a^2\sin^2\th (\pr_v\pr_{\rhoH}\psi)+ 2(r^2+a^2)(\pr_{\rhoH}^2\psi) + 6 r (\pr_{\rhoH}\psi) + 2\psi \Big)\Big|_{\HH^+}\\
       &+\pr_{\phiIEF}\Big(2a (\pr_v\pr_{\rhoH}\psi)+ 2a (\pr_{\rhoH}^2\psi) \Big) \Big|_{\HH^+}+\Big(\lapp_{(\theta,\phiIEF)}+2\Big)(\partial_{\rhoH}\psi)\Big|_{\HH^+}\\
    =&\That\Big(a^2\sin^2\th (\pr_v \pr_{\rhoH}\psi)+ 2(M^2+a^2)(\pr_{\rhoH}^2\psi) + 6 M (\pr_{\rhoH}\psi) + 2 \psi \Big)\Bigg|_{\HH^+}\\
       &+\pr_{\phiIEF}\Big(\frac{2M^2+2a^2-a^2\sin^2\th}{M^2+a^2} a  (\pr_v \pr_{\rhoH}\psi) -  \frac{6aM}{M^2+a^2}  (\pr_{\rhoH}\psi) -  \frac{2a}{M^2+a^2}  \psi \Big)\Bigg|_{\HH^+}\\
       &+\Big(\lapp_{\SSS^2}+2\Big)(\pr_{\rhoH}\psi)\Bigg|_{\HH^+}.
  \end{align*}
  Integrating the above on $S_\tau=\Sigma_{\tau}\cap \HH$ and in $\tau$, we see that 
  \begin{equation}
    \int^\tau_0 \int_{S_{\tau'}}\That \Big(\frac{1}{2} a^2\sin^2\th (\pr_v \pr_{\rhoH}\psi)+ (M^2+a^2)(\pr_{\rhoH}^2\psi) + 3 M (\pr_{\rhoH}\psi) +  \psi \Big)d\si d\tau' + \int_0^\tau\int_{S_{\tau'}} (\pr_{\rhoH}\psi) d\si d\tau'=0.
    \label{blowup-eKN}
  \end{equation}

  Suppose now that $H_0^{eKN}[\psi](0)\neq 0$. From \zcref{cor:pointwise-decay} we have that $\psi$ decay along $\HH$.
  Since $\partial_v$ and $\partial_{\phiIEF}$ commute with $\Box_g$, the decay statement of \zcref{cor:pointwise-decay} may be applied, with higher-order initial energy, to $\partial_v\psi$ and $\partial_{\phiIEF}\psi$, hence both $\partial_v\psi$ and $\That\psi$ decay along $\HH$.
  Then from the fact that $H_0^{eKN}[\psi]$ is conserved along $\EventHorizon$, we deduce from \zcref{cons-eKN} 
  \begin{equation*}
    \int_{S_\tau} (\pr_{\rhoH}\psi) d\si \rightarrow \frac{1}{2(M^2+a^2)}H_0^{eKN}[\psi](0)\neq 0, \quad \text{as}\ \tau\rightarrow \infty.
  \end{equation*}
  Finally, \zcref{blowup-eKN} then implies 
  \begin{equation*}
    \int^\tau_0 \That \int_{S_{\tau'}}\Big(\frac{1}{2} a^2\sin^2\th (\pr_v \pr_{\rhoH}\psi)+ (M^2+a^2)(\pr_{\rhoH}^2\psi) \Big)d\si d\tau' \rightarrow \infty,\quad \text{as}\ \tau\rightarrow \infty.
  \end{equation*}
  By fundamental theorem of calculus, there is a second order derivative of $\psi$ which blows up along $\HH$, as stated.
\end{proof}

\section{\texorpdfstring{$\HH^+$- and $\II^+$- w}{W}eighted hierarchies}
\label{sec:hierarchy}

In this section we prove weighted estimates near the two asymptotic
boundaries of the exterior region. In \zcref{sec:hierarchy-HH} we prove
a horizon hierarchy with weights in powers of $\rhoH$, corresponding to
the familiar $(r-M)^{-p}$-weighted hierarchy near the event horizon introduced in \cite{angelopoulosLatetimeAsymptoticsWave2020}.
In \zcref{sec:hierarchy-infinity} we prove the analogous hierarchy near
null infinity, corresponding to the $r^p$-weighted estimates introduced in \cite{dafermosNewPhysicalSpaceApproach2010}. The parameters $\alpha$ and $\beta$ below are related
to the traditional $p$-weights by $p=1-\alpha$ at the horizon and
$p=3-\beta$ at null infinity.

\subsection{The horizon hierarchy}\label{sec:hierarchy-HH}

In this section, we use ingoing Eddington-Finkelstein coordinates
$(v,r, \th,\phiIEF)$ near the event horizon.  We illustrate the main
domain of integration below, for some
$\rhoH \leq \rho_0 \leq \rho_\star$. 
\begin{figure}[ht]
  \centering
  \begin{tikzpicture}[scale=0.7,every node/.style={scale=0.7}]


  \def \s{3} 
  \def \exts{0.2} 
  \def \t{0.4}
  \def \Tlen{.5}
  \def \lenNull{0.1} 
  \def \lenDomain{0.2} 

  \coordinate (tInf) at (0,\s); 
  \coordinate (EventZero) at (-\s,0); 
  \coordinate (CosmoZero) at (\s,0); 
  \coordinate (tNegInf) at (0,-\s);
  \coordinate (SigmaZeroEvent) at (-\s + 0.15*\s, 0.15*\s);
  \coordinate (SigmaTEvent) at ($(SigmaZeroEvent) + (\lenDomain*\s, \lenDomain*\s)$);
  \coordinate (SigmaZeroCosmo) at (\s - 0.15*\s, 0.15*\s);
  \coordinate (SigmaTCosmo) at ($(SigmaZeroCosmo) + (-\lenDomain*\s, \lenDomain*\s)$);

  \draw[shorten >= -10,name path=EventFuture] (tInf) --
  node[pos=0.5,left]{$\EventHorizonFuture$} (EventZero) ;  
  \draw[shorten >= -10,name path=CosmoFuture,dashed] (tInf) --
  node[inner sep=10pt,scale=1.0,pos=0.5,right]{$\mathcal{I}^{+}$} (CosmoZero) ; 

  

  \draw[dashed] (tInf)    -- (-\s - \exts*\s,-\exts*\s);
  \draw[dashed] (tInf)    -- ( \s + \exts*\s,-\exts*\s);


  \path[name path=EventLowerBound] (-1.6*\s, 0.6*\s)-- (-\s, 0)--
  (tInf);
  \path[name path=CosmoLowerBound] (tInf) -- (\s, 0) -- (1.6*\s, 0.6*\s);


  \coordinate (SigmaZeroH) at ($(SigmaZeroEvent) + (4*\lenNull*\s, -2*\lenNull*\s)$);
  \coordinate (SigmaTH) at ($(SigmaTEvent) + (2*\lenNull*\s, -\lenNull*\s)$);
  \coordinate (SigmaZeroI) at ($(SigmaZeroCosmo) + (-4*\lenNull*\s, -2*\lenNull*\s)$);
  \coordinate (SigmaTI) at ($(SigmaTCosmo) + (-2*\lenNull*\s, -  \lenNull*\s)$);
  \def \bottombend{15}
  \def \topbend{10}
  \path[fill=blue, fill opacity=0.1, draw=blue, thick]
  (SigmaZeroEvent) to[bend right=\bottombend] (SigmaZeroH) to [bend left=10]  (SigmaTH) to[bend left=\topbend] (SigmaTEvent) -- cycle;
   \path[fill=black, fill opacity=0.1, draw=black!30, thin, dashed]
   (SigmaTCosmo) to[out=-135,in=15] (SigmaTI) to [bend right = 20] (SigmaTH) to [bend right=10]
   (SigmaZeroH) to [bend left = 10] (SigmaZeroI) to[in=-135,out=20] (SigmaZeroCosmo) -- cycle;


  \node[scale=0.5,fill=white,draw,circle,label=above:$i^+$]at(tInf){};

  \node[label=left:${r=M}$]at(SigmaZeroEvent){};
  \node[label=right:${r=\infty}$]at(SigmaZeroCosmo){};

\end{tikzpicture}

\caption{Schematic representation of the near-horizon region
$\Manifold_{\rhoH \leq \rho_0}(\tau_1,\tau_2)$, highlighted in blue. The region is bounded by
$\Sigma(\tau_1)$, $\Sigma(\tau_2)$, the event horizon, and the hypersurface
$\{\rhoH=\rho_0\}$.}
\end{figure}

The main horizon-weighted estimate is the following.
\begin{proposition}\label{prop:horizon-weighted-estimate}
  Let $\psi$ solve the inhomogeneous wave equation 
  \[\Box_\g \psi = F\]
  in an
  extremal Kerr--Newman spacetime with
  $\mathbf{a}\vcentcolon=\frac{|a|}{M}\ <\frac{1}{2}$. For
  \[\a \in \big(-1+4\mathbf{a}^2, 1-4\mathbf{a}^2\big)\] 
  the
  following $\EventHorizon$-weighted hierarchy holds true for $s \in \mathbb{N}$, $s \geq 1$ and any $\tau_1<\tau_2$:                   
  \begin{equation}
    \label{eq:horizon-weighted-estimate}
    \begin{split}
      &\norm*{\psi}_{H_{\EventHorizon}^{s, -\frac{\alpha+1}{2}}(\Sigma_{\EventHorizon}(\tau_2))}
        + \norm*{\psi}_{\bSobolevH{s, - \frac{\alpha+1}{2}}(\EventHorizon(\tau_1,\tau_2))}
        +\norm*{\psi}_{\bSobolevH{s,\,-\frac{\alpha}{2}}(\Manifold_{\EventHorizon}(\tau_1,\tau_2))}
      \\
      &\les{}
        \norm*{\psi}_{H_{\EventHorizon}^{s, -\frac{\alpha+1}{2}}(\Sigma_{\rhoH\leq 2\rho_0}(\tau_1))}
        + \norm*{F}_{\bSobolevH{s-1,-\frac{\alpha}{2}}(\Manifold_{\rhoH \leq 2\rho_0}(\tau_1,\tau_2))}
        + \rho_0^{-\frac 12}\norm*{\psi}_{\compactSobolev{s}(\Manifold_{ \rhoH \leq 2\rho_0}(\tau_1,\tau_2))}
    \end{split}
  \end{equation}
  where $\Sigma_{\HH}(\tau)$ and $\MM_{\HH}(\tau_1, \tau_2)$ denote the regions where $\rhoH \leq \rho_0$ for $\rho_0 \ll r_{+}$ sufficiently small.
\end{proposition}

The proof of \zcref{prop:horizon-weighted-estimate} is obtained as follows.
We first compute\footnote{This general computation will also be used for the degenerate redshift estimate in Section \ref{sec:deg-redshift}.} in \zcref{lemma:horizon:main-bulk} the divergence and boundary terms of the current associated to the vectorfield 
\[X_{\alpha, C} \vcentcolon=\rhoH^{\alpha+1}\big(-\partial_{\rhoH} + C r_+^{-2}\HawkingHorizon\big), \qquad C>0.\]
The proof of \zcref{prop:horizon-weighted-estimate} for $s=1$ is obtained by combining the above multiplier with a Hardy current which adds control of the zeroth-order
term. By applying higher order commutators, we obtain the proof of \zcref{prop:horizon-weighted-estimate} for $s\geq 2$.

\begin{remark}
  The horizon-weighted estimates appear in previous works as the $(r-M)^{-p}$-weighted hierarchy in the case of extremal Reissner-Nordstr\"om \cite{angelopoulosLatetimeAsymptoticsWave2020, angelopoulosNonlinearScalarPerturbations2020, gajicChargedScalarFields2026, angelopoulosSemilinearWaveEquations2025} for $p \in (0,2)$, and in the case of $m$-modes in extremal Kerr  \cite{gajicAzimuthalInstabilitiesExtremal2023}. 
  In our notation, 
  \[p=1-\alpha\]
  resulting in the range for p in \zcref{prop:horizon-weighted-estimate} of
  \[p \in \big(4\mathbf{a}^2, 2-4\mathbf{a}^2\big).\] 
  For $\mathbf{a}=0$, i.e. in extremal Reissner-Nordstr\"om, we recover the known range of $p \in (0, 2)$.

  Writing explicitly in terms of $p$, the above corresponds to the familiar multiplier $X=(r-M)^{2-p}\partial_r$ and we control in iEF coordinates
  \begin{align*}
    \int_{\MM_{\EventHorizon}(\tau_1, \tau_2)} (r-M)^{3-p} | \partial_r \psi|^2 +(r-M)^{1-p}\big( |\partial_v \psi|^2 + |\NablaAngular\psi|^2 +|\psi|^2 \big)  \les \text{RHS of \eqref{eq:horizon-weighted-estimate}}.
  \end{align*}
\end{remark}

\subsubsection{The horizon multiplier}

We first collect the following general computation. 
\begin{lemma}
  \label{lemma:horizon:main-bulk}
  \label{lemma:horizon-weighted-estimates:boundary-terms} 
  The current associated to the vectorfield
  \begin{equation*}
    X_{\alpha, C} \vcentcolon=\rhoH^{\alpha+1}\big(-\partial_{\rhoH} + C r_+^{-2}\HawkingHorizon\big), \qquad C>0
  \end{equation*}
  satisfies the following identity
  \begin{align}\label{eq:QQ-X-a-C-general}
    \begin{split}
      \abs*{q}^2\QQ^{(X_{\alpha, C},0,0)}[\psi]={}&\frac{1-\a}{2}\rhoH^\a |\rhoH\pr_{\rhoH}\psi|^2 +\frac{1+\a}{2}\rhoH^\a |\NablaAngular \psi|^2 +(1+\a)C r_+^{-2}\rhoH^\a (\HawkingHorizon\psi)^2\\
   &+2r\rhoH^{\alpha} (1+C r_+^{-2}\rhoH^2)\,\pr_v\psi\,\rhoH\pr_{\rhoH}\psi  +2C r_+^{-2} r\rhoH^{\alpha+1} (\HawkingHorizon\psi)\pr_v\psi  \\
          &+ C r_+^{-2}(1+\a)\rhoH^{\alpha+1}(\HawkingHorizon\psi)\rhoH\pr_{\rhoH}\psi  
    \end{split}
  \end{align}
  and
  \begin{align}\label{eq-boundary-terms-hierarchy-H}
    \begin{split}
      \JCurrent{(X_{\alpha, C},0,0)}[\psi]\cdot N_{\Sigma_{\EventHorizon}(\tau)}
      \ges{} & \rhoH^{\alpha + 1} \big(|\pr_{\rhoH}\psi|^2+ a^2     |\HawkingHorizon\psi|^2+|\NablaAngular \psi|^2\big), \\
      \JCurrent{(X_{\alpha, C},0,0)}[\psi]\cdot (- N_{\EventHorizon})
      \ges{}&\rhoH^{\alpha + 1} \big(|\rhoH\pr_{\rhoH}\psi|^2+      |\HawkingHorizon\psi|^2+|\NablaAngular \psi|^2\big)
              .   
    \end{split}
  \end{align}
  
\end{lemma}
\begin{proof}
  Applying \zcref{prop:fprr-HH} with
  $f=-\rhoH^{\alpha+1}$ and $h=\rhoH^{\a+1}$ and using that
  $\partial_r^{\mathrm{iEF}}=\partial_{\rhoH}$ gives that
  \begin{align*}
    |q|^2\D\cdot\JCurrent{(X^{\mathrm{iEF}}_{(1)},0,0)}[\psi]
    ={}& \frac{1}{2}(1-\alpha)\rhoH^{\alpha}|\rhoH\pr_{\rhoH}\psi|^2
         +\frac{1}{2}(1+\alpha)\rhoH^\alpha|\NablaAngular\psi|^2
    \\
       & +2r\rhoH^{\alpha}\,\pr_v\psi\,\rhoH\pr_{\rhoH}\psi
         -\rhoH^{\alpha+1}\pr_{\rhoH}\psi\,|q|^2\Box_\g\psi,\\
    |q|^2\D\cdot\JCurrent{(X^{\mathrm{iEF}}_{(2)},0,0)}[\psi]
    ={}& (1+\a)\rhoH^{\alpha}|\HawkingHorizon\psi|^2
         +2r\rhoH^{\alpha+1} (\HawkingHorizon\psi)\pr_v\psi
         + (1+\a)\rhoH^{\alpha+1}(\HawkingHorizon\psi)\rhoH\pr_{\rhoH}\psi\\
       & + 2r\rhoH^{\alpha+2}\,\pr_v\psi\,\rhoH\pr_{\rhoH}\psi
         +\rhoH^{\alpha+1} (\HawkingHorizon\psi)\,|q|^2\Box_\g\psi.
  \end{align*}
  Observing that
  \[X_{ \alpha, C} =X_{(1)}+Cr_+^{-2}X_{(2)} \]
  we deduce \zcref{eq:QQ-X-a-C-general}.
  
  For the boundary terms, using \zcref{lemma:boundary-terms-ieF}
  with $f=-\rhoH^{\a+1}$ and $h=\rhoH^{\a+1}$, we deduce
  \begin{align*}
    & \rhoH^{-\alpha-1}  \EMTensor(X_{\alpha, C} , |q|^2N_{\Sigma_{\EventHorizon}(\tau)})[\psi] \\
    ={}& \Big( \big(  1 +\frac{1}{2} C r_{+}^{-2}\rhoH^2\big) \big(r^2+a^2-\frac 1 2 \Delta h_{\HH}'\big)  +\frac{1}{4}C r_{+}^{-2} \rhoH^4 h_{\HH}'\Big)|\pr_{\rhoH}\psi|^2 +\frac{1}{2}\Big( C r_+^{-2}    \bigl(r^2+a^2\bigr) + h_{\HH}'   \Big)\abs*{\NablaAngular\psi}^2\\
    &+  C r_+^{-2}  h_{\HH}'   |\HawkingHorizon\psi|^2+  C r_+^{-2}  h_{\HH}' \rhoH^2 \HawkingHorizon\psi (\pr_{\rhoH} \psi)
      +a\sin\th  (\pr_{\rhoH}\psi )(\renormpphi\psi)
      -C r_+^{-2} a\sin\th (\HawkingHorizon\psi) (\renormpphi \psi).
  \end{align*}
  Completing the square above,
  we deduce
  \begin{align*}
    &\rhoH^{-\alpha-1}  \EMTensor(X_{ \alpha, C} , |q|^2N_{\Sigma_{\EventHorizon}(\tau)})[\psi]\\
    \geq{}& \big(  1 +\frac{1}{2} C r_{+}^{-2}\rhoH^2\big) \Big(r^2+a^2-\Delta\,h_{\HH}'\Big)
            |\pr_{\rhoH}\psi|^2 +\frac{1}{2}\Big(C r_+^{-2}    \bigl(r^2+a^2\bigr)+ h_{\HH}'    \Big)\abs*{\NablaAngular\psi}^2+  \frac{1}{2} C r_+^{-2}  h_{\HH}'    |\HawkingHorizon\psi|^2\\
    &+a\sin\th (\pr_{\rhoH}\psi )(\renormpphi\psi)-C r_+^{-2} a\sin\th (\HawkingHorizon\psi) (\renormpphi \psi)\\
    \geq{}& \frac{1}{2}\big(   1 + C r_{+}^{-2}\rhoH^2\big) \Big(r^2+a^2-\Delta\,h_{\HH}'\Big)
            |\pr_{\rhoH}\psi|^2 + \frac{1}{4}\big( 1 + 2C r_{+}^{-2}\rhoH^2\big) h_{\HH}' |\NablaAngular\psi|^2+  \frac{1}{4}C r_+^{-2}  h_{\HH}'    |\HawkingHorizon\psi|^2\\
    &+ \Big[ \frac{1}{4}C r_+^{-2}  h_{\HH}'    |\HawkingHorizon\psi|^2-C r_+^{-2} a\sin\th (\HawkingHorizon\psi) (\renormpphi \psi) + \frac{1}{2}C r_+^{-2} \big(r^2+a^2-\Delta h_{\HH}'\big) |\renormpphi\psi|^2\Big]\\
    &+ \Big[\frac{1}{2}\big(r^2+a^2-\Delta\,h_{\HH}'\big)|\pr_{\rhoH}\psi|^2+a\sin\th  (\pr_{\rhoH}\psi )(\renormpphi\psi)+\frac{1}{4}h_{\HH}'(r) |\renormpphi\psi|^2\Big].
  \end{align*}
  Observe that the last two lines are positive if $\Delta\,(h_{\HH}'(r))^2  -(r^2+a^2)  h_{\HH}'(r)+2a^2 \leq 0$
  which holds true for the condition \zcref{eq:condition-h-II}. Similarly, the first line is positive for all $r$, and since $h_\HH'=O(a^2)$ in
  particular this gives
  \begin{align*}
    \EMTensor(X_{\alpha, C} , |q|^2N_{\Sigma_{\EventHorizon}(\tau)})[\psi]
    \ges{}& \rhoH^{\alpha + 1} \big(|\pr_{\rhoH}\psi|^2
            +|\NablaAngular \psi|^2+ a^2  |\HawkingHorizon\psi|^2\big).
  \end{align*}
  Using \zcref{lemma:boundary-terms-ieF}, we also obtain
  \begin{align*}
    \JCurrent{(X_{\alpha, C},0,0)}[\psi]\cdot (-N_{\EventHorizon})
    \ges{}&\rhoH^{\alpha + 1} \big(|\rhoH\pr_{\rhoH}\psi|^2+      |\HawkingHorizon\psi|^2+|\NablaAngular \psi|^2\big)             ,
  \end{align*}
  as stated in \zcref{eq-boundary-terms-hierarchy-H}.
\end{proof}

\subsubsection{Proof of \texorpdfstring{\zcref{prop:horizon-weighted-estimate}}{} for \texorpdfstring{$s=1$}{}}

From \zcref{lemma:horizon:main-bulk}, \zcref{eq:QQ-X-a-C-general} we have
\begin{align*}
  \abs*{q}^2\QQ^{(X_{\alpha, C},0,0)}[\psi]
  ={}&\rho^\a \Big[\frac{1-\a}{2} |\rhoH\pr_{\rhoH}\psi|^2 +\frac{1+\a}{2}|\NablaAngular \psi|^2 +(1+\a)C r_+^{-2}(\HawkingHorizon\psi)^2+2r \,\pr_v\psi\,\rhoH\pr_{\rhoH}\psi \Big]\\
     &+\conormalSpaceH{\a+1}(\Manifold)(\bDiffH^{\le 1}\psi)^2\\
  ={}&\rhoH^\a \Big[\frac{1-\a}{2} |\rhoH\pr_{\rhoH}\psi|^2 +\frac{1+\a}{2} |\pr_\th \psi|^2+\frac{1+\a}{2} |\renormpphi \psi|^2 +(1+\a)C r_+^{-2} (\HawkingHorizon\psi)^2\\
     &+\frac{2r}{|q|^2} \,\HawkingHorizon \psi \,\rhoH\pr_{\rhoH}\psi -\frac{2ra\sin\th}{|q|^2}  \,\renormpphi\psi \,\rhoH\pr_{\rhoH}\psi  \Big]+\conormalSpaceH{\a+1}(\Manifold)(\bDiffH^{\le 1}\psi)^2,
\end{align*}
where we used that $  \pr_v =\frac{1}{|q|^2}\HawkingHorizon -\frac{a\sin\th}{|q|^2} \renormpphi$ and $|\NablaAngular\psi|^2=|\pr_\th \psi|^2 + |\renormpphi \psi|^2$.
Consider the quadratic form
\begin{align*}
  Q
  \vcentcolon= & \frac{1-\a}{2} |\rhoH\pr_{\rhoH}\psi|^2 +\frac{1+\a}{2} |\renormpphi \psi|^2 + (1+\a)C r_+^{-2}|\HawkingHorizon\psi|^2  +\frac{2r}{|q|^2} \,\HawkingHorizon \psi \,\rhoH\pr_{\rhoH}\psi -\frac{2ra\sin\th}{|q|^2}  \,\renormpphi\psi \,\rhoH\pr_{\rhoH}\psi .
\end{align*}
By Sylvester's criterion, the quadratic form $Q$ is positive definite, if, in addition to $-1<\a<1$, we have 
\begin{align}\label{eq:sylvester-criterion}
  \begin{split}
    \frac{1-\a}{2}\frac{1+\a}{2}-(\frac{a\sin\th}{|q|^2}r)^2 >0, \\
    C r_+^{-2} \Big(  \frac{1-\a}{2}\frac{1+\a}{2}-(\frac{a\sin\th}{|q|^2}r)^2\Big)-\frac 1 2(\frac{r}{|q|^2})^2>0.
  \end{split}
\end{align}
Using that $\frac{r}{|q|^2}\leq \frac{1}{r_{+}}= \frac{1}{M}$ and $a^2\sin^2\th \leq a^2$, we immediately deduce that \zcref{eq:sylvester-criterion} are implied by the following:
\begin{align*}
  \frac{1-\a}{2}\frac{1+\a}{2}-\abs*{\frac{a}{M}}^2 >0, \\
  C \Big(  \frac{1-\a}{2}\frac{1+\a}{2}-\abs*{\frac{a}{M}}^2\Big)-\frac 1 2>0.
\end{align*}
The first condition only admits solutions for $\alpha$ if
$\mathbf{a}^2<\frac 1 4$. Within this regime, the first condition is
satisfied if
$\a \in \big(-\sqrt{1-4\mathbf{a}^2},
\sqrt{1-4\mathbf{a}^2}\big)$, and since
$\sqrt{1-4\mathbf{a}^2}\geq 1-4\mathbf{a}^2$, it is also satisfied
for $\a \in \big(-1+4\mathbf{a}^2, 1-4\mathbf{a}^2\big)$. The
second condition is satisfied if
$C>\frac{2}{(1-\a) (1+\a)-4\mathbf{a}^2}$. Hence, we conclude that
the current associated to the vectorfield
\[    X_{\alpha, C} \vcentcolon=-\rhoH^{\alpha+1}\partial_{\rhoH} + C r_+^{-2}\rhoH^{\alpha+1}\HawkingHorizon, \qquad C\vcentcolon=\frac{3}{(1-\a) (1+\a)-4\mathbf{a}^2} \]
satisfies 
\begin{align*}
  \abs*{q}^2\QQ^{(X_{\alpha, C},0,0)}[\psi]
  \gtrsim {}\rhoH^\a \big( |\rhoH\pr_{\rhoH}\psi|^2 +|\NablaAngular \psi|^2+|\HawkingHorizon\psi|^2 \big)+\conormalSpaceH{\a+1}(\Manifold)(\bDiffH^{\le 1}\psi)^2.
\end{align*}
For $\rho_0 \ll r_{+}$ sufficiently small the lower order terms
$\conormalSpaceH{\a+1}(\Manifold)(\bDiffH^{\le 1}\psi)^2$ can be
absorbed by the first three terms on the RHS and therefore in the
region $\{ \rhoH \leq \rho_0\}$ we have
\begin{align*}
  \abs*{q}^2\QQ^{(X_{\alpha, C},0,0)}[\psi]
  \ges{}&  \rhoH^\a |\rhoH\pr_{\rhoH}\psi|^2+\rhoH^\a |\NablaAngular \psi|^2 +\rhoH^\a |\HawkingHorizon\psi|^2. 
\end{align*}
Using \zcref{lemma:hardy-pointwise-rhoH}, and recalling that $\PP_\mu^{(0,0,J)}[\psi]
=
J_\mu |\psi|^2$ we have for $J_{\a}=\frac 1 2 c_J(\a+1)\rhoH^{\a+1}|q|^{-2}\partial_{\rhoH}$ with $c_J>0$
\begin{align*}
  |q|^2 \D \c \PP^{(0, 0, J_{\a})}[\psi]\geq -c_J\rhoH^{\a}|\rhoH\pr_{\rhoH} \psi|^2+
  \frac14c_J(\a+1)^2 \rhoH^\a |\psi|^2.
\end{align*}
Adding this bound to the above for $c_J \ll 1$ sufficiently small we
obtain for the same range of $\a$ as above,
\begin{equation}
  \label{eq::horizon-hierarchy-hardy}
  \begin{split}
    \abs*{q}^2\QQ^{(X_{\alpha, C},0, J_{\alpha})}[\psi]
    \ges{}&  \rhoH^\a  \Big(  |\pr_v\psi|^2+|\NablaAngular\psi|^2 +|\rhoH\pr_{\rhoH}\psi|^2+
            |\psi|^2\Big),
  \end{split}    
\end{equation}
where we recognize the energy density of a weighted b-Sobolev norm of $\psi$.

We now localize \zcref{eq::horizon-hierarchy-hardy} using a cut-off
function.  For any non-negative bump function $\chi$ supported in
$[-2,2]$ and identically equal to $1$ on $[-1,1]$, we define
\begin{align*}
  \chi_{\HH}\vcentcolon=\chi\big(\frac{\rhoH}{\rho_0} \big), \qquad X_{\HH, \alpha}\vcentcolon= \chi_{\HH}   X_{\alpha, C}, \qquad J_{\EventHorizon, \alpha} \vcentcolon=\chi_{\HH} J_{\alpha}.
\end{align*}
From \zcref{eq::horizon-hierarchy-hardy, current-cutoff} 
we obtain
\begin{align*}
  \begin{split}
    & \int_{\MM(\tau_1, \tau_2)} \D\cdot \JCurrent{(X_{\HH, \a},  0, J_{\HH, \a})}[\psi]\\
    \ges{}& \int_{\MM(\tau_1, \tau_2)}\mathds{1}_{\{\rhoH < \rho_0\}} \rhoH^{\a}\Big( |\pr_v\psi|^2+|\NablaAngular\psi|^2 +|\rhoH\pr_{\rhoH}\psi|^2+
            |\psi|^2\Big)\\
    &-\rho_0^{-1} \int_{\MM(\tau_1, \tau_2)}\mathds{1}_{\{\rho_0 \leq \rhoH \leq 2\rho_0\}}\abs*{(\partial_v,\partial_{\rhoH},\NablaAngular)^{\le 1}\psi}^2 -\int_{\MM(\tau_1, \tau_2)}|X_{\Horizon, \alpha}(\psi) || \Box_{\Metric}\psi|\\
    \ges{} & \norm*{\psi}^2_{\bSobolevH{1,\,-\frac{\alpha}{2}}(\Manifold_{\rhoH \leq \rho_0}(\tau_1,\tau_2))}- \rho_0^{-1}\norm*{\psi}_{\compactSobolev{1}(\Manifold_{\rhoH\leq 2\rho_0}(\tau_1,\tau_2))}^2\\
    &-\norm*{\psi}_{\bSobolevH{1,-\frac{\alpha}{2}}(\Manifold_{\rhoH \leq 2\rho_0}(\tau_1,\tau_2))}
      \norm*{\Box_{\Metric}\psi}_{\bSobolevH{0,-\frac{\alpha}{2}}(\Manifold_{\rhoH \leq 2 \rho_0}(\tau_1,\tau_2))}\\
    \ges{} & \norm*{\psi}^2_{\bSobolevH{1,\,-\frac{\alpha}{2}}(\Manifold_{\EventHorizon}(\tau_1,\tau_2))}- \rho_0^{-1}\norm*{\psi}_{\compactSobolev{1}(\Manifold_{\rhoH\leq 2\rho_0}(\tau_1,\tau_2))}^2-
             \norm*{F}_{\bSobolevH{0,-\frac{\alpha}{2}}(\Manifold_{\rhoH \leq 2\rho_0}(\tau_1,\tau_2))}^2,
  \end{split}
\end{align*}
where we used Cauchy-Schwarz to bound the last term. 

For the boundary terms, we compute
\begin{align}\label{eq:boundary-J-HH}
  \begin{split}
    \JCurrent{(0,0,J_{\HH, \a})}[\psi]\cdot N_{\Sigma_{\EventHorizon}(\tau)}
    & =
      \chi_{\HH} (J_{\a}\cdot N_{\Sigma_{\EventHorizon}(\tau)})|\psi|^2 \\
    & =
      \frac12 \chi_{\HH} c_J(\a+1)\rhoH^{\a+1}|q|^{-4}\g(\partial_{\rhoH}, \left(-(r^2+a^2)+\De h_{\EventHorizon}'\right)\partial_{\rhoH}+
      h_{\EventHorizon}' \HawkingHorizon 
      -a\sin\th \renormpphi)|\psi|^2 \\
    & =
      \frac12 \chi_{\HH} c_J(\a+1)\rhoH^{\a+1}h_{\EventHorizon}'|q|^{-2}|\psi|^2 , \\
    \JCurrent{(0,0,J_{\HH, \a})}[\psi]\cdot(- N_{\EventHorizon})
    & =
      \frac12 \chi_{\HH} c_J(\a+1)\rhoH^{\a+1}|q|^{-4}\g(\partial_{\rhoH},  \HawkingHorizon )|\psi|^2 \\
    & =
      \frac12 \chi_{\HH} c_J(\a+1)\rhoH^{\a+1}|q|^{-2}|\psi|^2 
  \end{split}
\end{align}
where we used that $\g_{\mathrm{iEF}}(\partial_{\rhoH}, \partial_{\rhoH})=\g_{\mathrm{iEF}}(\partial_{\rhoH}, \renormpphi)=0$ and $\g_{\mathrm{iEF}}(\partial_{\rhoH}, V_{\HH})=|q|^2$. We therefore obtain from \zcref{eq-boundary-terms-hierarchy-H} and the above,
\begin{align*}
  \int_{\Sigma(\tau)}  \JCurrent{(X_{\Horizon, \alpha},0,J_{\HH, \a})}[\psi]\cdot  N_{\Sigma(\tau)}
  \ges{} & \int_{\Sigma(\tau)}\mathds{1}_{\{\rhoH < \rho_0\}} \rhoH^{\alpha + 1} \big(|\pr_v\psi|^2+|\NablaAngular \psi|^2+|\pr_{\rhoH}\psi|^2+     |\psi|^2\big)\\
  \ges{} &  \norm*{\psi}_{H_{\EventHorizon}^{1, -\frac{\alpha+1}{2}}(\Sigma_{\EventHorizon}(\tau))}^2, \nonumber\\
  \int_{\EventHorizon(\tau_1, \tau_2)}  \JCurrent{(X_{\Horizon, \alpha},0,J_{\HH, \a})}[\psi]\cdot  (-N_{\EventHorizon}) 
  \ges{} &  \int_{\EventHorizon(\tau_1, \tau_2)}\rhoH^{\alpha + 1} \big(|\rhoH\pr_{\rhoH}\psi|^2+      |\HawkingHorizon\psi|^2+|\NablaAngular \psi|^2+|\psi|^2\big)\nonumber \\
  \ges{}& \norm*{\psi}_{\bSobolevH{1, -\frac{\alpha+1}{2}}(\EventHorizon(\tau_1, \tau_2))}^2
\end{align*}
where we recognize the weighted standard Sobolev of $\psi$ along $\Sigma(\tau)$.

Performing the divergence theorem over $\Manifold(\tau_1, \tau_2)$, we obtain from \zcref{eq:general-divergence-theorem}
\begin{align*}
  &\int_{\Sigma(\tau_2)}\JCurrent{(X_{\HH, \a},  0, J_{\HH, \a})}[\psi]\cdot N_{\Sigma(\tau_2)}
    + \int_{\EventHorizon(\tau_1,\tau_2)}\JCurrent{(X_{\HH, \a},  0, J_{\HH, \a})}[\psi]\cdot (- N_{\EventHorizon})
    + \int_{\Manifold(\tau_1,\tau_2)}\D\cdot \JCurrent{(X_{\HH, \a},  0, J_{\HH, \a})}[\psi]\\
  ={}& \int_{\Sigma(\tau_1)}\JCurrent{(X_{\HH, \a},  0, J_{\HH, \a})}[\psi]\cdot N_{\Sigma(\tau_1)}.
\end{align*}
Using the bounds above and the trivial bound
\begin{align*}
  \int_{\Sigma(\tau_1)}\JCurrent{(X_{\HH, \a},  0, J_{\HH, \a})}[\psi]\cdot N_{\Sigma(\tau_1)}
  \les \norm*{\psi}_{H_{\EventHorizon}^{1, -\frac{\alpha+1}{2}}(\Sigma_{\rhoH \leq 2\rho_0}(\tau_1))}^2,
\end{align*}
we deduce \zcref{eq:horizon-weighted-estimate}.

\subsubsection{Proof of \texorpdfstring{\zcref{prop:horizon-weighted-estimate}}{} for \texorpdfstring{$s\geq 2$}{s>1}}
\label{sec:higher-order-horizon}

We first prove the estimate for $s=2$. 

\begin{enumerate}
\item \emph{Killing commutators.}
  Since $T=\partial_v$ and $\Phi=\partial_{\phiIEF}$ are Killing, applying \zcref{eq:horizon-weighted-estimate} for $s=1$ to $T\psi$ and $\Phi\psi$ gives
  \begin{align}
    \label{eq:HH-higher-Killing-commuted-estimate}
    \begin{split}
      &\norm*{(T,\Phi)\psi}_{H_{\EventHorizon}^{1,-\frac{\alpha+1}{2}}(\Sigma_{\EventHorizon}(\tau_2))}
        + \norm*{(T,\Phi)\psi}_{\bSobolevH{1, - \frac{\alpha+1}{2}}(\EventHorizon(\tau_1,\tau_2))}
        + \norm*{(T,\Phi)\psi}_{\bSobolevH{1,-\frac{\alpha}{2}}(\Manifold_{\EventHorizon}(\tau_1,\tau_2))}  \\
      \les{}&
              \norm*{\psi}_{H_{\EventHorizon}^{2,-\frac{\alpha+1}{2}}(\Sigma_{\rhoH\leq 2\rho_0}(\tau_1))}
              + \norm*{F}_{\bSobolevH{1,-\frac{\alpha}{2}}(\Manifold_{\rhoH\leq 2\rho_0}(\tau_1,\tau_2))}
              + \rho_0^{-\frac1 2}\norm*{\psi}_{\compactSobolev{2}(\Manifold_{\rhoH\leq 2\rho_0}(\tau_1,\tau_2))}.
    \end{split}
  \end{align}

\item \emph{Angular commutators.}
  We apply \zcref{eq::horizon-hierarchy-hardy}
  to $\RotationVF_1\psi$ and
  $\RotationVF_2\psi$ and sum them to obtain
  \begin{align}\label{eq:bound-commutator-angular}
    \begin{split}
      &  \abs*{q}^2\D\cdot\big(\JCurrent{(X_{\alpha, C},0, J_{\alpha})}[\RotationVF_1\psi]+\JCurrent{(X_{\alpha, C},0, J_{\alpha})}[\RotationVF_2\psi]\big)\\
      \ges{}&  \sum_{i=1}^2 \rhoH^\a  \Big(  |\pr_v\RotationVF_i\psi|^2+|\NablaAngular\RotationVF_i\psi|^2 +|\rhoH\pr_{\rhoH}\RotationVF_i\psi|^2+
              |\RotationVF_i\psi|^2\Big)\\
      &+X_{\alpha, C}(\RotationVF_1\psi)[ |q|^2 \Box_\g ,\RotationVF_1]\psi+X_{\alpha, C}(\RotationVF_2\psi) [|q|^2 \Box_\g,  \RotationVF_2]\psi\\
      &+X_{\alpha, C}(\RotationVF_1\psi) \RotationVF_1(|q|^2 \Box_\g \psi)+X_{\alpha, C}(\RotationVF_2\psi)\RotationVF_2( |q|^2 \Box_\g \psi). 
    \end{split}
  \end{align}
  Using \zcref{lem:weighted-hierarchy-commutators}, we compute the expression:
  \begin{align*}
    I\vcentcolon={}&X_{\alpha, C}(\RotationVF_1\psi)[ |q|^2 \Box_\g ,\RotationVF_1]\psi+X_{\alpha, C}(\RotationVF_2\psi) [|q|^2 \Box_\g,  \RotationVF_2]\psi\\
    ={}&X_{\alpha, C}(\RotationVF_1\psi)(-2a(\partial_{\rhoH}+T)\RotationVF_2\psi
         +
         O(a^2)T^2\psi)+X_{\alpha, C}(\RotationVF_2\psi) (2a(\partial_{\rhoH}+T)\RotationVF_1\psi
         +
         O(a^2)T^2\psi)\\
    ={}&\Big(-\rhoH^{\alpha+1}\partial_{\rhoH}\RotationVF_1\psi + C r_+^{-2}\rhoH^{\alpha+1}\HawkingHorizon\RotationVF_1\psi\Big)(-2a(\partial_{\rhoH}+T)\RotationVF_2\psi
         +
         O(a^2)T^2\psi)\\
                   &+\Big(-\rhoH^{\alpha+1}\partial_{\rhoH}\RotationVF_2\psi + C r_+^{-2}\rhoH^{\alpha+1}\HawkingHorizon\RotationVF_2\psi\Big) (2a(\partial_{\rhoH}+T)\RotationVF_1\psi
                     +
                     O(a^2)T^2\psi)\\
    ={}& 2a\rhoH^{\alpha}
         \Big(
         T\RotationVF_2\psi\,\rhoH\partial_{\rhoH}\RotationVF_1\psi
         -
         T\RotationVF_1\psi\,\rhoH\partial_{\rhoH}\RotationVF_2\psi
         \Big) \\
                   &-
                     2aCr_+^{-2}\rhoH^\alpha
                     \Big(
                     \rhoH\partial_{\rhoH}\RotationVF_2\psi\,\HawkingHorizon\RotationVF_1\psi
                     -
                     \rhoH\partial_{\rhoH}\RotationVF_1\psi\,\HawkingHorizon\RotationVF_2\psi
                     \Big) \\
                   &-
                     2a^2Cr_+^{-2}\rhoH^{\alpha+1}
                     \Big(
                     T\RotationVF_2\psi\,\Phi\RotationVF_1\psi
                     -
                     T\RotationVF_1\psi\,\Phi\RotationVF_2\psi
                     \Big)\\
                   &+O(a^2)\rhoH^{\a}\big(-\rhoH\partial_{\rhoH}(\RotationVF_1+\RotationVF_2)\psi + C r_+^{-2}\rhoH\HawkingHorizon(\RotationVF_1+\RotationVF_2)\psi\big)
                     T^2\psi,
  \end{align*}
  where we wrote
  $X_{\alpha, C}=-\rhoH^{\alpha+1}\partial_{\rhoH} + C
  r_+^{-2}\rhoH^{\alpha+1}\HawkingHorizon$ and
  $\HawkingHorizon=(r^2+a^2)T +a\Phi$. Notice that the terms
  involving two $\partial_{\rhoH}$ and $T$ derivatives cancel out in the
  sum. Using Cauchy-Schwarz, we can bound the above for $0<\lambda<1$ by
  \begin{align*}
    \abs*{I}\leq \lambda \rhoH^\a\sum_{i=1}^2|\rhoH\partial_{\rhoH}\RotationVF_i\psi|^2+\lambda^{-1} \rhoH^\a \big(\sum_{i=1}^2|T\RotationVF_i\psi|^2+\sum_{i=1}^2|\Phi\RotationVF_i\psi|^2+|T^2\psi|^2\big).
  \end{align*}
  For $\lambda$ sufficiently small, the first term above can be absorbed
  by the first line in \zcref{eq:bound-commutator-angular}, while the
  remaining terms only involve $T$ and $\Phi$ derivatives and therefore can be bounded, upon integration on
  $\Manifold(\tau_1, \tau_2)$, by the initial data norms using
  \zcref{eq:HH-higher-Killing-commuted-estimate}. Applying divergence
  theorem to \zcref{eq:bound-commutator-angular} and following the same
  steps as in proof of \zcref{prop:horizon-weighted-estimate}
  for $s=1$, we then deduce
  \begin{align}
    \label{eq:HH-angular-commuted-estimate}
    \begin{split}
      &\sum_{i=1}^2
        \norm*{\RotationVF_i\psi}_{H_{\EventHorizon}^{1,-\frac{\alpha+1}{2}}
        (\Sigma_{\EventHorizon}(\tau_2))}
        + \sum_{i=1}^2\norm*{\RotationVF_i\psi}_{\bSobolevH{1, - \frac{\alpha+1}{2}}(\EventHorizon(\tau_1,\tau_2))}+
        \sum_{i=1}^2
        \norm*{\RotationVF_i\psi}_{\bSobolevH{1,-\frac{\alpha}{2}}
        (\Manifold_{\EventHorizon}(\tau_1,\tau_2))} \\
      \les{}&
              \norm*{\psi}_{H_{\EventHorizon}^{2,-\frac{\alpha+1}{2}}
              (\Sigma_{\rhoH\leq 2\rho_0}(\tau_1))}
              +
              \norm*{F}_{\bSobolevH{1,-\frac{\alpha}{2}}
              (\Manifold_{\rhoH\leq 2\rho_0}(\tau_1,\tau_2))} 
              +
              \rho_0^{-\frac12}
              \norm*{\psi}_{\compactSobolev{2}
              (\Manifold_{\rhoH\leq 2\rho_0}(\tau_1,\tau_2))} .
    \end{split}
  \end{align}
  Angular elliptic estimates on $\mathbb{S}^2$ then give
  \begin{align}
    \label{eq:HH-angular-second-derivative-control}
    \begin{split}
      \norm*{\NablaAngular^2\psi}_{\bSobolevH{0,-\frac{\alpha}{2}}
      (\Manifold_{\EventHorizon}(\tau_1,\tau_2))}^2
      \les{}&
              \sum_{i=1}^2
              \norm*{\RotationVF_i\psi}_{\bSobolevH{1,-\frac{\alpha}{2}}
              (\Manifold_{\EventHorizon}(\tau_1,\tau_2))}^2 +
              \norm*{\Phi\psi}_{\bSobolevH{1,-\frac{\alpha}{2}}
              (\Manifold_{\EventHorizon}(\tau_1,\tau_2))}^2
              +
              \norm*{\psi}_{\bSobolevH{1,-\frac{\alpha}{2}}
              (\Manifold_{\EventHorizon}(\tau_1,\tau_2))}^2 \\
      \les{}&\norm*{\psi}_{H_{\EventHorizon}^{2,-\frac{\alpha+1}{2}}
              (\Sigma_{\rhoH\leq 2\rho_0}(\tau_1))}^2
              +
              \norm*{F}_{\bSobolevH{1,-\frac{\alpha}{2}}
              (\Manifold_{\rhoH\leq 2\rho_0}(\tau_1,\tau_2))}^2 
      \\
            &+
              \rho_0^{-1}
              \norm*{\psi}_{\compactSobolev{2}
              (\Manifold_{\rhoH\leq 2\rho_0}(\tau_1,\tau_2))}^2 
    \end{split}
  \end{align}
  and similarly on $\Sigma(\tau)$ and $\EventHorizon(\tau_1, \tau_2)$.

\item \emph{The second $\rhoH$-derivative.}
  We apply \zcref{eq::horizon-hierarchy-hardy}
  to $\Psi\vcentcolon=\rhoH \partial_{\rhoH}\psi$ and we obtain
  \begin{align}\label{eq:bound-commutator-rhoH}
    \begin{split}
      \abs*{q}^2\D\cdot\big(\JCurrent{(X_{\alpha, C},0, J_{\alpha})}[\rhoH \partial_{\rhoH}\psi]\big)  \ges{}&   \rhoH^\a  \Big(  |\pr_v\Psi|^2+|\NablaAngular\Psi|^2 +|\rhoH\pr_{\rhoH}\Psi|^2+
                                                                                                               |\Psi|^2\Big)\\
                                                                                                             &+X_{\alpha, C}(\Psi)[ |q|^2 \Box_\g ,\rhoH \partial_{\rhoH}]\psi+X_{\alpha, C}(\Psi) \rhoH \partial_{\rhoH}(|q|^2 \Box_\g \psi). 
    \end{split}
  \end{align}
  Using \zcref{lem:weighted-hierarchy-commutators}, we compute the expression:
  \begin{align*}
    Y\vcentcolon={}&X_{\alpha, C}(\Psi)[ |q|^2 \Box_\g ,\rhoH \partial_{\rhoH}]\psi\\
                ={}&X_{\alpha, C}(\Psi)\Big(
                  -
                  \rhoH\partial_{\rhoH}\Psi
                  -
                  \Psi -
                  \lapp_{\SSS^2}\psi -a^{2}\sin^{2}\theta T^2\psi
                  -2aT\Phi\psi
                  -2rT\psi
                  -4rT\rhoH \partial_{\rhoH}
                  -2\rhoH T\psi +|q|^2F\Big)\\
                ={}&\Big(-\rhoH^{\alpha+1}\partial_{\rhoH}\Psi + C r_+^{-2}\rhoH^{\alpha+1}\HawkingHorizon\Psi\Big)\Big(
                  -
                  \rhoH\partial_{\rhoH}\Psi
                  \Big)\\
                &+X_{\alpha, C}(\Psi)\Big(
                  -
                  \Psi -
                  \lapp_{\SSS^2}\psi -a^{2}\sin^{2}\theta T^2\psi
                  -2aT\Phi\psi
                  -2rT\psi
                  -4rT\rhoH \partial_{\rhoH}
                  -2\rhoH T\psi +|q|^2F\Big).
  \end{align*}
  Notice that the first term in $|\partial_{\rhoH}\Psi|^2$ gives a positive contribution while the remaining terms can be bounded using Cauchy-Schwarz for $0<\lambda<1$ by
  \begin{align*}
    Y \geq{}& \rhoH^\a \abs*{\rhoH \partial_{\rhoH}\Psi}^2-\lambda \rhoH^\a \abs*{\rhoH \partial_{\rhoH}\Psi}^2\\
            &-\lambda^{-1}\rhoH^{\alpha}\Big(|T\rhoH \partial_{\rhoH}\psi|^2+|\Phi\rhoH \partial_{\rhoH}\psi|^2+|\rhoH \partial_{\rhoH}\psi|^2 +
              |\lapp_{\SSS^2}\psi|^2
              +
              |T^2\psi|^2
              +
              |T\Phi\psi|^2
              +
              |T\psi|^2+|F|^2\Big).
  \end{align*}
  For $\lambda$ sufficiently small, the second term above can be absorbed by the first line in \zcref{eq:bound-commutator-rhoH}, while the second line only involves second derivatives in $T$, $\Phi$, $\NablaAngular$ and therefore can be bounded, upon integration on $\Manifold(\tau_1, \tau_2)$, by the initial data norms using \zcref{eq:HH-higher-Killing-commuted-estimate, eq:HH-angular-second-derivative-control}. Applying divergence theorem to \zcref{eq:bound-commutator-rhoH} and following the same steps as in proof of \zcref{prop:horizon-weighted-estimate} for $s=1$, we then deduce
  \begin{align}
    \label{eq:HH-rho-commuted-estimate}
    \begin{split}
      &\norm*{\rhoH \partial_{\rhoH}\psi}_{H_{\EventHorizon}^{1,-\frac{\alpha+1}{2}}
        (\Sigma_{\EventHorizon}(\tau_2))}
        + \norm*{\rhoH \partial_{\rhoH}\psi}_{\bSobolevH{1, - \frac{\alpha+1}{2}}(\EventHorizon(\tau_1,\tau_2))}+
        \norm*{\rhoH \partial_{\rhoH}\psi}_{\bSobolevH{1,-\frac{\alpha}{2}}
        (\Manifold_{\EventHorizon}(\tau_1,\tau_2))} \\
      \les{}&
              \norm*{\psi}_{H_{\EventHorizon}^{2,-\frac{\alpha+1}{2}}
              (\Sigma_{\rhoH\leq 2\rho_0}(\tau_1))}
              +
              \norm*{F}_{\bSobolevH{1,-\frac{\alpha}{2}}
              (\Manifold_{\rhoH\leq 2\rho_0}(\tau_1,\tau_2))} 
              +
              \rho_0^{-\frac 12}
              \norm*{\psi}_{\compactSobolev{2}
              (\Manifold_{\rhoH\leq 2\rho_0}(\tau_1,\tau_2))}
    \end{split}
  \end{align}
\end{enumerate}

Combining \zcref{eq:HH-higher-Killing-commuted-estimate, eq:HH-angular-commuted-estimate, eq:HH-rho-commuted-estimate} we prove \eqref{eq:horizon-weighted-estimate} for $s=2$.

The case $s\geq 3$ follows by induction, by commuting
$A\in\bDiffH^{s-1}$ and using
\zcref{lem:weighted-hierarchy-commutators}.

\subsection{The null infinity hierarchy}\label{sec:hierarchy-infinity}

In this section, we use outgoing Eddington-Finkelstein coordinates
$(u,r,\th,\phiOEF)$ near null infinity.  We illustrate the main domain
of integration below, for some $\rhoI \leq \rho_1$ specified below.

\begin{figure}[ht]
  \centering
  \begin{tikzpicture}[scale=0.7,every node/.style={scale=0.7}]


  \def \s{3} 
  \def \exts{0.2} 
  \def \t{0.4}
  \def \Tlen{.5}
  \def \lenNull{0.1} 
  \def \lenDomain{0.2} 

  \coordinate (tInf) at (0,\s); 
  \coordinate (EventZero) at (-\s,0); 
  \coordinate (CosmoZero) at (\s,0); 
  \coordinate (tNegInf) at (0,-\s);
  \coordinate (SigmaZeroEvent) at (-\s + 0.15*\s, 0.15*\s);
  \coordinate (SigmaTEvent) at ($(SigmaZeroEvent) + (\lenDomain*\s, \lenDomain*\s)$);
  \coordinate (SigmaZeroCosmo) at (\s - 0.15*\s, 0.15*\s);
  \coordinate (SigmaTCosmo) at ($(SigmaZeroCosmo) + (-\lenDomain*\s, \lenDomain*\s)$);

  \draw[shorten >= -10,name path=EventFuture] (tInf) --
  node[pos=0.5,left]{$\EventHorizonFuture$} (EventZero) ;  
  \draw[shorten >= -10,name path=CosmoFuture,dashed] (tInf) --
  node[inner sep=10pt,scale=1.0,pos=0.5,right]{$\mathcal{I}^{+}$} (CosmoZero) ; 

  

  \draw[dashed] (tInf)    -- (-\s - \exts*\s,-\exts*\s);
  \draw[dashed] (tInf)    -- ( \s + \exts*\s,-\exts*\s);


  \path[name path=EventLowerBound] (-1.6*\s, 0.6*\s)-- (-\s, 0)--
  (tInf);
  \path[name path=CosmoLowerBound] (tInf) -- (\s, 0) -- (1.6*\s, 0.6*\s);

  \coordinate (SigmaZeroH) at ($(SigmaZeroEvent) + (4*\lenNull*\s, -2*\lenNull*\s)$);
  \coordinate (SigmaTH) at ($(SigmaTEvent) + (2*\lenNull*\s, -\lenNull*\s)$);
  \coordinate (SigmaZeroI) at ($(SigmaZeroCosmo) + (-4*\lenNull*\s, -2*\lenNull*\s)$);
  \coordinate (SigmaTI) at ($(SigmaTCosmo) + (-2*\lenNull*\s, -  \lenNull*\s)$);
  \def \bottombend{15}
  \def \topbend{10}

  \path[fill=blue, fill opacity=0.1, draw=blue, thick]
  (SigmaTCosmo) to[out=-135,in=15] (SigmaTI) to[bend left =10] (SigmaZeroI) to[in=-135,out=20] (SigmaZeroCosmo) -- cycle;
  
   \path[fill=black, fill opacity=0.1, draw=black!30, thin, dashed]
   (SigmaZeroEvent) to[bend right=\bottombend] (SigmaZeroH) to[bend left = 20] (SigmaZeroI) to [bend right=10]  (SigmaTI) to[bend right =20] (SigmaTH)
   to[bend left=\topbend] (SigmaTEvent) -- cycle;

  \node[scale=0.5,fill=white,draw,circle,label=above:$i^+$]at(tInf){};

  \node[label=left:${r=M}$]at(SigmaZeroEvent){};
  \node[label=right:${r=\infty}$]at(SigmaZeroCosmo){};

\end{tikzpicture}

  \caption{Schematic representation of the near-infinity region
$\Manifold_{\rhoI \leq \rho_1}(\tau_1,\tau_2)$, highlighted in blue. The region is bounded by
$\Sigma(\tau_1)$, $\Sigma(\tau_2)$, the null infinity, and the hypersurface
$\{\rhoI=\rho_1\}$.}
\end{figure}

The main null-infinity-weighted estimate is the following. 
\begin{proposition}
  \label{prop:infinity-weighted-estimate}
  Let $\psi$ solve the inhomogeneous wave equation 
  \[\Box_\g \psi = F\]
  in an
  extremal Kerr--Newman spacetime and let $\widecheck{\psi}=r\psi$ be its radiation field. For 
  \[\b \in \big(1,3\big)\]
  the following
  $\NullInfinity$-weighted hierarchy holds true for $s \in \mathbb{N}$, $s \geq 1$ and any $\tau_1<\tau_2$:      
  \begin{equation}
    \label{eq:infinity-weighted-estimate-psic}
    \begin{split}
      & \norm*{\widecheck{\psi}}_{H_{\NullInfinity}^{s, -\frac{\b+3}{2}}(\Sigma_{\NullInfinity}(\tau_2))}
        +\norm*{\widecheck{\psi}}_{\bSobolevI{s, -\frac{\b+1}{2}}(\NullInfinity(\tau_1, \tau_2))}
        + \norm*{\widecheck{\psi}}_{\bSobolevI{s,-\frac{\b+2}{2}}(\Manifold_{\NullInfinity}(\tau_1,\tau_2))}\\
      \les{}& \norm*{\widecheck{\psi}}_{H_{\NullInfinity}^{s, -\frac{\b+3}{2}}(\Sigma_{\rhoI \leq 2\rho_1}(\tau_1))}
              + \norm*{F}_{\bSobolevI{s-1,-\frac{\b-4}{2}}(\Manifold_{\rhoI \leq 2 \rho_1}(\tau_1,\tau_2))}
              + \rho_1^{-\frac1 2}\norm*{\widecheck{\psi}}_{\compactSobolev{s}(\Manifold_{\rhoI \leq 2\rho_1}(\tau_1,\tau_2))}.
    \end{split}
  \end{equation}
  where $\Sigma_{\NullInfinity}(\tau)$ and
  $\MM_{\NullInfinity}(\tau_1, \tau_2)$ denote the regions where
  $\rhoI \leq \rho_1$ for $\rho_1 \ll r_{+}^{-1}$ sufficiently small.
\end{proposition}

The proof of \zcref{prop:infinity-weighted-estimate} is obtained as follows.
We combine the divergence and boundary terms of the current associated to the vectorfield 
\[X_{\beta} \vcentcolon=\rhoI^{\b+1}\big(-\partial_{\rhoI} + \frac{a^2}{\Upsilon}\InfinityHawking\big).\]
with a Hardy current which adds control of the zeroth-order
term. This proves \zcref{prop:infinity-weighted-estimate} for $s=1$. By applying higher order commutators, we obtain the proof of \zcref{prop:infinity-weighted-estimate} for $s\geq 2$.

\begin{remark}
  The $\NullInfinity$-weighted estimates are best known in
  asymptotically flat spacetimes as $r^p$-weighted hierarchy for $p \in (0,2)$, as
  introduced by Dafermos-Rodnianski in
  \cite{dafermosNewPhysicalSpaceApproach2010}, see also \cite{moschidis$r^p$WeightedEnergyMethod2016, angelopoulosVectorFieldApproach2018}. In our notation,
  \[p=3-\b\]
  resulting in the range of $p$ in \zcref{prop:infinity-weighted-estimate} of $p \in (0,2)$.

  Writing explicitly in terms of $p$, the above corresponds to the multiplier $X=r^{p-2}\partial_r$ applied to the radiation field and we control in oEF coordinates
  \begin{align*}
    \int_{\MM_{\NullInfinity}(\tau_1, \tau_2)} r^{p-3} | \partial_r \widecheck{\psi}|^2 +r^{p-5}\big( |\partial_v \widecheck{\psi}|^2 + |\NablaAngular\widecheck{\psi}|^2 +|\widecheck{\psi}|^2 \big) \les \text{RHS of \eqref{eq:infinity-weighted-estimate-psic}}.
  \end{align*}
\end{remark}

\subsubsection{Proof of \texorpdfstring{\zcref{prop:infinity-weighted-estimate}}{} for \texorpdfstring{$s=1$}{}}

Applying \zcref{prop:fprr-HH} to $f=r^{-\b+1}$ and $h=r^{-\b-3}$ to $\widecheck{\psi}=r\psi$, and using that $r\partial_{r}^{\mathrm{oEF}} = -\rhoI\partial_{\rhoI}$, gives that
\begin{align*}
  |q|^2\D\cdot \PP ^{(X^{\mathrm{oEF}}_{(1)}, 0, 0)}[\widecheck{\psi}]
  ={}& -\frac{1}{2} \rhoI^{\b}(1+\b + \conormalSpaceI{1})|\rhoI\partial_{\rhoI} \widecheck{\psi}|^2+\frac{1}{2} \rhoI^\b (\b-1) |\NablaAngular\widecheck{\psi}|^2\\
     &-2\rhoI^{\b-1} \rhoI \partial_{\rhoI} \widecheck{\psi} \pr_u\widecheck{\psi}
       - \rhoI^{\b+1}\partial_{\rhoI}\widecheck{\psi} \c |q|^2\Box_\g \widecheck{\psi}, \\
  |q|^2 \D\cdot \JCurrent{(X^{\mathrm{oEF}}_{(2)}, 0, 0)}[\widecheck{\psi}]={}& \rhoI^{\b}(\b+3) |\InfinityHawking\widecheck{\psi}|^2 -  2\rhoI^{\b}  (\InfinityHawking \widecheck{\psi}) \pr_u \widecheck{\psi}+   (\b+3+\conormalSpaceI{1})\rhoI^{\b+1}  (\InfinityHawking \widecheck{\psi}) \rhoI \partial_{\rhoI} \widecheck{\psi}\\
     &-  \rhoI^{\b+1} (2+\conormalSpaceI{1} ) \pr_u\widecheck{\psi} \rhoI \partial_{\rhoI} \widecheck{\psi} + \rhoI^{\b+1} (\InfinityHawking\widecheck{\psi}) \c|q|^2 \Box_\g \widecheck{\psi}.
\end{align*}
Using \zcref{lemma:wave-radiation-field} and the fact that
$\InfinityHawking = \partial_u + \rhoI^2\bDiffI$, the above becomes
\begin{align*}
  |q|^2\D\cdot \PP ^{(X^{\mathrm{oEF}}_{(1)}, 0, 0)}[\widecheck{\psi}]
  ={}& \frac{1}{2} \rhoI^{\b}(3-\b)|\rhoI\partial_{\rhoI} \widecheck{\psi}|^2+\frac{1}{2} \rhoI^\b (\b-1) |\NablaAngular\widecheck{\psi}|^2\\
     &
       +\conormalSpaceI{\b+1}(\Manifold)(\bDiffI^{\le 1}\widecheck{\psi})^2  - \rhoI^{\b}\partial_{\rhoI}\widecheck{\psi} \c \abs*{q}^2\Box_{\Metric}\psi
       , \\
  |q|^2 \D\cdot \JCurrent{(X^{\mathrm{oEF}}_{(2)}, 0, 0)}[\widecheck{\psi}]={}& \rhoI^{\b}(\b-1) |\partial_u\widecheck{\psi}|^2+\conormalSpaceI{\b+1}(\Manifold)(\bDiffI^{\le 1}\widecheck{\psi})^2 + \rhoI^{\b} (\InfinityHawking\widecheck{\psi}) \c \abs*{q}^2\Box_{\Metric}\psi.
\end{align*}
Therefore, for the vectorfield
\[X_{\b}\vcentcolon=X^{\mathrm{oEF}}_{(1)} + \frac{a^2}{\Upsilon}  X^{\mathrm{oEF}}_{(2)}=\rhoI^{\b+1}\big(-\partial_{\rhoI}+ \frac{a^2}{\Upsilon} \InfinityHawking\big)\] 
we have
\begin{equation}\label{eq:div-PP-Xbeta}
  \begin{split}
    \abs*{q}^2\D\cdot \JCurrent{(X_{\b}, 0, 0)}[\widecheck{\psi}]
    ={}& \rhoI^{\b}\Big(\frac{1}{2}(3-\b)|\rhoI\partial_{\rhoI}\widecheck{\psi}|^2
         + (\b-1)|\partial_u \widecheck{\psi}|^2
         + \frac{1}{2}(\b-1)|\NablaAngular\widecheck{\psi}|^2\Big)\\
       & 
         + \conormalSpaceI{\b+1}(\bDiffI^{\le 1}\widecheck{\psi})^2+X_{\b}(\widecheck{\psi})\cdot r \abs*{q}^2\Box_{\Metric}\psi
         .
  \end{split}    
\end{equation}

Using \zcref{lemma:boundary-terms-oEF}, we also deduce for
$\rhoI \leq \rho_1$ for $\rho_1 \ll r_{+}^{-1}$ sufficiently small,
\begin{align*}
  \JCurrent{(X_{\b}, 0, 0)}[\widecheck{\psi}]\cdot N_{\Sigma_\NullInfinity}
  \ges{}& \rhoI^{\b+3}\Big(\abs*{\partial_{\rhoI}\widecheck{\psi}}^2
          + |a|\abs*{\InfinityHawking\widecheck{\psi}}^2
          +|a|\abs*{\NablaAngular\widecheck{\psi}}^2\Big), \\
  \JCurrent{(X_{\b}, 0, 0)}[\widecheck{\psi}]\cdot (-N_{\NullInfinity})
  \ges{}&  \rhoI^{\b+1}\Big( \abs*{\rhoI\partial_{\rhoI}\widecheck{\psi}}^2
          + |a|\abs*{\InfinityHawking\widecheck{\psi}}^2
          +\abs*{\NablaAngular\widecheck{\psi}}^2\Big).
\end{align*}
Using \zcref{lemma:hardy-pointwise-rhoH} and recalling that
$ \PP_\mu^{(0,0,J)}[\psi] =  J_\mu |\psi|^2$, we have for
$J_\b=\frac 1 2 c_J(\b-1)\rhoI^{\b+1}|q|^{-2}\partial_{\rhoI}$ with $c_J>0$
\begin{align*}
  |q|^2 \D \c \PP^{(0, 0, J_\b)}[\widecheck{\psi}]\geq -c_J\rhoI^{\b}|\rhoI\pr_{\rhoI} \widecheck{\psi}|^2
  +
  \frac14c_J(\b-1)^2 \rhoI^\b |\widecheck{\psi}|^2.
\end{align*}
Adding this bound to \zcref{eq:div-PP-Xbeta} we obtain
\begin{align*}
  \abs*{q}^2\D\cdot \JCurrent{(X_{\b}, 0, J_{\b})}[\widecheck{\psi}]
  ={}& \rhoI^{\b}\Big[\frac{1}{2}(3-\b-2c_J)|\rhoI\partial_{\rhoI}\widecheck{\psi}|^2
       + (\b-1)|\partial_u \widecheck{\psi}|^2
       + \frac{1}{2}(\b-1)|\NablaAngular\widecheck{\psi}|^2\\
     &+
       \frac14c_J(\b-1)^2 |\widecheck{\psi}|^2\Big]
       + \conormalSpaceI{\b+1}(\bDiffI^{\le 1}\widecheck{\psi})^2+X_{\b}(\widecheck{\psi})\cdot r \abs*{q}^2\Box_{\Metric}\psi
       . 
\end{align*}
For $c_J \ll 1$ sufficiently small and $1<\b<3$, the above terms in parenthesis are all positive, and for $\rho_1 \ll r_{+}^{-1}$ sufficiently small the lower order terms $\conormalSpaceI{\b+1}(\bDiffI^{\le 1}\widecheck{\psi})^2$ can be absorbed by the positive ones and therefore we have, dividing by $|q|^2$,
\begin{align}\label{eq:bound-DD-P-rp}
  \begin{split}
    \D\cdot \JCurrent{(X_{\b}, 0, J_{\b})}[\widecheck{\psi}]
    \gtrsim{}& \rhoI^{\b+2}\Big(|\rhoI\partial_{\rhoI}\widecheck{\psi}|^2
               +|\partial_u \widecheck{\psi}|^2
               + |\NablaAngular\widecheck{\psi}|^2+
               |\widecheck{\psi}|^2\Big)
               +X_{\b}(r\psi)\cdot r \Box_{\Metric}\psi,
  \end{split}
\end{align}
where observe that throughout the infinity estimates, division by $|q|^2 \sim \rhoI^{-2}$ shifts weights by $+2$. Here, we recognize the energy density of a weighted $b$-Sobolev norm of $\widecheck{\psi}$.

We now localize the above using a cut-off function. For any non-negative bump function $\chi$ supported in $[-2,2]$ and identically equal to $1$ on $[-1,1]$, define 
\begin{align*}
  \chi_{\NullInfinity}\vcentcolon=\chi\big(\frac{\rhoI}{\rho_1} \big), \qquad X_{\NullInfinity, \beta}\vcentcolon= \chi_{\NullInfinity}   X_{\beta}, \qquad J_{\NullInfinity, \beta}\vcentcolon=\chi_{\NullInfinity} J_{\beta}.
\end{align*}
From \zcref{eq:bound-DD-P-rp, current-cutoff}  we obtain
\begin{align*}
  &\int_{\Manifold(\tau_1,\tau_2)} \D\cdot \JCurrent{(X_{\NullInfinity, \b},  0, J_{\NullInfinity, \beta})}[\widecheck{\psi}]\\
  \ges{}& \int_{\Manifold(\tau_1,\tau_2)}\mathds{1}_{\{\rhoI < \rho_1\}} \rhoI^{\b+2}\Big(|\rhoI\partial_{\rhoI}\widecheck{\psi}|^2
          +|\partial_u \widecheck{\psi}|^2
          + |\NablaAngular\widecheck{\psi}|^2+|\widecheck{\psi}|^2\Big)\\
  &-\rho_1^{-1}\int_{\Manifold(\tau_1,\tau_2)}\mathds{1}_{\{\rho_1 \leq \rhoI \leq 2\rho_1\}}\abs*{(\partial_u,\partial_{\rhoI},\NablaAngular)^{\leq 1}\widecheck{\psi}}^2-\int_{\Manifold(\tau_1,\tau_2)}|X_{\NullInfinity, \b}(\widecheck{\psi}) || r\Box_{\Metric}\psi|\\
  \ges{}&  \norm*{\widecheck{\psi}}_{\bSobolevI{1,-\frac{\b+2}{2}}(\Manifold_{\NullInfinity}(\tau_1,\tau_2))}^2- \rho_1^{-1}\norm*{\widecheck{\psi}}_{\compactSobolev{1}(\Manifold_{\rhoI \leq 2\rho_1}(\tau_1,\tau_2))}^2\\
  &-\norm*{\widecheck{\psi}}_{\bSobolevI{1,-\frac{\b+2}{2}}(\Manifold_{\rhoI \leq 2 \rho_1}(\tau_1,\tau_2))}
    \norm*{\Box_{\Metric}\psi}_{\bSobolevI{0,-\frac{\b-4}{2}}(\Manifold_{\rhoI \leq 2\rho_1}(\tau_1,\tau_2))}\\
  \ges{}&  \norm*{\widecheck{\psi}}_{\bSobolevI{1,-\frac{\b+2}{2}}(\Manifold_{\NullInfinity}(\tau_1,\tau_2))}^2- \rho_1^{-1}\norm*{\widecheck{\psi}}_{\compactSobolev{1}(\Manifold_{\rhoI \leq 2\rho_1}(\tau_1,\tau_2))}^2  -
          \norm*{F}_{\bSobolevI{0,-\frac{\b-4}{2}}(\Manifold_{\rhoI \leq 2\rho_1}(\tau_1,\tau_2))}^2,
\end{align*}
where we used Cauchy-Schwarz to bound the last term. 

For the boundary terms, we deduce
\begin{align*}
  \int_{\Sigma(\tau)}    \JCurrent{(X_{\NullInfinity, \b},0,0)}[\widecheck{\psi}]\cdot N_{\Sigma(\tau)}
  \ges{} &     \int_{\Sigma(\tau)}\mathds{1}_{\{\rhoI < \rho_1\}} \rhoI^{\b+3}\Big(\abs*{\partial_{\rhoI}\widecheck{\psi}}^2
           + \abs*{\InfinityHawking\widecheck{\psi}}^2
           +\abs*{\NablaAngular\widecheck{\psi}}^2\Big).
\end{align*}
We also compute
\begin{align*}
  \JCurrent{(0,0,J_{\NullInfinity,\b})}[\widecheck{\psi}]\cdot N_{\Sigma_{\NullInfinity}(\tau)}
  & =
    \chi_{\NullInfinity}(J_{\b}\cdot N_{\Sigma_{\NullInfinity}(\tau)})|\widecheck{\psi}|^2 \\
  & =
    \frac12 \chi_{\NullInfinity} c_J(\b-1)\rhoI^{\b+1}|q|^{-4}\g(\partial_{\rhoI}, -\partial_{\rhoI}+
    \frac{a^2}{\Upsilon}\InfinityHawking 
    -a\sin\th \renormpphi)|\widecheck{\psi}|^2 \\
  & =
    \frac12 \chi_{\NullInfinity} c_J(\b-1)\rhoI^{\b+1}\frac{a^2}{\Upsilon}|q|^{-2}|\widecheck{\psi}|^2 , \\
  \JCurrent{(0,0,J_{\NullInfinity,\b})}[\widecheck{\psi}]\cdot(- N_{\NullInfinity})
  & =
    \frac12 \chi_{\NullInfinity} c_J(\b-1)\rhoI^{\b+1}|q|^{-2}\g(\partial_{\rhoI},  \InfinityHawking )|\widecheck{\psi}|^2 \\
  & =
    \frac12 \chi_{\NullInfinity} c_J(\b-1)\rhoI^{\b+1}|\widecheck{\psi}|^2 
\end{align*}
where we used that $\g_{\mathrm{oEF}}(\partial_{\rhoI}, \partial_{\rhoI})=\g_{\mathrm{oEF}}(\partial_{\rhoI}, \renormpphi)=0$ and $\g_{\mathrm{oEF}}(\partial_{\rhoI}, V_{\NullInfinity})=|q|^2$.  We therefore obtain 
\begin{align*}
  \begin{split}
    \int_{\Sigma_\NullInfinity(\tau)}  \JCurrent{(X_{\NullInfinity, \b}, 0, J_{\NullInfinity, \b})}[\widecheck{\psi}]\cdot N_{\Sigma_\NullInfinity}
    \ges{}& \int_{\Sigma_\NullInfinity(\tau)}\rhoI^{\b+3}\Big(\abs*{\partial_{\rhoI}\widecheck{\psi}}^2
            + \abs*{\InfinityHawking\widecheck{\psi}}^2
            +\abs*{\NablaAngular\widecheck{\psi}}^2+\abs*{\widecheck{\psi}}^2\Big) , \\
    \ges{}& \norm*{\widecheck{\psi}}_{H_{\NullInfinity}^{1, -\frac{\b+3}{2}}(\Sigma_{\NullInfinity}(\tau))}^2\\
    \int_{\NullInfinity(\tau_1, \tau_2)} \JCurrent{(X_{\NullInfinity, \b}, 0, J_{\NullInfinity, \b})}[\widecheck{\psi}]\cdot (-N_{\NullInfinity})
    \ges{}&  \int_{\NullInfinity(\tau_1, \tau_2)}\rhoI^{\b+1}\Big( \abs*{\rhoI\partial_{\rhoI}\widecheck{\psi}}^2
            + \abs*{\InfinityHawking\widecheck{\psi}}^2
            +\abs*{\NablaAngular\widecheck{\psi}}^2+\abs*{\widecheck{\psi}}^2\Big)\\
    \ges{}&  \norm*{\widecheck{\psi}}_{\bSobolevI{1, -\frac{\b+1}{2}}(\NullInfinity(\tau_1, \tau_2))}^2.
  \end{split}
\end{align*}

Performing the divergence theorem over $\Manifold(\tau_1, \tau_2)$, we obtain from \zcref{eq:general-divergence-theorem}
\begin{align*}
  &\int_{\Sigma(\tau_2)}\JCurrent{(X_{\NullInfinity, \b},  0, J_{\NullInfinity, \b})}[\widecheck{\psi}]\cdot N_{\Sigma(\tau_2)}
    + \int_{\NullInfinity(\tau_1,\tau_2)}\JCurrent{(X_{\NullInfinity, \b},  0, J_{\NullInfinity, \b})}[\widecheck{\psi}]\cdot (- N_{\NullInfinity})
    + \int_{\Manifold(\tau_1,\tau_2)}\D\cdot \JCurrent{(X_{\NullInfinity, \b},  0, J_{\NullInfinity, \b})}[\widecheck{\psi}]\\
  &
    = \int_{\Sigma(\tau_1)}\JCurrent{(X_{\NullInfinity, \b},  0, J_{\NullInfinity, \b})}[\widecheck{\psi}]\cdot N_{\Sigma(\tau_1)}.
\end{align*}
Using the bounds above and the trivial bound
\begin{align*}
  \int_{\Sigma(\tau_1)}\JCurrent{(X_{\NullInfinity, \b},  0, J_{\NullInfinity, \b})}[\widecheck{\psi}]\cdot N_{\Sigma(\tau_1)}
  &\les \int_{\Sigma_{\NullInfinity}(\tau_1)}\rhoI^{\b+3}\Big(\abs*{\partial_{\rhoI}\widecheck{\psi}}^2
    + \abs*{\InfinityHawking\widecheck{\psi}}^2
    +\abs*{\NablaAngular\widecheck{\psi}}^2+\abs*{\widecheck{\psi}}^2\Big) \\
  &\les \norm*{\widecheck{\psi}}_{H_{\NullInfinity}^{1, -\frac{\b+3}{2}}(\Sigma_{\rhoI \leq 2 \rho_1}(\tau_1))}^2,
\end{align*}
we deduce \zcref{eq:infinity-weighted-estimate-psic}.

\subsubsection{Proof of Proposition \texorpdfstring{\ref{prop:infinity-weighted-estimate}}{} for \texorpdfstring{$s\geq 2$}{}}
\label{sec:higher-order-infinity}

We first prove the estimate for $s=2$.

\begin{enumerate}
\item \emph{Killing commutators.}
  Since $T=\partial_v$ and $\Phi=\partial_{\phiOEF}$ are Killing, applying \zcref{eq:infinity-weighted-estimate-psic} for $s=1$ to $T\widecheck{\psi}$ and $\Phi\widecheck{\psi}$ gives
  \begin{align}
    \label{eq:NI-higher-Killing-commuted-estimate}
    \begin{split}
      & \norm*{(T,\Phi)\widecheck{\psi}}_{H_{\NullInfinity}^{1, -\frac{\b+3}{2}}(\Sigma_{\NullInfinity}(\tau_2))}
        +\norm*{(T,\Phi)\widecheck{\psi}}_{\bSobolevI{1, -\frac{\b+1}{2}}(\NullInfinity(\tau_1, \tau_2))}
        + \norm*{(T,\Phi)\widecheck{\psi}}_{\bSobolevI{1,-\frac{\b+2}{2}}(\Manifold_{\NullInfinity}(\tau_1,\tau_2))}\\
      \les{}& \norm*{\widecheck{\psi}}_{H_{\NullInfinity}^{2, -\frac{\b+3}{2}}(\Sigma_{\rhoI \leq 2\rho_1}(\tau_1))}
              + \norm*{F}_{\bSobolevI{1,-\frac{\b-4}{2}}(\Manifold_{\rhoI \leq 2 \rho_1}(\tau_1,\tau_2))}
              + \rho_1^{-\frac12}\norm*{\widecheck{\psi}}_{\compactSobolev{2}(\Manifold_{\rhoI \leq 2\rho_1}(\tau_1,\tau_2))}. 
    \end{split}
  \end{align}

\item \emph{Angular commutators.}
  We apply \zcref{eq:bound-DD-P-rp}
  to $\RotationVF_1\widecheck{\psi}$ and
  $\RotationVF_2\widecheck{\psi}$ and we obtain
  \begin{align}\label{eq:bound-angular-infinity}
    \begin{split}
      & |q|^2  \D\cdot \big(\JCurrent{(X_{\b}, 0, J_{\b})}[\RotationVF_1\widecheck{\psi}]+\JCurrent{(X_{\b}, 0, J_{\b})}[\RotationVF_2\widecheck{\psi}]\big)\\
      \gtrsim{}&  \sum_{i=1}^2\rhoI^{\b}\Big(|\rhoI\partial_{\rhoI}\RotationVF_i\widecheck{\psi}|^2
                 +|\partial_u \RotationVF_i\widecheck{\psi}|^2
                 + |\NablaAngular\RotationVF_i\widecheck{\psi}|^2+
                 |\RotationVF_i\widecheck{\psi}|^2\Big)\\
      &
        +X_{\b}(\RotationVF_1\widecheck{\psi})\cdot r [|q|^2 \Box_{\Metric},\RotationVF_1]\psi +X_{\b}(\RotationVF_2\widecheck{\psi})\cdot r [|q|^2 \Box_{\Metric},\RotationVF_2]\psi\\
      &
        +X_{\b}(\RotationVF_1\widecheck{\psi})\cdot r \RotationVF_1(|q|^2 \Box_{\Metric}\psi )+X_{\b}(\RotationVF_2\widecheck{\psi})\cdot r \RotationVF_2(|q|^2 \Box_{\Metric}\psi).
    \end{split}
  \end{align}
  Using \zcref{lem:weighted-hierarchy-commutators}, we compute the expression:
  \begin{align*}
    I\vcentcolon={}&X_{\b}(\RotationVF_1\widecheck{\psi})\cdot r [|q|^2 \Box_{\Metric},\RotationVF_1]\psi +X_{\b}(\RotationVF_2\widecheck{\psi})\cdot r [|q|^2 \Box_{\Metric},\RotationVF_2]\psi\\
    ={}&r X_{\b}(\RotationVF_1\widecheck{\psi})(-2a(\rhoI^2\partial_{\rhoI}+T)\RotationVF_2\psi
         +
         O(a^2)T^2\psi)+rX_{\b}(\RotationVF_2\widecheck{\psi}) (2a(\rhoI^2\partial_{\rhoI}+T)\RotationVF_1\psi
         +
         O(a^2)T^2\psi)\\
    ={}&\Big(-\rhoI^{\b}\partial_{\rhoI}\RotationVF_1\widecheck{\psi} + \frac{a^2}{\Upsilon}\rhoI^{\b}\InfinityHawking\RotationVF_1\widecheck{\psi}\Big)(-2a\rhoI (\rhoI^2\partial_{\rhoI}+T)\RotationVF_2 \widecheck{\psi}-2a\rhoI^2\RotationVF_2\widecheck{\psi}
         +
         O(a^2)\rhoI T^2 \widecheck{\psi})\\
                   &+\Big(-\rhoI^{\b}\partial_{\rhoI}\RotationVF_2\widecheck{\psi} + \frac{a^2}{\Upsilon}\rhoI^{\b}\InfinityHawking\RotationVF_2\widecheck{\psi}\Big)(2a\rhoI(\rhoI^2\partial_{\rhoI}+T)\RotationVF_1\widecheck{\psi}+ 2 a\rhoI^2\RotationVF_1\widecheck{\psi}
                     +
                     O(a^2)\rhoI T^2 \widecheck{\psi})\\
    ={}& 2a\rhoI^{\b}
         \Big(
         T\RotationVF_2\widecheck{\psi}\,\rhoI\partial_{\rhoI}\RotationVF_1\widecheck{\psi}
         -
         T\RotationVF_1\widecheck{\psi}\,\rhoI\partial_{\rhoI}\RotationVF_2\widecheck{\psi}
         \Big) \\
                   &-
                     2a\frac{a^2}{\Upsilon}\rhoI^{\b+2}
                     \Big(
                     \rhoI\partial_{\rhoI}\RotationVF_2\widecheck{\psi}\,\InfinityHawking\RotationVF_1\widecheck{\psi}
                     -
                     \rhoI\partial_{\rhoI}\RotationVF_1\widecheck{\psi}\,\InfinityHawking\RotationVF_2\widecheck{\psi}
                     \Big) \\
                   &-
                     2a^2\frac{a^2}{\Upsilon}\rhoI^{\b+2}
                     \Big(
                     T\RotationVF_2\widecheck{\psi}\,\Phi\RotationVF_1\widecheck{\psi}
                     -
                     T\RotationVF_1\widecheck{\psi}\,\Phi\RotationVF_2\widecheck{\psi}
                     \Big)\\
                   &+O(a)\rhoI^{\b}\big(-\rhoI\partial_{\rhoI}(\RotationVF_1+\RotationVF_2)\widecheck{\psi} + \frac{a^2}{\Upsilon}\rhoI \InfinityHawking(\RotationVF_1+\RotationVF_2)\widecheck{\psi}\big)
                     \big( T^2\widecheck{\psi}+ \rhoI (\RotationVF_1+\RotationVF_2)\widecheck{\psi}\big),
  \end{align*}
  where we wrote $X_{\b}=\rhoI^{\b+1}\big(-\partial_{\rhoI}+ \frac{a^2}{\Upsilon} \InfinityHawking\big)$, $\InfinityHawking=\frac{r^2+a^2}{r^2}T+ \frac{a}{r^2}\Phi$ and $\psi=\rhoI \widecheck{\psi}$.
  Notice that the terms involving two $\partial_{\rhoI}$ and $T$ derivatives cancel out in the sum.
  Using Cauchy-Schwarz, we can bound the above for $0<\lambda<1$ by
  \begin{align*}
    \abs*{I}\leq \lambda \rhoI^\b\sum_{i=1}^2|\rhoI\partial_{\rhoI}\RotationVF_i\widecheck{\psi}|^2+\lambda^{-1} \rhoI^\b \big(\sum_{i=1}^2|T\RotationVF_i\widecheck{\psi}|^2+\sum_{i=1}^2|\Phi\RotationVF_i\widecheck{\psi}|^2+\sum_{i=1}^2|\RotationVF_i\widecheck{\psi}|^2+|T^2\widecheck{\psi}|^2\big).
  \end{align*}
  For $\lambda$ sufficiently small, the first term above can be absorbed by the first line in \zcref{eq:bound-angular-infinity}, while the remaining terms only involve second derivatives of $T$ and $\Phi$ and therefore can be bounded, upon integration on $\Manifold(\tau_1, \tau_2)$, by the initial data norms using \zcref{eq:NI-higher-Killing-commuted-estimate} and using \zcref{prop:infinity-weighted-estimate} for $s=1$. Applying divergence theorem to \zcref{eq:bound-angular-infinity} and following the same steps as in proof of \zcref{prop:infinity-weighted-estimate} for $s=1$, we then deduce
  \begin{align}
    \label{eq:II-angular-commuted-estimate}
    \begin{split}
      & \sum_{i=1}^2\norm*{\RotationVF_i\widecheck{\psi}}_{H_{\NullInfinity}^{1, -\frac{\b+3}{2}}(\Sigma_{\NullInfinity}(\tau_2))}
        +\sum_{i=1}^2\norm*{\RotationVF_i\widecheck{\psi}}_{\bSobolevI{1, -\frac{\b+1}{2}}(\NullInfinity(\tau_1, \tau_2))}
        + \sum_{i=1}^2\norm*{\RotationVF_i\widecheck{\psi}}_{\bSobolevI{1,-\frac{\b+2}{2}}(\Manifold_{\NullInfinity}(\tau_1,\tau_2))}\\
      \les{}& \norm*{\widecheck{\psi}}_{H_{\NullInfinity}^{2, -\frac{\b+3}{2}}(\Sigma_{\rhoI \leq 2\rho_1}(\tau_1))}
              + \norm*{F}_{\bSobolevI{1,-\frac{\b-4}{2}}(\Manifold_{\rhoI \leq 2 \rho_1}(\tau_1,\tau_2))}
              + \rho_1^{-\frac12}\norm*{\widecheck{\psi}}_{\compactSobolev{2}(\Manifold_{\rhoI \leq 2\rho_1}(\tau_1,\tau_2))}.
    \end{split}
  \end{align}
  The angular elliptic estimate then gives respective control of the second angular derivative of $\widecheck{\psi}$.

\item \emph{The second $\rhoI$-derivative.}
  We apply \zcref{eq:bound-DD-P-rp}
  to $\Psi\vcentcolon=\rhoI \partial_{\rhoI}\widecheck{\psi}$ and we obtain
  \begin{align}\label{eq:bound-commutator-rhoI}
    \begin{split}
      \abs*{q}^2\D\cdot\big(\JCurrent{(X_{\b},0, J_{\b})}[\rhoI \partial_{\rhoI}\widecheck{\psi}]\big)  \ges{}&   \rhoI^\b  \Big(  |\pr_u\Psi|^2+|\NablaAngular\Psi|^2 +|\rhoI\pr_{\rhoI}\Psi|^2+
                                                                                                                |\Psi|^2\Big)\\
                                                                                                              &+r X_{\b}(\Psi)[ |q|^2 \Box_\g ,\rhoI \partial_{\rhoI}]\psi+r X_{\b}(\Psi) \rhoI \partial_{\rhoI}(|q|^2 \Box_\g \psi). 
    \end{split}
  \end{align}
  Using \zcref{lem:weighted-hierarchy-commutators}, we compute the expression:
  \begin{align*}
    Y\vcentcolon={}&r X_{\b}(\Psi)[ |q|^2 \Box_\g ,\rhoI \partial_{\rhoI}]\psi\\
                ={}&r X_{\b}(\Psi)\Big[-(\rhoI\partial_{\rhoI})^2\psi
                  +(\rhoI\partial_{\rhoI}\psi)\\
                &-\lapp_{\SSS^2}\psi-a^2\sin^2\theta \pr_u^2\psi
                  -2a\pr_u\pr_{\phiOEF}\psi+\conormalSpaceI{1}(\Manifold)\bDiffI^1\psi+\conormalSpaceI{1}(\Manifold)(\rhoI \partial_{\rhoI})^{\leq 2}\psi+|q|^2F\Big]\\
                ={}&\Big(-\rhoI^{\b}\partial_{\rhoI}\Psi + \frac{a^2}{\Upsilon}\rhoI^{\b}\InfinityHawking\Psi \Big)\rhoI \Big(-\rhoI\partial_{\rhoI}\Psi\Big)\\
                &+ X_{\b}(\Psi)\Big( \Psi-\lapp_{\SSS^2}\widecheck{\psi}-a^2\sin^2\theta T^2\widecheck{\psi}
                  -2aT\Phi\widecheck{\psi}+\conormalSpaceI{1}(\Manifold)\bDiffI^1\widecheck{\psi}+\conormalSpaceI{1}(\Manifold)(\rhoI \partial_{\rhoI})^{\leq 2}\widecheck{\psi}+|q|^2r F\Big)
  \end{align*}
  Notice that the first term in $|\partial_{\rhoI}\Psi|^2$ gives a positive contribution while the remaining terms can be bounded using Cauchy-Schwarz for $0<\lambda<1$ by
  \begin{align*}
    Y \geq{}& \rhoI^\b(1 -\conormalSpaceI{1}(\Manifold))\abs*{\rhoI \partial_{\rhoI}\Psi}^2-\lambda \rhoI^\b \abs*{\rhoI \partial_{\rhoI}\Psi}^2\\
      &-\lambda^{-1}\rhoI^{\b}\Big[|T\rhoI \partial_{\rhoI}\widecheck{\psi}|^2+|\Phi\rhoI \partial_{\rhoI}\widecheck{\psi}|^2+|\rhoI \partial_{\rhoI}\widecheck{\psi}|^2 +
        |\lapp_{\SSS^2}\widecheck{\psi}|^2
        +
        |T^2\widecheck{\psi}|^2\\
      &
        +\conormalSpaceI{1}(\Manifold)|\bDiffI^1\widecheck{\psi}|^2
        +
        |T\widecheck{\psi}|^2+r^2|F|^2\Big].
  \end{align*}
  For $\lambda$ and $\rhoI$ sufficiently small, the second term
  can be absorbed by the first line in \zcref{eq:bound-commutator-rhoI},
  while the remaining  terms only involve second derivatives in $T$, $\Phi$ or $\NablaAngular$ and therefore can be bounded, upon integration on
  $\Manifold(\tau_1, \tau_2)$, by the initial data norms using
  \zcref{eq:NI-higher-Killing-commuted-estimate, eq:II-angular-commuted-estimate}. Applying divergence theorem
  to \zcref{eq:bound-commutator-rhoI} and following the same steps as in
  proof of \zcref{prop:horizon-weighted-estimate} for $s=1$, we then
  deduce
  \begin{align}
    \label{eq:II-rho-commuted-estimate}
    \begin{split}
      & \norm*{\rhoI \partial_{\rhoI}\widecheck{\psi}}_{H_{\NullInfinity}^{1, -\frac{\b+3}{2}}(\Sigma_{\NullInfinity}(\tau_2))}
        +\norm*{\rhoI \partial_{\rhoI}\widecheck{\psi}}_{\bSobolevI{1, -\frac{\b+1}{2}}(\NullInfinity(\tau_1, \tau_2))}
        + \norm*{\rhoI \partial_{\rhoI}\widecheck{\psi}}_{\bSobolevI{1,-\frac{\b+2}{2}}(\Manifold_{\NullInfinity}(\tau_1,\tau_2))}\\
      \les{}& \norm*{\widecheck{\psi}}_{H_{\NullInfinity}^{2, -\frac{\b+3}{2}}(\Sigma_{\rhoI \leq 2\rho_1}(\tau_1))}
              + \norm*{F}_{\bSobolevI{1,-\frac{\b-4}{2}}(\Manifold_{\rhoI \leq 2 \rho_1}(\tau_1,\tau_2))}
              + \rho_1^{-\frac12}\norm*{\widecheck{\psi}}_{\compactSobolev{2}(\Manifold_{\rhoI \leq 2\rho_1}(\tau_1,\tau_2))}. 
    \end{split}
  \end{align}
\end{enumerate}

Combining \zcref{eq:NI-higher-Killing-commuted-estimate,eq:II-angular-commuted-estimate,eq:II-rho-commuted-estimate} we prove \zcref{eq:infinity-weighted-estimate-psic} for $s=2$.

The case $s\geq 3$ follows by induction, by commuting with
$A\in\bDiffI^{s-1}$ and using
\zcref{lem:weighted-hierarchy-commutators}.

\section{Energy-Morawetz estimate}\label{sec:Morawetz}

The goal of this section is to prove the following energy-Morawetz estimate.
\begin{proposition}\label{prop:energy-Morawetz}
  Let $\psi$ solve the inhomogeneous wave equation 
  \[\Box_\g \psi = F\]
  in an
  extremal Kerr--Newman spacetime with $\frac{|a|}{M}\ll 1$. 
  Then for $\de_1 \ll 1$ the following energy-Morawetz estimate holds true for any $\tau_1<\tau_2$:
  \begin{align}\label{eq:final-theorem-energy-morawetz}
    \norm*{\psi}_{H_{\EventHorizon, \operatorname{b}\NullInfinity}^{1, -\frac{1-\de_1}{2}, -1}(\Sigma(\tau_2))}
    +\norm*{\psi}_{\bSobolevHITrap{1, \frac{\de_1}{2}, -\frac 5 2 }(\Manifold(\tau_1, \tau_2))}
    \les \norm*{\psi}_{H_{\EventHorizon, \operatorname{b}\NullInfinity}^{1, -\frac{1-\de_1}{2}, -1}(\Sigma(\tau_1))}
    +\norm*{F}_{\bSobolevHI{0,\frac{\delta_1}{2}, -\frac{1}{2}}(\Manifold(\tau_1,\tau_2))}.
  \end{align}
\end{proposition}

The proof of \zcref{prop:energy-Morawetz} consists of two parts:
\begin{enumerate}
\item the integrated local energy decay (or Morawetz estimates): 
  these are obtained by combining differential multipliers whose bulk is positive away from trapping with a pseudodifferential correction localized near the trapped set.
  More precisely:
  \begin{enumerate}
  \item outside trapping, the choice of differential multipliers is based on a choice in axial symmetry introduced in \cite{giorgiBoundednessDecayTeukolsky2026}. Nevertheless, outside axial symmetry a mixed term in the bulk cannot be absorbed near the event horizon, even when \(\frac{|a|}{M}\ll1\),
  \item the \(\a=-\de_1\) member of the \(\EventHorizon\)-hierarchy is used as a degenerate redshift estimate to absorb the mixed term near the event horizon. For this particular member, the hierarchy remains valid in a fixed neighborhood of the event horizon, rather than only in an arbitrarily small one, which makes it possible to absorb the problematic terms in that region.
  \item at trapping a pseudodifferential correction is needed, following \cite{tataruLocalEnergyEstimate2011}, to obtain positivity of the bulk. 
  \end{enumerate}
  The final estimates are stated in \zcref{lemma:Morawetz:cutoff-PsiDO:bulk} and proved in Section \ref{sec:ILED}.
\item the energy estimates: these are obtained through a globally timelike vectorfield, which is Killing outside a compact region away from the event horizon and trapping. The error term produced in the bulk can be absorbed for $\frac{|a|}{M} \ll 1$ by the Morawetz bulk obtained in the previous point and all the boundary terms are positive. This is proved in \zcref{subs:energy}.
\end{enumerate}
The proof of \zcref{prop:energy-Morawetz} is obtained by combining the Morawetz estimates with the energy estimates multiplied by a large constant, and it is proved in \zcref{sec:proof-e-mor}.

\subsection{Integrated local energy decay (Morawetz) estimates}\label{sec:ILED}

The goal of this subsection is to prove \zcref{lemma:Morawetz:cutoff-PsiDO:bulk}, stated at the end of the subsection. The proof is obtained by combining the choice of differential multipliers outside trapping (given by the axially symmetric choice and a degenerate redshift estimate) constructed in~\zcref{sec:positivity-bulk-outside-trapping} with a pseudodifferential correction at trapping obtained in~\zcref{sec:Morawetz:principal-order-bulk}.

\subsubsection{Positivity of the bulk outside trapping}\label{sec:positivity-bulk-outside-trapping}

In the following proposition we construct differential multipliers which have a positive lower bound outside trapping. 

\begin{proposition}
  \label{prop:Mor-axially-sym:main-bulk}
  On extremal Kerr--Newman spacetime with $\frac{|a|}{M}< \frac{1}{30}$, for any $\de_1 \ll 1$ there exists a regular function $u_\daxi(r)\in \conormalSpaceHI{0,-2}(\Manifold)$ and a multiplier triplet $(X_{\Mor}^\diff, w_{\Mor}^\diff, J_{\Mor}^\diff)$
  such that for every function $\psi$:
  \begin{equation}\label{eq:Mor-axially-sym:multiplier-choices:square:expansion}
    \begin{split}
      \QQ^{(X_{\Mor}^\diff, w_{\Mor}^\diff, J_{\Mor}^\diff)}[\psi]
      \geq{}&  c \Big(\AxiSquareSumCoeff_1^2\abs*{\That\psi}^2
              + \AxiSquareSumCoeff_2^2\abs*{\partial^{\mathrm{BL}}_r\psi}^2+ \widetilde{\AxiSquareSumCoeff}_2^2\abs*{\rhoH\partial_{\rhoH}\psi}^2
              + \AxiSquareSumCoeff_3^2\abs*{\nabla\psi}^2
              + \AxiSquareSumCoeff_4^2\abs*{\psi}^2\Big)
      \\
            &-    \frac{2aru_\daxi}{|q|^2(r^2+a^2)^2}\That \psi \pr_\phi \psi
              ,
    \end{split}    
  \end{equation}
  where
  \begin{align*}
    \AxiSquareSumCoeff_1^2={}& \frac{1}{r^3} \big( 1-\frac{r_{\trap}}{r}\big)^2 +\mathds{1}_{\{\rhoH < \frac{M}{80} \}} \rhoH^{-\de_1}\in \conormalSpaceHI{-\de_1,3}(\Manifold), \\
    \AxiSquareSumCoeff_2^2={}&  \frac{\rhoH^4}{r^7} \in \conormalSpaceHI{4,3}(\Manifold), \\
    \widetilde{\AxiSquareSumCoeff}_2^2={}& \mathds{1}_{\{\rhoH < \frac{M}{80} \}} \rhoH^{-\de_1} \in \conormalSpaceH{-\de_1}(\Manifold), \\
    \AxiSquareSumCoeff_3^2={}& \frac{1}{r^2}\rhoH \big( 1-\frac{r_{\trap}}{r}\big)^2 +\mathds{1}_{\{\rhoH < \frac{M}{80} \}} \rhoH^{-\de_1}\in \conormalSpaceHI{-\de_1,1}(\Manifold), \\
    \AxiSquareSumCoeff_4^2={}& \frac{\rhoH^2}{r^6}+c_{red}\mathds{1}_{\{\rhoH < \frac{M}{80} \}} \rhoH^{-\de_1}\in \conormalSpaceHI{-\de_1,4}(\Manifold),
  \end{align*}
  for some $c>0$.
\end{proposition}

The proof of \zcref{prop:Mor-axially-sym:main-bulk} is given at the end of this subsection, after combining the axially symmetric Morawetz multiplier with a degenerate redshift estimate.
Observe that for $|a|/M \ll 1$, the mixed term on the right hand side of \zcref{eq:Mor-axially-sym:multiplier-choices:square:expansion} can be absorbed \emph{outside the trapping region}.

\begin{corollary}\label{cor:control-outside-trapping}
  On extremal Kerr--Newman spacetime with $\frac{|a|}{M}\ll 1$, the same choice as in \zcref{prop:Mor-axially-sym:main-bulk} gives 
  \begin{equation*}
    \int_{\Manifold_{\nontrap}(\tau_1,\tau_2)}\KCurrent{(X_{\Mor}^\diff, w_{\Mor}^\diff, J_{\Mor}^\diff)}[\psi]
    \ges \norm*{\psi}^2_{\bSobolevHI{1,\frac{\delta_1}{2},-\frac{5}{2}}(\Manifold_{\nontrap}(\tau_1,\tau_2))}.
  \end{equation*}
\end{corollary}
\begin{proof}
  Using that $u_\daxi(r)\in \conormalSpaceHI{0,-2}(\Manifold)$, we can bound the mixed term on the right hand side of \zcref{eq:Mor-axially-sym:multiplier-choices:square:expansion} by Cauchy-Schwarz for $\rhoH \ll 1$:
  \begin{align*}
    \abs*{\frac{2aru_\daxi}{|q|^2(r^2+a^2)^2}\That \psi \pr_\phi \psi } \les  
    O(|a|)    |\That \psi \pr_\phi \psi| \les O(|a|) \big(\rhoH^{-\de_1} \abs*{\That\psi}^2 + \rhoH^{-\de_1} \abs*{\nabla\psi}^2\big)
  \end{align*}
  and for $\rhoI \ll 1$:
  \begin{align*}
    \abs*{\frac{2aru_\daxi}{|q|^2(r^2+a^2)^2}\That \psi \pr_\phi \psi } \les  
    O(|a|) \frac{1}{r^3}     |\That \psi \pr_\phi \psi| \les O(|a|) \big(\rhoI^{3} \abs*{\That\psi}^2 + \rhoI \abs*{\nabla\psi}^2\big)
  \end{align*}
  and similarly in a compact region. Hence for $\frac{|a|}{M}\ll 1$ outside the trapping region, the mixed term can be absorbed by the first line of \zcref{eq:Mor-axially-sym:multiplier-choices:square:expansion}. In particular, integrating on $\Manifold_{\nontrap}(\tau_1,\tau_2)$, we deduce, using \zcref{eq:relation-BL-iEF}, 
  \begin{align*}
    &\int_{\Manifold_{\nontrap}(\tau_1,\tau_2)}  \KCurrent{(X_{\Mor}^\diff, w_{\Mor}^\diff, J_{\Mor}^\diff)}[\psi] \\
    \ges{}&  \norm*{\psi}_{\compactSobolev{1}(\Manifold_{\nontrap}(\tau_1,\tau_2))}^2+\int_{\Manifold_{\nontrap}(\tau_1,\tau_2)\cap \{ \rhoH < \frac{M}{160}\}}\rhoH^{-\de_1}\big(\abs*{\partial_v\psi}^2
      + \abs*{\rhoH\partial_{\rhoH}\psi}^2
      +\abs*{\NablaAngular\psi}^2
      + \abs*{\psi}^2\big)\\
    & + \int_{\Manifold_{\nontrap}(\tau_1,\tau_2)\cap \{ \rhoI \ll 1 \}}  \rhoI^3\abs*{\partial_u\psi}^2
      + \rhoI^5\abs*{\rhoI\partial_{\rhoI}\psi}^2
      + \rhoI^3\abs*{\NablaAngular\psi}^2
      + \rhoI^4\abs*{\psi}^2+\norm*{\psi}_{\compactSobolev{1}(\Manifold_{\nontrap}(\tau_1,\tau_2))}^2\\
    \gtrsim{}&    \norm*{\psi}_{\bSobolevH{1, \frac{\de_1}{2}}(\Manifold_{\nontrap}(\tau_1,\tau_2))}^2+\norm*{\psi}_{\bSobolevI{1, -\frac{5}{2}}(\Manifold_{\nontrap}(\tau_1,\tau_2))}^2+\norm*{\psi}_{\compactSobolev{1}(\Manifold_{\nontrap}(\tau_1,\tau_2))}^2,
  \end{align*}
  as stated. 
\end{proof}

\paragraph*{The choice in axial symmetry}

We first recall the following basic lemma from \cite{giorgiBoundednessDecayTeukolsky2023}. In this subsection $u$ and $v$ are not the Eddington-Finkelstein coordinates, but given functions of $r$.

\begin{lemma}[Proposition B.10 in \cite{giorgiBoundednessDecayTeukolsky2023}]
  \label{lemma:identity:prop.Morawetz1}
  For a multiplier triplet $(X_{ax}, w_{ax}, J_{ax})$ given by
  \begin{equation*}
    X_{ax}= \FF \partial^{\mathrm{BL}}_r, \qquad \quad
    \FF=z u,  \qquad \quad
    2 w_{ax} = z \pr_r u , \qquad \quad
    J_{ax}=v \partial^{\mathrm{BL}}_r, 
  \end{equation*}
  for some functions $(z, u, v) = (z(r), u(r), v(r))$, the current
  $\QQ^{(X_{ax}, w_{ax}, J_{ax})}[\psi]$ satisfies 
  \begin{align}
    \label{identity:prop.Morawetz1}
    \begin{split}
      |q|^2\QQ^{(X_{ax}, w_{ax}, J_{ax})}[\psi]
      ={}&z^{\frac{1}{2} }\Delta^{\frac{3}{2} }\AA |\pr^{\mathrm{BL}}_r\psi|^2 + \UU^{\a\b}(\pr_\a \psi )(\pr_\b \psi )+\VV |\psi|^2+|q|^2\div(J_{ax} |\psi|^2), 
    \end{split}
  \end{align}
  where 
  \begin{align*}
    \AA={}& \partial_r\left( \frac{ z^{\frac{1}{2} }  u }{\Delta^{\frac{1}{2} }}  \right),  \\
    \UU^{\a\b}={}& -  \frac{ 1}{2}  u \pr_r\left( \frac z \De\RR^{\a\b}\right), \\
    \VV={}&-\frac{1}{2} \pr_r \big(\De \pr_r w_{ax} \big)= -\frac{1}{4} \pr_r\big(\De \pr_r \big(
            z \pr_ru  \big)  \big) ,\\
    |q|^2\div(J_{ax} |\psi|^2)={}&   |q|^2\Big( 2 v\psi \pr^{\mathrm{BL}}_r \psi + \big(\pr_r v+ \frac{2r}{|q|^2} v\big) |\psi|^2 \Big).
  \end{align*}
  where $\RR^{\a\b}$ is defined in \eqref{eq:RR-def}. 
\end{lemma}

We apply \zcref{lemma:identity:prop.Morawetz1} with the
choice of $z(r)$ given by the geodesic potential.
\begin{corollary}
  \label{coro:eq:Morawetz:axisymmetry:general-bulk}
  For $(X_{ax}, w_{ax}, J_{ax})$ as described in
  \zcref{lemma:identity:prop.Morawetz1}, for a suitable choice of
  $z(r)$, we have for any $\varepsilon>0$,
  \begin{align}
    \label{eq:Morawetz:axisymmetry:general-bulk}
    \begin{split}
      |q|^2\QQ^{(X_{ax}, w_{ax}, J_{ax})}[\psi]
      ={}& 2\varepsilon\frac{(r-M)^4}{(r^2+a^2)}\AA[u] |\pr^{\mathrm{BL}}_r\psi|^2 + \frac{u\TT}{(r^2+a^2)^{3}}   |\NablaAngular \psi|^2 +I_{u,v, \varepsilon}[\psi]\\
         &-    u\frac{2ar}{(r^2+a^2)^2}\That \psi \pr_\phi \psi,
    \end{split}    
  \end{align}
  where
  \begin{align}
    \AA[u]\vcentcolon= \partial_r \Big(\frac{u}{r^2+a^2}\Big), \qquad 
    \VV=\VV[u] \vcentcolon= -\frac14\,\partial_r \Big((r-M)^2\,\partial_r \Big(\frac{(r-M)^2}{(r^2+a^2)^2}\,\partial_r u\Big)\Big),\label{eq:definition-VV}\\
    I_{u,v, \varepsilon}[\psi]\vcentcolon=  \frac{(r-M)^4}{(r^2+a^2)}(1-2\varepsilon)\,\AA[u]\,|\partial^{\mathrm{BL}}_r\psi|^2 +\VV[u]\,|\psi|^2  + |q|^2\Big(2v\,\psi\,\partial^{\mathrm{BL}}_r\psi  +(\partial_r v + \tfrac{2r}{|q|^2}v)\,|\psi|^2\Big).\label{eq:I_uv-axisym}
  \end{align}
\end{corollary}
\begin{proof}
  We make the same choice of $z$ as in \cite{giorgiPhysicalspaceEstimatesAxisymmetric2024}, i.e. 
  \begin{align*}
    z=\frac{(r-M)^2}{(r^2+a^2)^2},
  \end{align*}
  so that $\AA=\AA[u]$ and $\VV=\VV[u]$ as defined in \zcref{eq:definition-VV}. 
  Using \zcref{eq:RR-def}, we compute
  \begin{align*}
    \pr_r\left( \frac z \De\RR^{\a\b}\right)={}&2a \big(\frac{2r}{(r^2+a^2)^2} \big) \pr_t^{(\a} \pr_\phi^{\b)}+ a^2\big( \frac{4r}{(r^2+a^2)^3} \big) \pr_\phi^{\a} \pr_\phi^{\b} -\frac{2\TT}{(r^2+a^2)^{3}} O^{\a\b}.
  \end{align*}
  Therefore,
  \begin{align*}
    \UU^{\a\b}(\pr_\a \psi )(\pr_\b \psi )
    =     \frac{u\TT}{(r^2+a^2)^{3}}  |\NablaAngular \psi|^2
    -    u\frac{2ar}{(r^2+a^2)^2}\That \psi \pr_\phi \psi.
  \end{align*}
  Substituting this into \zcref{identity:prop.Morawetz1}, we deduce
  \begin{align*}
    |q|^2\QQ^{(X_{ax}, w_{ax}, J_{ax})}[\psi]={}& \frac{(r-M)^4}{(r^2+a^2)}\AA[u] |\pr^{\mathrm{BL}}_r\psi|^2 + \frac{u\TT}{(r^2+a^2)^{3}}   |\NablaAngular \psi|^2 +\VV[u] |\psi|^2\\
                                                &+ |q|^2\Big(2v\,\psi\,\partial^{\mathrm{BL}}_r\psi  +(\partial_r v + \tfrac{2r}{|q|^2}v)\,|\psi|^2\Big)-    u\frac{2ar}{(r^2+a^2)^2}\That \psi \pr_\phi \psi \\
    ={}& 2\varepsilon\frac{(r-M)^4}{(r^2+a^2)}\AA[u] |\pr^{\mathrm{BL}}_r\psi|^2 + \frac{u\TT}{(r^2+a^2)^{3}}   |\NablaAngular \psi|^2 +I_{u,v, \varepsilon}[\psi]-    u\frac{2ar}{(r^2+a^2)^2}\That \psi \pr_\phi \psi
  \end{align*}
  for any $\varepsilon>0$, as stated.
\end{proof}

We now construct functions $u$ and $v$ such that the first line on the right-hand side of
\zcref{eq:Morawetz:axisymmetry:general-bulk} is (mostly) positive. We follow
the same strategy as in
\cite{giorgiPhysicalspaceEstimatesAxisymmetric2024} in the extremal
Kerr case.

\begin{proposition}\label{prop:ax-symm-choice}
  On extremal Kerr--Newman spacetime, there exists
  \(\daxi_1=\daxi_1(M,a)>0\) such that, for every
  \(0<\daxi\le \daxi_1\), there exist regular functions
  \(u_\daxi(r)\) and \(v(r)\) such that, for all \(r\ge M\),
  \begin{align}
    \frac{u_{\daxi} \TT}{(r^2+a^2)^3}
    &\gtrsim
      \frac{r-M}{r^2}
      \left(1-\frac{r_{\mathrm{trap}}}{r}\right)^2,
      \label{eq:condition-one-u-daxi}
    \\
    \AA[u_\daxi]
    \vcentcolon=
    \partial_r \left(\frac{u_\daxi}{r^2+a^2}\right)
    &\gtrsim
      \,\frac{r}{(r^2+a^2)^2}.
      \label{eq:condition-A-u-daxi}
  \end{align}

  Moreover, for every sufficiently small \(\varepsilon>0\), independent
  of \(\daxi\), there exists \(c=c(M,a,\varepsilon)>0\), independent of
  \(\daxi\), such that, for every \(\psi\),
  \begin{equation}\label{eq:I-u-daxi-v}
    I_{u_\daxi,v,\varepsilon}[\psi]
    \ge
    \left(
      c\,\frac{(r-M)^2}{r^4}
      -
      4\,\overline{\VV}_\daxi\,
      \mathds{1}_{\{\rho_{1,\daxi}\le \rhoH\le \rho_{2,\daxi}\}}
    \right)|\psi|^2 .
  \end{equation}
  Here \(\overline{\VV}_\daxi\) is a nonnegative function supported in $\{\rho_{1,\daxi}\le \rhoH\le \rho_{2,\daxi}\}$
  with
  \[
    \rho_{1,\daxi}\simeq_{M,a}\daxi,
    \qquad
    \rho_{2,\daxi}\simeq_{M,a}\daxi,
    \qquad
    \rho_{2,\daxi}-\rho_{1,\daxi}\simeq_{M,a}\daxi,
  \]
  and
  \[
    0\le \overline{\VV}_\daxi \lesssim_{M,a} 1.
  \]
\end{proposition}
\begin{proof}
  See \zcref{sec:proof-Lemma-ax-symm}. We remark that no restriction on the
  smallness of $a$ is needed in \zcref{prop:ax-symm-choice}.
\end{proof}

Choosing \(u=u_\daxi\) and \(v(r)\) as in \zcref{prop:ax-symm-choice}, we immediately deduce from \zcref{coro:eq:Morawetz:axisymmetry:general-bulk} that 
\begin{equation}\label{eq:bulklowerbound-Xax-wax-Jax}
  \begin{split}
    |q|^2\KCurrent{(X_{ax}, w_{ax}, J_{ax})}[\psi]
    \geq{}& c_0 \Big[\frac{(r-M)^4}{r^5} |\pr^{\mathrm{BL}}_r\psi|^2 +\frac{(r-M)}{r^2} \big( 1-\frac{r_{\trap}}{r}\big)^2 |\NablaAngular \psi|^2 +\frac{(r-M)^2}{r^4}|\psi|^2 \Big]\\
          &- 4\,\overline{\VV}_\daxi\,\mathds{1}_{\{\rho_{1,\daxi} \leq \rhoH \leq \rho_{2,\daxi}\}}\,|\psi|^2-    \frac{2aru_\daxi}{(r^2+a^2)^2}\That \psi \pr_\phi \psi,            
  \end{split}
\end{equation}
for some $c_0>0$.

\begin{remark}
  From the construction in \zcref{sec:proof-Lemma-ax-symm} we deduce in particular that 
  \begin{align*}
    X_{ax}
    =
    \frac{(r-M)^2}{(r^2+a^2)^2}\Bigl(-\frac{2M}{\varepsilon_0}\Bigr)\partial_r^{\mathrm{BL}},
    \qquad
    w_{ax}=0,
    \qquad
    J_{ax}
    =
    \frac{C\delta_1(r-M)^3}{r^2+M^2}\partial_r^{\mathrm{BL}}, \qquad \text{near $\EventHorizon$}, \\
    X_{ax}
    =
    \frac{(r-M)^2}{(r^2+a^2)^2}(r^2+C_2)\partial_r^{\mathrm{BL}}
    , \qquad 
    w_{ax}
    =
    \frac{r(r-M)^2}{(r^2+a^2)^2}
    \qquad
    J_{ax}=0, \qquad \text{near $\NullInfinity$}
  \end{align*}
  and therefore, using \zcref{eq:relation-BL-iEF},
  \begin{align}
    X_{ax} &\in \conormalSpaceH{1}(\Manifold)\rhoH\partial_{\rhoH}
             + \conormalSpaceH{0}(\Manifold)\HawkingHorizon,
             \qquad
             J_{ax}\in \conormalSpaceH{2}(\Manifold)\rhoH\partial_{\rhoH}
             + \conormalSpaceH{1}(\Manifold)\HawkingHorizon, \label{eq:axi-multipliers:boundary-behavior:horizon}\\
    X_{ax}
           &\in
             \conormalSpaceI{1}(\Manifold)\rhoI\partial_{\rhoI}
             +
             \conormalSpaceI{0}(\Manifold)\InfinityHawking, \qquad 
             w_{ax}
             \in \conormalSpaceI{1}(\Manifold).\label{eq:axi-multipliers:boundary-behavior:null-infinity}
  \end{align}
\end{remark}

\paragraph*{The degenerate redshift estimate}\label{sec:deg-redshift}

We now use the horizon weighted hierarchy from 
\zcref{sec:hierarchy-HH} to construct a multiplier which yields a
residual degenerate redshift near the horizon. 

\begin{lemma}
  \label{prop:extremal-redshift}
  Fix $0<\de_1<\frac{1}{2}$ and assume $ \frac{|a|}{M}<\frac{1}{30}$. Then the current associated to
  \begin{align}\label{eq:def-XHH-JHH}
    X_{red}\vcentcolon= \chi_{red} X_{\a, C}, \qquad J_{red}=\chi_{red} J_{\a}, \qquad \chi_{red}\vcentcolon=\chi\Big(\frac{\rhoH}{\frac{r_{+}}{80}}\Big), \qquad \a=-\de_1, \qquad C=800,
  \end{align}
  with $X_{\a, C}$ and $J_\a$ as in \zcref{lemma:horizon:main-bulk}, satisfies 
  \begin{align}\label{eq:bulklowerbound-XHH-JHH}
    \begin{split}
      |q|^2 \QQ^{(X_{red},0,J_{red})}[\psi]
      \ges{}&\mathds{1}_{\{\rhoH < \frac{r_{+}}{80}\}} \rhoH^{-\de_1}  \Big(  |\pr_v\psi|^2+|\NablaAngular\psi|^2 +|\rhoH\pr_{\rhoH}\psi|^2+
              |\psi|^2\Big)\\
            &- \mathds{1}_{\{\frac{r_{+}}{80} \leq \rhoH \leq \frac{r_{+}}{40}\}}\abs*{(\partial_v,\partial_{\rhoH},\NablaAngular)^{\le 1}\psi}^2 .
    \end{split}
  \end{align}
\end{lemma}
\begin{proof}
  Consider the current associated to the vectorfield $X_{\a, C}$. From \zcref{lemma:horizon:main-bulk} \zcref{eq:QQ-X-a-C-general}, 
  we can write by completing the square, 
  \begin{align*}
    \abs*{q}^2\QQ^{(X_{\alpha, C},0,0)}[\psi]
    ={}&\frac{3(1-\a)}{8}\rhoH^\a |\rhoH\pr_{\rhoH}\psi|^2 +\frac{1+\a}{2}\rhoH^\a |\NablaAngular \psi|^2 +\frac{1+\a}{2}C r_+^{-2}\rhoH^\a (\HawkingHorizon\psi)^2\\
       & +\frac{1-\a}{8}\rhoH^\a (\rhoH\pr_{\rhoH}\psi)^2+2r\rhoH^{\alpha} (1+C r_+^{-2}\rhoH^2)\,\pr_v\psi\,\rhoH\pr_{\rhoH}\psi\\
       & +\frac{1+\a}{4}C r_+^{-2}\rhoH^\a (\HawkingHorizon\psi)^2 +2C r_+^{-2} r\rhoH^{\alpha+1} (\HawkingHorizon\psi)\pr_v\psi\\
       & +\frac{1+\a}{4}C r_+^{-2}\rhoH^\a (\HawkingHorizon\psi)^2  + C r_+^{-2}(1+\a)\rhoH^{\alpha+1}(\HawkingHorizon\psi)\rhoH\pr_{\rhoH}\psi\\
    \ge {}&\frac{3(1-\a)}{8}\rhoH^\a |\rhoH\pr_{\rhoH}\psi|^2 +\frac{1+\a}{2}\rhoH^\a |\partial_\th\psi|^2+\frac{1+\a}{2}\rhoH^\a |\renormpphi\psi|^2+\frac{1+\a}{2}C r_+^{-2}\rhoH^\a (\HawkingHorizon\psi)^2\\
       &-\frac{8r^2(1+C r_+^{-2}\rhoH^2)^2}{1-\a}\rhoH^\a|\pr_v\psi|^2-C r_+^{-2}\frac{4r^2}{1+\a}\rhoH^{\a+2}|\pr_v\psi|^2-C r_+^{-2}(1+\a)\rhoH^{\a+2}|\rhoH\pr_{\rhoH}\psi|^2.
  \end{align*}
  Using the definition \zcref{eq:def-VHH-VII} in \iEF{} coordinates, we complete the square
  \begin{align*}
    A |\renormpphi\psi|^2
    + B |\HawkingHorizon\psi|^2 ={}& A\Big|a\sin\theta\,\pr_v\psi+\frac{1}{\sin\theta}\pr_{\phiIEF}\psi\Big|^2
                                     +B \Big|(r^2+a^2)\pr_v\psi+a\,\pr_{\phiIEF}\psi\Big|^2\\
    ={}& (A a^2\sin^2\theta+B(r^2+a^2)^2)|\pr_v\psi|^2
         + \Big(\frac{A}{\sin^2\theta}+Ba^2\Big)|\pr_{\phiIEF}\psi|^2\\
                                   &+2\Big(Aa+B a(r^2+a^2)\Big)(\pr_v\psi)(\pr_{\phiIEF}\psi) \\
    \geq{}&  (A a^2\sin^2\theta+B(r^2+a^2)^2)|\pr_v\psi|^2
            -\frac{(Aa+Ba(r^2+a^2))^2}{A/\sin^2\theta+Ba^2}|\pr_v\psi|^2\\
    ={}&A \frac{A|q|^4}{A/B+a^2\sin^2\theta}|\pr_v\psi|^2.
  \end{align*}
  Taking $A=\frac{1+\a}{4}\rhoH^\a$ and $B=\frac{1+\a}{4}C r_+^{-2}\rhoH^\a$, we obtain
  \begin{align*}
    \frac{1+\a}{4}\rhoH^\a |\renormpphi \psi|^2 + \frac{1+\a}{4}C r_+^{-2}\rhoH^\a |\HawkingHorizon\psi|^2\ge \frac{1+\a}{4}\frac{C|q|^4}{r_+^2+Ca^2\sin^2\th} \rhoH^\a |\pr_v\psi|^2.
  \end{align*}
  Combining this with the above inequality yields 
  \begin{equation*}
    \begin{split}
      \abs*{q}^2\QQ^{(X_{\alpha, C},0,0)}[\psi]
      \geq{}&  \frac{1-\a}{4}\rhoH^\a |\rhoH\pr_{\rhoH}\psi|^2+\frac{1+\a}{4}\rhoH^\a |\NablaAngular \psi|^2 +\frac{(1+\a)C}{4r_+^2}\rhoH^\a |\HawkingHorizon\psi|^2\\
            & +A_1^{(\a,C)}\rhoH^\a |\rhoH\pr_{\rhoH}\psi|^2+ A_2^{(\a,C)} \rhoH^\a |\pr_v\psi|^2,
    \end{split}    
  \end{equation*}
  where
  \begin{equation*}
    \begin{split}
      A_1^{(\a,C)}\vcentcolon={}&\frac{1-\a}{8}-(1+\a)Cr_+^{-2}\rhoH^2,\\
      A_2^{(\a,C)}\vcentcolon={}&\frac{1+\a}{4}\frac{C|q|^4}{r_+^2+Ca^2\sin^2\th} - \frac{8r^2(1+C r_+^{-2}\rhoH^2)^2}{1-\a}-C r_+^{-2} \frac{4r^2}{1+\a}\rhoH^2.
    \end{split}
  \end{equation*}
  Now consider the case of $\a=-\de_1$, $C=800$. Then
  \begin{equation*}
    \begin{split}
      \abs*{q}^2 \QQ^{(X_{-\de_1, 800},0,0)}[\psi]
      \geq{}&  \Big[\frac{1+\de_1}{4}+A_1^{(-\de_1,800)}\Big]\rhoH^{-\de_1} (\rhoH\pr_{\rhoH}\psi)^2
              + A_2^{(-\de_1,800)} \rhoH^{-\de_1} (\pr_v\psi)^2\\
            &+\frac{1-\de_1}{4}\rhoH^{-\de_1} |\NablaAngular \psi|^2
              +\frac{200(1-\de_1)}{r_+^2}\rhoH^{-\de_1} (\HawkingHorizon\psi)^2.
    \end{split}
  \end{equation*}
  We now show that 
  \[
    A_1^{(-\de_1,800)}>0,
    \qquad
    A_2^{(-\de_1,800)}>0.
  \]
  in the region $\{0\leq \rhoH\le \frac{r_+}{80}\}$.
  We compute
  \[A_1^{(-\de_1,800)}=\frac{1+\de_1}{8}-800(1-\de_1)r_+^{-2}\rhoH^2
    \geq \frac{1+\de_1}{8}-800(1-\de_1)\frac{1}{80^2}
    = \frac{\de_1}{4}>0.\]
  Since $Ca^2\le \frac{800}{900}r_+^2<r_+^2$ and $\frac{C\rhoH^2}{r_+^2}\le \frac{800}{6400}=\frac18$,
  we have
  \begin{align*}
    A_2^{(-\de_1,800)}={}&\frac{1-\de_1}{4}\frac{C|q|^4}{r_+^2+Ca^2\sin^2\th}
                           -\frac{8r^2(1+C\rhoH^2/r_+^2)^2}{1+\de_1}
                           -\frac{4Cr^2}{1-\de_1}\rhoH^2/r_+^2\\
    >{}&\frac{100r^4}{r_+^2+Ca^2}
         -8r^2(1+C\rhoH^2/r_+^2)^2
         -8r^2C\rhoH^2/r_+^2\\
    >{}&50r^2-8r^2(1+1/8)^2-r^2>0.
  \end{align*}
  We therefore deduce that in the region $\{0\leq \rhoH\leq \frac{r_+}{80}\}$, we have
  \begin{equation*}
    \begin{split}
      \abs*{q}^2 \QQ^{(X_{-\de_1, 800},0,0)}[\psi] 
      \geq{}&  \frac{1}{8}\rhoH^{-\de_1}\Big((\rhoH\pr_{\rhoH}\psi)^2+\abs*{\NablaAngular\psi}^2+r_+^{-2}|\HawkingHorizon\psi|^2\Big).
    \end{split}
  \end{equation*}
  Using \zcref{lemma:hardy-pointwise-rhoH}, we have for $J_{-\de_1}=\frac 1 2 c_J(1-\de_1)\rhoH^{1-\de_1}|q|^{-2}\partial_{\rhoH}$ for some $c_J>0$ that
  \begin{align*}
    |q|^2 \D \c \PP^{(0, 0, J_{-\de_1})}[\psi]\geq -c_J\rhoH^{-\de_1}|\rhoH\pr_{\rhoH} \psi|^2+
    \frac14c_J(1-\de_1)^2 \rhoH^{-\de_1} |\psi|^2.
  \end{align*}
  Adding those two bounds together we obtain for $c_J < \frac 1 8$
  that in the region $\{0\leq \rhoH\leq \frac{r_+}{80}\}$
  \begin{equation*}
    \begin{split}
      \abs*{q}^2\QQ^{(X_{-\de_1, 800},0, J_{-\de_1})}[\psi]
      \ges{}&  \rhoH^{-\de_1}  \Big(  |\pr_v\psi|^2+|\NablaAngular\psi|^2 +|\rhoH\pr_{\rhoH}\psi|^2+
              |\psi|^2\Big).
    \end{split}    
  \end{equation*}
  Adding the cutoff $\chi_{red}$ using \zcref{current-cutoff} we then
  obtain \zcref{eq:bulklowerbound-XHH-JHH}, completing the proof.
\end{proof}

\paragraph*{Proof of Proposition \ref{prop:Mor-axially-sym:main-bulk}}

We are now ready to prove \zcref{prop:Mor-axially-sym:main-bulk}.
\begin{proof}[Proof of \zcref{prop:Mor-axially-sym:main-bulk}]

  First of all, we use a small Lagrangian corrector to add control of the $\That$ derivatives outside the
  trapping region in the bound \zcref{eq:bulklowerbound-Xax-wax-Jax}.
  The multiplier triplet $(X,w,J)=(0,w_{\That},0)$ with
  \begin{equation}\label{eq:def-wThat}
    2 w_{\That}=-\frac{(r-M)^2(r-r_{\trap})^2}{r^7},
  \end{equation}
  using \zcref{eq:inverse-metric-BL}  gives 
  \begin{align}\label{eq:bulklowerbound-wThat}
    \begin{split}
      |q|^2 \QQ^{(0, w_{\That}, 0)}[\psi]
      =&  \De  \, w_{\That}     |\pr^{\mathrm{BL}}_r\psi|^2-\frac{w_{\That} (r^2+a^2)^2}{  \De}|\That \psi|^2 + w_{\That} O^{\a\b}\pr_\a\psi\pr_\b \psi -\frac{1}{2} |q|^2\square_\g w_{\That} |\psi|^2 \\
      =&\frac{(r-r_{\trap})^2(r^2+a^2)^2}{2r^7}|\That \psi|^2-\frac{1}{2}  \, \frac{(r-M)^4(r-r_{\trap})^2}{r^7}     |\pr^{\mathrm{BL}}_r\psi|^2 \\
       &-\frac{1}{2}  \frac{(r-M)^2(r-r_{\trap})^2}{r^7} \abs*{\NablaAngular\psi}^2-\frac{1}{2} |q|^2\square_\g w_{\That} |\psi|^2. 
    \end{split}
  \end{align}
  We now combine the bulks of the axially symmetric Morawetz, the
  degenerate redshift and the above Lagrangian. In particular, for
  small positive constants $c_{red}$ and $c_{\That}$, define
  \begin{align*}
    X_{\Mor}^{\diff}\vcentcolon= X_{ax}+c_{red}X_{red}, \qquad w_{\Mor}^{\diff}\vcentcolon= w_{ax} +c_{\That} w_{\That}, \qquad J_{\Mor}^{\diff}=J_{ax} + c_{red}J_{red}.
  \end{align*}
  By summing \zcref{eq:bulklowerbound-Xax-wax-Jax,eq:bulklowerbound-XHH-JHH,eq:bulklowerbound-wThat}, we obtain
  \begin{equation}
    \label{eq:degenerate-redshift:final-big-computation}
    \begin{split}
      |q|^2\QQ^{(X_{\Mor}^\diff, w_{\Mor}^\diff, J_{\Mor}^\diff)}[\psi]
      \geq{}& c_0 \Big[\frac{\rhoH^4}{r^5} |\pr^{\mathrm{BL}}_r\psi|^2 +\big( 1-\frac{r_{\trap}}{r}\big)^2\big( \frac{\rhoH}{r^2}\abs*{\NablaAngular\psi}^2 + \frac{1}{r} |\That \psi|^2\big) +\frac{\rhoH^2}{r^4}|\psi|^2 \Big]\\
            &+c_{red}\mathds{1}_{\{\rhoH \leq \frac{r_{+}}{80} \}} \rhoH^{-\de_1}\Big( |\rhoH\pr_{\rhoH}\psi|^2+ \abs*{\NablaAngular\psi}^2 +|\partial_v\psi|^2+|\psi|^2\Big)\\
            &- 4\,\overline{\VV}_\daxi\,\mathds{1}_{\{\rho_{1,\daxi} \leq \rhoH \leq \rho_{2,\daxi}\}}\,|\psi|^2-    \frac{2aru_\daxi}{(r^2+a^2)^2}\That \psi \pr_\phi \psi\\
            &-O(c_{red}) \mathds{1}_{\{\frac{r_{+}}{80}  \leq \rhoH \leq \frac{r_{+}}{40} \}}\abs*{(\partial_v,\partial_{\rhoH},\NablaAngular)^{\le 1}\psi}^2\\
            &-c_{\That} \Big[\frac{1}{2}  \, \frac{\rhoH^4(r-r_{\trap})^2}{r^7}     |\pr^{\mathrm{BL}}_r\psi|^2 +\frac{1}{2}  \frac{\rhoH^2(r-r_{\trap})^2}{r^7} \abs*{\NablaAngular\psi}^2+\frac{1}{2} |q|^2\square_\g w_{\That} |\psi|^2\Big].
    \end{split}
  \end{equation}
  We now choose the parameters so that the
  first two lines of the \RHS{} of
  \zcref{eq:degenerate-redshift:final-big-computation}, which are
  positive, dominate over the remaining terms. 

  First, take $c_{\That}\ll 1$ small enough so that
  \begin{equation}
    \label{eq:degenerate-redshift:final-big-computation:aux1}
    \begin{split}
      & c_{\That} \Big[\frac{1}{2}  \, \frac{\rhoH^4(r-r_{\trap})^2}{r^7}     |\pr^{\mathrm{BL}}_r\psi|^2 +\frac{1}{2}  \frac{\rhoH^2(r-r_{\trap})^2}{r^7} \abs*{\NablaAngular\psi}^2+\frac{1}{2} |q|^2\square_\g w_{\That} |\psi|^2\Big]\\
      \ll{}& c_0 \Big[\frac{\rhoH^4}{r^5} |\pr^{\mathrm{BL}}_r\psi|^2 +\big( 1-\frac{r_{\trap}}{r}\big)^2\big( \frac{\rhoH}{r^2}\abs*{\NablaAngular\psi}^2 + \frac{1}{r} |\That \psi|^2\big) +\frac{\rhoH^2}{r^4}|\psi|^2 \Big],
    \end{split}    
  \end{equation}
  making use of the fact that 
  $\square_\g w_{\That}\in \conormalSpaceHI{2,3}$.

  Next, take $c_{red} \ll 1$ small enough so that 
  \begin{equation}
    \label{eq:degenerate-redshift:final-big-computation:aux2}
    \begin{split}
      &O(c_{red}) \mathds{1}_{\{\frac{r_{+}}{80}  \leq \rhoH \leq \frac{r_{+}}{40} \}}\abs*{(\partial_v,\partial_{\rhoH},\NablaAngular)^{\le 1}\psi}^2\\
      \ll{}& c_0 \Big[\frac{\rhoH^4}{r^5} |\pr^{\mathrm{BL}}_r\psi|^2 +\big( 1-\frac{r_{\trap}}{r}\big)^2\big( \frac{\rhoH}{r^2}\abs*{\NablaAngular\psi}^2 + \frac{1}{r} |\That \psi|^2\big) +\frac{\rhoH^2}{r^4}|\psi|^2 \Big].
    \end{split}    
  \end{equation}

  Finally, for any $\de_1\ll1$, take $\daxi\ll 1$ small enough so that $\rho_{2,\daxi}\leq \frac{r_{+}}{80}$
  and
  \begin{equation}
    \label{eq:degenerate-redshift:final-big-computation:aux3}
    4\,\overline{\VV}_\daxi\,\mathds{1}_{\{\rho_{1,\daxi} \leq \rhoH \leq \rho_{2,\daxi}\}}\,|\psi|^2\ll c_{red}\mathds{1}_{\{\rhoH \leq \frac{M}{160} \}}\rhoH^{-\de_1}|\psi|^2.
  \end{equation}
  Combining
  \zcref{eq:degenerate-redshift:final-big-computation,eq:degenerate-redshift:final-big-computation:aux1,eq:degenerate-redshift:final-big-computation:aux2,eq:degenerate-redshift:final-big-computation:aux3},
  we obtain that
  \begin{align*}
    |q|^2\QQ^{(X_{\Mor}^\diff, w_{\Mor}^\diff, J_{\Mor}^\diff)}[\psi]
    \geq{}& c \Big[\frac{\rhoH^4}{r^5} |\pr^{\mathrm{BL}}_r\psi|^2 +\big( 1-\frac{r_{\trap}}{r}\big)^2\big( \frac{\rhoH}{r^2}\abs*{\NablaAngular\psi}^2 + \frac{1}{r} |\That \psi|^2\big) +\frac{\rhoH^2}{r^4}|\psi|^2 \\
          &+\mathds{1}_{\{\rhoH < \frac{M}{80} \}} \rhoH^{-\de_1}\Big( |\rhoH\pr_{\rhoH}\psi|^2+ \abs*{\NablaAngular\psi}^2 +|\partial_v\psi|^2+|\psi|^2\Big)\Big]\\
          &-    \frac{2aru_\daxi}{(r^2+a^2)^2}\That \psi \pr_\phi \psi.
  \end{align*}
  By writing explicitly the coefficients of the derivatives, we conclude the proof of \zcref{prop:Mor-axially-sym:main-bulk}.
\end{proof}

\subsubsection{Non-negativity of the bulk near trapping}
\label{sec:Morawetz:principal-order-bulk}

We now construct a pseudodifferential correction of
\((X_{\Mor}^\diff,w_{\Mor}^\diff,J_{\Mor}^\diff)\) which restores
positivity of the Morawetz bulk near the trapped set. This correction is
needed because, for $a\neq0$, the trapped set depends on the temporal
and azimuthal frequencies. The construction follows the strategy of
\cite{tataruLocalEnergyEstimate2011} for slowly rotating Kerr
spacetimes.

\begin{lemma}
  \label{lemma:ILED-KdS:Bulk}
  There exist
  \begin{enumerate}
  \item a pseudo-differential operator
    $\widetilde{\MorawetzVF}\in \MixedOpClass{1}{1}(\mathcal{M})$ such that
    \begin{equation}
      \label{eq:ILED-KdS:X-pert-def}
      \widetilde{\MorawetzVF} = \widetilde{\MorawetzVF}_0\partial_t +  \widetilde{\MorawetzVF}_1,
    \end{equation}
    where $\widetilde{\MorawetzVF}_i\in \TanOpClass{i}(\mathcal{M})$
    is an anti-symmetric operator with symbol
    $\widetilde{\MorawetzSym}_i$, 
  \item a pseudo-differential operator
    $\widetilde{\MorawetzLagrangeCorr}\in
    \MixedOpClass{0}{1}(\mathcal{M})$
    such that
    \begin{equation}
      \label{eq:ILED-KdS:q-pert-def}
      \widetilde{\MorawetzLagrangeCorr}
      = \widetilde{\MorawetzLagrangeCorr}_0
      + \widetilde{\MorawetzLagrangeCorr}_{-1}\partial_{t},
    \end{equation}
    where
    $\widetilde{\MorawetzLagrangeCorr}_i\in
    \TanOpClass{i}(\mathcal{M})$ is a self-adjoint
    pseudo-differential operator with symbol
    $\widetilde{\MorawetzLagrangeCorrSym}_i$,
  \end{enumerate}
  such that
  \begin{equation}
    \label{eq:Morawetz:KdS:bulk:square-sum-decomposition}
    \abs*{q}^2\left(
      \frac{1}{2\ImagUnit}H_p\MorawetzSym + \PrinSymb \MorawetzLagrangeCorrSym
    \right)
    = \sum_{j=1}^{7}\SquareDecomp_j^2,
  \end{equation}
  where $H_p$ denotes the Hamiltonian vectorfield associated to $p$ the
  principal symbol of $\Box_{\Metric}$, and
  $\SquareDecomp_j\in \MixedSymClass{1}{1}(T^{*}\mathcal{M})$ are
  principally scalar, and
  \begin{gather*}
    \MorawetzSym = \MorawetzSym_{\Mor} + a\widetilde{\MorawetzSym}, \qquad
    \MorawetzSym_{\Mor} = \ImagUnit \mathcal{F}(r)\xi, \\
    \MorawetzLagrangeCorrSym = \MorawetzLagrangeCorrSym_{\Mor} + a\widetilde{\MorawetzLagrangeCorrSym},\qquad
    \MorawetzLagrangeCorrSym_{\Mor} = \MorawetzLagrangeCorr_{\Mor}^\diff - \frac{1}{2}\CovariantDeriv_{\Metric}\cdot \MorawetzVF_{\Mor}^\diff,
  \end{gather*}
  where $\MorawetzSym_{Mor}$ agrees with the principal symbol of $X_{ax}$
  in $\abs*{r-\rTrapping} < \delta_{\trap}$.
  
  Moreover, the decomposition in
  \zcref{eq:Morawetz:KdS:bulk:square-sum-decomposition} extends the
  decomposition in
  \zcref{eq:Mor-axially-sym:multiplier-choices:square:expansion} in the
  sense that if $a=0$, then
  $\widetilde{\MorawetzLagrangeCorrSym} = \widetilde{\MorawetzSym} = 0$,
  and
  \begin{equation*}    
    \sum_{j=1}^7\SquareDecomp_j^2 = \mathfrak{k}^{\alpha\beta}\zeta_{\alpha}\zeta_{\beta},
  \end{equation*}
  where $\zeta = \{\sigma,\xi,\eta\}\in T^{*}\Manifold$, and
  $\mathfrak{k}^{\alpha\beta}\partial_{\alpha}\psi\partial_{\beta}\psi
  =
  \KCurrent{X_{\Mor}^\diff, w_{\Mor}^\diff,J_{\Mor}^\diff}[\psi] + O(\abs*{\psi}^2)$.
\end{lemma}
\begin{remark}
  Observe that the specific forms of
  $\widetilde{\MorawetzVF}$, $\widetilde{\MorawetzLagrangeCorr}$ in
  \zcref{eq:ILED-KdS:X-pert-def} and \zcref{eq:ILED-KdS:q-pert-def}
  are not unique due to their pseudodifferential nature.
\end{remark}
\begin{remark}
  The definitions of $\MorawetzSym_{\Mor}$,
  $\MorawetzLagrangeCorrSym_{\Mor}$ are chosen so that
  \begin{equation*}
    \WeylQ{\MorawetzSym_{\Mor} + \MorawetzLagrangeCorrSym_{\Mor}} = \MorawetzVF_{\Mor}^\diff + \MorawetzLagrangeCorr_{\Mor}^\diff. 
  \end{equation*}
\end{remark}

\begin{proof}[Proof of \zcref{lemma:ILED-KdS:Bulk}]

  For the sake of simplifying some of our ensuing calculations, define
  \begin{equation*}
    \MorawetzLagrangeCorrSym'_{\Mor} = \MorawetzLagrangeCorrSym_{\Mor}
    - 2\PoissonB{\ln q, \MorawetzSym_{\Mor}}, \qquad
    \tilde{\MorawetzLagrangeCorrSym}'_{\Mor} =
    \tilde{\MorawetzLagrangeCorrSym}_{\Mor} -2\PoissonB{\ln q, \tilde{\MorawetzSym}}, 
  \end{equation*}
  so that
  \begin{equation*}
    \abs*{q}^2\left(
      \frac{1}{2\ImagUnit}H_\PrinSymb
      (\MorawetzSym_{\Mor}+a\tilde{\MorawetzSym})
      +(\MorawetzLagrangeCorrSym_{\Mor}+a\tilde{\MorawetzLagrangeCorrSym})\PrinSymb  
    \right)
    = \frac{1}{2\ImagUnit}H_{\abs*{q}^2\PrinSymb} (\MorawetzSym_{\Mor}+a\tilde{\MorawetzSym}) +
    \left(\tilde{\MorawetzLagrangeCorrSym}'_{\Mor}
      +a\tilde{\MorawetzLagrangeCorrSym}'\right)(\abs*{q}^2\PrinSymb).
  \end{equation*}
  We first choose $\tilde{\MorawetzSym}$ so that
  $H_{\RescaledPrinSymb} (\MorawetzSym_{\Mor}+a\tilde{\MorawetzSym})$
  vanishes at $\{r=\rTrapping(\sigma_i, \FreqPhi), \sigma=\sigma_i\}$.
  The most immediate choice for this is the symbol
  \begin{equation*}
    \mathfrak{x}'\vcentcolon=\ImagUnit \Delta\abs*{q}^{-2}\left(r-\rTrapping(\sigma,\FreqPhi)\right)\xi
    =\ImagUnit\frac{r-\rTrapping(\sigma,\FreqPhi)}{2\abs*{q}^2}H_{\abs*{q}^2\PrinSymb}r.
  \end{equation*}
  This symbol is clearly well defined and smooth in a neighborhood of
  the trapped set. Moreover, we can calculate that on the
  characteristic set $\PrinSymb=0$, we have 
  \begin{equation*}
    2 H_{\RescaledPrinSymb}\MorawetzSym' =
    \left(\frac{1}{\abs*{q}^2} - \frac{2(r-\rTrapping)\partial_r\abs*{q}}{\abs*{q}^3}\right)(H_{\RescaledPrinSymb} r)^2
    + \frac{r-\rTrapping(\sigma,\FreqPhi)}{\abs*{q}^2}H_{\RescaledPrinSymb}^2r.
  \end{equation*}
  We can compute that for
  $\PrinSymb=0$, we have that
  \begin{equation*}
    H_{\RescaledPrinSymb}^2r
    = 2\Delta\partial_r\left(\frac{(1+\gamma)^2}{\Delta}\left((r^2+a^2)\sigma +a\FreqPhi\right)\right).
  \end{equation*}
  Since $\rTrapping$ is the unique minimum of
  $\Delta^{-1}\left((r^2+a^2)\sigma
    +a\FreqPhi\right)$, and we are in a $\delta_{\operatorname{trap}}$ neighborhood of
  $r=2M$, there exist positive homogeneous symbols
  $\alpha,\beta\in \Psi^0_{\hom}(r,\sigma,\eta_{\varphi})$ such that
  on $\PrinSymb=0$, near $r=2M$,
  \begin{equation}
    \label{eq:ILED-near:sum-of-squares:KdS-ILED-sym-characteristic}
    H_{\RescaledPrinSymb}\MorawetzSym' =
    \alpha^2(r,\sigma, \FreqPhi) (r-\rTrapping)^2
    + \beta^2(r,\sigma, \FreqPhi)\xi^2.
  \end{equation}  
  Unfortunately, the problem with $\MorawetzSym'$ is that it is not a
  polynomial in $\sigma$ unless $a=0$. Thus, for $a\neq 0$,
  $\MorawetzSym'$ cannot be directly used in conjunction with our
  integration-by-parts or divergence theorem method to produce a
  Morawetz estimate. We will overcome this difficulty with the aid of
  the Malgrange preparation theorem.
  Observe that we defined $\MorawetzSym$ so that it is smooth in $a$,
  and so that
  \begin{equation*}
    \MorawetzSym'-\MorawetzSym_{\Mor} \in a\MixedSymClass{2}{0}(T^{*}\Manifold).
  \end{equation*}
  Thus the Malgrange preparation theorem gives us the existence of
  homogeneous
  $\tilde{\MorawetzSym}_i\in \TanSymClass{i}(T^{*}\mathcal{M})$, $i=0,1$
  and homogeneous  $\rAux \in S^{-1}(T^{*}\mathcal{M})$
  such that
  \begin{equation*}
    \frac{1}{\ImagUnit}\left(\MorawetzSym' - \MorawetzSym_{\Mor}\right)
    = a\left(\tilde{\MorawetzSym}_1
      + \tilde{\MorawetzSym}_0\sigma\right)
    + a \rAux\PrinSymb.
  \end{equation*}
  Now, we define
  \begin{equation*}
    \frac{1}{\ImagUnit}\tilde{\MorawetzSym} = \tilde{\MorawetzSym}_1+\tilde{\MorawetzSym}_0\sigma,
  \end{equation*}
  so that on $\PrinSymb=0$,
  \begin{equation*}
    \MorawetzSym = \MorawetzSym_{\Mor} + a\tilde{\MorawetzSym} = \MorawetzSym'.
  \end{equation*}
  Thus $\MorawetzSym$ is a symbol which is a polynomial in $\sigma$
  and moreover vanishes at
  $\{r = \rTrapping(\sigma_i, \FreqPhi), \sigma=\sigma_i\}$.
  
  \textit{A priori}, $H_{\RescaledPrinSymb} \tilde{\MorawetzSym}$ is a
  third degree polynomial in $\sigma$. Applying the Malgrange
  preparation theorem again yields that there exist some
  $\gamma_1\in \TanSymClass{1}(T^{*}\mathcal{M}), \gamma_2 \in
  \TanSymClass{2}(T^{*}\mathcal{M})$,
  $f_0\in \TanSymClass{0}(T^{*}\mathcal{M}), f_{-1}\in \sigma
  \TanSymClass{-1}(T^{*}\mathcal{M})$ such that
  \begin{equation*}
    \frac{1}{2\ImagUnit\abs*{q}^2} H_{\RescaledPrinSymb}(\MorawetzSym_{\Mor}+a\tilde{\MorawetzSym})
    + \MorawetzLagrangeCorrSym_{\Mor}' (\RescaledPrinSymb)
    = \gamma_2+\gamma_1\sigma  + \left(
      e_{\Mor} + a(f_0+f_{-1}\sigma)\right)(\sigma-\sigma_1)(\sigma-\sigma_2),
  \end{equation*}
  observing that
  \begin{equation*}
    e_{\Mor} \vcentcolon= c_1\alpha_{\Mor}^2
  \end{equation*}
  is the coefficient for $\sigma^2$ in the expression for
  $\frac{1}{2\ImagUnit}H_{\PrinSymb}\MorawetzSym_{\Mor} +
  \MorawetzLagrangeCorrSym_{\Mor}\PrinSymb$ (see \zcref{eq:Mor-axially-sym:multiplier-choices:square:expansion}).                 
  It now remains to demonstrate that $\gamma_2+\gamma_1\sigma
  +e_{\Mor}(\sigma-\sigma_1)(\sigma-\sigma_2)$ can be expressed as a
  sum of squares up to some error in
  $a\PrinSymb\MixedSymClass{0}{1}(T^{*}\mathcal{M})$. If this were true, we could write
  \begin{equation}
    \label{eq:ILED-near:sum-of-squares:aux-1}
    \gamma_2+\gamma_1\sigma + e_{\Mor}(\sigma-\sigma_1)(\sigma-\sigma_2)
    = \sum \SquareDecomp_j^2 + a(g_0+g_{-1}\sigma)(\sigma-\sigma_1)(\sigma-\sigma_2). 
  \end{equation}
  We could then define $\tilde{\MorawetzLagrangeCorrSym}$ such that
  \begin{equation*}
    \tilde{\MorawetzLagrangeCorrSym}' = -2\left(f_0+g_0+(f_{-1}+g_{-1})\sigma\right),
  \end{equation*}
  so that the
  $a\PrinSymb\MixedSymClass{0}{1}(T^{*}\mathcal{M})$ terms are all
  canceled.

  We now return to showing
  \zcref{eq:ILED-near:sum-of-squares:aux-1}. Recall that on
  $\PrinSymb=0$,
  \begin{equation*}
    H_{\RescaledPrinSymb}(\MorawetzSym_{\Mor} + a\tilde{\MorawetzSym}) = H_{\RescaledPrinSymb}\MorawetzSym'.
  \end{equation*}
  As a result of
  \zcref{eq:ILED-near:sum-of-squares:KdS-ILED-sym-characteristic}, we
  now have that if $\sigma=\sigma_i$, which in particular implies that for $\PrinSymb=0$,
  \begin{equation*}
    \gamma_2+\gamma_1\sigma =\alpha^2(r, \sigma, \FreqPhi)(r-\rTrapping)^2 + \beta^2(r,\sigma,\FreqPhi)\xi^2.
  \end{equation*}
  We can solve for $\gamma_2, \gamma_1$ explicitly now by considering
  the two-dimensional system of equations
  \begin{equation}
    \label{eq:ILED:top-order-trapping:final-2d-system}
    \begin{split}
      \gamma_2+\gamma_1\sigma_i ={}& \frac{1}{4}\alpha_i^2(\sigma_1-\sigma_2)^2 + \beta_i^2\xi^2,\\
      \alpha_i ={}& \frac{2|\sigma_i|}{\sigma_1-\sigma_2}\alpha(r,
                    \sigma_i, \FreqPhi)(r-\rTrapping(\sigma_i, \FreqPhi))\in \TanSymClass{0}(T^{*}\mathcal{M}),\\
      \beta_i ={}& \beta(r, \sigma_i, \FreqPhi) \in \TanSymClass{0}(T^{*}\mathcal{M}).
    \end{split}
  \end{equation}
  Solving the system yields 
  \begin{equation}
    \begin{split}
      \label{eq:ILED-near:sum-of-squares:gamma-def}
      \gamma_2 ={}& \frac{1}{4}(\sigma_1-\sigma_2)(\alpha_2^2\sigma_1 - \alpha_1^2\sigma_2) + \frac{\sigma_1\beta_2^2 - \sigma_2\beta_1^2}{\sigma_1-\sigma_2}\xi^2,  \\
      \gamma_1 ={}& \frac{1}{4}(\sigma_1-\sigma_2)(\alpha_1^2-\alpha_2^2)+\frac{\beta_1^2 - \beta_2^2}{\sigma_1-\sigma_2}\xi^2. 
    \end{split} 
  \end{equation}
  We first add together the first two terms in $\gamma_i$ to see that
  \begin{equation}
    \label{eq:ILED-near:sum-of-squares:gamma-sum-first-terms}
    \begin{split}
      (\sigma_1-\sigma_2)\left(\alpha_2^2\sigma_1 -\alpha_1^2\sigma_2 + \sigma(\alpha_1^2-\alpha_2^2)\right)
      ={}&(1-{\delta_1})(\alpha_1(\sigma-\sigma_2)-\alpha_2(\sigma-\sigma_1))^2\\
         &+ {\delta_1}\left(\alpha_1(\sigma-\sigma_2)+\alpha_2(\sigma-\sigma_1)\right)^2\\
         &- 4\mathfrak{e}(\sigma-\sigma_1)(\sigma-\sigma_2),
    \end{split}
  \end{equation}
  where %
  \begin{equation*}
    \mathfrak{e} = \frac{(\alpha_1-\alpha_2)^2}{4}+ {\delta_1}\alpha_1\alpha_2. 
  \end{equation*}
  Recall\footnote{In particular, recall the definitions of
    $\AxiSquareSumCoeff_i$ from
    \zcref{prop:Mor-axially-sym:main-bulk}.} that when
  $a\FreqPhi = 0$, $\alpha_1=\alpha_2=\AxiSquareSumCoeff_1$,
  $\sigma_2=-\sigma_1$, $\beta_1=\beta_2=\AxiSquareSumCoeff_2$, and
  that $\MorawetzLagrangeCorrSym_{\Mor} = \delta_1\alpha_{\Mor}^2$.
  This implies that
  \begin{equation*}
    \mathfrak{e}-e_{\Mor}\in a\MixedSymClass{0}{1}(T^{*}\mathcal{M})
  \end{equation*}
  as desired.  We now add together the second terms in the $\gamma_i$
  given in \zcref{eq:ILED-near:sum-of-squares:gamma-def} to see that
  \begin{equation}
    \label{eq:ILED-near:sum-of-squares:gamma-sum-second-terms}
    \begin{split}
      \frac{\sigma_1\beta_2^2-\sigma_2\beta_1^2}{\sigma_1-\sigma_2}
      + \sigma\frac{\beta_1^2-\beta_2^2}{\sigma_1-\sigma_2}
      ={}&
           \frac{1}{2}\left(\beta_1^2+\beta_2^2 -Ca\right) +
           \frac{(Ca-\beta_2^2+\beta_1^2)(\sigma-\sigma_2)^2}{2(\sigma_1-\sigma_2)^2}\\
         &+
           \frac{(Ca-\beta_1^2+\beta_2^2)(\sigma-\sigma_1)^2}{2(\sigma_1-\sigma_2)^2}
           + O(a)\PrinSymb.
    \end{split}
  \end{equation}
  Summing \zcref{eq:ILED-near:sum-of-squares:gamma-sum-first-terms}
  and \zcref{eq:ILED-near:sum-of-squares:gamma-sum-second-terms}
  together, we have that
  \begin{equation*}
    \begin{split}
      &\frac{1}{2\ImagUnit}H_{\RescaledPrinSymb}(
        \MorawetzSym_{\Mor}+a\tilde{\MorawetzSym}) + \RescaledPrinSymb\MorawetzLagrangeCorrSym'\\
      ={}& \frac{1-{\delta_1}}{4}\left(
           \alpha_1(\sigma-\sigma_2)-\alpha_2(\sigma-\sigma_1)
           \right)^2
           + \frac{{\delta_1}}{4}\left(\alpha_1(\sigma-\sigma_2) + \alpha_2(\sigma-\sigma_1)\right)^2\\
      &+ \frac{1}{2}\left(\beta_1^2+\beta_2^2 - Ca\right)\xi^2
        + \frac{(Ca-\beta_2^2+\beta_1^2)(\sigma-\sigma_2)^2}{2(\sigma_1-\sigma_2)^2}\xi^2\\
      &+ \frac{(Ca-\beta_1^2+\beta_2^2)(\sigma-\sigma_1)^2}{2(\sigma_1-\sigma_2)^2}\xi^2
        +a(\sigma-\sigma_1)(\sigma-\sigma_2)\MixedSymClass{0}{1}(T^{*}\mathcal{M}).
    \end{split}
  \end{equation*}
  We then pick
  \begin{align*}
    \SquareDecomp_1^2 ={}& \frac{\delta_1}{4}\left(\alpha_1(\sigma-\sigma_2)+\alpha_2(\sigma-\sigma_1)\right)^2,\\
    \SquareDecomp_2^2 ={}& \frac{1}{2}\left(\beta_1^2 + \beta_2^2 - Ca\right)\xi^2,\\
    \SquareDecomp_{3}^2 ={}& \frac{\FreqTheta^2}{|\FreqAngular|^2 + \Delta\xi^2}\frac{(1-\delta_1)}{4} \left(\alpha_1(\sigma-\sigma_2)-\alpha_2(\sigma-\sigma_1)\right)^2, \\ 
    \SquareDecomp_{4}^2 ={}& \frac{\FreqPhi^2}{|\FreqAngular|^2 + \Delta\xi^2}\frac{(1-\delta_1)}{4} \left(\alpha_1(\sigma-\sigma_2)-\alpha_2(\sigma-\sigma_1)\right)^2,\\
    \SquareDecomp_{5}^2 ={}& \frac{\Delta\xi^2}{|\FreqAngular|^2 + \Delta\xi^2}\frac{(1-\delta_1)}{4}\left(\alpha_1(\sigma-\sigma_2)-\alpha_2(\sigma-\sigma_1)\right)^2,\\
    \SquareDecomp_6^2 ={}& \frac{\left(Ca-\beta_2^2 + \beta_1^2\right)\left(\sigma-\sigma_2\right)^2}{2\left(\sigma_1-\sigma_2\right)^2}\xi^2,\\
    \SquareDecomp_7^2 ={}& \frac{\left(Ca-\beta_1^2 + \beta_2^2\right)\left(\sigma-\sigma_1\right)^2}{2\left(\sigma_1-\sigma_2\right)^2}\xi^2,
  \end{align*}
  concluding the proof of \zcref{lemma:ILED-KdS:Bulk}.

  We observe that in particular, when $a = 0$, $\sigma_1 = -\sigma_2$,
  $\alpha_1 = \alpha_2$, $\beta_1 = \beta_2$, and
  $\SquareDecomp_4 = \SquareDecomp_5 = \SquareDecomp_6 = \SquareDecomp_7 = 0$. 
\end{proof}

\begin{corollary}
  \label{coro:Morawetz:elliptic-equivalence}
  The family of symbols $\{\SquareDecomp_j\}_{j=1,7}$ is elliptically
  equivalent to the family of symbols given by
  $\curlyBrace*{\ell_1(\sigma-\sigma_2), \ell_2(\sigma-\sigma_1),
    \xi^2}$ where $\ell_i$, $\sigma_j$, are as defined in
  \zcref{eq:trapping-norm:elli-def} and
  \zcref{eq:trapping-norm:sigmai-def} respectively, in the
  sense that there is a representation of the form
  \begin{equation*}
    \SquareDecomp = \mathfrak{M} \mathfrak{b},\qquad
    \mathfrak{b} =
    \begin{pmatrix*}
      \ell_1(\sigma-\sigma_2)\\
      \ell_2(\sigma-\sigma_1)\\
      \xi
    \end{pmatrix*},
  \end{equation*}
  where the symbol valued matrix $\mathfrak{M}$ has maximum rank 3 everywhere. 
\end{corollary}

\begin{proof}
  Recall from \zcref{eq:ILED:top-order-trapping:final-2d-system, eq:trapping-norm:elli-def} that 
  \begin{equation*}
    \frac{2\abs*{\sigma_i}}{\sigma_1-\sigma_2}\alpha(r, \sigma_i, \eta_{\varphi})\ell_i
    = \alpha_i.
  \end{equation*}
  As a result, it suffices to show that there exists some
  $\widetilde{\mathfrak{M}}$ such that
  \begin{equation*}
    \mathfrak{a}=\widetilde{\mathfrak{M}}
    \begin{pmatrix*}
      \alpha_1(\sigma-\sigma_2)\\
      \alpha_2(\sigma-\sigma_1)\\
      \xi
    \end{pmatrix*}.
  \end{equation*}
  To this end, consider
  \begin{equation*}
    \widetilde{\mathfrak{M}}
    =
    \begin{pmatrix}
      \frac{\sqrt{\delta_1}}{2} & \frac{\sqrt{\delta_1}}{2}& 0\\
      0& 0 & \frac{1}{\sqrt{2}}\sqrt{\beta_1^2+\beta_2^2-Ca}    \\
      \frac{\FreqTheta}{\sqrt{|\FreqAngular|^2 + \Delta\xi^2}}\frac{\sqrt{1-\delta_1}}{2} & -\frac{\FreqTheta}{\sqrt{|\FreqAngular|^2 + \Delta\xi^2}}\frac{\sqrt{1-\delta_1}}{2} & 0\\
      \frac{\FreqPhi}{\sqrt{|\FreqAngular|^2 + \Delta\xi^2}}\frac{\sqrt{1-\delta_1}}{2} & -\frac{\FreqTheta}{\sqrt{|\FreqAngular|^2 + \Delta\xi^2}}\frac{\sqrt{1-\delta_1}}{2} & 0\\
      \frac{\sqrt{\Delta}\xi}{\sqrt{|\FreqAngular|^2 + \Delta\xi^2}}\frac{\sqrt{1-\delta_1}}{2} & -\frac{\FreqTheta}{\sqrt{|\FreqAngular|^2 + \Delta\xi^2}}\frac{\sqrt{1-\delta_1}}{2} & 0\\
      0 & 0 & \frac{\sqrt{Ca-\beta_2^2 + \beta_1^2}(\sigma-\sigma_2)}{\sqrt{2}(\sigma_1-\sigma_2)}\\
      0 & 0 & \frac{\sqrt{Ca-\beta_1^2 + \beta_2^2}(\sigma-\sigma_1)}{\sqrt{2}(\sigma_1-\sigma_2)},
    \end{pmatrix}
  \end{equation*}
  which is of rank 3 everywhere. 
\end{proof}

\subsubsection{Combining trapped and untrapped estimates}

Here we combine the bulk estimates that we obtained outside trapping in \zcref{prop:Mor-axially-sym:main-bulk} with the ones at trapping in
\zcref{lemma:ILED-KdS:Bulk}.

As constructed in \zcref[cap]{lemma:ILED-KdS:Bulk}, both
$\widetilde{\MorawetzSym}$ and $\widetilde{\MorawetzLagrangeCorrSym}$
are homogeneous symbols and can be made smooth simply by a standard
truncation away from low frequencies. Since they are also only defined
near $r=2M$, we will also truncate the symbols in physical
space. Define
\begin{equation*}
  \mathring{\chi} =
  \begin{cases}
    1 & \abs*{r-2M} < \delta_{\operatorname{trap}}\\
    0 & \abs*{r-2M} > 2\delta_{\operatorname{trap}},
  \end{cases}
\end{equation*}
and define $\breve{\chi}$ so that
\begin{equation*}
  1 = \mathring{\chi}^2 + \breve{\chi}^2. 
\end{equation*}
Now, we truncate our pseudo-differential multipliers, so that
\begin{equation}
  \label{eq:Morawetz:morawetz-VF-LagrangeCorr-with-cutoff:def}
  \widetilde{\MorawetzVF} = \mathring{\chi}\WeylQ{\widetilde{\MorawetzSym}}\mathring{\chi}, \qquad
  \widetilde{\MorawetzLagrangeCorr} = \mathring{\chi}\WeylQ{\widetilde{\MorawetzLagrangeCorr}}\mathring{\chi}.
\end{equation}
From the Weyl calculus, we can only say that
there exist
$K^W_i\in \TanOpClass{i}(\Manifold)$ 
such that
\begin{equation*}
  \left(\left[\Box_{\Metric}, \widetilde{\MorawetzVF}\right]
    + \Box_{\Metric}\widetilde{\MorawetzLagrangeCorr}
    + \widetilde{\MorawetzLagrangeCorr}\Box_{\Metric}
  \right)
  = K^W_2 + 2K_1^W\partial_t + K_0^W\partial_t^2 + K^W_{-1}\partial_t^3. 
\end{equation*}

In order to eliminate the $K^W_{-1}\partial_t^3$ term, we
slightly alter the choice of pseudo-differential Lagrangian corrector
so that
\begin{equation*}
  \widetilde{\MorawetzLagrangeCorr}
  = \mathring{\chi}\widetilde{\MorawetzLagrangeCorrSym}\mathring{\chi}
  - \widetilde{\MorawetzLagrangeCorrSym}_{\Aux}\partial_t,
\end{equation*}
where the operator $\widetilde{\MorawetzLagrangeCorrSym}_{\Aux}$ is chosen so that
\begin{equation*}
  \Metric^{-1}(dt, dt)\WeylQ{\widetilde{\MorawetzLagrangeCorrSym}_{\Aux}}
  + \WeylQ{\widetilde{\MorawetzLagrangeCorrSym}_{\Aux}}\Metric^{-1}(dt, dt)
  = K^W_{-1}.
\end{equation*}
With this choice of $\widetilde{\MorawetzLagrangeCorr}$, we can
enforce that actually $K^W_{-1}=0$, so that
\begin{equation*}
  \left(\left[\widetilde{\MorawetzVF}, \Box_{\Metric}\right]
    + \Box_{\Metric}\widetilde{\MorawetzLagrangeCorr}
    + \widetilde{\MorawetzLagrangeCorr}\Box_{\Metric}
  \right)
  = K^W_2 + 2K_1^W\partial_t + K_0^W\partial_t^2. 
\end{equation*}
We now have that the principal symbol of the bilinear product
$\KCurrentPert{\widetilde{\MorawetzVF},
  \widetilde{\MorawetzLagrangeCorr}}[\psi]$ in \zcref{eq:KCurrentPert-def}
will be equal to
\begin{equation*}
  \KCurrentPertSym{\widetilde{\MorawetzVF}, \widetilde{\MorawetzLagrangeCorr}} \vcentcolon=
  \mathring{\chi}^2\left(\frac{1}{2\ImagUnit}\PoissonB*{\PrinSymb, \widetilde{\MorawetzSym}}
    + \PrinSymb\widetilde{\MorawetzLagrangeCorrSym} \right)
  + \frac{1}{\ImagUnit}\mathring{\chi}\PoissonB{\PrinSymb, \mathring{\chi}}.
\end{equation*}
We now split
\begin{equation}
  \label{eq:KCurrent-princ-aux:decomposition}
  \KCurrentPert{\widetilde{\MorawetzVF},\widetilde{\MorawetzLagrangeCorr}}[\psi]
  = \KCurrentPert{\widetilde{\MorawetzVF},\widetilde{\MorawetzLagrangeCorr}}_{\Main}[\psi]
  + \KCurrentPert{\widetilde{\MorawetzVF},\widetilde{\MorawetzLagrangeCorr}}_{\Aux}[\psi],
\end{equation}
where
\begin{equation*}
  \begin{split}
    \KCurrentPert{\widetilde{\MorawetzVF},\widetilde{\MorawetzLagrangeCorr}}_{\Main}[\psi]
    \vcentcolon={}& \bangle*{\widetilde{K}^W_{2,\Main}\psi, \psi}_{L^2(\Manifold(\tau_1,\tau_2))}
                    + 2\Re\bangle*{\widetilde{K}^W_{1,\Main}\psi, \partial_t\psi}_{L^2(\Manifold(\tau_1,\tau_2))}\\
                  &+ \bangle*{\widetilde{K}^W_{0,\Main}\partial_t\psi, \partial_t\psi}_{L^2(\Manifold(\tau_1,\tau_2))},    
  \end{split}
\end{equation*}
where $\widetilde{K}^W_{2,\Main}$, $i=0,1,2$ are defined so that
\begin{equation*}
  \sum_{i=0}^2
  \begin{pmatrix}
    2\\ i
  \end{pmatrix}
  \widetilde{K}^W_{i,\Main}\partial_t^i
  = \mathring{\chi}\WeylQ{\frac{1}{2\ImagUnit}\PoissonB*{\PrinSymb,\widetilde{\MorawetzSym}} + \PrinSymb\widetilde{\MorawetzLagrangeCorrSym}}\mathring{\chi}.
\end{equation*}
We similarly define $K^W_{i,\Aux}, i=0,1,2$ so that
\begin{equation*}
  \begin{split}
    \KCurrentPert{\widetilde{\MorawetzVF},\widetilde{\MorawetzLagrangeCorr}}_{\Aux}[\psi](\tau_1,\tau_2)
    ={}& \bangle*{K_{2,\Aux}^W\psi,\psi}_{L^2(\Manifold(\tau_1,\tau_2))}
         + 2\Re \bangle*{K_{1,\Aux}^W\partial_t\psi,\psi}_{L^2(\Manifold(\tau_1,\tau_2))}\\
       & + \bangle*{K_{0,\Aux}^W\partial_t\psi,\partial_t\psi}_{L^2(\Manifold(\tau_1,\tau_2))}
         - \bangle*{\WeylQ{\frac{1}{\ImagUnit}\mathring{\chi}\widetilde{\MorawetzVF}H_{\PrinSymb}\mathring{\chi}}\psi,\psi}_{L^2(\Manifold(\tau_1,\tau_2))}.
  \end{split}  
\end{equation*}
Then observe from \zcref{cor:control-outside-trapping}
that we already control the
$\bSobolevHI{1, \frac{\delta_1}{2},-\frac{5}{2}}(\Manifold(\tau_1,\tau_2))$
norm outside an $O(a)$ neighborhood of $\{r=2M\}$, and that for $a$
sufficiently small,
$\KCurrentPert{\widetilde{\MorawetzVF},\widetilde{\MorawetzLagrangeCorr}}_{\Aux,\DomainOfIntegration}[\psi]$
has principal symbols with support away from
$\Manifold_{\operatorname{trap}}, \Manifold_{\EventHorizon},
\Manifold_{\NullInfinity}$. As a result, we can write that
\begin{equation}
  \label{eq:Morawetz:K-ext-bound}
  \abs*{\KCurrentPert{\widetilde{\MorawetzVF},\widetilde{\MorawetzLagrangeCorr}}_{\Aux}[\psi](\tau_1,\tau_2)}
  \les \norm*{\psi}_{\bSobolevHITrap{1,\frac{\delta_1}{2},-\frac{5}{2}}(\Manifold(\tau_1,\tau_2))}^2
  + \norm*{\partial_t\psi}_{H^{-1}_c(\Manifold(\tau_1,\tau_2))}^2.
\end{equation}

We now show how to control the norm $\norm*{\partial_t\psi}_{H^{-1}_c(\Manifold(\tau_1,\tau_2))}$ above.
\begin{lemma}
  \label{lemma:ILED-near:LoT-control:zero-order-mixed-term}
  We have that
  \begin{equation*}
    \norm*{\partial_t\psi}_{H^{-1}_{c}(\Manifold(\tau_1,\tau_2))}
    \les \norm*{\psi}_{L^2_{c}(\Manifold(\tau_1,\tau_2))}
    + \norm*{F}_{L^2(\Manifold(\tau_1,\tau_2))}
    + \sup_{[\tau_1,\tau_2]}\norm*{\psi}_{\bSobolevHI{1, 0, -1}(\Sigma(\tau))}.
  \end{equation*}
\end{lemma}
\begin{proof}
  Consider some compactly supported self-adjoint elliptic operator
  $Q \in \TanOpClass{-1}(\Manifold)$. We
  use $Q^2$ as a multiplier to see via integration by parts that
  \begin{equation*}
    \begin{split}
      &2\Re\bangle*{\Box_{\Metric}\psi - V\psi, \frac{1}{\Metric^{-1}(dt,dt)}Q^2\psi}_{L^2(\Manifold(\tau_1,\tau_2))}\\
      ={}& \norm*{Q\partial_t\psi}_{L^2(\Manifold(\tau_1,\tau_2))}^2
           + O\left(\norm*{Q\partial_t\psi}_{L^2(\Manifold(\tau_1,\tau_2))}\norm*{\psi}_{L^2_c(\Manifold(\tau_1,\tau_2))}\right)\\
      & + O\left(\norm*{\psi}^2_{L^2_c(\Manifold(\tau_1,\tau_2))}
        + \sup_{[\tau_1,\tau_2]}\norm*{\psi}_{\bSobolevHI{1, 0, -1}(\Sigma(\tau))}^2
        \right).
    \end{split}    
  \end{equation*}
  It then follows that
  \begin{equation*}
    \norm*{Q\partial_t\psi}_{L(\Manifold(\tau_1,\tau_2))}
    \les \norm*{\psi}_{L_c(\Manifold(\tau_1,\tau_2))}
    + \varepsilon \norm*{\partial_t\psi}_{H^{-1}_{\operatorname{c}}(\Manifold(\tau_1,\tau_2))}
    + \norm*{\Box_{\Metric}\psi}_{L(\Manifold(\tau_1,\tau_2))}
    + \sup_{\tau\in[\tau_1,\tau_2]}\norm*{\psi}_{\bSobolevHI{1, 0, -1}(\Sigma(\tau))}.
  \end{equation*}
  Since $Q$ is arbitrary, we can absorb the
  $\varepsilon
  \norm*{\partial_{t}\psi}_{H^{-1}_c(\Manifold(\tau_1,\tau_2))}^2$
  term from the right-hand side onto the left-hand side for
  $\varepsilon$ sufficiently small, and prove the statement.
\end{proof}

We are finally ready to state and prove \zcref{lemma:Morawetz:cutoff-PsiDO:bulk}.
\begin{lemma}
  \label{lemma:Morawetz:cutoff-PsiDO:bulk}  
  Let
  $(\MorawetzVF_{\Mor}^\diff, \MorawetzLagrangeCorr_{\Mor}^\diff, \MorawetzOneForm_{\Mor}^\diff)$ be as in
  \zcref{prop:Mor-axially-sym:main-bulk},
  and let $\widetilde{\MorawetzVF}$ and
  $\widetilde{\MorawetzLagrangeCorr}$ be as in
  \zcref{eq:Morawetz:morawetz-VF-LagrangeCorr-with-cutoff:def}, where
  $\widetilde{\MorawetzSym}$ and
  $\widetilde{\MorawetzLagrangeCorrSym}$ are as in 
  \zcref[cap]{lemma:ILED-KdS:Bulk}. Then,
  \begin{equation*}
    \begin{split}
      & a \KCurrentPert{\widetilde{\MorawetzVF}, \widetilde{\MorawetzLagrangeCorr}}[\psi](\tau_1,\tau_2)
        + \int_{\Manifold(\tau_1,\tau_2)}\KCurrent{\MorawetzVF_{\Mor}^\diff, \MorawetzLagrangeCorr_{\Mor}^\diff, J_{\Mor}^\diff}[\psi]\ges{}  \norm*{\psi}_{\bSobolevHITrap{1,\frac{\delta_1}{2},-\frac{5}{2}}(\Manifold(\tau_1,\tau_2))}^2\\
      &
        - O(a)\left(\norm*{\psi}_{\bSobolevHITrap{1,\frac{\delta_1}{2},-\frac{5}{2}}(\Manifold(\tau_1,\tau_2))}^2+
        \norm*{\psi}_{H^0_c(\Manifold(\tau_1,\tau_2))}^2
        + \norm*{F }_{L^2(\Manifold(\tau_1,\tau_2))}^2
        + \sup_{[\tau_1,\tau_2]}\norm*{\psi}_{\bSobolevHI{1, 0, -1}(\Sigma(\tau))}^2
        \right).
    \end{split}    
  \end{equation*}
\end{lemma}
\begin{proof}
  We first decompose
  \begin{equation*}
    \begin{split}
      \int_{\Manifold(\tau_1,\tau_2)}\KCurrent{\MorawetzVF_{\Mor}^\diff, \MorawetzLagrangeCorr_{\Mor}^\diff, J_{\Mor}^\diff}[\psi]
      ={}& \int_{\Manifold(\tau_1,\tau_2)}\mathring{\chi}^2\KCurrent{\MorawetzVF_{\Mor}^\diff, \MorawetzLagrangeCorr_{\Mor}^\diff, J_{\Mor}^\diff}[\psi]\\
         & + \int_{\Manifold(\tau_1,\tau_2)}\breve{\chi}^2\KCurrent{\MorawetzVF_{\Mor}^\diff, \MorawetzLagrangeCorr_{\Mor}^\diff, J_{\Mor}^\diff}[\psi].
    \end{split}    
  \end{equation*}
  Then, since
  $\KCurrent{\MorawetzVF_{\Mor},
    \MorawetzLagrangeCorr_{\Mor},
    \MorawetzOneForm_{\Mor}}[\psi]$ is pointwise positive
  outside of a neighborhood of $r=2M$, then from \zcref{cor:control-outside-trapping}
  \begin{equation}
    \label{eq:Morawetz:PsiDO-bulk:outer:representation}
    \begin{split}
      \int_{\Manifold(\tau_1,\tau_2)}\breve{\chi}^2\KCurrent{\MorawetzVF_{\Mor}^\diff, \MorawetzLagrangeCorr_{\Mor}^\diff, J_{\Mor}^\diff}[\psi]
      &\ges \norm*{\psi}_{\bSobolevHI{1,\frac{\delta_1}{2},-\frac{5}{2}}(\Manifold_{\nontrap}(\tau_1,\tau_2))}^2.
    \end{split}
  \end{equation}
  On the other hand, we see that
  \begin{align*}
    &a\KCurrentPert{\widetilde{\MorawetzVF},\widetilde{\MorawetzLagrangeCorr}}_{\Main}[\psi](\tau_1,\tau_2)
      + \int_{\Manifold(\tau_1,\tau_2)}\mathring{\chi}^2\KCurrent{\MorawetzVF_{\Mor}^\diff, \MorawetzLagrangeCorr_{\Mor}^\diff, \MorawetzOneForm_{\Mor}^\diff}[\psi]\\
    ={}& \Re\int_{\Manifold(\tau_1,\tau_2)}\mathring{\chi}^2\left(\KCurrentSym{\MorawetzVF_{\Mor}^\diff, \MorawetzLagrangeCorr_{\Mor}^\diff, \MorawetzOneForm_{\Mor}^\diff}_{(2)}\right)^{\alpha\beta}\partial_{\alpha}\psi\cdot \overline{\partial_{\beta}\psi}
         + \mathring{\chi}^2\KCurrentSym{\MorawetzVF_{\Mor}^\diff, \MorawetzLagrangeCorr_{\Mor}^\diff, \MorawetzOneForm_{\Mor}^\diff}_{(0)}\abs*{\psi}^2\\
    &+ a\Re \int_{\Manifold(\tau_1,\tau_2)}K^W_{2,\Main}\psi\cdot \overline{\psi} + 2 K^W_{1,\Main}\partial_t\psi\cdot \overline{\psi} + K^W_{0,\Main}\partial_t\psi\cdot\overline{\partial_t\psi},
  \end{align*}
  where for $a$ sufficiently small, we have that
  \begin{equation*}
    \left(\KCurrentSym{\MorawetzVF_{\Mor}^\diff, \MorawetzLagrangeCorr_{\Mor}^\diff, \MorawetzOneForm_{\Mor}^\diff}_{(2)}\right)^{\alpha\beta}\eta_{\alpha}\eta_{\beta} = \frac{1}{2\ImagUnit}H_{\PrinSymb}\MorawetzVF_{\Mor} + \PrinSymb \MorawetzLagrangeCorr_{\Mor}, \qquad
    \KCurrentSym{\MorawetzVF_{\Mor}^\diff, \MorawetzLagrangeCorr_{\Mor}^\diff, \MorawetzOneForm_{\Mor}^\diff}_{(0)} > 0.
  \end{equation*}
  Moreover, recall that by construction,
  \begin{equation*}
    K^W_{2,\Main} + 2K^W_{1,\Main}\partial_t + K^W_{0,\Main}\partial_t^2
    = \mathring{\chi}\WeylQ{\frac{1}{2\ImagUnit}H_{\PrinSymb}\widetilde{\MorawetzSym} + \PrinSymb\widetilde{\MorawetzLagrangeCorrSym}}\mathring{\chi}.
  \end{equation*}
  Observe that by \zcref{lemma:ILED-KdS:Bulk}, we have that
  $\SquareDecomp_k(a)$ are in general symbols of pseudodifferential
  operators. However, $\evalAt*{\SquareDecomp_k}_{a=0}$ are symbols of
  differential operators. As a result, we can decompose
  \begin{equation*}
    \SquareDecomp_k =
    \evalAt*{\SquareDecomp_k}_{a=0}
    + \left( \SquareDecomp_k
      - \evalAt*{\SquareDecomp_k}_{a=0} \right),
  \end{equation*}
  where $\evalAt*{\SquareDecomp_k}_{a=0}$ is the symbol of
  a differential operator, and
  \begin{equation*}
    \SquareDecomp_k - \evalAt*{\SquareDecomp_k}_{a=0}
    \in a \MixedSymClass{2}{2}\left(T^{*}\Manifold\right).
  \end{equation*}
  Then, define the operators
  \begin{equation*}
    \SquareDecompOp_k
    = \mathring{\chi}\evalAt*{\SquareDecomp_k}_{a=0}(x, D)
    + \WeylQ{\SquareDecomp_k(a) - \evalAt*{\SquareDecomp_k}_{a=0}}\mathring{\chi}.
  \end{equation*}
  Then from the Weyl calculus, we have that
  \begin{equation}
    \label{eq:Morawetz:PsiDO-bulk:inner:representation}
    \begin{split}
      &a\KCurrentPert{\widetilde{\MorawetzVF},\widetilde{\MorawetzLagrangeCorr}}_{\Main}[\psi](\tau_1,\tau_2)
        + \int_{\Manifold(\tau_1,\tau_2)}\mathring{\chi}^2\KCurrent{\MorawetzVF_{\Mor}^\diff, \MorawetzLagrangeCorr_{\Mor}^\diff, \MorawetzOneForm_{\Mor}^\diff}[\psi] \\
      ={}& \int_{\Manifold(\tau_1,\tau_2)}\sum_k\abs*{\SquareDecompOp_k\psi}^2
           + \KCurrentSym{\MorawetzVF_{\Mor}^\diff, \MorawetzLagrangeCorr_{\Mor}^\diff, \MorawetzOneForm_{\Mor}^\diff}_{(0)}\mathring{\chi}^2\abs*{\psi}^2\\
      & + \Re \int_{\Manifold(\tau_1,\tau_2)}R^W_2\psi\cdot \overline{\psi} + 2R^W_1\partial_t\psi\cdot\overline{\psi} + R^W_0\partial_t\psi\cdot\overline{\partial_t\psi},      
    \end{split}
  \end{equation}
  where the remainder terms $R^W_i$ have symbols
  $r_j\in a \TanSymClass{j-2}(\Manifold)$ from the Weyl calculus.
  Combining
  \zcref{eq:Morawetz:PsiDO-bulk:inner:representation,eq:Morawetz:PsiDO-bulk:outer:representation}
  and using \zcref{coro:Morawetz:elliptic-equivalence}, and then using
  \zcref{lemma:ILED-near:LoT-control:zero-order-mixed-term} to control
  the remainder terms we deduce
  \begin{equation*}
    \begin{split}
      & a \KCurrentPert{\widetilde{\MorawetzVF}, \widetilde{\MorawetzLagrangeCorr}}_{\Main}[\psi](\tau_1,\tau_2)
        + \int_{\Manifold(\tau_1,\tau_2)}\KCurrent{\MorawetzVF_{\Mor}^\diff, \MorawetzLagrangeCorr_{\Mor}^\diff, J_{\Mor}^\diff}[\psi]\\
      \ges{}&  \norm*{\psi}_{\bSobolevHITrap{1,\frac{\delta_1}{2},-\frac{5}{2}}(\Manifold(\tau_1,\tau_2))}^2
              - O(a)\left(
              \norm*{\psi}_{H^0_c(\Manifold(\tau_1,\tau_2))}^2
              + \norm*{F }_{L^2(\Manifold(\tau_1,\tau_2))}^2
              + \sup_{[\tau_1,\tau_2]}\norm*{\psi}_{\bSobolevHI{1, 0, -1}(\Sigma(\tau))}^2
              \right).
    \end{split}    
  \end{equation*}
  We can add the control of
  $\KCurrentPert{\widetilde{\MorawetzVF},\widetilde{\MorawetzLagrangeCorr}}_{\Aux}[\psi]$
  by using \zcref{eq:Morawetz:K-ext-bound} and
  \zcref{lemma:ILED-near:LoT-control:zero-order-mixed-term} to
  conclude the proof of \zcref{lemma:Morawetz:cutoff-PsiDO:bulk}.
\end{proof}

\subsection{Energy estimates}\label{subs:energy}

In this section we derive the energy estimates and show that the
boundary terms from Morawetz are controlled by the positive energy
boundary terms. By combining the energy estimates with the positivity
for the Morawetz bulk obtained in
\zcref{lemma:Morawetz:cutoff-PsiDO:bulk} we obtain the proof of
\zcref{prop:energy-Morawetz} in \zcref{sec:proof-e-mor}.

\subsubsection{The \texorpdfstring{$\That_\chi$}{T} energy}

Recall that due to the presence of an ergoregion, $T$ is not timelike
on the entire domain of exterior communication and thus cannot be used
to immediately deduce an energy inequality via a multiplier argument.
We will nonetheless be able to prove a suitable energy inequality by
using a modified version of $T$.
\begin{definition}
  \label{def:That-chi:def}
  Define the vectorfield 
  \begin{equation*}
    \That_\chi\vcentcolon= T + \omega_{+} \chi_{\That} \Phi, \qquad
    \omega_{+}\vcentcolon=\frac{a}{r_+^2+a^2},
  \end{equation*}
  where $\chi_{\That}(r)$ is a smooth decreasing cut-off function
  which is equal to $1$ for $\rhoH \leq A_1$ and $0$ for
  $\rhoH \geq \frac{3}{2}A_1$ with
  $\frac 32A_1 \leq \frac{r_{+}}{80} <\rho_0$.
\end{definition}
We note that by our construction, $\That_\chi$ is timelike for all
$r>r_{+}$, Killing everywhere except for
$\rhoH \in ( A_1, \frac{3}{2}A_1)$, equal to $\That_{\HH}$ on
$\rhoH \leq A_1$, and equal to $T$ on $\rhoH \geq \frac{3}{2}A_1$, in
particular on the trapping region $\MM_{\trap}$.

\begin{lemma}
  \label{lemma:That-chi-deformation-tensor}
  For $\That_{\chi}$ as defined in \zcref{def:That-chi:def}, we have that
  \begin{align*}
    \int_{\Manifold(\tau_1, \tau_2)} \big| \QQ^{(\That_\chi, 0, 0)}[\psi]\big|
    \les  {}&  O(a) \norm*{\psi}_{\bSobolevH{1, -\frac 1 2}(\Manifold_{\rhoH \leq \frac 3 2 A_1}(\tau_1, \tau_2))}^2.
  \end{align*}
\end{lemma}
\begin{proof}
  The deformation tensor of $\That_{\chi}$ is given by
  \begin{align*}
    |q|^2\, ^{(\That_\chi)} \pi^{\a\b}={}&\frac{2a \rhoH^2}{M^2+a^2}\chi_{\That}'  \,   \pr_\phi^{(\a} {\pr^{\mathrm{BL}}_r}^{\b)},
  \end{align*}
  and therefore from \zcref{le:divergPP-gen} we have 
  \begin{align*}
    \QQ^{(\That_\chi, 0, 0)}[\psi]={}&  \frac{2a\rhoH^2}{|q|^2(M^2+a^2)}\chi_{\That}'  \pr_\phi\psi \pr^{\mathrm{BL}}_r \psi,
  \end{align*}
  where recall that $\chi_{\That}'$ is supported in $A_1\leq \rhoH\leq \frac{3}{2}A_1$. In particular, 
  \begin{align*}
    \big| \QQ^{(\That_\chi, 0, 0)}[\psi]\big|\les {}&  O(a) \mathds{1}_{\{A_1 \leq \rhoH \leq  \frac{3}{2}A_1\}} \rhoH^2 | \pr_\phi\psi| |\pr^{\mathrm{BL}}_r \psi|\\
    \les {}&  O(a) \mathds{1}_{\{A_1 \leq \rhoH \leq  \frac{3}{2}A_1\}}\Big( \rhoH|\pr_\phi\psi||\rhoH\pr_{\rhoH}
             \psi |+|\pr_\phi\psi|| \HawkingHorizon\psi|\Big)\\
    \les  {}&  O(a) \mathds{1}_{\{A_1 \leq \rhoH \leq  \frac{3}{2}A_1\}}\rhoH \abs*{A^{\Horizon}_1\psi}^2,
  \end{align*}
  and therefore, upon integrating on $\Manifold(\tau_1, \tau_2)$, gives the stated bound.
\end{proof}

We now compute the non-negative boundary terms of the current $\JCurrent{(\That_\chi, 0, 0)}[\psi]$.

\begin{lemma}\label{lemma:boundary-terms-energy}
  For $|a|/M\ll 1$ sufficiently small, we have 
  \begin{align*}
    \int_{\Sigma(\tau)}  \JCurrent{(\That_\chi, 0, 0)}[\psi]\cdot N_{\Sigma(\tau)}
    &\ges  \norm*{\psi}_{\bSobolevHI{1, 0, -1}(\Sigma(\tau))}^2.
  \end{align*}
\end{lemma}
\begin{proof}
  Since $\That_\chi$ and $N_{\Sigma(\tau)}$ are globally timelike vectorfields, by the dominant energy condition of the energy momentum tensor we deduce that 
  \begin{align*}
    \JCurrent{(\That_\chi, 0, 0)}[\psi]\cdot N_{\Sigma(\tau)}={}& \EMTensor( \That_\chi, N_{\Sigma}(\tau))[\psi] \geq 0,
  \end{align*}
  and control all derivatives.  In the region where $\rhoH \leq A_1$, we have $\That_\chi={}\frac{1}{r^2+a^2}\HawkingHorizon-\rhoH \frac{\omega_{+}(r+r_{+})}{r^2+a^2} \,\partial_{\phiIEF}$, so using \zcref{eq:EMtensor-rhoH-H} we deduce
  \begin{align*}
    \JCurrent{(\That_\chi, 0, 0)}[\psi]\cdot N_{\Sigma_{\EventHorizon}(\tau)}={}& \EMTensor( \That_\chi, N_{\Sigma_{\EventHorizon}(\tau)})[\psi] \\
    ={}& \EMTensor(\conormalSpaceH{0}\InfinityHawking+\conormalSpaceH{1}\partial_{\phiIEF}, \conormalSpaceH{0}\partial_{\rhoI}+\conormalSpaceH{0}\HawkingHorizon+\conormalSpaceH{0}\renormpphi)[\psi] \\
    \ges{}&  \abs*{\rhoH\partial_{\rhoH}\psi}^2+  |\pr_v \psi|^2 +|\NablaAngular \psi|^2.
  \end{align*}
  In the region where $\rhoI \leq r_{+}^{-2}\rho_0$, we have $\That_\chi=\frac{r^2}{r^2+a^2}\InfinityHawking-\frac{a}{r^2+a^2}\partial_{\phiIEF}$, 
  so using \eqref{eq:T-outgoing} we deduce
  \begin{align*}
    \JCurrent{(\That_\chi, 0, 0)}[\psi]\cdot N_{\Sigma_{\NullInfinity}(\tau)}={}& \EMTensor( \That_\chi, N_{\Sigma_{\NullInfinity}(\tau)})[\psi] \\
    ={}& \EMTensor(\conormalSpaceI{0}\InfinityHawking+\conormalSpaceI{2}\partial_{\phiOEF}, \conormalSpaceI{2}\partial_{\rhoI}+\conormalSpaceI{2}\InfinityHawking+\conormalSpaceI{2}\renormpphi)[\psi] \\
    \ges{} & \rhoI^{2} \left( \abs*{\rhoI\partial_{\rhoI}\psi}^2+ |\pr_u \psi|^2 +|\NablaAngular \psi|^2\right).
  \end{align*}
  Combining the above control close to the event horizon and close to null infinity, we obtain the stated control of the first order derivatives of $\psi$. 
  We are now left to add control of the zero-th order term. Applying \zcref{lemma:hardy-dejan} with $\rho=\rhoH$, $a=0$, $b=\infty$ and $\gamma=0$
  \begin{align*}
    \int_0^\infty  |\psi|^2 d\rhoH
    & \leq 4\int_0^\infty \rhoH^2\big| \partial_{\rhoH} \psi\big|^2 d\rhoH +2 \lim_{r \to \infty} (r|\psi|^2) \leq 4\int_0^\infty \big|\rhoH \partial_{\rhoH} \psi\big|^2 d\rhoH
  \end{align*}
  where we used that by regularity assumption $\lim_{r\to \infty} r \psi^2=0$.
  In particular,  rearranging the volume form, we have that
  \begin{equation*}
    \int_M^{\infty}\int_{S^2}\frac{1}{\abs*{q}^2}\abs*{\psi}^2\,\abs*{q}^2drd\mathring{\gamma}
    \le 4\int_M^{\infty}\int_{S^2}\frac{1}{\abs*{q}^2}\frac{(r-M)^2}{r^2}\abs*{r\partial_r\psi}^2\,\abs*{q}^2drd\mathring{\gamma}.
  \end{equation*}
  This then implies that
  \begin{equation*}
    \norm*{\psi}_{\bSobolevHI{0, 0, -1}(\Sigma(\tau))}^2
    \les \int_{\Sigma}\JCurrent{(\That_\chi, 0, 0)}[\psi]\cdot N_{\Sigma(\tau)}.
  \end{equation*}
  Combining with the above we obtain \zcref{lemma:boundary-terms-energy}. 
\end{proof}

\subsubsection{Boundary terms of Morawetz}\label{sec:boundary-terms-Morawetz}

We now compute the boundary terms of the Morawetz multiplier triplet.
Recall that the Morawetz multiplier triplets are
\begin{align*}
  (X,w,J)
  ={}&
       \bigl(X_{\Mor}^\diff+a\widetilde X,\; w_{\Mor}^\diff+a\widetilde w,\; J_{\Mor}^\diff\bigr) \\
  ={}&
       (X_{ax},w_{ax},J_{ax})
       +c_{red}(X_{red},0,J_{red})
       +c_{\That}(0,w_{\That},0)
       +a(\widetilde X,\widetilde w,0),
\end{align*}
where the quantities are defined in \zcref{sec:ILED}. 
By \zcref{definition-of-P}, for any boundary normal $N$,
\begin{equation}\label{eq:boundary-identity-general}
  \JCurrent{(X,w,J)}[\psi]\cdot N
  =
  \EMTensor(X,N)[\psi]
  + w\,\psi\,N(\psi)
  -\bigl(\frac 1 2 N(w)-J\cdot N\bigr)|\psi|^2.
\end{equation}
We now compute the boundary terms separately on each boundary component. All boundary fluxes at \(\EventHorizon\) and \(\NullInfinity\) are understood
with respect to the compactified boundary densities.

\paragraph*{Boundary terms through $\EventHorizon$}

\begin{lemma}\label{lem:boundary-H} 
  For the three triplets $(X_{ax},w_{ax},J_{ax})$, $(X_{red},0,J_{red})$, and $(0,w_{\That},0)$, we have
  \begin{align*}
    \JCurrent{(X_{ax},w_{ax},J_{ax})}[\psi]\cdot (-N_{\EventHorizon})
    \ges{}&-|\HawkingHorizon\psi|^2-\rhoH^2\Bigl(|\rhoH\partial_{\rhoH}\psi|^2+|\NablaAngular\psi|^2+|\psi|^2\Bigr) ,\\
    \JCurrent{(X_{red},0,J_{red})}[\psi]\cdot (-N_{\EventHorizon})
    \ges{}&
            \rhoH^{1-\delta_1}\Bigl(|\rhoH\partial_{\rhoH}\psi|^2+ |\HawkingHorizon\psi|^2+|\NablaAngular\psi|^2+|\psi|^2\Bigr),\\
    \JCurrent{(0,w_{\That},0)}[\psi]\cdot (-N_{\EventHorizon})
    \ges{}& -\rhoH^2|\psi||\HawkingHorizon\psi|.
  \end{align*}
  Moreover the energy multiplier $\That_\chi$ has flux through $\HH^+$:
  \begin{equation*}
    \JCurrent{(\That_\chi,0,0)}[\psi]\cdot (-N_{\EventHorizon})
    \ges{} |\HawkingHorizon\psi|^2+\rhoH^2\Bigl(|\rhoH\partial_{\rhoH}\psi|^2+|\NablaAngular\psi|^2\Bigr).
  \end{equation*}
\end{lemma}
\begin{proof}
  Using \zcref{eq:EMtensor-rhoH-H} and \zcref{eq:axi-multipliers:boundary-behavior:horizon} we have the lower bound
  \begin{align*}
    \JCurrent{(X_{ax},w_{ax},J_{ax})}[\psi]\cdot (-N_{\EventHorizon})
    ={}& |q|^{-2}\JCurrent{(X_{ax}, w_{ax}, J_{ax})}[\psi] \c \HawkingHorizon \\
    ={}&|q|^{-2}\EMTensor(X_{ax},\HawkingHorizon)[\psi]
         + |q|^{-2}J_{ax}\cdot \HawkingHorizon|\psi|^2\\
    \ges{}&-|\HawkingHorizon\psi|^2-\rhoH^2\Bigl(|\rhoH\partial_{\rhoH}\psi|^2+|\NablaAngular\psi|^2+|\psi|^2\Bigr)
  \end{align*}
  as stated.
  Using \zcref{eq:EMtensor-rhoH-H} and \eqref{eq:def-XHH-JHH},
  \begin{align*}
    \JCurrent{(X_{red},0,J_{red})}[\psi]\cdot (-N_{\EventHorizon})
    ={}&
         |q|^{-2}\EMTensor( X_{red},\HawkingHorizon)[\psi]
         +|q|^{-2}c_{red} J_{red}\cdot \HawkingHorizon\,|\psi|^2\\
    ={}&
         \frac{1}{2}|q|^{-2}\rhoH^{1-\delta_1}
         \Bigl(
         |\rhoH\partial_{\rhoH}\psi|^2+|\NablaAngular\psi|^2
         \Bigr)
         +800|q|^{-2}r_+^{-2}\rhoH^{1-\delta_1}|\HawkingHorizon\psi|^2\\
       &+\frac12|q|^{-2}(1-\de_1)\rhoH^{1-\delta_1}|\psi|^2
         -\rhoH^{2-\delta_1}\bigl(\bDiffH\psi\bigr)^2\\
    \ges{}&
            \rhoH^{1-\delta_1} \Bigl(|\rhoH\partial_{\rhoH}\psi|^2+ |\HawkingHorizon\psi|^2 + |\NablaAngular\psi|^2 + |\psi|^2 \Bigr),
  \end{align*}
  as stated.  From \eqref{eq:def-wThat}, we have
  \begin{align*}
    \JCurrent{(0,w_{\That},0)}[\psi]\cdot (-N_{\EventHorizon})
    ={}&  w_{\That}\psi |q|^{-2}\HawkingHorizon(\psi)
         -\frac12 |q|^{-2}\HawkingHorizon(w_{\That})|\psi|^2 \ges{} -\rhoH^2|\psi||\HawkingHorizon\psi|,
  \end{align*}
  as stated.
  Finally, for $\That_\chi=\frac{1} {r_+^2+a^2}\HawkingHorizon$ along $\HH$, using \zcref{eq:EMtensor-rhoH-H}
  \begin{align*}
    \JCurrent{(\That_\chi,0,0)}[\psi]\cdot (-N_{\EventHorizon})
    ={}&\EMTensor(\frac{1}{r_+^2+a^2}\HawkingHorizon, \frac{1}{|q|^2}\HawkingHorizon)\ges |\HawkingHorizon\psi|^2+\rhoH^2\Bigl(|\rhoH\partial_{\rhoH}\psi|^2+|\NablaAngular\psi|^2\Bigr),
  \end{align*}
  as stated.
\end{proof}

As a direct consequence of \zcref{lem:boundary-H}, for a sufficiently large $C_{\That}$, we have
\begin{align}\label{eq:bound-PP-horizon}
  \begin{split}
    &\int_{\EventHorizon(\tau_1,\tau_2)}    \JCurrent{(C_{\That} \That_{\chi} + X_{\Mor}^\diff, w_{\Mor}^\diff, J_{\Mor}^\diff)}[\psi]
      \cdot \left( -N_{\EventHorizon} \right)\\
    \gtrsim{}&   \int_{\EventHorizon(\tau_1,\tau_2)} (C_{\That}-1)\abs*{\HawkingHorizon\psi}^2+ \rhoH^{1-\delta_1}\Bigl(|\rhoH\partial_{\rhoH}\psi|^2+ |\HawkingHorizon\psi|^2+|\NablaAngular\psi|^2+|\psi|^2\Bigr)  -\rhoH^2\bigl(\bDiffH\psi\bigr)^2\\
    \gtrsim{}&
               \int_{\EventHorizon(\tau_1,\tau_2)}\abs*{\HawkingHorizon\psi}^2
               +  \rhoH^{1-\delta_1}\Bigl(|\rhoH\partial_{\rhoH}\psi|^2+|\NablaAngular\psi|^2+|\psi|^2\Bigr).
  \end{split}
\end{align}

\paragraph*{Boundary terms through $\NullInfinity$}

\begin{lemma}\label{lem:boundary-I}
  For the three triplets $(X_{ax},w_{ax},J_{ax})$, $(X_{red},0,J_{red})$, and $(0,w_{\That},0)$, we have
  \begin{align*}
    \JCurrent{(X_{ax},w_{ax},J_{ax})}[\psi]\cdot (-N_{\NullInfinity})
    \gtrsim{}& -\abs*{\InfinityHawking\psi}^2- \rhoI^2\left( \abs*{\rhoI\partial_{\rhoI}\psi}^2 + \abs*{\NablaAngular\psi}^2+|\psi|^2 \right),  
    \\
    \JCurrent{(X_{red},0,J_{red})}[\psi]\cdot (-N_{\NullInfinity})
    ={}&0,\\
    \JCurrent{(0,w_{\That},0)}[\psi]\cdot (-N_{\NullInfinity})
    \ges{}& -\rhoI^3|\psi||\InfinityHawking\psi|.
  \end{align*}
  Moreover the energy multiplier $\That_\chi$ has flux through $\II^+$:
  \begin{equation*}
    \JCurrent{(\That_{\chi}, 0, 0)}[\psi]\cdot(-N_{\NullInfinity})
    \gtrsim{} \abs*{\InfinityHawking\psi}^2+ \rhoI^2\left(  |\rhoI\partial_{\rhoI}\psi|^2
      + \abs*{\NablaAngular\psi}^2
    \right).
  \end{equation*}
\end{lemma}
\begin{proof}
  Using \zcref{eq:T-outgoing} and \zcref{eq:axi-multipliers:boundary-behavior:null-infinity}, we have
  \begin{align*}
    \JCurrent{(X_{ax},w_{ax},J_{ax})}[\psi]\cdot (-N_{\NullInfinity})
    ={}& \JCurrent{(X_{ax}, w_{ax}, J_{ax})}[\psi]\cdot \InfinityHawking \\
    ={}& \EMTensor(X_{ax},\InfinityHawking)[\psi]+ w_{ax}\psi \InfinityHawking\psi\\
    \gtrsim{}& -\abs*{\InfinityHawking\psi}^2- \rhoI^2\left( \abs*{\rhoI\partial_{\rhoI}\psi}^2 + \abs*{\NablaAngular\psi}^2+|\psi|^2 \right)
  \end{align*}
  as stated.
  Since $(X_{red},0,J_{red})$ is supported away from $\II^+$, the second identity is trivially satisfied.
  From  \eqref{eq:def-wThat}, we have
  \begin{align*}
    \JCurrent{(0,c_{\That}w_{\That},0)}[\psi]\cdot (-N_{\NullInfinity})
    ={}&  c_{\That}w_{\That}\psi \InfinityHawking(\psi)
         -\frac12 c_{\That}\InfinityHawking(w_{\That})|\psi|^2 \ges{} -\rhoI^3|\psi||\InfinityHawking\psi|,
  \end{align*}
  as stated.
  Finally, for $\That_{\chi} = \partial_u$ along $\NullInfinity$,
  \begin{align*}
    \JCurrent{(\That_{\chi}, 0, 0)}[\psi]\cdot(-N_{\NullInfinity})
    ={}& \EMTensor\left(\partial_u, \InfinityHawking  \right)[\psi]= \EMTensor\left( \InfinityHawking, \InfinityHawking \right)
         - a\rhoI^2\EMTensor\left(aT + \Phi,  \InfinityHawking\right)\\
    \gtrsim{}& \abs*{\InfinityHawking\psi}^2+ \rhoI^2\left(  |\rhoI\partial_{\rhoI}\psi|^2
               + \abs*{\NablaAngular\psi}^2
               \right),
  \end{align*}
  as stated.
\end{proof}

As a direct consequence of \zcref{lem:boundary-I}, for a sufficiently large $C_{\That}$, we have
\begin{align}\label{eq:bound-P-II}
  \begin{split}
    & \int_{\NullInfinity(\tau_1,\tau_2)}    \JCurrent{(C_{\That} \That_{\chi} + X_{\Mor}^\diff, w_{\Mor}^\diff, J_{\Mor}^\diff)}[\psi]
      \cdot \left( -N_{\NullInfinity} \right)\\
    \gtrsim{}& \int_{\NullInfinity(\tau_1,\tau_2)} (C_{\That}-1)\abs*{\InfinityHawking\psi}^2+(C_{\That}-1) \rhoI^2\left(  |\rhoI\partial_{\rhoI}\psi|^2
               + \abs*{\NablaAngular\psi}^2
               \right)  - \conormalSpaceI{2}(\Manifold)\abs*{\psi}^2\\
    \gtrsim{}&
               \int_{\NullInfinity(\tau_1,\tau_2)}\abs*{\InfinityHawking\psi}^2
               + \rhoI^{2}\left(\abs*{\rhoI\partial_{\rhoI}\psi}^2 + \abs*{\NablaAngular\psi}^2 \right),
  \end{split}
\end{align}
where in the last step, we used the fact that
\begin{align*}
  \int_{\NullInfinity(\tau_1,\tau_2)}\conormalSpaceI{2}(\Manifold)\abs*{\psi}^2
  \les{}& \int_{\tau_1}^{\tau_2}\int_{\Sphere^2}\abs*{\psi}^2(\tau, \rhoI=0, \omega) d\tau d\omega= 0,
\end{align*}
where we used that by regularity assumption $\lim_{r\to \infty} r \psi^2=0$.

\paragraph*{Boundary terms through $\Sigma(\tau)$}

\begin{lemma}\label{lem:boundary-Sigma}
  For the three triplets $(X_{ax},w_{ax},J_{ax})$,
  $(X_{red},0,J_{red})$, and $(0,w_{\That},0)$, we
  have
  \begin{align*}
    \Big| \JCurrent{(X_{ax},w_{ax},J_{ax})}[\psi]\cdot N_{\Sigma(\tau)}\Big|
    \les{}& \JCurrent{(\That_\chi)}[\psi]\cdot N_{\Sigma(\tau)},\\
    \JCurrent{(X_{red},0,J_{red})}[\psi]\cdot N_{\Sigma(\tau)}
    \ges{}& \mathds{1}_{\{\rhoH < \frac{r_{+}}{80}\}} \rhoH^{1-\de_1}\Big(
            \abs*{\partial_{\rhoH}\psi}^2
            + \abs*{\NablaAngular\psi}^2
            + |\HawkingHorizon\psi|^2 +|\psi|^2\Big)
            ,\\
    \JCurrent{(0,w_{\That},0)}[\psi]\cdot N_{\Sigma(\tau)}
    \in{}&\conormalSpaceH{1}(\Manifold)\left(\abs*{\rhoH\partial_{\rhoH}\psi}\abs*{\psi} + \abs*{\psi}^2\right)
           \cap{}\conormalSpaceI{2}(\Manifold)\left( \abs*{\psi}\abs*{\rhoI\partial_{\rhoI}\psi} + \abs*{\psi}^2 \right).
  \end{align*}
\end{lemma}
\begin{proof}
  We first consider $\JCurrent{(X_{ax},w_{ax},J_{ax})}[\psi]\cdot N_{\Sigma(\tau)}$. Near $\EventHorizon$,
  from
  \zcref{lemma:boundary-terms-ieF} and \eqref{eq:relation-BL-iEF}, we obtain
  \begin{align*}
    | \PP^{(X_{ax}, w_{ax}, J_{ax})}[\psi] \c N_{\Sigma_{\EventHorizon}(\tau)}|
    &\les |\rhoH \pr_{\rhoH} \psi|^2+ |\HawkingHorizon\psi|^2+|\NablaAngular \psi|^2 +|\psi|^2.     
  \end{align*}
  Near $\NullInfinity$, using that $\abs*{q}^2N_{\Sigma_\NullInfinity}
  = - \partial_{\rhoI} + \conormalSpaceI{0}\bDiffI$,
  and using
  \zcref{eq:axi-multipliers:boundary-behavior:null-infinity} and
  \zcref{eq:T-outgoing}, we have
  \begin{align}
    \EMTensor(X_{ax}, \abs*{q}^2N_{\Sigma_{\NullInfinity}})
    ={}& \conormalSpaceI{0}\EMTensor(\rhoI\partial_{\rhoI}, \rhoI\partial_{\rhoI})
         +\conormalSpaceI{0}\EMTensor(\InfinityHawking, \partial_{\rhoI})
         + \conormalSpaceI{1}\abs*{\bDiffI\psi}^2
         \notag\\
    \les{}& \abs*{\rhoI\partial_{\rhoI}\psi}^2 + \abs*{\NablaAngular\psi}^2 + \conormalSpaceI{1}\abs*{\bDiffI\psi}^2.
            \label{eq:axi-multipliers:boundary-behavior:sigma:infinity:X}
  \end{align}
  We also have that
  \begin{align}
    \JCurrent{(0, w_{ax}, J_{ax})}[\psi]\cdot \abs*{q}^2N_{\Sigma_{\NullInfinity}}
    ={}& w_{ax}\psi \abs*{q}^2N_{\Sigma_{\NullInfinity}}\psi
         - \abs*{q}^2(\frac 1 2 N_{\Sigma_{\NullInfinity}}w_{ax} - J\cdot N_{\Sigma_{\NullInfinity}})\abs*{\psi}^2
         \notag\\
    \les{}& \abs*{\psi}\abs*{\rhoI\partial_{\rhoI}\psi}
            + \abs*{\psi}^2.
            \label{eq:axi-multipliers:boundary-behavior:sigma:infinity:w-J}
  \end{align}
  Summing together
  \zcref{eq:axi-multipliers:boundary-behavior:sigma:infinity:X,eq:axi-multipliers:boundary-behavior:sigma:infinity:w-J},
  we have that
  \begin{equation*}
    \big|  \JCurrent{(X_{ax}, w_{ax}, J_{ax})}[\psi]\cdot \abs*{q}^2N_{\Sigma_{\NullInfinity}}\big|
    \les (\bDiffI^{\le 1}\psi)^2.
  \end{equation*}
  Dividing both sides by $\abs*{q}^2$ and combining with the bound near $\EventHorizon$ and the trivial bound in the compact region, we deduce the stated bound from \zcref{lemma:boundary-terms-energy}.
  From \zcref{eq-boundary-terms-hierarchy-H} and \eqref{eq:boundary-J-HH} we deduce 
  \begin{align*}
    \JCurrent{( X_{red})}[\psi]\cdot N_{\Sigma_{\rhoH\le \rho_0}(\tau)}
    \ges{} & \mathds{1}_{\{\rhoH < \frac{r_{+}}{80}\}}  \rhoH^{1-\de_1}\Big(
             \abs*{\partial_{\rhoH}\psi}^2
             + \abs*{\NablaAngular\psi}^2
             + |\HawkingHorizon\psi|^2 +\abs*{\psi}^2\Big).
  \end{align*}
  From \zcref{eq:def-wThat}, since
  \begin{equation*}
    2w_{\That}= - \frac{(r_+ - r_{\trap})^2}{r_+^7}\rhoH^2 + \conormalSpaceH{3}(\Manifold),
    \text{ near } \EventHorizon,
    \qquad
    2 w_{\That}= - \frac{1}{r^3} + \conormalSpaceI{4}(\Manifold),
    \text{ near } \NullInfinity.
  \end{equation*}
  applying \eqref{eq:boundary-identity-general} with $N=N_{\Sigma(\tau)}$, we obtain
  \begin{align*}
    \JCurrent{(0,c_{\That}w_{\That},0)}[\psi]\cdot N_{\Sigma}
    \in{}\conormalSpaceH{1}(\Manifold)\left(\abs*{\rhoH\partial_{\rhoH}\psi}\abs*{\psi} + \abs*{\psi}^2\right)
    \cap{}\conormalSpaceI{2}(\Manifold)\left( \abs*{\psi}\abs*{\rhoI\partial_{\rhoI}\psi} + \abs*{\psi}^2 \right),
  \end{align*}
  as stated.
\end{proof}

As a direct consequence of
\zcref{lem:boundary-Sigma, lemma:boundary-terms-energy}, for a
sufficiently large $C_{\That}$ and for $\frac{|a|}{M}\ll 1$, we have
\begin{align}\label{eq:bound-Sigma-f}
  \begin{split}
    & \int_{\Sigma(\tau)}  \JCurrent{(C_{\That} \That_{\chi} + X_{\Mor}^\diff, w_{\Mor}^\diff, J_{\Mor}^\diff)}[\psi]\cdot N_{\Sigma(\tau)}\\
    \gtrsim{}& \int_{\Sigma(\tau)}(C_{\That}-1)\JCurrent{(\That_\chi)}[\psi]\cdot N_{\Sigma(\tau)}+\int_{\Sigma(\tau)}\mathds{1}_{\{\rhoH < \frac{r_{+}}{80}\}} \rhoH^{1-\de_1}\Big(
               \abs*{\partial_{\rhoH}\psi}^2
               + \abs*{\NablaAngular\psi}^2
               + |\HawkingHorizon\psi|^2 +|\psi|^2\Big)\\
    &-\int_{\Sigma(\tau)}\conormalSpaceH{1}(\Manifold)\left(\abs*{\rhoH\partial_{\rhoH}\psi}\abs*{\psi} + \abs*{\psi}^2\right)
      \cap{}\conormalSpaceI{2}(\Manifold)\left( \abs*{\psi}\abs*{\rhoI\partial_{\rhoI}\psi} + \abs*{\psi}^2 \right)\\
    \gtrsim{}& \norm*{\psi}_{\bSobolevHI{1, 0, -1}(\Sigma(\tau))}^2+\norm*{\psi}_{H_{\EventHorizon}^{1, -\frac{1-\de_1}{2}}(\Sigma(\tau))}^2\\
    \gtrsim{}&  \norm*{\psi}_{H_{\EventHorizon, b\NullInfinity}^{1, -\frac{1-\de_1}{2}, -1}(\Sigma(\tau))}^2.
  \end{split}
\end{align}

\subsubsection{Proof of \texorpdfstring{\zcref{prop:energy-Morawetz}}{}}\label{sec:proof-e-mor}

We are now ready to prove \zcref{prop:energy-Morawetz}. 
Consider the  multiplier
\[
  X = C_{\That}\That_{\chi}+X_{\Mor}^\diff+a\widetilde X,
  \qquad
  w = w_{\Mor}^\diff+a\widetilde w,
  \qquad
  J = J_{\Mor}^\diff ,
\]
where $(\widetilde X,\widetilde w)$ are the pseudodifferential correction terms constructed in
\zcref{lemma:ILED-KdS:Bulk}, after the cutoff procedure of
\zcref{eq:Morawetz:morawetz-VF-LagrangeCorr-with-cutoff:def}.
Applying the divergence identity for the differential part together with
the pseudodifferential divergence theorem \zcref{eq:PsiDO-divergence-theorem}, we obtain
\begin{align*}
  & \bangle*{\Box_{\Metric}\psi, (X+w)\psi}_{L^2(\Manifold(\tau_1,\tau_2))}\\
  ={}& \bangle*{\Box_{\Metric}\psi, (C \That_{\chi} + X_{\Mor}^\diff +w_{\Mor}^\diff)\psi}_{L^2(\Manifold(\tau_1,\tau_2))}+\bangle*{\Box_{\Metric}\psi, a (\widetilde{X}+ \widetilde{w})\psi}_{L^2(\Manifold(\tau_1,\tau_2))}\\
  ={}& \int_{\mathcal{M}(\tau_1,\tau_2)}\D\cdot\JCurrent{(C \That_{\chi} + X_{\Mor}^\diff, w_{\Mor}^\diff, J_{\Mor}^\diff)}[\psi]
       - \int_{\mathcal{M}(\tau_1,\tau_2)}\QQ^{(C \That_{\chi} + X_{\Mor}^\diff, w_{\Mor}^\diff, J_{\Mor}^\diff)}[\psi]\\
  &+a\KCurrentPert{\widetilde{\MorawetzVF}, \widetilde{\MorawetzLagrangeCorr}}[\psi](\tau_1,\tau_2)
    - a\JCurrentPert{\widetilde{\MorawetzVF}, \widetilde{\MorawetzLagrangeCorr}}[\psi](\tau_2)
    + a\JCurrentPert{\widetilde{\MorawetzVF}, \widetilde{\MorawetzLagrangeCorr}}[\psi](\tau_1).
\end{align*}
Using the divergence theorem in \zcref{eq:general-divergence-theorem}
we then obtain
\begin{equation*}
  \begin{split}
    & P_{\Mor}(\tau_1,\tau_2)
      +\int_{\mathcal{M}(\tau_1,\tau_2)}\QQ^{(C \That_{\chi} + X_{\Mor}^\diff, w_{\Mor}^\diff, J_{\Mor}^\diff)}[\psi]
      +a\KCurrentPert{\widetilde{X}, \widetilde{w}}[\psi](\tau_1,\tau_2)\\
    ={}& -\bangle*{\Box_{\Metric}\psi, (X+w)\psi}_{L^2(\Manifold(\tau_1,\tau_2))}, 
  \end{split}
\end{equation*}
where
\begin{equation*}
  \begin{split}
    P_{\Mor}(\tau_1, \tau_2)
    \vcentcolon={}& P_{\Mor, \Sigma}(\tau_2)
                    - P_{\Mor, \Sigma}(\tau_1)
                    + P_{\Mor, \NullInfinity}(\tau_1,\tau_2)
                    + P_{\Mor, \EventHorizon}(\tau_1,\tau_2),\\
    P_{\Mor,\Sigma}(\tau)
    \vcentcolon={}&  \int_{\Sigma(\tau)}\JCurrent{(C_{\That} \That_{\chi} + X_{\Mor}^\diff, w_{\Mor}^\diff, J_{\Mor}^\diff)}[\psi]\cdot N_{\Sigma}
                    + a \JCurrentPert{\widetilde{X}, \widetilde{w}}[\psi](\tau),\\
    P_{\Mor, \NullInfinity}(\tau_1,\tau_2)
    \vcentcolon={}& \int_{\NullInfinity(\tau_1,\tau_2)}\JCurrent{(C_{\That} \That_{\chi} + X_{\Mor}^\diff, w_{\Mor}^\diff, J_{\Mor}^\diff)}[\psi]\cdot (-N_{\NullInfinity}),\\
    P_{\Mor, \EventHorizon}(\tau_1,\tau_2)
    \vcentcolon={}& \int_{\EventHorizon(\tau_1,\tau_2)}\JCurrent{(C_{\That} \That_{\chi} + X_{\Mor}^\diff, w_{\Mor}^\diff, J_{\Mor}^\diff)}[\psi]\cdot (-N_{\EventHorizon}).
  \end{split}    
\end{equation*}
We first control the boundary terms. Observe that from
\zcref{eq:bound-PP-horizon}, \eqref{eq:bound-P-II}, \eqref{eq:bound-Sigma-f} and \eqref{eq:JCurrentPert-def} we have that for $\frac{\abs*{a}}{M}\ll1 $, 
\begin{align}
  \label{eq:energy-morawetz:proof:boundary-terms}
  \begin{split}
    P_{\Mor,\EventHorizon}(\tau_1,\tau_2) &\gtrsim \int_{\EventHorizon(\tau_1,\tau_2)}\abs*{\HawkingHorizon\psi}^2
                                            +  \rhoH^{1-\delta_1}\Bigl(|\rhoH\partial_{\rhoH}\psi|^2+|\NablaAngular\psi|^2+|\psi|^2\Bigr)\\
    P_{\Mor, \NullInfinity}(\tau_1,\tau_2) &\gtrsim \int_{\NullInfinity(\tau_1,\tau_2)}\abs*{\InfinityHawking\psi}^2
                                             + \rhoI^{2}\left(\abs*{\rhoI\partial_{\rhoI}\psi}^2 + \abs*{\NablaAngular\psi}^2 \right)\\
    P_{\Mor, \Sigma}(\tau_2)
                                          &\gtrsim \norm*{\psi}_{H_{\EventHorizon, b\NullInfinity}^{1, -\frac{1-\de_1}{2}, -1}(\Sigma(\tau_2))}^2. 
  \end{split}
\end{align}

Observe that we also have the trivial control 
\begin{equation}
  \label{eq:energy-morawetz:proof:boundary-terms-Sigma1}
  P_{\Mor,\Sigma}(\tau_1) \les \norm*{\psi}_{H_{\EventHorizon, b\NullInfinity}^{1, -\frac{1-\de_1}{2}, -1}(\Sigma(\tau_1))}^2.
\end{equation}

With regards to the bulk terms, we have from \zcref{eq:KCurrent-princ-aux:decomposition} that
\begin{equation*}
  \begin{split}
    &\int_{\mathcal{M}(\tau_1,\tau_2)}\KCurrent{(C_{\That}\That_{\chi}+X_{\Mor}^\diff, w_{\Mor}^\diff, J_{\Mor}^\diff)}[\psi]
      +a\KCurrentPert{\widetilde{X}, \widetilde{w}}[\psi](\tau_1,\tau_2)\\
    ={}&  \int_{\mathcal{M}(\tau_1,\tau_2)}\KCurrent{(C_{\That}\That_{\chi}+X_{\Mor}^\diff, w_{\Mor}^\diff, J_{\Mor}^\diff)}[\psi]
         +a\KCurrentPert{\widetilde{\MorawetzVF},\widetilde{\MorawetzLagrangeCorr}}_{\Main}[\psi]
         +a\KCurrentPert{\widetilde{\MorawetzVF},\widetilde{\MorawetzLagrangeCorr}}_{\Aux}[\psi].
  \end{split}    
\end{equation*}
Combining 
\zcref{lemma:Morawetz:cutoff-PsiDO:bulk, lemma:That-chi-deformation-tensor}  we have that for some $\delta_{*}>0$
\begin{equation*}
  \begin{split}
    &\int_{\mathcal{M}(\tau_1,\tau_2)}\KCurrent{(C_{\That}\That_{\chi}+X_{\Mor}^\diff, w_{\Mor}^\diff, J_{\Mor}^\diff)}[\psi]
      +a\KCurrentPert{\widetilde{X}, \widetilde{w}}[\psi](\tau_1,\tau_2)\\
    \ge{}& \delta_{*} \norm*{\psi}_{\bSobolevHITrap{1,\frac{\delta_1}{2},-\frac 52}(\Manifold(\tau_1,\tau_2))}^2-O(a) \norm*{\psi}_{\bSobolevH{1, -\frac 1 2}(\Manifold_{\rhoH \leq \frac 3 2 A_1}(\tau_1, \tau_2))}^2
    \\
    &- O(a)\left(\norm*{\psi}_{\bSobolevHITrap{1,\frac{\delta_1}{2},-\frac{5}{2}}(\Manifold(\tau_1,\tau_2))}^2+
      \norm*{\psi}_{H^0_c(\Manifold(\tau_1,\tau_2))}^2
      + \norm*{F }_{L^2(\Manifold(\tau_1,\tau_2))}^2
      + \sup_{[\tau_1,\tau_2]}\norm*{\psi}_{\bSobolevHI{1, 0, -1}(\Sigma(\tau))}^2
      \right)
  \end{split}
\end{equation*}
For $\frac{|a|}{M} \ll \de_*$ sufficiently small, we have that
\begin{equation}
  \label{eq:energy-morawetz:proof:bulk-terms}
  \begin{split}
    &\int_{\mathcal{M}(\tau_1,\tau_2)}\KCurrent{(C_{\That}\That_{\chi}+X_{\Mor}^\diff, w_{\Mor}^\diff, J_{\Mor}^\diff)}[\psi]
      +a\KCurrentPert{\widetilde{X}, \widetilde{w}}[\psi](\tau_1,\tau_2)\\
    \ge{}& \delta_{*} \norm*{\psi}_{\bSobolevHITrap{1,\frac{\delta_1}{2},-\frac 52}(\Manifold(\tau_1,\tau_2))}^2
           - O(a)\left(
           \norm*{F }_{L^2(\Manifold(\tau_1,\tau_2))}^2
           + \sup_{[\tau_1,\tau_2]}\norm*{\psi}_{\bSobolevHI{1, 0, -1}(\Sigma(\tau))}^2
           \right).
  \end{split}     
\end{equation}
Finally, we control 
\begin{equation}
  \label{eq:energy-morawetz:proof:inhomogeneity-terms}
  \bangle*{\Box_{\Metric}\psi, (X+w)\psi}_{L^2(\Manifold(\tau_1,\tau_2))}
  \les \varepsilon \norm*{\psi}_{\bSobolevHITrap{1,\frac{\delta_1}{2}, -\frac{5}{2}}(\Manifold(\tau_1,\tau_2))}^2
  + \frac{1}{\varepsilon}\norm*{F}_{\bSobolevHI{0,\frac{\delta_1}{2}, -\frac{1}{2}}(\Manifold(\tau_1,\tau_2))}^2, 
\end{equation}
using the fact that
\begin{equation*}
  X = -\rhoH^{-\delta_1}\rhoH\partial_{\rhoH} + \conormalSpaceH{0}\bDiffH \text{ near } \EventHorizon,\qquad
  X = - \rhoI^{2}\partial_{\rhoI} + \conormalSpaceI{3}\bDiffI \text{ near }\NullInfinity.
\end{equation*}
Combining
\zcref{eq:energy-morawetz:proof:inhomogeneity-terms,eq:energy-morawetz:proof:bulk-terms,eq:energy-morawetz:proof:boundary-terms-Sigma1,eq:energy-morawetz:proof:boundary-terms}
concludes the proof of \zcref{prop:energy-Morawetz}.

\section{Proof of the main theorem}\label{sec:proof-main-theorem}

In this section we prove \zcref{thm:main-thm}, first for $s=1$ in
\zcref{proof:main-them-s1} and then for $s \geq 2$ in
\zcref{proof:main-them-s2}. Finally, in \zcref{proof:corollary-decay}
we prove \zcref{cor:pointwise-decay}.

\subsection{Proof of Theorem \ref{thm:main-thm} for \texorpdfstring{$s=1$}{s=1}}\label{proof:main-them-s1}

Here we prove \zcref{eq:main-theorem-s=1}, which implies
\zcref{thm:main-thm} for $s=1$, see \zcref{rem:main-theorem-s=1}. The proof is obtained by combining the
Energy-Morawetz estimates of \zcref{prop:energy-Morawetz} with the
$\EventHorizon$ and $\NullInfinity$-weighted hierarchies of
\zcref{sec:hierarchy}. Since the Morawetz bulk of
\zcref{prop:energy-Morawetz} has weight $\frac{\de_1}{2}$ at the
horizon, the only norms from the $\EventHorizon$-weighted hierarchies
which are stronger than the Morawetz bulk are the ones with weight
$-\frac{\a}{2}$ for $\alpha \in (-1+4\mathbf{a}^2, -\de_1)$,
restricting therefore the range of the hierarchy roughly in half.

Sum \zcref{eq:final-theorem-energy-morawetz},
multiplied by a large constant $C_{Mor}$, with the
$\EventHorizon$-weighted hierarchy
\zcref{eq:horizon-weighted-estimate} for $s=1$ and 
$\alpha \in (-1+4\mathbf{a}^2, -\de_1)$, and the $\NullInfinity$-weighted hierarchy
\zcref{eq:infinity-weighted-estimate-psic} for $s=1$ and $\beta \in (1,3)$. Therefore for an extremal Kerr--Newman with $\mathbf{a}\ll 1$,
$\delta_1\ll 1$, $-1+4\mathbf{a}^2<\alpha< -\de_1$,
$1<\b<3$ and any $\tau_1<\tau_2$ we have
\begin{align*}
  &  C_{Mor} \norm*{\psi}_{H_{\EventHorizon, \operatorname{b}\NullInfinity}^{1, -\frac{1-\de_1}{2}, -1}(\Sigma(\tau_2))}+ \norm*{\psi}_{H_{\EventHorizon}^{1, -\frac{\alpha+1}{2}}(\Sigma_{\EventHorizon}(\tau_2))}+\norm*{\widecheck{\psi}}_{H_{\NullInfinity}^{1, -\frac{\b+3}{2}}(\Sigma_{\NullInfinity}(\tau_2))}\\
  &
    +C_{Mor}\norm*{\psi}_{\bSobolevHITrap{1, \frac{\de_1}{2}, -\frac 52}(\Manifold(\tau_1, \tau_2))}+\norm*{\psi}_{\bSobolevH{1,\,-\frac{\alpha}{2}}(\Manifold_{\EventHorizon}(\tau_1,\tau_2))}+   \norm*{\widecheck{\psi}}_{\bSobolevI{1,-\frac{\b+2}{2}}(\Manifold_{\NullInfinity}(\tau_1,\tau_2))}\\
  &+  \norm*{\psi}_{\bSobolevH{1, - \frac{\alpha+1}{2}}(\EventHorizon(\tau_1,\tau_2))} +\norm*{\widecheck{\psi}}_{\bSobolevI{1, -\frac{\b+1}{2}}(\NullInfinity(\tau_1, \tau_2))}\\
  \les{}&  C_{Mor} \norm*{\psi}_{H_{\EventHorizon, \operatorname{b}\NullInfinity}^{1, -\frac{1-\de_1}{2}, -1}(\Sigma(\tau_1))}+\norm*{\psi}_{H_{\EventHorizon}^{1, -\frac{\alpha+1}{2}}(\Sigma_{\rhoH \leq 2 \rho_0}(\tau_1))}+\norm*{\widecheck{\psi}}_{H_{\NullInfinity}^{1, -\frac{\b+3}{2}}(\Sigma_{\rhoI \leq 2\rho_1}(\tau_1))}\\
  &+C_{Mor}\norm*{F}_{\bSobolevHI{0,\frac{\delta_1}{2}, -\frac{1}{2}}(\Manifold(\tau_1,\tau_2))} + \norm*{F}_{\bSobolevH{0,-\frac{\alpha}{2}}(\Manifold_{\rhoH \leq 2\rho_0}(\tau_1,\tau_2))}+ \norm*{F}_{\bSobolevI{0,-\frac{\b-4}{2}}(\Manifold_{\rhoI \leq 2 \rho_1}(\tau_1,\tau_2))}\\
  &+ \rho_0^{-\frac 12 }\norm*{\psi}_{\compactSobolev{1}(\Manifold_{ \rhoH \leq 2\rho_0}(\tau_1,\tau_2))}+ \rho_1^{-\frac12}\norm*{\widecheck{\psi}}_{\compactSobolev{1}(\Manifold_{\rhoI \leq 2\rho_1}(\tau_1,\tau_2))}.
\end{align*}         
For $C_{Mor} \gg 1$ sufficiently large, the trapped Morawetz bulk on the second line on the left hand side can absorb the last line on the right hand side, which is supported away from trapping. For $-1+4\mathbf{a}^2<\alpha< -\de_1$,  $1<\b<3$, we can bound the bulk norms:
\begin{align*}
  & C_{Mor}\norm*{\psi}_{\bSobolevHITrap{1, \frac{\de_1}{2}, -\frac 52}(\Manifold(\tau_1, \tau_2))}+\norm*{\psi}_{\bSobolevH{1,\,-\frac{\alpha}{2}}(\Manifold_{\EventHorizon}(\tau_1,\tau_2))}+ \norm*{\widecheck{\psi}}_{\bSobolevI{1,-\frac{\b+2}{2}}(\Manifold_{\NullInfinity}(\tau_1,\tau_2))} \\
  \gtrsim{}& \norm*{\psi}_{\bSobolevTrap{1}(\Manifold(\tau_1, \tau_2))}+\norm*{\psi}_{\bSobolevH{1, -\frac{\a}{2}}(\Manifold(\tau_1, \tau_2))}+ \norm*{\psi}_{\bSobolevI{1,-\frac{\b}{2}}(\Manifold_{\NullInfinity}(\tau_1,\tau_2))}=\norm*{\psi}_{\bSobolevHITrap{1, -\frac{\a}{2}, -\frac \b 2}(\Manifold(\tau_1, \tau_2))}.
\end{align*}
For the energy norms on $\Sigma(\tau)$ we have for $-1+4\mathbf{a}^2<\alpha< -\de_1$, $1<\b<3$, 
\begin{align*}
  &   C_{Mor} \norm*{\psi}_{H_{\EventHorizon, \operatorname{b}\NullInfinity}^{1, -\frac{1-\de_1}{2}, -1}(\Sigma(\tau))}+ \norm*{\psi}_{H_{\EventHorizon}^{1, -\frac{\alpha+1}{2}}(\Sigma_{\EventHorizon}(\tau))}+\norm*{\widecheck{\psi}}_{H_{\NullInfinity}^{1, -\frac{\b+3}{2}}(\Sigma_{\NullInfinity}(\tau))} \\
  \gtrsim{}& \norm*{\psi}_{H_{\EventHorizon}^{1, -\frac{1-\de_1}{2}}(\Sigma_{\EventHorizon}(\tau))}+\norm*{\psi}_{H_{\EventHorizon}^{1, -\frac{\alpha+1}{2}}(\Sigma_{\EventHorizon}(\tau))}+\norm*{\psi}_{\bSobolevI{1, -1}(\Sigma_{\NullInfinity}(\tau))}+\norm*{\widecheck{\psi}}_{H_{\NullInfinity}^{1, -\frac{\b+3}{2}}(\Sigma_{\NullInfinity}(\tau))}+\norm*{\psi}_{\compactSobolev{1}(\Sigma_{\NullInfinity}(\tau))}\\
  \gtrsim{}& \norm*{\psi}_{H_{\EventHorizon}^{1, -\frac{\alpha+1}{2}}(\Sigma_{\EventHorizon}(\tau))}+\norm*{\widecheck{\psi}}_{H_{\NullInfinity}^{1, -\frac{\b+3}{2}}(\Sigma_{\NullInfinity}(\tau))}+\norm*{\psi}_{\compactSobolev{1}(\Sigma_{\NullInfinity}(\tau))}.
\end{align*}
Similarly,
\begin{align*}
  &C_{Mor}\norm*{F}_{\bSobolevHI{0,\frac{\delta_1}{2}, -\frac{1}{2}}(\Manifold(\tau_1,\tau_2))} + \norm*{F}_{\bSobolevH{0,-\frac{\alpha}{2}}(\Manifold_{\rhoH \leq 2\rho_0}(\tau_1,\tau_2))}+ \norm*{F}_{\bSobolevI{0,-\frac{\b-4}{2}}(\Manifold_{\rhoI \leq 2 \rho_1}(\tau_1,\tau_2))} \\
  \les{}& \norm*{F}_{\bSobolevHI{0,-\frac \a 2,-\frac{\b-4}{2}}(\Manifold(\tau_1,\tau_2))}.
\end{align*}
Combining the above and recalling the definition of
$\norm*{\psi}_{\SobolevfinHI{s, -\frac{\alpha+1}{2},
    -\frac{\beta+1}{2}}(\Sigma(\tau))}$, we finally obtain
\zcref{eq:main-theorem-s=1}.

\subsection{Proof of \zcref{thm:main-thm} for \texorpdfstring{$s\geq 2$}{s>1}}\label{proof:main-them-s2}

In this section, we prove \zcref{thm:main-thm} for $s\geq 2$. 
We first recover the higher-order bulk quantities in the following lemma. 
\begin{lemma}
  \label{lemma:higher-order:deg-morawetz-bulk}
  For $s\in \mathbb{N}$, $s\ge 1$, and $\tau_1<\tau_2$, we have the
  following inequality:
  \begin{equation}
    \label{eq:lemma:higher-order:deg-morawetz-bulk}
    \begin{split}
      \norm*{\psi}_{\SobolevDeg{s}(\Manifold(\tau_1,\tau_2))}^2
      \lesssim{}& \norm*{\psi}_{\SobolevfinHI{s+1, -\frac{\alpha+1}{2}, -\frac{\beta+1}{2}}(\Sigma(\tau_1))}^2
                  + \norm*{F}_{\bSobolevHI{s, -\frac{\alpha}{2}, -\frac{\beta-4}{2}}(\Manifold(\tau_1,\tau_2))}^2\\
                & +  O\left(\norm*{\psi}_{\bSobolevH{s,-1}(\EventHorizon(\tau_1,\tau_2))}\norm*{\psi}_{\bSobolevH{s-1,-1}(\EventHorizon(\tau_1,\tau_2))}\right)\\
                & + O\left(\norm*{\psi}_{\bSobolevI{s,-\frac{5}{2}}(\NullInfinity(\tau_1,\tau_2))}\norm*{\psi}_{\bSobolevI{s-1,-\frac{5}{2}}(\NullInfinity(\tau_1,\tau_2))}\right)\\
                & + O\left(
                  \norm*{\psi}_{H_c^s(\Manifold(\tau_1,\tau_2))}\norm*{\psi}_{H_c^{s-1}(\Manifold(\tau_1,\tau_2))}
                  \right).
    \end{split}
  \end{equation}
\end{lemma}
\begin{proof}
  Since $T$ and $\Phi$ are Killing vectorfields, we can apply
  \zcref{thm:main-thm} for $s=1$ to $T\psi$ and $\Phi\psi$ and obtain
  \begin{equation}
    \label{eq:higher-order-morawetz:T-Phi-s1}
    \begin{split}
      \norm*{(T, \Phi)\psi}_{\bSobolevHI{0, -\frac{\a}{2}, -\frac \b 2}(\Manifold(\tau_1, \tau_2))}
      &\les \norm*{\psi}_{\SobolevfinHI{2, -\frac{\alpha+1}{2}, -\frac{\beta+1}{2}}(\Sigma(\tau_1))}
        + \norm*{F}_{\bSobolevHI{1, -\frac{\alpha}{2}, -\frac{\beta-4}{2}}(\Manifold(\tau_1,\tau_2))}.
    \end{split}
  \end{equation}
  To estimate the remaining derivatives, we will use elliptic theory,
  taking advantage of the fact that the stationary portion of
  $\abs*{q}^2\Box_{\Metric}$ is an elliptic $b$-operator with respect
  to both the event horizon and null infinity.

  We first show how to recover the derivatives in a neighborhood of
  the event horizon. Observe that, by integration by parts,
  \begin{equation}
    \label{eq:higher-order-morawetz:horizon:elliptic}
    \begin{split}
      & -\int_{\Manifold_{r\le 4M}(\tau_1,\tau_2)}
        \rhoH^2\psi\left( \left( \rhoH\partial_{\rhoH} \right)^2 + \lapp_{\SSS^2} \right)\psi
      \\
      ={}& \norm*{\rhoH(\rhoH\partial_{\rhoH})\psi}^2_{L^2(\Manifold_{r\le 4M}(\tau_1,\tau_2))}
           + \norm*{\rhoH \NablaAngular\psi}^2_{L^2(\Manifold_{r\le 4M}(\tau_1,\tau_2))}\\
      & + O\left(\norm*{\psi}_{\bSobolevH{1,-1}(\EventHorizon(\tau_1,\tau_2))}\norm*{\psi}_{\bSobolevH{0,-1}(\EventHorizon(\tau_1,\tau_2))}\right)
        + O\left(
        \norm*{\psi}_{\bSobolevH{1,-1}(\Manifold_{r\le 4M}(\tau_1,\tau_2))}\norm*{\psi}_{\bSobolevH{0,-1}(\Manifold_{r\le 4M}(\tau_1,\tau_2))}
        \right)\\
      & + O\left(
        \norm*{\psi}_{H^1_c(\Manifold(\tau_1,\tau_2))}\norm*{\psi}_{H^0_c(\Manifold(\tau_1,\tau_2))}
        \right).
    \end{split}    
  \end{equation}
  We also have that, again by integration by parts,
  \begin{align*}
    \begin{split}
      &\int_{\Manifold_{r\le 4M}(\tau_1,\tau_2)}\rhoH^2\psi\left(2\HawkingHorizon\partial_{\rhoH}+2r\pr_v  + a^{2}\sin^{2}\theta \pr_v^2 + 2a\pr_v\pr_{\phiIEF}\right)\psi
      \\
      \lesssim{}& \int_{\Manifold_{r\le 4M}(\tau_1,\tau_2)}\left( \varepsilon^{-1}\rhoH^2\left(\abs*{T\psi}^2 + \abs*{\Phi\psi}^2+|\psi|^2\right)+\varepsilon \rhoH |\rhoH \partial_{\rhoH}\psi|^2 \right)
                  +\sup_{\tau\in [\tau_1,\tau_2]}\norm*{\psi}_{H_{\EventHorizon}^{1, -\frac{1}{2}}(\Sigma_{r\le 4M}(\tau))}^2.  
    \end{split}
  \end{align*}
  Applying the $s=1$ of \zcref{thm:main-thm}, we see that
  \begin{align}
    \label{eq:higher-order-morawetz:horizon:dv-dphi-terms-s1}
    \begin{split}
      &\int_{\Manifold_{r\le 4M}(\tau_1,\tau_2)}\rhoH^2\psi\left(2\HawkingHorizon\partial_{\rhoH}+2r\pr_v  + a^{2}\sin^{2}\theta \pr_v^2 + 2a\pr_v\pr_{\phiIEF}\right)\psi
      \\
      \lesssim{}& \int_{\Manifold_{r\le 4M}(\tau_1,\tau_2)}\left( \varepsilon^{-1}\rhoH^2\left(\abs*{T\psi}^2 + \abs*{\Phi\psi}^2+|\psi|^2\right)+\varepsilon \rhoH |\rhoH \partial_{\rhoH}\psi|^2 \right)\\
      & + \norm*{\psi}_{\SobolevfinHI{2, -\frac{\alpha+1}{2}, -\frac{\beta+1}{2}}(\Sigma(\tau_1))}^2
        + \norm*{F}_{\bSobolevHI{1, -\frac{\alpha}{2}, -\frac{\beta-4}{2}}(\Manifold(\tau_1,\tau_2))}^2.
    \end{split}
  \end{align}

  We now use the fact that from the form of the wave operator in
  \zcref{wave-iEF} we know that
  \begin{equation}
    \label{eq:higher-order-morawetz:horizon-elliptic-decomposition-s1} 
    \begin{split}
      & \abs*{q}^2\Box_{\Metric}
        - \left(2\HawkingHorizon\partial_{\rhoH}+2r\pr_v  + a^{2}\sin^{2}\theta \pr_v^2 + 2a\pr_v\pr_{\phiIEF}\right)\\
      ={}& \left( \rhoH\partial_{\rhoH} \right)^2 + \rhoH\partial_{\rhoH}
           +\lapp_{\SSS^2},
    \end{split}    
  \end{equation}
  to write that
  \begin{equation}
    \label{eq:higher-order-morawetz:horizon:integrated-wave-decomp-s1}
    \begin{split}
      & -\int_{\Manifold_{r\le 4M}(\tau_1,\tau_2)}
        \rhoH^2\psi\left( \left( \rhoH\partial_{\rhoH} \right)^2 + \rhoH\partial_{\rhoH}+\lapp_{\SSS^2} \right)\psi\\
      ={}& \int_{\Manifold_{r\le 4M}(\tau_1,\tau_2)}
           \rhoH^2\psi\abs*{q}^2\Box_{\Metric}\psi\\
      & -\int_{\Manifold_{r\le 4M}(\tau_1,\tau_2)}
        \rhoH^2\psi\left(2\HawkingHorizon\partial_{\rhoH}+2r\pr_v  + a^2\sin^{2}\theta \pr_v^2 + 2a\pr_v\pr_{\phiIEF}\right)\psi.
    \end{split}   
  \end{equation}
  By combining
  \zcref{eq:higher-order-morawetz:T-Phi-s1, eq:higher-order-morawetz:horizon:integrated-wave-decomp-s1,eq:higher-order-morawetz:horizon:dv-dphi-terms-s1,eq:higher-order-morawetz:horizon:elliptic}, 
  we deduce
  \begin{equation}
    \label{eq:higher-order-morawetz:s=1:preliminary}
    \begin{split}
      \norm*{\rhoH(\rhoH\partial_{\rhoH}, \NablaAngular)\psi}_{L^2(\Manifold_{r\le 4M}(\tau_1,\tau_2))}^2
      \lesssim{}& \norm*{\psi}_{\SobolevfinHI{2, -\frac{\alpha+1}{2}, -\frac{\beta+1}{2}}(\Sigma(\tau_1))}^2
                  + \norm*{F}_{\bSobolevHI{1, -\frac{\alpha}{2}, -\frac{\beta-4}{2}}(\Manifold(\tau_1,\tau_2))}^2
      \\
                & + O\left(\norm*{\psi}_{\bSobolevH{1,-1}(\EventHorizon(\tau_1,\tau_2))}\norm*{\psi}_{\bSobolevH{0,-1}(\EventHorizon(\tau_1,\tau_2))}\right)\\
                & + O\left(
                  \norm*{\psi}_{\bSobolevH{1, -1}(\Manifold(\tau_1,\tau_2))}\norm*{\psi}_{\bSobolevH{0, -1}(\Manifold(\tau_1,\tau_2))}
                  \right)\\
                & + O\left(
                  \norm*{\psi}_{H^1_c(\Manifold(\tau_1,\tau_2))}\norm*{\psi}_{H^0_c(\Manifold(\tau_1,\tau_2))}
                  \right)
                  .
    \end{split}    
  \end{equation}
  This proves \zcref{eq:lemma:higher-order:deg-morawetz-bulk} for
  $s=1$.  We now show how to prove
  \zcref{eq:lemma:higher-order:deg-morawetz-bulk} for $s=2$. As above,
  by first commuting with $T$ and $\Phi$, we deduce
  \begin{equation}
    \label{eq:higher-order-morawetz:T-Phi-s2}
    \begin{split}
      \norm*{(T, \Phi)^{\leq 2}\psi}_{\bSobolevHITrap{1, -\frac{\a}{2}, -\frac \b 2}(\Manifold(\tau_1, \tau_2))}
      &\les \norm*{\psi}_{\SobolevfinHI{3, -\frac{\alpha+1}{2}, -\frac{\beta+1}{2}}(\Sigma(\tau_1))}
        + \norm*{F}_{\bSobolevHI{2, -\frac{\alpha}{2}, -\frac{\beta-4}{2}}(\Manifold(\tau_1,\tau_2))}.
    \end{split}    
  \end{equation}
  To control the remaining derivatives we write
  \begin{equation}
    \label{eq:higher-order-morawetz:horizon:elliptic-s2}
    \begin{split}
      & \int_{\Manifold_{r\le 4M}(\tau_1,\tau_2)}
        \rhoH^2\left( \left( \rhoH\partial_{\rhoH} \right)^2 + \lapp_{\SSS^2} \right)\psi
        \left(\left( \rhoH\partial_{\rhoH} \right)^2 + \lapp_{\SSS^2} \right)\psi\\
      ={}& \norm*{\rhoH(\rhoH\partial_{\rhoH})^2\psi}^2_{L^2(\Manifold_{r\le 4M}(\tau_1,\tau_2))}
           + \norm*{\rhoH \NablaAngular^2\psi}^2_{L^2(\Manifold_{r\le 4M}(\tau_1,\tau_2))}
           + \norm*{\rhoH\rhoH\partial_{\rhoH}\NablaAngular\psi}^2_{L^2(\Manifold_{r\le 4M}(\tau_1,\tau_2))}\\
      & + O\left(\norm*{\psi}_{\bSobolevH{2,-1}(\EventHorizon(\tau_1,\tau_2))}\norm*{\psi}_{\bSobolevH{1,-1}(\EventHorizon(\tau_1,\tau_2))}\right)
        + O\left(
        \norm*{\psi}_{\bSobolevH{2,-1}(\Manifold_{r\le 4M}(\tau_1,\tau_2))}\norm*{\psi}_{\bSobolevH{1, -1}(\Manifold_{r\le 4M}(\tau_1,\tau_2))}
        \right)\\
      & + O\left(
        \norm*{\psi}_{H^2_c(\Manifold(\tau_1,\tau_2))}\norm*{\psi}_{H^1_c(\Manifold(\tau_1,\tau_2))}
        \right).
    \end{split}    
  \end{equation}
  We also have that
  \begin{align*}
    2\rhoH^2\HawkingHorizon\partial_{\rhoH}
    ={}& \left(\HawkingHorizon + \rhoH\rhoH\partial_{\rhoH}\right)^2
         - \HawkingHorizon^2
         - \rhoH^2\left( \rhoH\partial_{\rhoH} \right)^2
         ,
  \end{align*}
  and
  \begin{align*}
    &\rhoH^2\left( 2r\pr_v  + a^2\sin^2\theta \pr_v^2 + 2a\pr_v\pr_{\phiIEF}\right)\psi
      \left(\left( \rhoH\partial_{\rhoH} \right)^2 + \rhoH\partial_{\rhoH}+\lapp_{\SSS^2} \right)\psi\\
    \lesssim{}& \rhoH\left(\abs*{\bDiffH^{\le 1}(T\psi)}^2 + \abs*{\bDiffH^{\le 1}(\Phi\psi)}^2 \right)
                + \rhoH^3
                \left(
                \abs*{\left( \rhoH\partial_{\rhoH} \right)^2\psi}^2
                + \abs*{\NablaAngular\psi}^2
                + \abs*{\rhoH\partial_{\rhoH}\psi}^2
                \right),
  \end{align*}
  so we in fact have that for any $\varepsilon>0$
  \begin{equation}
    \label{eq:higher-order-morawetz:horizon:dv-dphi-terms-s2}
    \begin{split}
      &\rhoH^2\left(2\HawkingHorizon\partial_{\rhoH}+2r\pr_v  + a^{2}\sin^{2}\theta \pr_v^2 + 2a\pr_v\pr_{\phiIEF}\right)\psi
        \left(\left( \rhoH\partial_{\rhoH} \right)^2 + \rhoH\partial_{\rhoH}+\lapp_{\SSS^2} \right)\psi
      \\
      \lesssim{}& \varepsilon^{-1}\left(\abs*{\bDiffH^{\le 1}(T\psi)}^2 + \abs*{\bDiffH^{\le 1}(\Phi\psi)}^2 \right)
                  + \varepsilon
                  \left(
                  \abs*{\left( \rhoH\partial_{\rhoH} \right)^2\psi}^2
                  + \abs*{\NablaAngular\psi}^2
                  + \abs*{\rhoH\partial_{\rhoH}\psi}^2
                  \right).
    \end{split}    
  \end{equation}
  Using again \zcref{eq:higher-order-morawetz:horizon-elliptic-decomposition-s1} we write
  \begin{equation}
    \label{eq:higher-order-morawetz:horizon:integrated-wave-decomp-s2}
    \begin{split}
      & \int_{\Manifold_{r\le 4M}(\tau_1,\tau_2)}
        \rhoH^2\abs*{\left( \left( \rhoH\partial_{\rhoH} \right)^2 + \rhoH\partial_{\rhoH}+\lapp_{\SSS^2} \right)\psi}^2\\
      ={}& \int_{\Manifold_{r\le 4M}(\tau_1,\tau_2)}
           \rhoH^2\abs*{q}^2\Box_{\Metric}\psi\left( \left( \rhoH\partial_{\rhoH} \right)^2 + \rhoH\partial_{\rhoH}+\lapp_{\SSS^2} \right)\psi\\
      & -\int_{\Manifold_{r\le 4M}(\tau_1,\tau_2)}
        \rhoH^2\left(2\HawkingHorizon\partial_{\rhoH}+2r\pr_v  + a^{2}\sin^{2}\theta \pr_v^2 + 2a\pr_v\pr_{\phiIEF}\right)\psi
        \left( \left( \rhoH\partial_{\rhoH} \right)^2 + \rhoH\partial_{\rhoH}+\lapp_{\SSS^2} \right)\psi.
    \end{split}    
  \end{equation}
  Using
  \zcref{eq:higher-order-morawetz:horizon:elliptic-s2,eq:higher-order-morawetz:horizon:dv-dphi-terms-s2,eq:higher-order-morawetz:horizon:integrated-wave-decomp-s2,eq:higher-order-morawetz:s=1:preliminary,eq:higher-order-morawetz:T-Phi-s2}, we then have that
  \begin{equation}
    \label{eq:higher-order-morawetz:horizon-s2}
    \begin{split}
      \norm*{\psi}_{\SobolevDeg{2}(\Manifold_{r\le 4M}(\tau_1,\tau_2))}^2
      \lesssim{}& \norm*{\psi}_{\SobolevfinHI{3, -\frac{\alpha+1}{2}, -\frac{\beta+1}{2}}(\Sigma(\tau_1))}^2
                  + \norm*{F}_{\bSobolevHI{2, -\frac{\alpha}{2}, -\frac{\beta-4}{2}}(\Manifold(\tau_1,\tau_2))}^2\\
                & + O\left(\norm*{\psi}_{\bSobolevH{2,-1}(\EventHorizon(\tau_1,\tau_2))}\norm*{\psi}_{\bSobolevH{1,-1}(\EventHorizon(\tau_1,\tau_2))}\right)\\
                & + O\left(
                  \norm*{\psi}_{H^1_c(\Manifold(\tau_1,\tau_2))}\norm*{\psi}_{H^0_c(\Manifold(\tau_1,\tau_2))}
                  \right)\\
                & + O\left(
                  \norm*{\psi}_{H^2_c(\Manifold(\tau_1,\tau_2))}\norm*{\psi}_{H^1_c(\Manifold(\tau_1,\tau_2))}
                  \right)
                  .
    \end{split}    
  \end{equation}
  Similarly, we have that
  \begin{equation}
    \label{eq:higher-order-morawetz:infinity-s2}
    \begin{split}
      \norm*{\psi}_{\SobolevDeg{2}(\Manifold_{r\ge 4M}(\tau_1,\tau_2))}^2
      \lesssim{}& \norm*{\psi}_{\SobolevfinHI{3, -\frac{\alpha+1}{2}, -\frac{\beta+1}{2}}(\Sigma(\tau_1))}^2
                  + \norm*{F}_{\bSobolevHI{2, -\frac{\alpha}{2}, -\frac{\beta-4}{2}}(\Manifold(\tau_1,\tau_2))}^2\\
                & + O\left(\norm*{\psi}_{\bSobolevI{2,-\frac{5}{2}}(\NullInfinity(\tau_1,\tau_2))}\norm*{\psi}_{\bSobolevI{1,-\frac{5}{2}}(\NullInfinity(\tau_1,\tau_2))}\right)\\
                & + O\left(
                  \norm*{\psi}_{H^1_c(\Manifold(\tau_1,\tau_2))}\norm*{\psi}_{H^0_c(\Manifold(\tau_1,\tau_2))}
                  \right)\\
                & + O\left(
                  \norm*{\psi}_{H^2_c(\Manifold(\tau_1,\tau_2))}\norm*{\psi}_{H^1_c(\Manifold(\tau_1,\tau_2))}
                  \right)
                  .
    \end{split}    
  \end{equation}
  Combining
  \zcref{eq:higher-order-morawetz:horizon-s2,eq:higher-order-morawetz:infinity-s2}
  then concludes the proof for $s=2$. The
  higher-order inequalities can be proven with an induction argument.
\end{proof}

We now show how to recover the higher-order energy estimates in the
following lemma.
\begin{lemma}
  \label{lemma:higher-order-morawetz:energy}
  For $s \in \mathbb{N}$,  $s\ge 2$, and $\tau_1<\tau_2$, we have the following inequality:
  \begin{equation*}
    \begin{split}
      \norm*{\psi}_{\bSobolevHI{s, -1, -\frac{5}{2}}(\Sigma(\tau_2))}^2
      \lesssim{}& \norm*{\psi}_{\SobolevfinHI{s, -\frac{\alpha+1}{2}, -\frac{\beta+1}{2}}(\Sigma(\tau_1))}^2
                  + \norm*{F}_{\bSobolevHI{s-1, -\frac{\alpha}{2}, -\frac{\beta-4}{2}}(\Manifold(\tau_1,\tau_2))}^2\\
                & +  O\left(\norm*{\psi}_{\bSobolevH{s,-1}(\EventHorizon(\tau_1,\tau_2))}\norm*{\psi}_{\bSobolevH{s-1,-1}(\EventHorizon(\tau_1,\tau_2))}\right)\\
                & + O\left(\norm*{\psi}_{\bSobolevI{s,-\frac{5}{2}}(\NullInfinity(\tau_1,\tau_2))}\norm*{\psi}_{\bSobolevI{s-1,-\frac{5}{2}}(\NullInfinity(\tau_1,\tau_2))}\right).
    \end{split}    
  \end{equation*}
\end{lemma}
\begin{proof}
  As before, we show how to prove the estimate when $s=2$. The
  higher-order inequalities can be proven with an induction argument.
  The proof follows closely the proof of
  \zcref{lemma:higher-order:deg-morawetz-bulk}.  Let
  $\tau_{*}\in [\tau_0,\tau_0+1]$ for arbitrary
  $\tau_0\in (\tau_1,\tau_2-1)$, assuming without loss of generality
  that $\tau_2-\tau_1>1$.

  Then we first observe that since $T, \Phi$ are Killing vectorfields, from \zcref{thm:main-thm} applied to $T\psi$ and $\Phi \psi$ we
  immediately have
  \begin{align*}
    \norm*{(T, \Phi)\psi}_{\bSobolevHI{1, 0, -\frac 5 2 }(\Sigma(\tau_2))} &\lesssim\norm*{(T, \Phi)\psi}_{\SobolevfinHI{1, -\frac{\alpha+1}{2}, -\frac{\beta+1}{2}}(\Sigma(\tau_2))}\\
                                                                           &\les  \norm*{\psi}_{\SobolevfinHI{2, -\frac{\alpha+1}{2}, -\frac{\beta+1}{2}}(\Sigma(\tau_1))}+\norm*{F}_{\bSobolevHI{1,-\frac \a 2,-\frac{\b-4}{2}}(\Manifold(\tau_1,\tau_2))}
                                                                             .
  \end{align*}

  Next, integrating
  \zcref{eq:higher-order-morawetz:horizon:integrated-wave-decomp-s2,eq:higher-order-morawetz:horizon:dv-dphi-terms-s2,eq:higher-order-morawetz:horizon:elliptic-s2}
  across $\Manifold_{r\le 4M}(\tau_0,\tau_0+1)$ and
  $\Manifold_{r\ge 4M}(\tau_0,\tau_0+1)$ instead of
  $\Manifold_{r\le 4M}(\tau_1,\tau_2)$ and
  $\Manifold_{r\ge 4M}(\tau_1,\tau_2)$ respectively and combining
  them with \zcref{thm:main-thm} for $s=1$ to control lower-order terms, we see that
  \begin{equation*}
    \begin{split}
      & \int_{\Manifold_{\le 4M}(\tau_0,\tau_0+1)}\rhoH^2\left(
        \abs*{\left((\rhoH\partial_{\rhoH})^2\psi\right)}^2
        + \abs*{\rhoH\partial_{\rhoH}\NablaAngular\psi}^2
        + \abs*{\NablaAngular^2\psi}^2
        \right)\\
      & + \int_{\Manifold_{\ge 4M}(\tau_0,\tau_0+1)}\rhoI^5\left(
        \abs*{\left((\rhoI\partial_{\rhoI})^2\psi\right)}^2
        + \abs*{\rhoI\partial_{\rhoI}\NablaAngular\psi}^2
        + \abs*{\NablaAngular^2\psi}^2
        \right)\\
      \lesssim {}& \norm*{\psi}_{\SobolevfinHI{2, -\frac{\alpha+1}{2}, -\frac{\beta+1}{2}}(\Sigma(\tau_1))}^2
                   + \norm*{F}_{\bSobolevHI{1, -\frac{\alpha}{2}, -\frac{\beta-4}{2}}(\Manifold(\tau_0,\tau_0+1))}^2\\
      & + O\left(\norm*{\psi}_{\bSobolevH{2,-1}(\EventHorizon(\tau_0,\tau_0+1))}\norm*{\psi}_{\bSobolevH{1,-1}(\EventHorizon(\tau_0,\tau_0+1))}\right)\\
    \end{split}
  \end{equation*}
  We thus in particular have that
  \begin{equation*}
    \begin{split}
      \inf_{\tau\in [\tau_0,\tau_0+1]}\norm*{\psi}^2_{\bSobolevHI{2,-1,-\frac{5}{2}}(\Sigma(\tau))}
      \lesssim {}& \norm*{\psi}_{\SobolevfinHI{2, -\frac{\alpha+1}{2}, -\frac{\beta+1}{2}}(\Sigma(\tau_1))}^2
                   + \norm*{F}_{\bSobolevHI{1, -\frac{\alpha}{2}, -\frac{\beta-4}{2}}(\Manifold(\tau_0,\tau_0+1))}^2\\
                 & +  O\left(\norm*{\psi}_{\bSobolevH{2,-1}(\EventHorizon(\tau_0,\tau_0+1))}\norm*{\psi}_{\bSobolevH{1,-1}(\EventHorizon(\tau_0,\tau_0+1))}\right)\\
                 & + O\left(\norm*{\psi}_{\bSobolevI{2,-\frac{5}{2}}(\NullInfinity(\tau_0,\tau_0+1))}\norm*{\psi}_{\bSobolevI{1,-\frac{5}{2}}(\NullInfinity(\tau_0,\tau_0+1))}\right).
    \end{split}
  \end{equation*}
  We then conclude recalling that $\tau_0$ was arbitrary. 
\end{proof}

We are finally ready to prove \zcref{thm:main-thm} for $s\geq 2$.
\begin{proof}[Proof of Theorem \ref{thm:main-thm} for $s\geq 2$]
  As before, we will show how to prove \zcref{thm:main-thm} with $s=2$
  using the fact that we have already proven \zcref{thm:main-thm} with
  $s=1$. Higher-order estimates are proven inductively.

  Let $\tau_{*}\in [\tau_0,\tau_0+1]$ be such that
  \begin{equation*}
    \norm*{\psi}_{\SobolevfinHI{s, -\frac{\alpha+1}{2}, -\frac{\beta+1}{2}}(\Sigma(\tau_*))}
    = \int_{\tau\in [\tau_0,\tau_0+1]}\norm*{\psi}_{\SobolevfinHI{s, -\frac{\alpha+1}{2}, -\frac{\beta+1}{2}}(\Sigma(\tau))},
  \end{equation*}
  where $\tau_0\in [\tau_1,\tau_2-1]$, where we assume without loss of
  generality that $\tau_2-\tau_1>1$.

  Combining \zcref{lemma:higher-order:deg-morawetz-bulk,lemma:higher-order-morawetz:energy}, we have that
  \begin{equation*}
    \begin{split}
      \norm*{\psi}_{\SobolevDeg{2}(\Manifold(\tau_0,\tau_*))}^2
      +   \norm*{\psi}_{\bSobolevHI{2, -1, -\frac{5}{2}}(\Sigma(\tau_*))}^2
      \lesssim{}& \norm*{\psi}_{\SobolevfinHI{2, -\frac{\alpha+1}{2}, -\frac{\beta+1}{2}}(\Sigma(\tau_0))}^2
                  + \norm*{F}_{\bSobolevHI{1, -\frac{\alpha}{2}, -\frac{\beta-4}{2}}(\Manifold(\tau_0,\tau_*))}^2\\
                & +  O\left(\norm*{\psi}_{\bSobolevH{2,-1}(\EventHorizon(\tau_0,\tau_*))}\norm*{\psi}_{\bSobolevH{1,-1}(\EventHorizon(\tau_0,\tau_*))}\right)\\
                & + O\left(\norm*{\psi}_{\bSobolevI{2,-\frac{5}{2}}(\NullInfinity(\tau_0,\tau_*))}\norm*{\psi}_{\bSobolevI{1,-\frac{5}{2}}(\NullInfinity(\tau_0,\tau_*))}\right).
    \end{split}
  \end{equation*}
  Then combining with the $s=2$ statement of the weighted hierarchies, i.e. \zcref{eq:horizon-weighted-estimate, eq:infinity-weighted-estimate-psic}, and adding a
  large amount of \zcref{thm:main-thm} with $s=1$ to
  control any lower order terms, concludes the proof of
  \zcref{thm:main-thm} with $s=2$. 
\end{proof}
\subsection{Proof of \zcref{cor:pointwise-decay}}\label{proof:corollary-decay}

Here we assume $F=0$.
By definition of the norms, we observe that
\begin{align}\label{eq:bound-blk-integral-boundary}
  \norm*{\psi}_{\bSobolevHITrap{s, -\frac{\a}{2}, -\frac{\b}{2} }(\Manifold(\tau_1, \tau_2))}^2 \gtrsim \int_{\tau_1}^{\tau_2}  \norm*{\psi}_{\SobolevfinHI{s-1, -\frac{\alpha+2}{2}, -\frac{\beta+2}{2}}(\Sigma(\tau))}^2 d\tau.
\end{align}

Combining \zcref{thm:main-thm} and
\zcref{eq:bound-blk-integral-boundary}, we deduce for
$\a \in (-1+4\mathbf{a}^2, -\de_1)$ and $\b \in (1,3)$
\begin{align}\label{bound-in-boundary-terms-flux}
  \norm*{\psi}_{\SobolevfinHI{s, -\frac{\alpha+1}{2}, -\frac{\beta+1}{2}}(\Sigma(\tau_2))}^2+
  \int_{\tau_1}^{\tau_2}  \norm*{\psi}_{\SobolevfinHI{s-1, -\frac{\alpha+2}{2}, -\frac{\beta+2}{2}}(\Sigma(\tau))}^2 d\tau\les  \norm*{\psi}_{\SobolevfinHI{s, -\frac{\alpha+1}{2}, -\frac{\beta+1}{2}}(\Sigma(\tau_1))}^2.
\end{align} 

From now on we consider $\beta=\alpha+2\in (1+4\mathbf{a}^2,2-\de_1)$.
We collect the following lemma to perform interpolation of the norms. 
\begin{lemma}\label{lemma:interpolation}
  Let $\a_- <\gamma < \a_+$ with 
  \[
    \gamma=\theta\a_+ +(1-\theta)\a_-,
    \qquad
    0<\theta<1.
  \]
  Then
  \[
    \norm*{\psi}_{\SobolevfinHI{s, -\frac{\gamma+1}{2}, -\frac{\gamma+3}{2}}(\Sigma(\tau))}^2
    \le
    \Big( \norm*{\psi}_{\SobolevfinHI{s, -\frac{\a_{+}+1}{2}, -\frac{\a_{+}+3}{2}}(\Sigma(\tau))}^2\Big)^\theta
    \Big( \norm*{\psi}_{\SobolevfinHI{s, -\frac{a_{-}+1}{2}, -\frac{a_{-}+3}{2}}(\Sigma(\tau))}^2\Big)^{1-\theta}.
  \]
\end{lemma}
\begin{proof}
  By definition, we have 
  \begin{align*}
    \norm*{\psi}_{\SobolevfinHI{s, -\frac{\gamma+1}{2}, -\frac{\gamma+3}{2}}(\Sigma(\tau))}\vcentcolon =\norm*{\psi}_{H_{\EventHorizon}^{s, -\frac{\gamma+1}{2}}(\Sigma_{\EventHorizon}(\tau))}+\norm*{\widecheck{\psi}}_{H_{\NullInfinity}^{s, -\frac{\gamma+5}{2}}(\Sigma_{\NullInfinity}(\tau))}+\norm*{\psi}_{\compactSobolev{s}(\Sigma_{\NullInfinity}(\tau))}
  \end{align*}
  and in particular, by writing 
  \begin{align*}
    \rhoH^{\gamma+1}
    =
    (\rhoH^{\a_++1})^\theta(\rhoH^{\a_-+1})^{1-\theta}, \qquad \rhoI^{\gamma+5}
    =
    (\rhoI^{\a_++5})^\theta
    (\rhoI^{\a_-+5})^{1-\theta},
  \end{align*}
  applying Hölder gives
  \begin{align*}
    \norm*{\psi}_{H_{\EventHorizon}^{1, -\frac{\gamma+1}{2}}(\Sigma_{\EventHorizon}(\tau))}^2 ={}&  \int_{\Sigma(\tau)\cap \{\rhoH\leq \rho_0\}}\left(\rhoH^{\a_++1}\Big(\abs*{\partial_{\rhoH}\psi}^2+\abs*{\psi}^2+\abs*{\partial_v\psi}^2+\abs*{\NablaAngular\psi}^2 \Big)\right)^\theta
    \\ 
       &\times  \left(\rhoH^{\a_-+1}\Big(\abs*{\partial_{\rhoH}\psi}^2+\abs*{\psi}^2+\abs*{\partial_v\psi}^2+\abs*{\NablaAngular\psi}^2 \Big)\right)^{1-\theta} \\
    \leq{}&  \Big( \norm*{\psi}_{H_{\EventHorizon}^{1, -\frac{\alpha_{+}+1}{2}}(\Sigma_{\EventHorizon}(\tau))}^2\Big)^{\theta} \Big(\norm*{\psi}_{H_{\EventHorizon}^{s, -\frac{\alpha_{-}+1}{2}}(\Sigma_{\EventHorizon}(\tau))}^2\Big)^{1-\theta} , 
    \end{align*}
    and similarly for $\norm*{\widecheck{\psi}}_{H_{\NullInfinity}^{1, -\frac{\gamma+5}{2}}(\Sigma_{\NullInfinity}(\tau))}^2$
    and for higher derivatives norms.
  Adding the two norms together and using the
  discrete Hölder inequality
  \[
    \sum_j a_j^\theta b_j^{1-\theta}
    \le
    \big(\sum_j a_j\big)^\theta
    \big(\sum_j b_j\big)^{1-\theta},
  \]
  gives the result.
\end{proof}

\subsubsection{Decay of the energy flux}

We are ready to prove the decay in time of the energy flux. Here we take $\delta>4\mathbf{a}^2$ with $\de+\de_1 <1$.
Applying \zcref{bound-in-boundary-terms-flux} to $\alpha=-1+\de$, we obtain
\begin{align*}
  \norm*{\psi}_{\SobolevfinHI{s, -\frac{\delta}{2}, -\frac{2+\delta}{2}}(\Sigma(\tau_2))}^2+
  \int_{\tau_1}^{\tau_2}  \norm*{\psi}_{\SobolevfinHI{s-1, -\frac{1+\delta}{2}, -\frac{3+\delta}{2}}(\Sigma(\tau))}^2 d\tau\les  \norm*{\psi}_{\SobolevfinHI{s, -\frac{\delta}{2}, -\frac{2+\delta}{2}}(\Sigma(\tau_1))}^2.
\end{align*}
The above implies that
\begin{itemize}
\item the energy $\norm*{\psi}_{\SobolevfinHI{s-1, -\frac{\delta}{2}, -\frac{2+\delta}{2}}(\Sigma(\tau))}^2$ is bounded by initial data 
  , i.e.
  \begin{align}
    \norm*{\psi}_{\SobolevfinHI{s-1, -\frac{\delta}{2}, -\frac{2+\delta}{2}}(\Sigma(\tau_2))}^2 \les   \norm*{\psi}_{\SobolevfinHI{s, -\frac{\delta}{2}, -\frac{2+\delta}{2}}(\Sigma(\tau_2))}^2 \les   \norm*{\psi}_{\SobolevfinHI{s, -\frac{\delta}{2}, -\frac{2+\delta}{2}}(\Sigma(\tau_1))}^2, \label{eq:energy-bounded-initial-data}
  \end{align}
\item the energy norm $ \norm*{\psi}_{\SobolevfinHI{s-1, -\frac{1+\delta}{2}, -\frac{3+\delta}{2}}(\Sigma(\tau))}^2$ is integrable in time, i.e.
  \begin{align}\label{eq:integrable-in-time-energy}
    \int_{\tau_1}^{\tau_2}  \norm*{\psi}_{\SobolevfinHI{s-1, -\frac{1+\delta}{2}, -\frac{3+\delta}{2}}(\Sigma(\tau))}^2 d\tau\les  \norm*{\psi}_{\SobolevfinHI{s, -\frac{\delta}{2}, -\frac{2+\delta}{2}}(\Sigma(\tau_1))}^2.
  \end{align} 
\end{itemize}

Applying \zcref{lemma:interpolation} with
\[
  \a_-=-1+\de,
  \qquad
  \a_+=\de,
  \qquad
  \gamma=-\de_1,
\]
and using the fact that 
\[
  -\de_1
  =
  (1-\de-\de_1)\de
  +
  (\de+\de_1)(-1+\de),
\]
we get
\[
  \norm*{\psi}_{\SobolevfinHI{s-1, -\frac{1-\de_1}{2}, -\frac{3-\de_1}{2}}(\Sigma(\tau))}^2
  \le
  \left( \norm*{\psi}_{\SobolevfinHI{s-1, -\frac{1+\de}{2}, -\frac{3+\de}{2}}(\Sigma(\tau))}^2\right)^{1-\de-\de_1}
  \left( \norm*{\psi}_{\SobolevfinHI{s-1, -\frac{\de}{2}, -\frac{2+\de}{2}}(\Sigma(\tau))}^2\right)^{\de+\de_1}.
\]
Using \zcref{eq:energy-bounded-initial-data} to bound the last term we get
\begin{align*}
  \left( \norm*{\psi}_{\SobolevfinHI{s-1, -\frac{1-\de_1}{2}, -\frac{3-\de_1}{2}}(\Sigma(\tau))}^2\right)^{\frac{1}{1-\de-\de_1}}
  \les
  \norm*{\psi}_{\SobolevfinHI{s-1, -\frac{1+\de}{2}, -\frac{3+\de}{2}}(\Sigma(\tau))}^2
  \left(\norm*{\psi}_{\SobolevfinHI{s, -\frac{\delta}{2}, -\frac{2+\delta}{2}}(\Sigma(\tau_1))}^2\right)^{\frac{\de+\de_1}{1-\de-\de_1}}.
\end{align*}
By integrating the above in time from $\tau_1$ to $\infty$ and using \zcref{eq:integrable-in-time-energy}, we obtain
\begin{align}\label{eq:bound-interm-inte}
  \int_{\tau_1}^{\infty}\left( \norm*{\psi}_{\SobolevfinHI{s-1, -\frac{1-\de_1}{2}, -\frac{3-\de_1}{2}}(\Sigma(\tau))}^2\right)^{\frac{1}{1-\de-\de_1}} d\tau \les
  \left(\norm*{\psi}_{\SobolevfinHI{s, -\frac{\delta}{2}, -\frac{2+\delta}{2}}(\Sigma(\tau_1))}^2\right)^{\frac{1}{1-\de-\de_1}}.
\end{align}
Then, using \zcref{thm:main-thm} applied to $\a=-\de_1$ and $\b=2-\de_1$ we have
\begin{align*}
  \norm*{\psi}_{\SobolevfinHI{s-1, -\frac{1-\de_1}{2}, -\frac{3-\de_1}{2}}(\Sigma(\tau_2))}^2\les  \norm*{\psi}_{\SobolevfinHI{s-1, -\frac{1-\de_1}{2}, -\frac{3-\de_1}{2}}(\Sigma(\tau_1))}^2. 
\end{align*}
Therefore for any $\tau \geq 2\tau_1$ and any $t \in [\frac{\tau}{2}, \tau]$, we have 
\begin{align*}
  \norm*{\psi}_{\SobolevfinHI{s-1, -\frac{1-\de_1}{2}, -\frac{3-\de_1}{2}}(\Sigma(\tau))}^2\les  \norm*{\psi}_{\SobolevfinHI{s-1, -\frac{1-\de_1}{2}, -\frac{3-\de_1}{2}}(\Sigma(t))}^2. 
\end{align*}
Raising both sides to the power $\frac{1}{1-\de-\de_1}$ and
integrating over $t \in [\frac{\tau}{2}, \tau]$ we get, using
\zcref{eq:bound-interm-inte},
\begin{align*}
  \frac{\tau}{2} \left( \norm*{\psi}_{\SobolevfinHI{s-1, -\frac{1-\de_1}{2}, -\frac{3-\de_1}{2}}(\Sigma(\tau))}^2\right)^{\frac{1}{1-\de-\de_1}}&\les \int_{\frac \tau 2}^\tau  \left(\norm*{\psi}_{\SobolevfinHI{s-1, -\frac{1-\de_1}{2}, -\frac{3-\de_1}{2}}(\Sigma(t))}^2\right)^{\frac{1}{1-\de-\de_1}} dt\\
                                                                                                                                                &\les
                                                                                                                                                  \left(\norm*{\psi}_{\SobolevfinHI{s-1, -\frac{\delta}{2}, -\frac{2+\delta}{2}}(\Sigma(\tau_1))}^2\right)^{\frac{1}{1-\de-\de_1}}.
\end{align*}
Hence, we obtain
\begin{align}\label{eq:decay-energy-flux-energy-fin}
  \norm*{\psi}_{\SobolevfinHI{s-1, -\frac{1-\de_1}{2}, -\frac{3-\de_1}{2}}(\Sigma(\tau))}^2
  &\les \tau^{-(1-\de-\de_1)}E_{\operatorname{init}}^s
    ,
\end{align}
where
\begin{align*}
  E_{\operatorname{init}}^s=\norm*{\psi}_{\SobolevfinHI{s, -\frac{\delta}{2}, -\frac{2+\delta}{2}}(\Sigma(\tau_1))}^2.
\end{align*}
From the definition of $ \norm*{\psi}_{\SobolevfinHI{s-1, -\frac{1-\de_1}{2}, -\frac{3-\de_1}{2}}(\Sigma(\tau))}^2$, we deduce
\begin{align}
  \norm*{\psi}_{H_{\EventHorizon}^{s-1, -\frac{1-\de_1}{2}}(\Sigma_{\EventHorizon}(\tau))}^2 \les \tau^{-(1-\de-\de_1)}E_{\operatorname{init}}^s \label{eq:energy-flux-HH}\\
  \norm*{\widecheck{\psi}}_{H_{\NullInfinity}^{s-1, -\frac{5-\de_1}{2}}(\Sigma_{\NullInfinity}(\tau))}^2 \les \tau^{-(1-\de-\de_1)}E_{\operatorname{init}}^s.\label{eq:energy0flux-decay}
\end{align}

\subsubsection{Pointwise decay of the solution}

We now deduce pointwise decay in $\rhoH$ and $\rhoI$ for the solution.  Since the proof is symmetric in the two ends, we collect here the following general one-dimensional computations:
\begin{enumerate}
\item Applying \zcref{corollary-hardy-dejan-cutoff} to $\gamma=-1+\de_1$ gives
  \begin{align}\label{eq:bound-1-decay}
    \begin{split}
      \frac{\de_1^2}{4}
      \int_0^{\frac{c}{2}}\rho^{-1+\de_1}
      \abs*{\psi}^2\,d\rho &\leq   \int_0^{c}\rho^{-1+\de_1}
 \abs*{\rho\partial_{\rho}\psi}^2\,d\rho+O(c^{-2})
                             \int_{\frac12 c}^{c}
                             \rho^{1+\de_1}|\psi|^2\,d\rho\\
                           &\leq   \int_0^{c}\rhoH^{1+\de_1}
                             \abs*{\partial_{\rho}\psi}^2\,d\rho+O(c^{-2})
                             \int_{\frac12 c}^{c}
                             \rho^{1+\de_1}|\psi|^2\,d\rho.
    \end{split}
  \end{align}
\item By fundamental theorem of calculus and Cauchy-Schwarz, we write for any $\rho \in (0, \rho_0)$
  \begin{align*}
    |\psi(\tau, 0, \omega)|^2 &\les |\psi(\tau, \rho, \omega)|^2 + \Big(\int_0^{\rho} |\partial_{\rho'}\psi(\tau, \rho', \omega)|d\rho' \Big)^2\\
                              &\les |\psi(\tau, \rho, \omega)|^2 + \big(\int_0^{\rho} \rho'^{-1+\de_1}d\rho'\big)\big( \int_0^{\rho}\rho'^{1-\de_1}|\partial_{\rho'}\psi(\tau, \rho', \omega)|^2d\rho' \big)\\
                              &\les |\psi(\tau, \rho, \omega)|^2 + \frac{\rho^{\de_1}}{\de_1}\Big( \int_0^{\rho_0}\rho'^{1-\de_1}|\partial_{\rho'}\psi(\tau, \rho', \omega)|^2d\rho' \Big)
  \end{align*}
  Multiplying both sides by $\rho^{-1+\de_1}$ and integrating over $\rho \in (0, \rho_0)$ we obtain
  \begin{align}\label{eq:bound-on-rho=0}
    \begin{split}
      |\psi(\tau, 0, \omega)|^2 \int_0^{\rho_0} \rho^{-1+\de_1}  d \rho
      &\les \int_0^{\rho_0}\rho^{-1+\de_1}|\psi(\tau, \rho, \omega)|^2 d\rho \\
      &+ \big(\int_0^{\rho_0}\frac{1}{\de_1}\rho^{-1+2\de_1}d\rhoH\big)\Big( \int_0^{\rho_0}\rho'^{1-\de_1}|\partial_{\rho'}\psi(\tau, \rho', \omega)|^2d\rho' \Big)\\
      &\les \int_0^{\rho_0}\big(\rho^{1-\de_1}|\partial_{\rho}\psi(\tau, \rho, \omega)|^2+\rho^{-1+\de_1}|\psi(\tau, \rho, \omega)|^2 \big)d\rho.   
    \end{split}
  \end{align}
\end{enumerate}

\paragraph*{At the event horizon}
First, we focus on the solution close to $\EventHorizon$. Observe that \zcref{eq:energy-flux-HH} explicitly gives 
\begin{align}\label{eq:control-horizon}
  \int_{\Sigma_{\EventHorizon}(\tau)}\rhoH^{1-\de_1}\big(|\partial_{\rhoH} \psi|^2+|\psi|^2\big) \leq \tau^{-(1-\de-\de_1)}E_{\operatorname{init}}^2.
\end{align}
We can improve the control of the zero-th order term using \zcref{eq:bound-1-decay} for $\rho=\rhoH$. Indeed, since on $\Sigma_{\EventHorizon}$, we have $\rhoH^{1+\de_1}\leq \rhoH^{1-\de_1}$, from \zcref{eq:control-horizon} we obtain
\begin{align*}
  \int_{\Sigma_{\EventHorizon}(\tau)}\rhoH^{-1+\de_1} |\psi|^2 \leq \tau^{-(1-\de-\de_1)}E_{\operatorname{init}}^2
\end{align*}
where we have bounded the second integral in \zcref{eq:bound-1-decay} using \zcref{eq:decay-energy-flux-energy-fin}. Summing with \zcref{eq:control-horizon} we deduce 
\begin{align}\label{eq:control-horizon-1}
  \int_{\Sigma_{\EventHorizon}(\tau)}\rhoH^{1-\de_1}|\partial_{\rhoH} \psi|^2+\rhoH^{-1+\de_1}|\psi|^2 \leq \tau^{-(1-\de-\de_1)}E_{\operatorname{init}}^2.
\end{align}
By fundamental theorem of calculus
\begin{align*}
  \int_{\Sphere^2}  |\psi(\tau, \rhoH, \omega)|^2 d \mathring{\gamma}-\int_{\Sphere^2}|\psi(\tau, 0, \omega)|^2d \mathring{\gamma}&=-\int_{\Sphere^2}\int^0_{\rhoH} \partial_{\rho'}\big(|\psi|^2\big)d \rho' d\mathring{\gamma} \les \int_{\Sigma_{\EventHorizon}(\tau)}  |\psi||\partial_{\rhoH} \psi| \\
                                                                                                                      &\les \int_{\Sigma_{\EventHorizon}(\tau)} \big(\rhoH^{1-\de_1}|\partial_{\rhoH} \psi|^2 + \rhoH^{-1+\de_1}|\psi|^2\big) .
\end{align*}
From \eqref{eq:control-horizon-1}, we obtain
\begin{align}\label{eq:boundedness-psi}
  \int_{\Sphere^2}   |\psi(\tau, \rhoH, \omega)|^2d \mathring{\gamma}-\int_{\Sphere^2}|\psi(\tau, 0, \omega)|^2d \mathring{\gamma}
  &\les \tau^{-(1-\de-\de_1)}E_{\operatorname{init}}^2. 
\end{align}
We now bound the function at the horizon. Integrating on the spheres \zcref{eq:bound-on-rho=0} gives
\begin{align*}
  \int_{\Sphere^2}|\psi(\tau, 0, \omega)|^2d \mathring{\gamma} 
  &\les \int_{\Sphere^2}\int_0^{\rho_0}\big(\rhoH^{1-\de_1}|\partial_{\rhoH}\psi(\tau, \rhoH, \omega)|^2+\rhoH^{-1+\de_1}|\psi(\tau, \rhoH, \omega)|^2\big) d\rhoH d \mathring{\gamma}\\
  &\leq \tau^{-(1-\de-\de_1)}E_{\operatorname{init}}^2
\end{align*}
where we used \zcref{eq:control-horizon-1}.
Therefore from \zcref{eq:boundedness-psi} we deduce for $\rhoH \geq 0$
\begin{align*}
  \int_{\Sphere^2}   |\psi(\tau, \rhoH, \omega)|^2d \mathring{\gamma}
  &\les \tau^{-(1-\de-\de_1)}E_{\operatorname{init}}^2,
\end{align*}
and 
by Sobolev embedding on $\Sphere^2$ 
\begin{align*}
  |\psi(\tau, \rhoH, \omega)| \les
  \frac{\sqrt{E_{\operatorname{init}}^{4}}}
  {\tau^{\frac{1-\de-\de_1}{2}}},
\end{align*}
as stated.

\paragraph*{At null infinity}
We now focus on the solution close to $\NullInfinity$. Observe that \zcref{eq:energy0flux-decay} explicitly gives
\begin{align}\label{eq:control-near-infinity}
  \int_{\Sigma(\tau)}\rhoI^{5-\de_1}\big(| \partial_{\rhoI}\widecheck{\psi}|^2+|\widecheck{\psi}|^2 \big) \les \tau^{-(1-\de-\de_1)}E_{\operatorname{init}}^2. 
\end{align}
We can improve the control of the zero-th order term using \zcref{eq:bound-1-decay} for $\widecheck{\psi}$ and for $\rho=\rhoI$. Indeed, since on $\Sigma_{\NullInfinity}$ we have $\rhoI^{1+\de_1}\leq \rhoI^{1-\de_1}$ and 
using that $|q|^2 dr d\mathring{\gamma}=-(\rhoI^{-2}+a^2\cos^2\th)\rhoI^{-2}d\rhoI d\mathring{\gamma}$,
we obtain
\begin{align*}
  \int_{\Sigma_{\NullInfinity}(\tau)}\rhoI^{3+\de_1}
  \abs*{\widecheck{\psi}}^2\,d\rhoI 
  &\les \tau^{-(1-\de-\de_1)}E_{\operatorname{init}}^2,
\end{align*}
where we have bounded the second integral in \zcref{eq:bound-1-decay} using \zcref{eq:decay-energy-flux-energy-fin}. Summing with \zcref{eq:control-near-infinity} we deduce
\begin{align}\label{eq:control-near-infinity-1}
  \int_{\Sigma_\NullInfinity(\tau)}\rhoI^{5-\de_1}| \partial_{\rhoI}\widecheck{\psi}|^2+\rhoI^{3+\de_1}|\widecheck{\psi}|^2 \les \tau^{-(1-\de-\de_1)}E_{\operatorname{init}}^2. 
\end{align}
By fundamental theorem of calculus
\begin{align*}
  \int_{\Sphere^2}  |\widecheck{\psi}(\tau, \rhoI, \omega)|^2 d \mathring{\gamma}-\int_{\Sphere^2}|\widecheck{\psi}_{\NullInfinity}(\tau, \omega)|^2d \mathring{\gamma}&=-\int_{\Sphere^2}\int^0_{\rhoI} \partial_{\rho'}\big(|\widecheck{\psi}|^2\big)d \rho' d\mathring{\gamma}\\
                                                                                                                                                           & \les \int_{\Sigma_{\NullInfinity}(\tau)} \rhoI^4 |\widecheck{\psi}||\partial_{\rhoI} \widecheck{\psi}| \\
                                                                                                                                                           &\les \int_{\Sigma_{\NullInfinity}(\tau)} \big(\rhoI^{5-\de_1}|\partial_{\rhoI} \widecheck{\psi}|^2 + \rhoI^{3+\de_1}|\widecheck{\psi}|^2\big) ,
\end{align*}
where $\widecheck{\psi}_{\NullInfinity}(\tau, \omega)\vcentcolon =\lim_{\rhoI\to 0}\widecheck{\psi}(\tau, \rhoI, \omega)$ is the radiation field at null infinity.
From \zcref{eq:control-near-infinity-1}, we obtain
\begin{align}\label{eq:boundedness-psi-II}
  \int_{\Sphere^2}   |\widecheck{\psi}(\tau, \rhoI, \omega)|^2d \mathring{\gamma}-\int_{\Sphere^2}|\widecheck{\psi}_{\NullInfinity}(\tau, \omega)|^2d \mathring{\gamma}
  &\les \tau^{-(1-\de-\de_1)}E_{\operatorname{init}}^2. 
\end{align}
We now bound $\widecheck{\psi}_{\NullInfinity}(\tau, \omega)$. Integrating on the spheres \zcref{eq:bound-on-rho=0} and using again that $|q|^2 dr d\mathring{\gamma}=-(\rhoI^{-2}+a^2\cos^2\th)\rhoI^{-2}d\rhoI d\mathring{\gamma}$, 
\begin{align*}
  \int_{\Sphere^2}|\widecheck{\psi}_{\NullInfinity}(\tau, \omega)|^2d \mathring{\gamma} 
  &\les \int_{\Sphere^2}\int_0^{\rho_0}\big(\rhoI^{1-\de_1}|\partial_{\rhoI}\widecheck{\psi}(\tau, \rhoI, \omega)|^2+\rhoI^{-1+\de_1}|\widecheck{\psi}(\tau, \rhoI, \omega)|^2\big) d\rhoH d \mathring{\gamma} \\
  &\les \int_{\Sigma_{\NullInfinity}(\tau)}\rhoI^{3+\de_1}|\widecheck{\psi}(\tau, \rhoI, \omega)|^2 +\rhoI^{5-\de_1}|\partial_{\rhoI}\widecheck{\psi}(\tau, \rhoI, \omega)|^2 \\
  &\leq \tau^{-(1-\de-\de_1)}E_{\operatorname{init}}^2
\end{align*}
where we used \zcref{eq:control-near-infinity-1}. Therefore from \zcref{eq:boundedness-psi-II} we deduce for $\rhoI \geq 0$
\begin{align*}
  \int_{\Sphere^2}   |\widecheck{\psi}(\tau, \rhoI, \omega)|^2d \mathring{\gamma}
  &\les \tau^{-(1-\de-\de_1)}E_{\operatorname{init}}^2,
\end{align*}
and 
by Sobolev embedding on $\Sphere^2$ 
\begin{align*}
  |\widecheck{\psi}(\tau, \rhoI, \omega)| \les
  \frac{\sqrt{E_{\operatorname{init}}^{4}}}
  {\tau^{\frac{1-\de-\de_1}{2}}}.
\end{align*}
Writing that $\widecheck{\psi}=r\psi$, we obtain the stated.

\appendix
\section{Choice of multiplier in axial symmetry}\label{sec:proof-Lemma-ax-symm}

In this section we prove \zcref{prop:ax-symm-choice}.  We follow the
physical-space construction of Stogin, with a regularization near the
horizon and a final Hardy correction.  We do not require any smallness
assumption on \(a\), beyond the subextremal/extremal range \(a^2+Q^2\le M^2\).
For the analogous estimate on extremal Kerr, see also
\cite{giorgiPhysicalspaceEstimatesAxisymmetric2024}.

Throughout the section, for any radial function \(u\), we set
\begin{equation*}
  2 w \vcentcolon= \frac{(r-M)^2}{(r^2+a^2)^2}\,\partial_r u .
\end{equation*}
With this notation the potential term in \zcref{eq:definition-VV} takes
the form
\begin{equation}\label{eq:V-in-terms-of-w-axial}
  \mathcal V[w]
  =-\frac12\partial_r\big((r-M)^2\partial_r w\big).
\end{equation}

\paragraph*{Step 1: Stogin's singular multiplier $u_S$}

We first construct a multiplier which has the desired positivity
properties away from the horizon, but which is singular at \(r=M\).  Given
an auxiliary positive function \(w_S\), define
\begin{equation}\label{eq:def-u}
  u_S(r)
  = \int_{r_{\mathrm{trap}}}^r
  \frac{(s^2+a^2)^2}{(s-M)^2}2w_S(s)\,ds,
  \qquad
  r_{\mathrm{trap}}\vcentcolon=M+\sqrt{M^2+a^2}.
\end{equation}
Here \(r_{\mathrm{trap}}\) is the largest root of the trapping polynomial
\(\mathcal T\), and the only one larger than \(M\).  Since \(w_S>0\), we
have
\begin{equation*}
  \partial_r u_S>0,
  \qquad
  u_S(r_{\mathrm{trap}})=0.
\end{equation*}
Thus \(u_S\mathcal T\ge0\) on \(r\ge M\).  

We next choose \(w_S\) so that \(\mathcal A[u_S]\ge0\).  Set
\begin{equation*}
  \widetilde{\mathcal A}[u_S]
  \vcentcolon=
  \frac{(r^2+a^2)^2}{2r}
  \partial_r\left(\frac{u_S}{r^2+a^2}\right).
\end{equation*}
Using \zcref{eq:def-u}, we compute
\begin{equation*}
  \partial_r\widetilde{\mathcal A}[u_S]
  =(r^2+a^2)\partial_r\left(
    \frac{w_S(r^2+a^2)^2}{r(r-M)^2}
  \right).
\end{equation*}
Let \(r_*\) be the point at which $\frac{r(r-M)^2}{(r^2+a^2)^2}$
attains its maximum.  Equivalently, \(r_*\) is the largest root of
\begin{equation*}
  r^3-3Mr^2-3a^2r+Ma^2=0,
\end{equation*}
and satisfies \(r_*>r_{\mathrm{trap}}\).  We define
\begin{equation}\label{eq:def-w}
  w_S(r)=
  \begin{cases}
    \displaystyle
    \frac{r_*(r_*-M)^2}{(r_*^2+a^2)^2}=:\widetilde C,
    & r\le r_*,\\[0.8em]
    \displaystyle
    \frac{r(r-M)^2}{(r^2+a^2)^2},
    & r>r_*.
  \end{cases}
\end{equation}
Then \(w_S\) is positive and \(C^1\).  Since \(r_*\) is the minimum point
of $\frac{(r^2+a^2)^2}{r(r-M)^2}$, 
we have
\begin{equation*}
  \partial_r\left(
    \frac{w_S(r^2+a^2)^2}{r(r-M)^2}
  \right)
  \begin{cases}
    <0, & r<r_*,\\
    =0, & r\ge r_*.
  \end{cases}
\end{equation*}
Hence \(\widetilde{\mathcal A}[u_S]\) is decreasing on \([M,r_*]\) and
constant on \([r_*,\infty)\).  It remains only to check that this constant
value is positive.  Evaluating at \(r_*\) and using \zcref{eq:def-u},
\begin{align*}
  \widetilde{\mathcal A}[u_S](r_*)
  & =
    \frac{(r_*^2+a^2)}{2r_*}\partial_r u_S(r_*)-u_S(r_*)                                      \\
  & =
    \frac{(r_*^2+a^2)^3}{r_*(r_*-M)^2}w_S(r_*)
    -2w_S(r_*)\int_{r_{\mathrm{trap}}}^{r_*}
    \frac{(r^2+a^2)^2}{(r-M)^2}\,dr .
\end{align*}
The function \((r^2+a^2)^2/(r-M)^2\) is increasing on
\([r_{\mathrm{trap}},r_*]\).  Therefore
\begin{align*}
  \widetilde{\mathcal A}[u_S](r_*)
  &\ge
    \frac{(r_*^2+a^2)^2}{r_*(r_*-M)^2}w_S(r_*)
    \big(r_*^2+a^2-2r_*(r_*-r_{\mathrm{trap}})\big)                       \\
  &=
    \widetilde C\big(a^2+r_*(2r_{\mathrm{trap}}-r_*)\big)
    \ge c(M,a)>0.
\end{align*}
Consequently there exists a constant \(c_A=c_A(M,a)>0\) such that
\begin{equation}\label{eq:bounds-AA}
  \mathcal A[u_S](r)
  \ge c_A\frac{r}{(r^2+a^2)^2},
  \qquad r\ge M.
\end{equation}

The same choice of \(w_S\) gives positivity of the potential for
\(r\ge r_*\).  Indeed,
\begin{equation}\label{eq:ddw-uS}
  \partial_r\big((r-M)^2\partial_r w_S\big)
  =
  \begin{cases}
    0, & r\le r_*,\\[0.4em]
    \displaystyle
    -\frac{6(r-M)^2}{(r^2+a^2)^4}
    \Big(Mr^4-2(M^2-a^2)r^3-6a^2Mr^2
    +2a^2(M^2-a^2)r+a^4M\Big),
       & r>r_*.
  \end{cases}
\end{equation}
It remains to verify that the quartic in parentheses is positive on
\(r\ge r_*\).  Set \(\alpha=(a/M)^2\in[0,1]\) and \(x=r/M\).  The quartic
is \(M^5p(x)\), where
\begin{equation*}
  p(x)=x^4-2(1-\alpha)x^3-6\alpha x^2
  +2\alpha(1-\alpha)x+\alpha^2.
\end{equation*}
For \(x\ge3\),
\begin{align*}
  p(3)&=27+6\alpha-5\alpha^2\ge22,\\
  p'(3)&=54+20\alpha-2\alpha^2\ge52,\\
  p''(3)&=72+24\alpha\ge72,\\
  p'''(3)&=60+12\alpha\ge60,
\end{align*}
and \(p''''=24>0\).  Hence \(p(x)>0\) for all \(x\ge3\).  Since
\(r_*\ge3M\), \zcref{eq:V-in-terms-of-w-axial, eq:ddw-uS} imply that, for some \(c_V=c_V(M,a)>0\),
\begin{equation}\label{eq:V-uS-lower}
  \mathcal V[u_S](r)
  \ge 1_{\{r\ge r_*\}}\frac{c_V}{r^2} .
\end{equation}

Integrating the relation $u_S'=\frac{(r^2+M^2)^2}{(r-M)^2}2w_S$ from \zcref{eq:def-u} and using \zcref{eq:def-w} we deduce 
\begin{align}\label{eq:def-u-explicit}
  u_S=\begin{cases}
    -\frac{\widetilde C(M^2+a^2)^2}{r-M}+O(|\log(r-M)|)
    \qquad\text{as }r\downarrow M \\
    r^2+C\ \qquad \qquad r\geq r_*,
  \end{cases}
\end{align}
for some suitable constant $C$ such that $u$ is continuous at $r_{*}$. With this choice, we also deduce that
$\frac{u_{S} \TT}{(r^2+a^2)^3} \ges \frac{(r-M)}{r^2}\big( 1-\frac{r_{trap}}{r}\big)^2$, which is \zcref{eq:condition-one-u-daxi}.
Notice that \(u_S\) is singular at the horizon.

\paragraph*{Step 2: Regularization near the event horizon}

Fix \(\daxi>0\).  Let \(F:\mathbb R\to\mathbb R\) be a fixed smooth
nondecreasing function such that
\begin{equation*}
  F(x)=x \quad\text{for }x\le1,
  \qquad
  F(x)=2 \quad\text{for }x\ge3.
\end{equation*}
Define
\begin{equation*}
  u_{\daxi}
  \vcentcolon=
  -\frac{M^2}{\daxi}
  F\left(-\frac{\daxi}{M^2}u_S\right).
\end{equation*}
Then
\begin{equation*}
  u_{\daxi}=u_S
  \quad\text{where }u_S\ge-\frac{M^2}{\daxi},
  \qquad
  u_{\daxi}=-\frac{2M^2}{\daxi}
  \quad\text{where }u_S\le-\frac{3M^2}{\daxi}.
\end{equation*}
Since \(u_S\) is increasing, the transition radii
\begin{equation*}
  r_{1,\daxi}
  \vcentcolon=u_S^{-1}\left(-\frac{3M^2}{\daxi}\right),
  \qquad
  r_{2,\daxi}
  \vcentcolon=u_S^{-1}\left(-\frac{M^2}{\daxi}\right)
\end{equation*}
are well defined for \(\daxi\) sufficiently small.  By
\zcref{eq:def-u-explicit},
\begin{equation}\label{eq:rminusM-scale}
  r_{1,\daxi}-M\simeq_{M,a}\daxi,
  \qquad
  r_{2,\daxi}-M\simeq_{M,a}\daxi,
  \qquad
  r_{2,\daxi}-r_{1,\daxi}\simeq_{M,a}\daxi.
\end{equation}
As before, define
\begin{equation*}
  2 w_{\daxi}
  \vcentcolon=
  \frac{(r-M)^2}{(r^2+a^2)^2}\partial_r u_{\daxi}.
\end{equation*}
We estimate \(\mathcal A[u_{\daxi}]\) and \(\mathcal V[u_{\daxi}]\) in the
three natural regions.

First, on \(r\le r_{1,\daxi}\), the function \(u_{\daxi}\) is constant and
\(w_{\daxi}=0\).  Hence
\begin{equation*}
  \mathcal A[u_{\daxi}]
  =\partial_r\left(\frac{u_{\daxi}}{r^2+a^2}\right)
  =\frac{4M^2r}{\daxi(r^2+a^2)^2},
  \qquad
  \mathcal V[u_{\daxi}]=0.
\end{equation*}

Second, on \(r_{1,\daxi}\le r\le r_{2,\daxi}\), the cut-off is active.  We
have
\begin{equation*}
  \partial_r u_{\daxi}
  =\partial_r u_S
  F'\left(-\frac{\daxi}{M^2}u_S\right),
  \qquad
  2 w_{\daxi}
  =\widetilde C
  F'\left(-\frac{\daxi}{M^2}u_S\right).
\end{equation*}
Differentiating twice, using \(\partial_r u_S\lesssim_{M,a}(r-M)^{-2}\),
and using that \(F''\) and \(F'''\) are supported in the transition
region, gives
\begin{equation*}
  \left|
    \partial_r\big((r-M)^2\partial_r w_{\daxi}\big)
  \right|
  \lesssim_{M,a}
  \daxi\left(1+\frac{\daxi}{(r-M)^2}\right)
  1_{\{r_{1,\daxi}\le r\le r_{2,\daxi}\}}.
\end{equation*}
Combining this with \zcref{eq:rminusM-scale}, and defining
\begin{equation*}
  \overline{\mathcal V}_{\daxi}
  \vcentcolon=
  \frac14\left|
    \partial_r\big((r-M)^2\partial_r w_{\daxi}\big)
  \right|,
\end{equation*}
we obtain
\begin{equation}\label{eq:Vbar-size}
  0\le\overline{\mathcal V}_{\daxi}\lesssim_{M,a}1,
  \qquad
  \operatorname{supp}\overline{\mathcal V}_{\daxi}
  \subset [r_{1,\daxi},r_{2,\daxi}],
  \qquad
  r_{2,\daxi}-r_{1,\daxi}\simeq_{M,a}\daxi.
\end{equation}
For \(\mathcal A[u_{\daxi}]\), we compute
\begin{align*}
  \partial_r\left(\frac{u_{\daxi}}{r^2+a^2}\right)
  &=-\frac{2r}{(r^2+a^2)^2}u_{\daxi}
    +\frac{1}{r^2+a^2}
    F'\left(-\frac{\daxi}{M^2}u_S\right)\partial_r u_S.
\end{align*}
Since \(F'\ge0\), \(\partial_r u_S>0\), and
\(u_{\daxi}\ge-2M^2/\daxi\), we have
\begin{equation*}
  \mathcal A[u_{\daxi}]
  \ge
  \frac{2M^2r}{\daxi(r^2+a^2)^2}
  \qquad
  \text{on }r_{1,\daxi}\le r\le r_{2,\daxi}.
\end{equation*}

Third, on \(r\ge r_{2,\daxi}\), we have \(u_{\daxi}=u_S\) and
\(w_{\daxi}=w_S\).  Therefore \zcref{eq:bounds-AA, eq:V-uS-lower} apply.

Combining the three regions, and taking \(\daxi\) sufficiently small if
necessary, there exists \(c_A=c_A(M,a)>0\) such that
\begin{equation}\label{eq:AA-udelta-global}
  \mathcal A[u_{\daxi}](r)
  \ge c_A\frac{r}{(r^2+a^2)^2},
  \qquad r\ge M,
\end{equation}
which is \eqref{eq:condition-A-u-daxi}.
Moreover,
\begin{equation}\label{eq:bound-VV-ude}
  \mathcal V[u_{\daxi}]
  \ge
  \frac{c_V}{r^2} 1_{\{r\ge r_*\}}
  -\overline{\mathcal V}_{\daxi}
  1_{\{r_{1,\daxi}\le r\le r_{2,\daxi}\}},
\end{equation}
where \(\overline{\mathcal V}_{\daxi}\) satisfies
\zcref{eq:Vbar-size}.

It remains to check that the lower bound in Condition 1 is preserved by the
regularization.  For \(r\ge r_{2,\daxi}\), this follows from
\(u_\daxi=u_S\).  For \(M\le r\le r_{2,\daxi}\), choosing
\(\daxi_0\) sufficiently small ensures \(r_{2,\daxi}<r_{\mathrm{trap}}\).
Hence \(u_\daxi\le0\) and \(\TT\le0\) on this interval.  Since
\(|u_\daxi|\gtrsim M^2/\daxi\) and \(-\TT\sim_{M,a} r-M\) near
\(r=M\), the desired bound follows, after decreasing \(\daxi\) if
necessary.

\begin{figure}[t]
  \begin{tikzpicture}[scale=0.85, transform shape, x=1.2cm,y=0.9cm,>=Stealth]

    \pgfmathsetmacro{\rH}{0.0}      
    \pgfmathsetmacro{\rone}{1.3}    
    \pgfmathsetmacro{\rtwo}{2.5}    
    \pgfmathsetmacro{\rtrap}{4.5}   
    \pgfmathsetmacro{\rstar}{6}   
    \pgfmathsetmacro{\ydelta}{-2.4} 

    \pgfmathdeclarefunction{uu}{1}{%
      \pgfmathparse{ln(#1 - \rH + 0.09) + 0.2*(#1-1)^3 - 3.1}%
    }

    \pgfmathdeclarefunction{sigma}{1}{%
      \pgfmathparse{3*#1*#1 - 2*#1*#1*#1}%
    }

    \pgfmathdeclarefunction{ud}{1}{%
      \pgfmathparse{
        (#1 <= \rone) *
        (\ydelta)
        +
        (#1 > \rone && #1 < \rtwo) *
        ( \ydelta + sigma((#1-\rone)/(\rtwo-\rone)) * (uu(#1)-\ydelta) )
        +
        (#1 >= \rtwo) *
        ( uu(#1) )
      }%
    }

    \draw[->] (-1.5,0) -- (9,0) node[below] {$r$};
    \draw[->] (-1,-5.1) -- (-1,5.1) node[left] {$u,\ u_\daxi$};

    \draw[blue!70, very thick] (0,-1.2) -- (8.6,-1.2);
    \node[blue!70,anchor=west] at (8.1,-1.5) {\tiny $-\dfrac{M^2}{\daxi}$};

    \draw[blue!70, very thick] (0,-3.4) -- (8.6,-3.4);
    \node[blue!70,anchor=west] at (8.1,-3.0) {\tiny $-\dfrac{3M^2}{\daxi}$};

    \foreach \x/\lab in {%
      -1/0,
      \rH/$r_{\Horizon}$,
      \rone/$r_{1, \daxi}$,
      \rtwo/$r_{2, \daxi}$,
      \rtrap/$r_{\mathrm{trap}}$,
      \rstar/$r_*$}
    {
      \draw[dashed,blue!45] (\x,-5) -- (\x,5);
      \draw (\x,0) -- (\x,0.08);
      \node[below] at (\x,-0.05) {\lab};
    }

    \draw[red, very thick, smooth, samples=250,
    domain=\rH+0.03:4.5]
    plot(\x,{uu(\x)});

    \node[red,anchor=west] at (4,3.6)
    {\small $u\sim\frac{1}{r-r_{\Horizon}}$ near $r_{\Horizon}$, $u\sim r^2+C_2$ for large $r$};

    \draw[green!50!black, very thick, smooth, samples=300,
    domain=\rH:4.5]
    plot(\x,{ud(\x)});

    \node[green!50!black,anchor=south]
    at ({0.5*(\rH+\rone)},\ydelta+0.25)
    {\tiny $u_\daxi=-\dfrac{2M^2}{\daxi}$};

    \node[green!50!black,anchor=west]
    at (5.4,2.0) {\small $u_\daxi = u\ \text{for } r \geq r_{2, \daxi}$};

    \node[blue!70,anchor=west] at (6.1,-0.6)
    {\tiny $-\dfrac{\daxi}{M^2}u \leq 1$};
    \draw[->,blue!70] (6.1,-0.6) -- (5.4,0.7);

    \node[blue!70,anchor=west] at (6.1,-3.8)
    {\tiny $-\dfrac{\daxi}{M^2}u  \geq 3$};
    \draw[->,blue!70] (6.1,-3.8) -- (5.5,-4.5);

    \node[black!70,anchor=west] at (1.2,0.6) {\small transition};

  \end{tikzpicture}
  \centering
\end{figure}

\paragraph*{Step 3: Hardy correction.}

We now add a lower-order current in order to recover positivity of the
potential in \(M\le r\le r_*\).  Since \zcref{eq:bound-VV-ude} already
provides positivity for \(r\ge r_*\), we choose \(v\) supported in
\(\{r\le r_*+M/2\}\).

Consider \zcref{eq:I_uv-axisym} with \(u=u_{\daxi}\) and initially
\(\varepsilon=0\).  Using \zcref{eq:bound-VV-ude} and completing the square,
we obtain
\begin{align}\label{eq:Iuv-after-square}
  I_{u_{\daxi},v}[\psi]
  &\ge
    |q|^2
    \left(
    \partial_r v+\frac{2r}{|q|^2}v
    -\frac{|q|^2v^2}{4\frac{(r-M)^4}{r^2+a^2}\mathcal A[u_{\daxi}]}
    \right)|\psi|^2                                   
    +\left(
    \frac{c_V}{r^2}1_{\{r\ge r_*\}}
    -\overline{\mathcal V}_{\daxi}
    1_{\{r_{1,\daxi}\le r\le r_{2,\daxi}\}}
    \right)|\psi|^2 .
\end{align}
On any region where \(v\ge0\) and
\(\partial_r((r^2+a^2)v)\ge0\), we have
\begin{equation}\label{eq:v-derivative-lower}
  |q|^2\left(\partial_r v+\frac{2r}{|q|^2}v\right)
  \ge
  \frac12\partial_r\big((r^2+a^2)v\big),
\end{equation}
using \(|q|^2\ge r^2\) and \(r^2+a^2\le2r^2\).

By \zcref{eq:AA-udelta-global}, on the compact interval
\([M,r_*+M/2]\) there exists \(C=C(M,a)>0\) such that
\begin{equation*}
  \frac{1}{r^2+a^2}\mathcal A[u_{\daxi}]
  \ge C.
\end{equation*}
Thus \zcref{eq:Iuv-after-square, eq:v-derivative-lower} reduce
the problem, on \([M,r_*]\), to choosing \(v\) so that
\begin{equation*}
  \frac12\partial_r\left((r^2+a^2)\frac{v}{C}\right)
  -\frac{|q|^4}{4(r-M)^4}\left(\frac{v}{C}\right)^2
\end{equation*}
is positive.  We set
\begin{equation*}
  (r^2+a^2)\frac{v}{C}
  =\delta_1(r-M)^3
  \qquad\text{for }M\le r\le r_*,
\end{equation*}
where \(\delta_1>0\) will be chosen sufficiently small.  Then
\begin{align*}
  &\frac12\partial_r\left((r^2+a^2)\frac{v}{C}\right)
    -\frac{|q|^4}{4(r-M)^4}\left(\frac{v}{C}\right)^2  
  \\
  &\qquad=
    \frac14\delta_1(r-M)^2
    \left(6-\frac{|q|^4}{(r^2+a^2)^2}\delta_1\right)
    \ge
    \frac54\delta_1(r-M)^2,
\end{align*}
provided \(\delta_1\) is chosen small enough.  We extend \(v\) smoothly to
zero on \([r_*,r_*+M/2]\), with
\begin{equation*}
  |v|+|\partial_rv|\lesssim_{M,a}\delta_1,
  \qquad
  v=0\quad\text{for }r\ge r_*+M/2.
\end{equation*}
The cut-off region produces an error of size \(O(\delta_1)\), which is
absorbed by the positive term \(c_Vr^{-2}1_{\{r\ge r_*\}}\) in
\zcref{eq:Iuv-after-square}, after choosing \(\delta_1\) sufficiently
small.  Consequently, for some \(c=c(M,a)>0\) independent of
\(\daxi\),
\begin{equation*}
  I_{u_{\daxi},v}[\psi]
  \ge
  \left(
    c\frac{(r-M)^2}{r^4}
    -\overline{\mathcal V}_{\daxi}
    1_{\{r_{1,\daxi}\le r\le r_{2,\daxi}\}}
  \right)|\psi|^2,
\end{equation*}
which is \zcref{eq:I-u-daxi-v} for $\varepsilon=0$.
Finally, the same conclusion holds after splitting
\begin{equation*}
  \mathcal A[u_{\daxi}]
  =(1-2\varepsilon)\mathcal A[u_{\daxi}]
  +2\varepsilon\mathcal A[u_{\daxi}],
\end{equation*}
and applying the above argument only to the first term, provided
\(\varepsilon>0\) is sufficiently small.  Defining
\begin{equation*}
  \rho_{1,\daxi}\vcentcolon=r_{1,\daxi}-M,
  \qquad
  \rho_{2,\daxi}\vcentcolon=r_{2,\daxi}-M,
\end{equation*}
we obtain the asserted statement of \zcref{prop:ax-symm-choice}.

\printbibliography

\end{document}